\documentclass[openany]{amsbook}
\usepackage{amsmath}%,makeidx}

\usepackage[psamsfonts]{amssymb}
\usepackage{hyperref}
\usepackage[all]{xy}

\numberwithin{equation}{section}
\newtheorem{theorem}{Theorem}[section]
\newtheorem{corollary}[theorem]{Corollary}
\newtheorem{proposition}[theorem]{Proposition}
\newtheorem{lemma}[theorem]{Lemma}
\newtheorem{claim}[theorem]{Claim}

\theoremstyle{definition}
\newtheorem{example}[theorem]{Example}
\newtheorem{definition}[theorem]{Definition}

\newtheorem{problem}[theorem]{Problem}
\newtheorem{remark}[theorem]{Remark}
\theoremstyle{plain}

\newcommand{\w}{\omega}
\newcommand{\K}{\mathcal K}
\newcommand{\A}{\mathcal A}
\newcommand{\Bas}{\mathcal B}
\newcommand{\C}{\mathcal C}
\newcommand{\N}{\mathcal N}
\newcommand{\Ra}{\Rightarrow}
\newcommand{\IZ}{\mathbb Z}
\newcommand{\IR}{\mathbb R}
\newcommand{\IN}{\mathbb N}

\newcommand{\IQ}{\mathbb Q}

\newcommand{\V}{\mathcal V}

\newcommand{\U}{\mathcal U}
\newcommand{\F}{\mathcal F}
\newcommand{\E}{\mathcal E}

\newcommand{\e}{\varepsilon}

\newcommand{\W}{\mathcal{W}}

\newcommand{\cs}{\mathsf{cs}}
\newcommand{\ccs}{\overline{\mathsf{cs}}}
\newcommand{\cccs}{\ccs}
\newcommand{\css}{\mathsf{s}}
\newcommand{\as}{\mathsf{p}}

\newcommand{\St}{\mathcal{S}t}
\newcommand{\supp}{\mathrm{supp}}

\newcommand{\Fat}{\mathsf{Fat}}

\newcommand{\add}{\mathsf{add}}
\newcommand{\cof}{\mathsf{cof}}
\newcommand{\cov}{\mathsf{cov}}
\newcommand{\cf}{\mathsf{cof}}
\newcommand{\PM}{\mathcal{P\!M}}

\newcommand{\pr}{\mathrm{pr}}

\newcommand{\Tau}{\mathcal{T}}

\newcommand{\diam}{\mathrm{diam}}

\newcommand{\id}{\mathrm{id}}
\makeindex


\title{Topological spaces with an $\w^\w$-base}
\author{Taras Banakh}
\address{Jan Kochanowski University in Kielce and
Ivan Franko National University in Lviv}
\email{t.o.banakh@gmail.com}

\begin{document}

\maketitle
\begin{abstract} Given a partially ordered set $P$ we study properties of topological spaces $X$ admitting a $P$-base, i.e., an indexed family $(U_\alpha)_{\alpha\in P}$ of subsets of $X\times X$ such that $U_\beta\subset U_\alpha$ for all $\alpha\le\beta$ in $P$ and for every $x\in X$ the family $\{U_\alpha[x]\}_{\alpha\in P}$ of balls $U_\alpha[x]=\{y\in X:(x,y)\in U_\alpha\}$ is a neighborhood base at $x$. A $P$-base $(U_\alpha)_{\alpha\in P}$ for $X$ is called {\em locally uniform} if the family of entourages $(U_\alpha U_\alpha^{-1}U_\alpha)_{\alpha\in P}$ remains a $P$-base for $X$. A topological space is first-countable if and only if it has an $\w$-base. By Moore's Metrization Theorem, a $T_0$-space is metrizable if and only if it has  a locally uniform $\w$-base.

In the paper we shall study topological spaces possessing a (locally uniform) $\w^\w$-base. Our results show that  spaces with an $\w^\w$-base share some common properties with first countable spaces, in particular, many known upper bounds on the cardinality of first-countable spaces remain true for countably tight $\w^\w$-based topological spaces. On the other hand, topological spaces with a locally uniform $\w^\w$-base have many properties, typical for generalized metric spaces. Also we study Tychonoff spaces whose universal (pre- or quasi-) uniformity has an $\w^\w$-base and show that such spaces are close to being $\sigma$-compact.
\end{abstract}
%\makeabstract

\tableofcontents

\chapter*{Introduction}

This paper-book is devoted to studying topological spaces whose topological structure is Tukey dominated by the poset $\w^\w$. By a \index{poset}{\em poset} we understand a non-empty set endowed with a partial preorder (i.e., a reflexive transitive binary relation). %The set $\w^\w$ of all self-maps of the smallest infinite cardinal $\w$ is endowed with the partial order $\le$ defined by $f(n)\le g(n)$ for all $n\in\w$.

We say that a poset $Q$ is \index{poset!Tukey dominated by}{\em Tukey dominated by} a poset $Q$ and write $Q\le_T P$  if there exists a function $f:P\to Q$ that maps every cofinal subset of $P$ to a cofinal subset of $Q$. This is equivalent to saying that the preimage $f^{-1}(B)$ of every bounded set $B\subset Q$ is bounded in $P$.
We recall that a subset $A$ of a poset $(P,\le)$ is
\begin{itemize}
\item \index{subset of a poset!cofinal}{\em cofinal} in $P$ if  $\forall x\in P\;\exists a\in A\;\;(x\le a)$;
\item \index{subset of a poset!bounded}{\em bounded} in $P$ if $\exists x\in P\;\forall a\in A\;\;(a\le x)$.
\end{itemize}

We say that two posets $P,Q$ are \index{posets!Tukey equivalent}{\em Tukey equivalent} and denote this by $P\equiv_T Q$ if $P\le_T Q$ and $Q\le_T P$. It is known \cite{Tuk}, \cite{Day} that two posets $P,Q$ are Tukey equivalent if and only if they are order isomorphic to cofinal subsets of some poset $Z$. For lower complete posets the Tukey domination can be equivalently defined using monotone cofinal maps (instead of maps that preserve cofinal subsets).

We say that a function $f:P\to Q$ between posets $P,Q$ is
\begin{itemize}
\item \index{function between posets!monotone}{\em monotone} if for any elements $x\le y$ in $P$ we get $f(x)\le f(y)$ is $Q$;
\item \index{function between posets!cofinal}{\em cofinal} if the image $f(P)$ is a cofinal subset of $Q$.
\end{itemize}
A poset $P$ is called \index{poset!lower complete}{\em lower complete} if each non-empty subset $A\subset P$ has the greatest lower bound. In Lemma~\ref{l:Tuckey} we shall observe that a (lower complete) poset $Q$ is Tukey dominated by a poset $P$ if (and only if) there exists a monotone cofinal function $f:P\to Q$. In the latter case we shall write $P\succcurlyeq Q$ and say that the poset $Q$ is \index{poset!$P$-dominated}{\em $P$-dominated}.

The study of the Tukey preorder $\le_T$ between posets is an active area of research in Set-Theoretic Topology, see \cite{DT}, \cite{Frem}, \cite{GM16}, \cite{H11}, \cite{LT}, \cite{LV99}, \cite{Mama}, \cite{Sol15}, \cite{ST}, \cite{Tod85}. In particular, the papers \cite{LV99}, \cite{Sol15}, \cite{ST} are devoted to the classification of definable posets up to the Tukey equivalence. An important role in this classification belongs to the posets $\w$ and $\w^\w$. Here $\w$ is the smallest infinite cardinal endowed with its natural well-order. The countable power $\w^\w$ of $\w$ is endowed with the coordinatewise partial order $\le$ defined by $f(n)\le g(n)$ for all $n\in\w$.

Observe that a poset $P$ is Tukey dominated by $\w$ if and only if $P$ is a directed poset of countable cofinality. A poset $(P,\le)$ is \index{poset!directed}{\em directed} if for any points $x,y\in P$ there is a point $z\in P$ such that $x\le z$ and $y\le z$.

%In a sense, the poset $\w^\w$ is the simplest poset, which is not Tukey dominated by $\w$. More precisely, any poset $P$, which is Tukey dominated by $\w^\w$ is Tukey equivalent to $1$, $\w$ or $\w^\w$, see \cite{??}.

Many properties of a topological space depend on the cofinality properties of some posets determined by the topological structure of the space. In particular, for each point $x\in X$ of a topological space $X$ the family $\Tau_x(X)$ of all neighborhoods of $x$ can be considered as a (directed lower complete) poset endowed with the partial order $\le$ of reverse inclusion ($U\le V$ iff $V\subset U$). Observe that the poset $\Tau_x(X)$ is Tukey dominated by $\w$ if and only if $X$ is first-countable at $x$ (i.e., $X$ has a countable neighborhood base at $x$).

Given a poset $P$, we shall say that a topological space $X$ has a \index{$P$-base!neighborhood}{\em neighborhood $P$-base} at a point $x\in X$ if the poset $\Tau_x(X)$ is Tukey dominated by $P$. Since the poset $\Tau_x(X)$ is lower complete this happens if and only if at each point $x\in X$ the space $X$ has a neighborhood base $(U_\alpha[x])_{\alpha\in P}$ such that $U_\beta[x]\subset U_\alpha[x]$ for all $\alpha\le\beta$ in $P$ (see Lemma~\ref{l:Tuckey}). If a space $X$ has a neighborhood $P$-base $(U_\alpha[x])_{\alpha\in P}$ at each point $x\in X$, then all these neighborhood bases can be encoded by the family $(U_\alpha)_{\alpha\in P}$ of the entourages $U_\alpha=\{(x,y)\in X\times X:y\in U_\alpha[x]\}$, $\alpha\in P$. Such family $(U_\alpha)_{\alpha\in P}$ will be called a \index{$P$-base}{\em $P$-base} for $X$, and the topological space $X$ endowed with a $P$-base $(U_\alpha)_{\alpha\in P}$ will be called a \index{topological space!$P$-based}{\em $P$-based topological space}.

The notion of a $P$-base allows us to impose some relations between the balls $U_\alpha[x]$ for various centers $x\in X$. In particular, a $P$-base $(U_\alpha)_{\alpha\in P}$ for a topological space $X$ will be called
\begin{itemize}
\item \index{$P$-base!quasi-uniform}{\em quasi-uniform} if $\forall\alpha\in P\;\exists\beta\in P\;\;(U_\beta U_\beta\subset U_\alpha)$;
\item \index{$P$-base!uniform}{\em uniform} if $\forall\alpha\in P\;\exists\beta\in P\;\;(U_\beta U_\beta^{-1}U_\beta\subset U_\alpha)$;
\item \index{$P$-base!locally quasi-uniform}{\em locally quasi-uniform} if $\forall x\in X\;\forall\alpha\in P\;\exists\beta\in P\;\;(U_\beta U_\beta[x]\subset U_\alpha[x])$;
\item \index{$P$-base!locally uniform}{\em locally uniform} if $\forall x\in X\;\forall\alpha\in P\;\exists\beta\in P\;\;(U_\beta U_\beta^{-1}U_\beta[x]\subset U_\alpha[x])$.
\end{itemize}
For each $P$-base we get the following implications:
$$
\xymatrix{
\mbox{uniform}\ar@{=>}[r]\ar@{=>}[d]&\mbox{locally uniform}\ar@{=>}[d]\\
\mbox{quasi-uniform}\ar@{=>}[r]&\mbox{locally quasi-uniform.}
}
$$

So, a topological space $X$ has a neighborhood $P$-base at each point if and only if it has a $P$-base. Topological spaces with an $\w$-base are precisely first-countable spaces. The Moore Metrization Theorem~\cite[5.4.2]{Eng} implies that a topological space is metrizable if and only if it is a $T_0$-space with a locally uniform $\w$-base.

In this paper we shall systematically study topological spaces with a (locally [quasi-]uniform) $\w^\w$-base.
In fact, spaces with a (uniform) $\w^\w$-base are known in Functional Analysis since 2003 when Cascales, K\c akol, and Saxon \cite{CKS} characterized quasi-barreled  locally convex spaces with an $\w^\w$-base. In the papers \cite{feka}, \cite{GabKak_2}, \cite{GK_L(X)}, \cite{GabKakLei_2}, \cite{LRZ}  spaces with an $\w^\w$-base are called spaces with a $\mathfrak G$-base but we prefer to use the more self-suggesting and flexible terminology of $\w^\w$-bases. In \cite{CO} Cascales and Orihuela proved that compact subsets of locally convex spaces with a $\mathfrak G$-base are metrizable. This important fact was deduced from a more general theorem on the metrizability of compact Hausdorff spaces whose uniformity has an $\w^\w$-base.

The uniformity $\U_X$ of any uniform space $X$  can (and will) be considered as a (directed lower complete) poset endowed with the partial order $\le$ of reverse inclusion ($U\le V$ iff $V\subset U$).  Given a poset $P$ we shall say that a uniform space $X$ has a \index{uniform space!$P$-base of}{\em $P$-base} if its uniformity $\U_X$ is Tukey dominated by $P$. Since the poset $\U_X$ is lower complete this happens if and only if the uniformity $\U_X$ has a base $\{U_\alpha\}_{\alpha\in P}$ such that $U_\beta\subset U_\alpha$ for all $\alpha\le\beta$ in $P$. Such base $\{U_\alpha\}_{\alpha\in P}$ will be called a \index{$P$-base}{\em $P$-base} of the uniformity $\U_X$. Each $P$-base of a uniformity is necessarily uniform, so the study of uniform spaces with a $P$-base reduces to studying topological spaces with a uniform $P$-base, which is one of the main topics considered in this paper.

In spite of the fact that the index set $P=\w^\w$ is uncountable, it carries a natural Polish topology with a countable base indexed by the set $\w^{<\w}=\bigcup_{n\in\w}\w^n$. Namely, each finite sequence $\beta\in\w^n\subset\w^{<\w}$ determines a basic clopen set ${\uparrow}\beta=\{\alpha\in\w^\w:\alpha|n=\beta\}$. Given any $\w^\w$-base $(U_\alpha)_{\alpha\in \w^\w}$ for a topological space $X$ we can define a countable family $(U_\beta)_{\beta\in\w^{<\w}}$ of sets $U_\beta=\bigcap_{\alpha\in{\uparrow}\beta}U_\alpha$. One of the most important results of this paper is Theorem~\ref{t:lP*} saying that for every point $x\in X$ the family $\mathcal N_x=\{U_{\beta}[x]\}_{\beta\in\w^{<\w}}$ is a countable $\css^*$-network at $x$ (which means that for any neighborhood $O_x\subset X$ of $x$ and any sequence $(x_n)_{n\in\w}$ accumulating at $x$ there is a set $N\in\mathcal N_x$ such that $x\in N\subset O_x$ and $N$ contains infinitely many points $x_n$, $n\in\w$). Many nice properties of topological spaces with an $\w^\w$-base can be derived from properties of the countable family $(U_\beta)_{\beta\in\w^{<\w}}$ or, more precisely, from properties of the pair $\big(\{U_\beta\}_{\beta\in\w^{<\w}},\{U_\alpha\}_{\alpha\in\w^\w}\big)$. Such pairs will be studied in the framework of the theory of netbases developed in Chapter~\ref{ch:netbase}.
\smallskip

Now we briefly describe the structure of this paper-book. In Chapter~\ref{s:prelim} we collect some preliminary information that will be used in the next chapters. In Chapter~\ref{ch:poset} we study the (Tukey) reductions between posets in more details. In Section~\ref{s:Set} we detect $\w^\w$-dominated cardinals and $\w^\w$-dominated powers of cardinals. In particular, we reprove known results on the $\w^\w$-dominacy of the cardinals $\mathfrak b$, $\mathfrak d$, $\cf(\mathfrak d)$ and establish that the existence of an uncountable regular cardinal $\kappa$ with $\w^\w$-dominated power $\kappa^\kappa$ is independent of ZFC.

In Chapter~\ref{Ch:pu} we present the necessary information on preuniform spaces. In particular, we study the canonical (quasi-)uniformity of a preuniform space, define locally (quasi-)uniform preuniformities, universal, complete preuniformities, and also study uniformly continuous and $\kappa$-continuous maps between preuniform spaces.

In Chapter~\ref{ch:portator} we introduce a new notion of a baseportator, which is a far generalization of the notion of a left-topological group. A {\em baseportator} is a pair $(X,(t_x)_{x\in X})$ consisting of a topological space $X$ with a distinguished point $e$ called the {\em unit} of $X$ and a family $(t_x)_{x\in X}$ of monotone cofinal maps $t_x:\Tau_e(X)\to\Tau_e(X)$ transforming neighborhoods of the unit into neighborhoods of a given point $x\in X$. Each baseportator carries a canonical preuniformity $\vec\Tau_X$ generated by the base $\{\vec V:V\in\Tau_e(X)\}$ consisting of the entourages $\vec V=\{(x,y)\in X\times X:y\in t_x(V)\}$. Some additional conditions on the transport structure $(t_x)_{x\in X}$ (like the local uniformity) will help us to convert local properties (like the first countability at $e$) into global properties (like the metrizability) of the whole space. A baseportator $X$ is called a {\em portator} if its transport function is given by a set-valued binary operation $\mathbf{xy}:X\times X\multimap X$, $\mathbf{xy}:(x,y)\mapsto xy\subset X$, such that $t_x(V)=\bigcup_{y\in V}xy$ for all $x\in X$ and $V\in\Tau_e(X)$. Many well-known structures of topological algebra (topological groups, topological loops, rectifiable spaces) are natural examples of portators.

In Chapter~\ref{ch:netbase} we introduce a new notion of a netbase and establish many specific properties of topological spaces admitting countable netbases with some additional properties (like local uniformity). For a family $\C$ of subsets of a topological space $X$ a {\em $\C^*$-netbase} is a pair $(\N,\Bas)$ consisting of a family $\N$ of entourages on $X$ and an entourage base $\Bas$ for $X$ such that for any entourage $B\in\Bas$ and point $x\in X$ the family $\N_B[x]:=\{N[x]:N\in\N,\;\;N\subset B\}$ is a $\C^*$-network at $x$ (which means that for any set $C\in\C$ accumulating at $x$ there is a set $N\in\N_B[x]$ such that $N\cap C$ is infinite). A netbase $(\N,\Bas)$ is {countable} if the family $\N$ is countable; $(\N,\Bas)$ is uniform (resp. locally uniform, locally quasi-uniform) if so is the entourage base $\Bas$. The most interesting choices for the family $\C$ are: the family $\cs$ of all convergent sequences in $X$, the family $\ccs$ of all countable sets with countably compact closure in $X$, the family $\css$ of all sequences in $X$, and the family $\as$ of all subsets of $X$. It is clear that $\cs\subset\ccs\subset\css\subset\as$. Our interest to netbases is motivated and warmed up by Theorem~\ref{t:ww=>netbase} saying that for any $\w^\w$-base $(U_\alpha)_{\alpha\in\w^\w}$ for a topological space $X$ the pair $\big(\{U_{\beta}\}_{\beta\in\w^{<\w}},\{U_\alpha\}_{\alpha\in\w^\w}\big)$ is an $\css^*$-netbase for $X$. The most important results of Chapter~\ref{ch:netbase} are:
\begin{itemize}
\item the first-countability of $q$-spaces with a countable locally uniform $\ccs^*$-netbase;
\item the metrizability of first-countable closed-$\bar G_\delta$ $T_0$-spaces with a countable locally uniform $\cs^*$-netbase;
\item the metrizability of $M$-spaces with a countable locally uniform $\ccs^*$-netbase;
\item the $\sigma$-space property of $\Sigma$-spaces with a countable locally uniform $\ccs^*$-netbase;
\item the construction of a $\sigma$-discrete $\C^*$-network in strong $\sigma$-space with a locally quasi-uniform $\C^*$-netbase.
\end{itemize}
In the final section of Chapter~\ref{ch:netbase} we apply $\as^*$-netbases to generalize some known upper bounds on the cardinality of a topological space $X$ by replacing the character $\chi(X)$ of $X$ in these bounds by the $\as^*$-character $\chi_{\as^*}(X)$ of $X$ (defined as the smallest cardinality $|\N|$ of a $\as^*$-netbase $(\N,\Bas)$ for the space $X$). In particular we improve the famous Arhangelskii's upper bound $|X|\le 2^{\chi(X)\cdot L(X)}$ to the upper bound $|X|\le 2^{\chi_{\as^*}(X)\cdot L(X)}$, where $L(X)$ is the Lindel\"of number of a Hausdorff space $X$.

In Chapter~\ref{Ch:ww-base} $\css^*$-netbases are applied to establishing generalized metric properties of topological spaces with a locally (quasi-) uniform $\w^\w$-bases. Applying the obtained results to portators we conclude that many local topological properties of portators lift the their global topological properties. Under some additional set-theoretic assumptions (like $\w_1<\mathfrak b$ or PFA) we shall also detect certain properties of spaces with a locally uniform $\w^\w$-base $(U_\alpha)_{\alpha\in\w^\w}$, which depend on some specific set-theoretic properties of the poset $\w^\w$ and cannot be derived from the properties of the associated netbase $(\{U_\beta\}_{\beta\in\w^{<\w}},\{U_\alpha\}_{\alpha\in\w^\w}\big)$.
In Sections~\ref{s:ub}~--~\ref{s:ww-small} of Chapter~\ref{Ch:ww-base} we study spaces with a uniform $\w^\w$-base and show that for such spaces many topological countability conditions (like separability, cosmicity, being a Lindel\"of $\Sigma$-space, etc.) are equivalent. The crucial role in establishing these equivalences belongs to the spaces $C_u(X)$ of uniformly continuous real-valued functions on spaces with a uniform $\w^\w$-base.

In  Chapter~\ref{Ch:univer} we study topological spaces with a universal $\w^\w$-base, i.e., $\w^\w$-based topological spaces on which every continuous function $f:X\to Y$ to a metric space $Y$ is uniformly continuous.  The results of Chapter~\ref{Ch:univer} show that topological spaces with a universal $\w^\w$-base are close to being $\sigma'$-compact, i.e., have $\sigma$-compact set of non-isolated points. In Sections~\ref{s:pu-ww}, \ref{s:qu-ww}, \ref{s:u-ww} we detect topological spaces $X$ whose universal preuniformity $p\U_X$, universal quasi-uniformity $q\U_X$ and universal uniformity $\U_X$ have an $\w^\w$-base.

In the final Chapter~\ref{Ch:ex} we detect $\w^\w$-bases in $\sigma$-products of cardinals and thus find a consistent example of a non-separable Lindel\"of $P$-space whose universal uniformity has an $\w^\w$-base. Also we characterize La\v snev spaces which have an $\w^\w$-base or are universally $\w^\w$-based. In particular, we construct an example of a sequential countable $\aleph_0$-space with a unique non-isolated point, which has no $\w^\w$-base. In Section~\ref{s:GO} we detect generalized ordered spaces with an $\w^\w$-base. In Section~\ref{s:ww-cardinal} we find some upper bounds on the cardinality of countably tight spaces with an $\w^\w$-bases, which improve known upper bounds on the cardinality of first-countable spaces.

The results of this paper-book have been applied in the paper \cite{BL-LG} devoted to detecting free (locally convex) topological vector spaces and free topological (aeolian) groups possessing an $\w^\w$-base.
\smallskip

\noindent{\bf Acknowledgement.} The author expresses his sincere thanks to Arkady Leiderman for stimulating discussions related to $\w^\w$-based uniform and topological spaces, to Lyubomyr Zdomskyy for the help with set-theoretic questions and improvement of the proof and statement of Theorem~\ref{t:contra}, to Zolt\'an Vidny\'anszky who suggested to the author the proof of Lemma~\ref{l:nona}, and to Paul Szeptycki who suggested Example~\ref{ex:szeptycki} thus resolving a problem posed in the initial version of the manuscript.

\chapter{Preliminaries}\label{s:prelim}

In this section we collect notations that will be used throughout the paper. Also we remind some definitions and prove some auxiliary results that will be used in subsequent sections.

Cardinals are identified with the smallest ordinals of a given cardinality. For a set $X$ by $\mathcal P(X)$ we denote the power-set of $X$, i.e., the family of all subsets of $X$.

For a function $f:X\to Y$ and a set $A$ we put $f[A]=\{f(x):x\in X\cap A\}$ and $f^{-1}[A]=\{x\in X:f(x)\in A\}$.

\section{Separation properties of topological spaces}\label{ss:gms}

For a subset $A$ of a topological space $X$ its closure and interior in $X$ will be denoted by $\bar A$ and $A^\circ$, respectively. So, $\overline{A}^\circ$ is the interior of the closure of $A$ in $X$.

A topological space $X$ is called
\begin{itemize}
\item a \index{$T_0$-space}\index{topological space!$T_0$}{\em $T_0$-space} if for any distinct points $x,y\in X$ there exists an open set $U\subset X$ such that $U\cap\{x,y\}$ is a singleton;
\item \index{topological space!$T_1$}{\em $T_1$ at a point} $x\in X$ if the singleton $\{x\}$ is closed in $X$;
\item \index{topological space!$R_0$}{\em $R_0$ at a point} $x\in X$ if each neighborhood $O_x\subset X$ contains the closure $\overline{\{x\}}$ of the singleton $\{x\}$;
\item  \index{topological space!Hausdorff}{\em Hausdorff at a point} $x\in X$ if for each $y\in X\setminus \{x\}$ there exists a neighborhood $O_x$ of $x$ such that $\overline{O}_x\subset X\setminus\{x\}$;
\item  \index{topological space!semi-Hausdorff}{\em semi-Hausdorff at a point} $x\in X$ if for each $y\in X\setminus \{x\}$ there exists a neighborhood $O_x$ of $x$ such that $\overline{O}^\circ_x\subset X\setminus\{x\}$;
\item  \index{topological space!functionally Hausdorff} {\em functionally Hausdorff at a point} $x\in X$ if for any point $y\in X\setminus \{x\}$ there exists a continuous function $f:X\to [0,1]$ such that $f(x)=0$ and $f(y)=1$;
\item  \index{topological space!Urysohn}{\em Urysohn at a point} $x\in X$ if for each $y\in X\setminus \{x\}$ there exist a neighborhood $O_x$ of $x$ and a neighborhood $O_y$ of $y$ such that $\bar O_x\cap\bar O_y=\emptyset$;
\item  \index{topological space!regular}{\em regular at a point} $x\in X$ if for each neighborhood $O_x$ of $x$ there exists a neighborhood $U_x$ of $x$ such that $\overline{U}_x\subset O_x$;
\item  \index{topological space!semi-regular}{\em semi-regular at a point} $x\in X$ if for each neighborhood $O_x$ of $x$ there exists a neighborhood $U_x$ of $x$ such that $\overline{U}^\circ_x\subset O_x$;
\item  \index{topological space!completely regular}{\em completely regular at a point} $x\in X$ if for each neighborhood $O_x$ of $x$ there exists a continuous function $f:X\to [0,1]$ such that $f(x)=0$ and $f[X\setminus O_x]\subset \{1\}$;
\item  \index{topological space!Tychonoff}{\em Tychonoff at a point} $x\in X$ if $X$ is completely regular and $T_1$ at the point $x$
\item {\em a $T_1$-space} (resp. {\em $R_0$-space}, {\em Hausdorff}, {\em semi-Hausdorff}, {\em functionally Hausdorff}, {\em Urysohn}, {\em regular}, {\em semi-regular}, {\em completely regular}, {\em Tychonoff\/}) if $X$ is $T_1$ (resp. $R_0$, Hausdorff, semi-Hausdorff, functionally Hausdorff, Urysohn,  regular, semi-regular, completely regular, Tychonoff\/) at each point $x\in X$.
\end{itemize}

For a pseudometric $d$ on a set $X$ by $B_d(x;\e)$  we shall denote the open $\e$-ball $\{y\in X:d(y,x)<\e\}$ centered at a point $x\in X$.

We shall say that a subset $A$ of a topological space $X$ \index{subset of a topological space!accumulating at a point}{\em accumulates} at a point $x\in X$ if each neighborhood of $x$ contains infinitely many points of the set $A$.
A sequence $(x_n)_{n\in\w}$ of points in a topological space $X$ {\em accumulates} at a point $x\in X$ if each neighborhood of $x$ contains infinitely many points $x_n$, $n\in\w$. %A sequence $(F_n)_{n\in\w}$ of subsets of a topological space $X$  {\em accumulates} at a point $x\in X$ if each neighborhood of $x$ intersects infinitely many sets $F_n$, $n\in\w$.

A subset $A$ of a topological space $X$ will be called
\begin{itemize}
\item a \index{convergent sequence}{\em convergent sequence} if $A$ is infinite and can be written as $A=\{x_n\}_{n\in\w}$ for some convergent sequence $(x_n)_{n\in\w}$ in $X$;
\item \index{sequentially compact set}\index{subset of a topological space!sequentially compact}{\em sequentially compact in} $X$ if each infinite subset of $A$ contains a convergent sequence;
\item \index{countably compact set}\index{subset of a topological space!countably compact}{\em countably compact in} $X$ if each infinite subset $I\subset A$ accumulates at some point of $X$.
\end{itemize}

For a topological space $X$ by $C(X)$ we denote the linear space of all continuous real-valued functions on $X$. Endowed with the topology of pointwise convergence (inherited from the Tychonoff product $\IR^X$ of the real lines) the linear topological space $C(X)$ is denoted by $C_p(X)$.

For a cover $\U$ of a space $X$, a point $x\in X$ and a set $A\subset X$ we put $\St(x;\U)=\bigcup\{U\in\U:x\in U\}$ be the {\em $\U$-star} of $x$ and $\St(A;\U)=\bigcup_{x\in A}\St(x;\U)$ be the {\em $\U$-star} of the subset $A$ in $X$. We say that a family of sets $\V$ {\em refines} (resp. {\em star-refines}) a family of sets $\U$ if for every $V\in \V$ there exists $U\in\U$ such that $V\subset U$ (resp. $\St(V;\V)\subset U$).

A subset $B$ of a topological space $X$ is called
\begin{itemize}
\item \index{subset of a topological space!functionally bounded}{\em functionally bounded} in $X$ if for every continuous function $f:X\to\IR$ the set $f[B]$ is bounded in the real line $\IR$;
\item \index{subset of a topological space!$\w$-narrow}{\em $\w$-narrow} in $X$ if for every continuous map $f:X\to M$ to a metric space $M$ the image $f[B]$ is separable.
\end{itemize}

A family $\mathcal F$ of subsets of a topological space $X$ is called
\begin{itemize}
\item \index{family of sets!discrete}{\em discrete} in $X$ if each point $x\in X$ has a neighborhood $O_x\subset X$ that meets at most one set $F\in\mathcal F$.
\item \index{family of sets!strongly discrete}{\em strongly discrete} in $X$ if each set $F\in\mathcal F$ has an open neighborhood $U_F\subset X$ such that the family $(U_F)_{F\in\mathcal F}$ is discrete in $X$;
\item \index{family of sets!$\sigma$-discrete}\index{family of sets!strongly $\sigma$-discrete}({\em strongly}) {\em $\sigma$-discrete} if $\F$ can be written as the countable union $\F=\bigcup_{i\in\w}\F_i$ of (strongly) discrete families.
\end{itemize}

A subset $D$ of a topological space $X$ is called \index{subset of a topological space!strongly discrete}{\em strongly discrete} if the family of singletons $(\{x\})_{x\in X}$ is strongly discrete in $X$.

A subset $F$ of a topological space $X$ is called a \index{subset of a topological space!$\bar G_\delta$-set}{\em $\bar G_\delta$-set} if $F=\bigcap_{n\in\w}W_n=\bigcap_{n\in\w}\overline{W}_n$ for some sequence $(W_n)_{n\in\w}$ of open sets in $X$. It is easy to see that a subset $F$ of a (perfectly normal) space is a $\bar G_\delta$-set (if and) only if $F$ a closed $G_\delta$-set.

\begin{lemma}\label{l:d+bG=>sD} A countable $\bar G_\delta$-subset $D$ of a regular topological space $X$ is discrete if and only if it is strongly discrete.
\end{lemma}

\begin{proof} Assume that the countable $\bar G_\delta$-set $D$ is discrete. If $D$ is finite, then $D$ is strongly discrete in $X$ by the regularity of $X$. So, assume that $D$ is infinite and let $D=\{x_n\}_{n\in\w}$ be an enumeration of $D$ such that $x_n\ne x_m$ for any numbers $n\ne m$. Write the set $D$ as the intersection $D=\bigcap_{n\in\w}W_n=\bigcap_{n\in\w}\overline{W}_n$ of a decreasing sequence $(W_n)_{n\in\w}$ of open sets in $X$. Put $U_{-1}\cap V_{-1}=W_0$ and construct inductively two sequences of open sets $(U_n)_{n\in\w}$ and $(V_n)_{n\in\w}$ in $X$ such that for every $n\in\w$ the following conditions are satisfied:
\begin{enumerate}
\item $x_n\in U_n\subset V_{n-1}$;
\item $\{x_k\}_{k>n}\subset V_n\subset V_{n-1}\cap W_n$;
\item $\overline{U}_n\cap \overline{V}_n=\emptyset$.
\end{enumerate}
It can be shown that the family of open sets $(U_n)_{n\in\w}$ is discrete in $X$ and witnesses that the set $D=\{x_n\}_{n\in\w}$ is strongly discrete in $X$.
\end{proof}

A topological space is called
\begin{itemize}
\item \index{topological space!collectionwise normal}{\em collectionwise normal\/} if each discrete family of closed sets in $X$ is strongly discrete;
\item \index{topological space!collectionwise Urysohn}{\em collectionwise Urysohn\/} if each closed discrete subset of $X$ is strongly discrete;
\item \index{topological space!$\kappa$-Urysohn}{\em $\kappa$-Urysohn\/} for a cardinal $\kappa$ if each closed discrete subset of cardinality $\kappa$ in $X$ contains a strongly discrete subset of cardinality $\kappa$ in $X$.
\end{itemize}
It is well-known that each discrete family in a paracompact space is strongly discrete, which implies that for any $T_1$-space $X$ and any cardinal $\kappa$ we have the implications:
{%\scriptsize
$$
\xymatrix
{%\mbox{metrizable}\ar@{=>}[r]&
\mbox{paracompact}\ar@{=>}[r]&
\mbox{collectionwise}\atop\mbox{normal}\ar@{=>}[r]&
\mbox{collectionwise}\atop\mbox{Urysohn}\ar@{=>}[r]&
\mbox{$\kappa$-Urysohn}\ar@{=>}[r]&\mbox{2-Urysohn}\ar@{<=>}[r]&\mbox{Urysohn}.
}
$$
}

A regular $T_0$-space $X$ is called
\begin{itemize}
\item a \index{topological space!La\v snev}\index{La\v snev space}{\em La\v snev space} if $X$ is the image of a metrizable space under a closed continuous map;
\item a \index{topological space!Moore}\index{Moore space}{\em Moore space} if $X$ admits a sequence $(\U_n)_{n\in\w}$ of open covers such that for every point $x\in X$ the family $\{\St(x,\U_n)\}_{n\in\w}$ is a neighborhood base at $x$;
 \item a \index{topological space!$w\Delta$-space}\index{$w\Delta$-space}{\em $w\Delta$-space} if $X$ admits a sequence $(\U_n)_{n\in\w}$ of open covers such that for every point $x\in X$, any sequence $(x_n)_{n\in\w}\in \prod_{n\in\w}\St(x;\U_n)$ has an accumulation  point in $X$;
 \item an \index{topological space!$M$-space}\index{$M$-space}{\em $M$-space} if $X$ admits a sequence $(\U_n)_{n\in\w}$ of open covers such that each $\U_{n+1}$ star-refines $\U_n$ and for every point $x\in X$ any sequence $(x_n)_{n\in\w}\in \prod_{n\in\w}\St(x;\U_n)$ has an accumulation  point in $X$;
 \item \index{topological space!with a $G_\delta$-diagonal} a space with a {\em $G_\delta$-diagonal} if the diagonal $\Delta_X=\{(x,y)\in X\times X:x=y\}$ is a $G_\delta$-set in $X\times X$;
 \item \index{topological space!submetrizable}{\em submetrizable} if $X$ admits a continuous injective map to a metrizable space;
 \item \index{topological space!$\sigma$-metrizable}{\em $\sigma$-metrizable} if $X$ admits a countable cover by closed metrizable subspaces;
 \item \index{topological space!closed-$\bar G_\delta$}{\em closed-$\bar G_\delta$} if each closed subset of $X$ is a $\bar G_\delta$-set in $X$.
\end{itemize}
It is clear that each $M$-space is a $w\Delta$-space and each submetrizable space has a $G_\delta$-diagonal. By \cite[3.8]{Grue} a topological space $X$ is metrizable if and only if it is an $M$-space with a $G_\delta$-diagonal.
\smallskip

Lemma~\ref{l:d+bG=>sD} implies the following fact.

\begin{proposition} Each regular closed-$\bar G_\delta$ space is $\w$-Urysohn.
\end{proposition}

A subset $A$ of a topological space $X$ is called {\em $\w$-Urysohn} if each infinite closed discrete subset $D\subset A$ of $X$ contains an infinite strongly discrete set $S\subset D$ in $X$.

\begin{lemma}\label{l:hL+bG=>wU} Each hereditarily Lindel\"of $\bar G_\delta$-subset $B$ of a regular space $X$ is $\w$-Urysohn.
\end{lemma}

\begin{proof} Given an infinite countable closed discrete subset $D\subset B$ we shall prove that $D$ is strongly discrete in $X$. By Lemma~\ref{l:d+bG=>sD}, it suffices to show that $D$ is $\bar G_\delta$-set in $X$. Since $B$ is a $\bar G_\delta$-set in $X$, there exists a countable family $\W$ of open sets in $X$ such that $B=\bigcap\W=\bigcap_{W\in\W}\overline{W}$.
By the regularity of the space $X$, for every $x\in B\setminus D$ there exists an open neighborhood $U_x\subset X$ whose closure $\overline{U}_x$ is disjoint with the closed set $D$. Since the space $B$ is hereditarily Lindel\"of, the open cover $\{U_x:x\in B\setminus D\}$ of the Lindel\"of space $B\setminus D$ contains a countable subcover $\{U_x:x\in C\}$ (here $C$ is a suitable countable subset of $B\setminus D$). Then $\U=\W\cup\{X\setminus \overline{U}_x:x\in C\}$ is a countable family of open sets in $X$ such that $D=\bigcap\U=\bigcap_{U\in\U}\overline{U}$, witnessing that $D$ is a $\bar G_\delta$-set in $X$. By Lemma~\ref{l:d+bG=>sD}, the set $D$ is strongly discrete in $X$ and the set $B$ is $\w$-Urysohn.
\end{proof}

A topological space $X$ is called a \index{topological space!$P$-space}{\em $P$-space} if each point $x\in X$ is a $P$-point in $X$; a point $x\in X$ is called a \index{$P$-point}{\em $P$-point} if for any neighborhoods $U_n$, $n\in\w$, of $x$ in $X$ the intersection $\bigcap_{n\in\w}U_n$ is a neighborhood of $x$ in $X$.

For a topological space $X$ by $X'$ we denote the set of non-isolated points in $X$ and by $X^{\prime P}$ the set of points, which are not $P$-points in $X$. It is clear that $X^{\prime P}\subset X'$ and the set $X'$ is closed in $X$.

\begin{lemma}\label{l:P=>Gd=bGd} If for a regular space $X$ the subspace $X^{\prime P}$ is Lindel\"of, then each closed $G_\delta$-set $F$ in $X$ is a $\bar G_\delta$-set in $X$.
\end{lemma}

\begin{proof} Let $\W$ be a countable family of open sets in $X$ such that $X\in\W$ and $\bigcap\W=F$.
For every point $z\in X^{\prime P}\setminus F$ use the regularity of $X$ to choose an open neighborhood $O_z\subset X$ of $z$ such that $\overline{O}_z\cap F=\emptyset$. The complement $X^{\prime P}\setminus F$, being an $F_\sigma$-subset of the Lindel\"of space $X^{\prime P}$, is Lindel\"of. Consequently, we can find a countable set $Z\subset X^{\prime P}\setminus F$ such that $X^{\prime P}\setminus F\subset \bigcup_{z\in Z}O_z$. Let $\{V_n\}_{n\in\w}$ be an enumeration of the countable family $\V=\W\cup \{W\setminus\overline{O}_z:W\in\W,\;z\in Z\}$ of open sets in $X$.

For every $n\in\w$, use the Lindel\"of property of the closed subset $F\cap X^{\prime P}$ of the Lindel\"of space $X^{\prime P}$ and find a countable cover $\U_n$ of $F$ by open subsets $U\subset X$ such that $\overline{U}\subset V_n$. Consider the open neighborhood $U_n=\bigcup\U_n\subset V_n$ of the set $F$. We claim that $\bigcap_{n\in\w}\overline{U}_n=F$. Given any point $x\in X\setminus F$, we should find $n\in\w$ such that $x\notin \overline{U}_n$.

If $x\in X^{\prime P}$ then we can find a point $z\in Z$ with $x\in O_z$ and then find $n\in\w$ such that $V_n=X\setminus\overline{O}_z$. In this case $\overline{U}_n\subset \overline{V}_n\subset \overline{X\setminus O_z}=X\setminus O_z\subset X\setminus\{x\}$ and hence $x\notin \overline{U}_n$.
If $x\notin X^{\prime P}$, then find $W\in\W$ such that $x\notin W$ and choose a number $n\in\w$ such that $V_n=W$. Since $x$ is a $P$-point in $X$, the set $O_x=\bigcap_{U\in\U_n}(X\setminus \overline{U})$ is a neighborhood of $x$, disjoint with the set $U_n=\bigcup\U_n$, which implies that $x\notin U_n$.
\end{proof}

 We shall need the following (probably) known folklore result.

\begin{lemma}\label{l:para'} A regular topological space $X$ with Lindel\"of set $X'$ of non-isolated points is paracompact.
\end{lemma}

\begin{proof} Given an open cover $\U$ of $X$, for every $x\in X'$ find a set $U_x\subset\U$ containing $x$ and using the regularity of $X$, choose an open neighborhood $V_x\subset X$ of $x$ such that $\bar V_x\subset U_x$.   By the Lindel\"of property, the cover $\{V_x:x\in X'\}$ of $X'$ has a countable subcover $\{V_{x_n}\}_{n\in\w}$. It follows that $V=\bigcup_{n\in\w}V_{x_n}$ is an open neighborhood of $X'$ in $X$ and the family $\{W_n\}_{n\in\w}$ of the open sets  $W_n=V\cap U_{x_n}\setminus \bigcup_{k<n}\bar V_{x_k}$ is a locally finite countable cover of $V$. Then $\mathcal W=\{W_n\}_{n\in\w}\cup\big\{\{x\}:x\in X\setminus V\}$ is a locally finite open cover of $X$, refining the cover $\U$.
\end{proof}

\section{Some local and network properties of topological spaces}

A topological space $X$ is defined to be
\begin{itemize}
\item \index{topological space!first countable}{\em first-countable} at a point $x\in X$ if $X$ has a countable neighborhood base at $x$;
\item a \index{topological space!$q$-space}{\em $q$-space at} a point $x\in X$ if there exists a sequence $(U_n)_{n\in\w}$ of neighborhoods of $x$ such every sequence $(x_n)_{n\in\w}\in\prod_{n\in\w}U_n$ accumulates at some point $x'\in X$;
%\item {\em first-countable} if $X$ is first-countable at each point $x\in X$;
\item  \index{topological space!Fr\'echet-Urysohn}{\em Fr\'echet-Urysohn at} a point $x\in X$ if for each set $A\subset X$ with $x\in\bar A$ there is a sequence $(a_n)_{n\in\w}\in A^\w$ converging to $x$;
\item \index{topological space!sequential}{\em sequential} if for each non-closed set $A\subset X$ there exists a sequence $\{a_n\}_{n\in\w}\subset A$, convergent to some point $x\in X\setminus A$;
\item \index{topological space!strong Fr\'echet}{\em strong Fr\'echet at} a point $x\in X$ if for any decreasing sequence $(A_n)_{n\in\w}$ of subsets of $X$ with $x\in\bigcap_{n\in\w}\bar A_n$ there exists a sequence $(x_n)_{n\in\w}\in\prod_{n\in\w}A_n$ that converges to $x$;
\item \index{topological space!countably tight}{\em countably tight at} a point $x\in X$ if each subset $A\subset X$ with $x\in\bar A$ contains a countable subset $B\subset A$ such that $x\in\bar B$;
%\item a \index{topological space!$\bar q$-space}{\em $\bar q$-space at} a point $x\in X$ if there exists a sequence $(U_n)_{n\in\w}$ of closed neighborhoods of $x$ such every sequence $(x_n)_{n\in\w}\in\prod_{n\in\w}U_n$ accumulates at some point $x'\in X$;
%\item a {\em $q$-space} if $X$ is a $q$-space at each point $x\in X$;
\item \index{topological space!countably fan tight}{\em countably fan tight} at a point $x\in X$ if for any decreasing sequence $(A_n)_{n\in\w}$ of subsets of $X$ with $x\in\bigcap_{n\in\w}\bar A_n$ there exists a sequence $(F_n)_{n\in\w}$ of finite subsets $F_n\subset A_n$, $n\in\w$, such that each neighborhood of $x$ intersects infinitely many sets $F_n$, $n\in\w$;
%\item \index{topological space!countably ofan tight}{\em countably ofan tight} at a point $x\in X$ if for any decreasing sequence $(A_n)_{n\in\w}$ of open subsets of $X$ with $x\in\bigcap_{n\in\w}\bar A_n$ there exists a sequence $(F_n)_{n\in\w}$ of finite subsets $F_n\subset A_n$, $n\in\w$, accumulating at $x\in X$;
\item  {\em first-countable} (resp. {\em Fr\'echet-Urysohn, strong Fr\'echet, countably tight, countably fan-tight, a $q$-space}) if $X$ is first-countable (resp. Fr\'echet-Urysohn, strong Fr\'echet, countably tight, countably fan-tight, a $q$-space) at each point $x\in X$.
\end{itemize}

Strong Fr\'echet spaces and countably fan-tight spaces are partial cases of fan $\C$-tight spaces defined as follows.

\begin{definition}\label{d:fan-C-tight} Let $X$ be a topological space and $\C$ be a family of subsets of $X$. The space $X$ is defined to be
\begin{itemize}
\item \index{topological space!fan $\C$-tight}\index{topological space!ofan $\C$-tight}{\em fan $\C$-tight} (resp. {\em ofan $\C$-tight}) at a point $x\in X$  if for any decreasing sequence $(A_n)_{n\in\w}$ of (open) subsets of $X$ with $x\in\bigcap_{n\in\w}\bar A_n$ there exists a set $C\in\C$ accumulating at $x$ such that $C\setminus A_n$ is finite for all $n\in\w$;
\item ({\em o}){\em fan $\C$-tight} if $X$ is (o)fan $\C$-tight at each point $x\in X$.
\end{itemize}
\end{definition}

 Observe that a topological space $X$ is strong Fr\'echet if and only if $X$ is fan $\cs$-tight for the family $\cs$ of convergent sequences in $X$.   A topological space $X$ is countably fan-tight if and only if it is fan $\css$-tight for the family $\css$ of all countable subsets of $X$. Ofan $\css$-tight spaces were introduced by Sakai \cite{MSak} as spaces with the property $(\#)$.
In \cite{Ban} such spaces were called countably fan open-tight.

For a topological space $X$ and a family $\C$ of subsets of $X$ with $\cs\subset \C\subset \css$, we get the following implications between local properties of $X$.
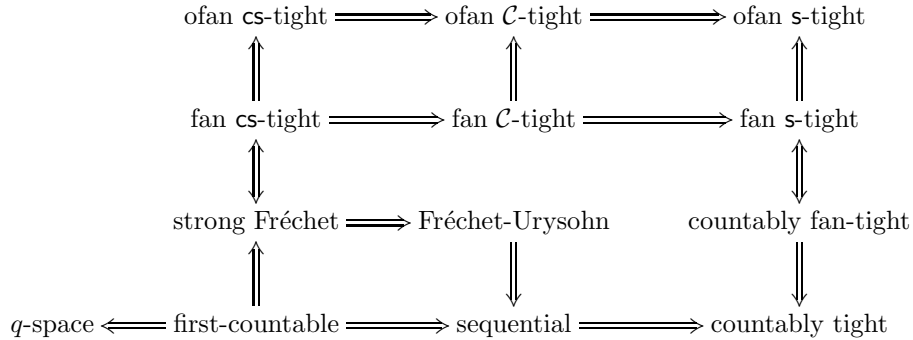
$$\xymatrix{
&\mbox{ofan $\cs$-tight}\ar@{=>}[r]&\mbox{ofan $\C$-tight}\ar@{=>}[r]&\mbox{ofan $\css$-tight}\\
&\mbox{fan $\cs$-tight}\ar@{=>}[r]\ar@{=>}[u]&\mbox{fan $\C$-tight}\ar@{=>}[r]\ar@{=>}[u]&\mbox{fan $\css$-tight}\ar@{=>}[u]\\
&\mbox{strong Fr\'echet}\ar@{=>}[r]\ar@{<=>}[u]&\mbox{Fr\'echet-Urysohn}\ar@{=>}[d]&\mbox{countably fan-tight}\ar@{<=>}[u]\ar@{=>}[d]\\
\mbox{$q$-space}&\mbox{first-countable}\ar@{=>}[u]\ar@{=>}[r]\ar@{=>}[l]&\mbox{sequential}\ar@{=>}[r]&\mbox{countably tight}
}
$$

%For a topological space $X$ its \index{topological space!character of}{\em character} $\chi(x;X)$ at a point $x\in X$ is the smallest cardinality of a neighborhood base at $x$. The cardinal $\chi(X)=\sup_{x\in X}\chi(x;X)$ is called the {\em character} of $X$. Observe that a topological space $X$ is first-countable iff $\chi(X)\le\w$.

In some sense, the fan $\C$-tightness is an ``orthogonal'' notion to that of a $\C^*$-network.

\begin{definition}\label{d:C-network} Let $\C$ be a family of subsets of a topological space $X$.
A family $\mathcal N$ of subsets of $X$ is called
\begin{itemize}\itemsep=2pt
\item a \index{network}{\em network} if for any open set $U\subset X$ and point $x\in U$ there exists a set $N\in\mathcal N$ with $x\in N\subset U$;
\item a \index{network!$k$-network}{\em $k$-network} if for any open set $U\subset X$ and compact subset $K\subset U$ there exists a finite subfamily $\F\subset \mathcal N$ with $K\subset\bigcup\F\subset U$;
\item a \index{network!$\C$-network}{\em $\C$-network} if for any set $C\in\C$ and any open neighborhood $U\subset X$ of $C$ there is a set $N\in\mathcal N$ such that $C\subset N\subset U$;
\item a \index{network!$\C^*$-network}{\em $\C^*$-network at a point} $x\in X$ if for any neighborhood $O_x\subset X$ of $x$ and any set $C\in\C$ accumulating at $x$, there exists a set $N\in\mathcal N$ such that $x\in N\subset O_x$ and $N\cap C$ is infinite;
\item a \index{$\C^*$-network}\index{network!$\C^*$-network}{\em $\C^*$-network} for $X$ if $\mathcal N$ is a network and a $\C^*$-network at each point $x\in X$.
\end{itemize}
\end{definition}

The precise meaning of ``orthogonality'' of (o)fan $\C$-tightness and $\C^*$-networks is described in the following theorem. For some concrete families $\C$ this theorem has been proved in \cite[3.2]{MSak},  \cite[1.6]{BL}, \cite[1.12]{Ban}.

\begin{theorem}\label{t:1=C*+fan} Let $X$ be a (semi-regular) topological space and $\C\supset \cs$ be a family of subsets of $X$, containing all convergent sequences in $X$.
\begin{enumerate}
\item $X$ is first-countable at a point $x\in X$ if and only if $X$ is fan $\C$-tight at $x$ and $X$ has a countable $\C^*$-network at $x$ (if and) only if $X$ is ofan $\C$-tight at $x$ and $X$ has a countable $\C^*$-network at $x$.
\item $X$ is first-countable if and only if $X$ is fan $\C$-tight and has a countable $\C^*$-network at each point (if and) only if $X$ is ofan $\C$-tight and has a countable $\C^*$-network at each point.
\item $X$ is second-countable if and only if $X$ is fan $\C$-tight and has a countable $\C^*$-network (if and) only if $X$ is ofan $\C$-tight and has a countable $\C^*$-network.
\end{enumerate}
\end{theorem}

\begin{proof} The ``only if'' parts of all statements is trivial. It remains to prove the ``if'' parts.
\smallskip

1. Let $\mathcal N$ be a countable $\C^*$-network at $x$. We lose no generality assuming that the family $\N$ is closed under finite unions. For every set $N\in{\mathcal N}$ denote by $N^\circ$ the interior of $N$ in $X$ and let ${\mathcal N}^\circ=\{N^\circ:N\in{\mathcal N},\;x\in N^\circ\}$.

We claim that ${\mathcal N}^\circ$ is a neighborhood base at $x$ if $X$ is fan $\C$-tight at $x$.
Given a neighborhood $O_x\subset X$ of $x$, we should find a set $N\in{\mathcal N}$ such that $x\in N^\circ\subset O_x$. Since the family $\mathcal N'=\{N\in{\mathcal N}:N\subset O_x\}$ is countable and closed under finite unions, we can choose an increasing sequence of sets $\{N_k\}_{k\in\w}\subset\mathcal N'$ such that each set $N\in\mathcal N'$ is contained in some set $N_k$, $k\in\w$. We claim that $x\in N_k^\circ$ for some $k\in\w$. To derive a contradiction, assume that $x\notin N_k^\circ$ for all $k\in\w$. In this case $(X\setminus N_k)_{k\in\w}$ is a decreasing sequence of sets with $x\in\bigcap_{n\in\w}\overline{X\setminus N_k}$. By the fan $\C$-tightness of $X$ at $x$, there exists a set $C\in\C$ accumulating at $x$ and such that for every $k\in\w$ the set $F_k=C\setminus (X\setminus N_k)=C\cap N_k$ is finite. The $\C^*$-network $\N$ contains a set $N\subset O_x$ such that $N\cap C$ is infinite.  It follows that $N\in\mathcal N'$ and hence $N\subset N_k$ for some $k\in\w$. Then $C\cap N\subset C\cap N_k=F_k$ is finite, which contradicts the choice of the set $N$. This contradiction shows that the countable family ${\mathcal N}^\circ$ is a neighborhood base at $x$ and hence $X$ is first-countable at $x$.
\smallskip

Now assume that the space $X$ is semi-regular at $x$ and $X$ is ofan $\C$-tight at $x$. For every $N\in\mathcal N$ let $\overline{N}^\circ$ be the interior of the closure of the set $N$ in $X$. We claim that the family $\overline{\mathcal N}^\circ=\{\overline{N}^\circ:N\in\mathcal N,\;x\in\overline{N}^\circ\}$ is a neighborhood base at $x$.

Given a neighborhood $O_x\subset X$ of $x$, we should find a set $N\in\mathcal N$ such that $x\in \overline{N}^\circ\subset O_x$. By the semi-regularity of $X$ at $x$, the point $x$ has a neighborhood $U_x$ such that $\overline{U}^\circ_x\subset O_x$. Since the family $\mathcal N'=\{N\in\mathcal N:N\subset U_x\}$ is countable and closed under finite unions, there exists an increasing sequence of sets $\{N_k\}_{k\in\w}\subset \mathcal N'$ such that each set $N\in\mathcal N'$ is contained in some set $N_k$. We claim that $x\in\overline{N}_k^\circ$ for some $k\in\w$. In the opposite case $x\in\bigcap_{k\in\w}X\setminus\overline{N}_k$ and by the ofan $\C$-tightness of $X$ at $x$, there exists a set $C\in\C$ accumulating at $x$ and such that for every $k\in\w$ the set $F_k=C\setminus (X\setminus \overline{N}_k)=C\cap\overline{N}_k$ is finite. The family $\N$, being  a $\C^*$-network at $x$, contains a set $N\in\mathcal N$ such that $N\subset U_x$ and $N\cap C$ is infinite. Since $N\in\mathcal N'$, there exists a number $k\in\w$ such that $N\subset N_k$. Then the set $C\cap N\subset C\cap \overline{N}_k= F_k$ is finite, which is a desired contradiction showing that for some $k\in\w$ we get the required inclusions: $x\in\overline{N}^\circ_k\subset \overline{U}^\circ_x\subset O_x$.
\smallskip

2. The second statement of the theorem follows immediately from the first one.
\smallskip

3. Assume that $\mathcal N$ is a countable $\C^*$-network for $X$ and consider the countable families $\mathcal N^\circ=\{N^\circ:N\in\N\}$ and $\overline{\N}^\circ=\{\overline{N}^\circ:N\in\N\}$. If the space $X$ is fan $\C$-tight (resp. ofan $\C$-tight and semi-regular), then by the first statement, for every point $x\in X$ the subfamily $\N^\circ_x=\{N^\circ\in\N^\circ:x\in N^\circ\}$ (resp. $\overline{\N}^\circ_x=\{\overline{N}^\circ\in\overline{\N}^\circ:x\in\overline{N}^\circ\}$) is a neighborhood base at $x$, which implies that $\N^\circ$ (resp. $\overline{\N}^\circ$) is a countable base for $X$.
\end{proof}

In each topological space $X$ for the role of the family $\C$ in Definitions~\ref{d:fan-C-tight}, \ref{d:C-network} we shall consider four families:
\begin{itemize}
\item the family $\as$ of all subsets of $X$ (i.e., $\as$ is the $\mathsf p$ower-set of $X$);
\item the family $\css$ of all countable subsets of $X$ (i.e., $\css$ is the family of all ``$\mathsf s$equences'' in $X$);
\item the family $\cs$ of all $\mathsf c$onvergent $\mathsf s$equences in $X$ and
\item the family $\ccs$ of $\mathsf c$ountable $\mathsf s$ub$\mathsf s$ets with $\mathsf c$ountably $\mathsf c$ompact $\mathsf c$losure in $X$.
\end{itemize}
For each Hausdorff space the obvious inclusions $\cs\subset\ccs\subset \css\subset\as$ yield the implications:
$$
\mbox{$\as^*$-network $\Ra$ $\css^*$-network $\Ra$ $\ccs^*$-network $\Ra$ $\cs^*$-network $\Ra$ network}.
$$

%If the family $\mathcal N$ is $\sigma$-discrete, then this diagram simplifies to the following linear form (see \cite{??}):
   %$$\mbox{$\mathfrak P^*$-network $\Ra$ $\w\mathfrak P^*$-network $\Ra$  $\ccs^*$-network $\Ra$ $\cccs^*$-network $\Leftrightarrow$ $\mathsf{cs}^*$-network $\Leftrightarrow$ $k$-network $\Ra$ network}.$$

Regular $T_0$-spaces with countable or $\sigma$-discrete $\C^*$-networks for various families $\C$ have special names.

 A regular $T_0$-space $X$ is called
 \begin{itemize}\itemsep=2pt
  \item \index{topological space!cosmic}\index{cosmic space}{\em cosmic} if $X$ has a countable network;
  \item a \index{topological space!$\sigma$-space}\index{$\sigma$-space}{\em $\sigma$-space} if $X$ has a $\sigma$-discrete network;
\item a \index{topological space!strong $\sigma$-space}\index{strong $\sigma$-space}{\em strong $\sigma$-space} if $X$ has a strongly $\sigma$-discrete network;
  \item a \index{topological space!$\Sigma$-space}\index{$\Sigma$-space}{\em $\Sigma$-space} if $X$ has a $\sigma$-discrete $\C$-network for some cover $\C$ of $X$ by closed countably compact subspaces;
 \item an \index{topological space!$\aleph_0$-space}\index{$\aleph_0$-space}{\em $\aleph_0$-space} if $X$ has a countable $\cs^*$-network (equivalently, $X$ has a countable $k$-network \cite{Foged});
 \item an \index{topological space!$\aleph$-space}\index{$\aleph$-space}{\em $\aleph$-space} if $X$ has a $\sigma$-discrete $\cs^*$-network (equivalently, $X$ has a $\sigma$-discrete $k$-network \cite{Foged});
 \item a \index{topological space!$\mathfrak P_0$-space}\index{$\mathfrak P_0$-space}{\em $\mathfrak P_0$-space} if $X$ has a countable $\as^*$-network (equivalently, $X$ has a countable $\css^*$-network \cite{Ban2});
 \item a \index{topological space!$\mathfrak P$-space}\index{$\mathfrak P$-space}{\em $\mathfrak P$-space} if $X$ has a $\sigma$-discrete $\as^*$-network;
 \item a \index{topological space!$\mathfrak P^*$-space}\index{$\mathfrak P^*$-space}{\em $\mathfrak P^*$-space} if $X$ has a $\sigma$-discrete $\css^*$-network;
 \item a space with the \index{strong Pytkeev property}\index{strong Pytkeev$^*$ property}{\em strong Pytkeev property} (resp. {\em the strong Pytkeev$^*$ property}) if $X$ has a countable $\as^*$-network (resp. a countable $\css^*$-network) at each point $x\in X$.
    \end{itemize}
For every topological space these properties relate as follows.
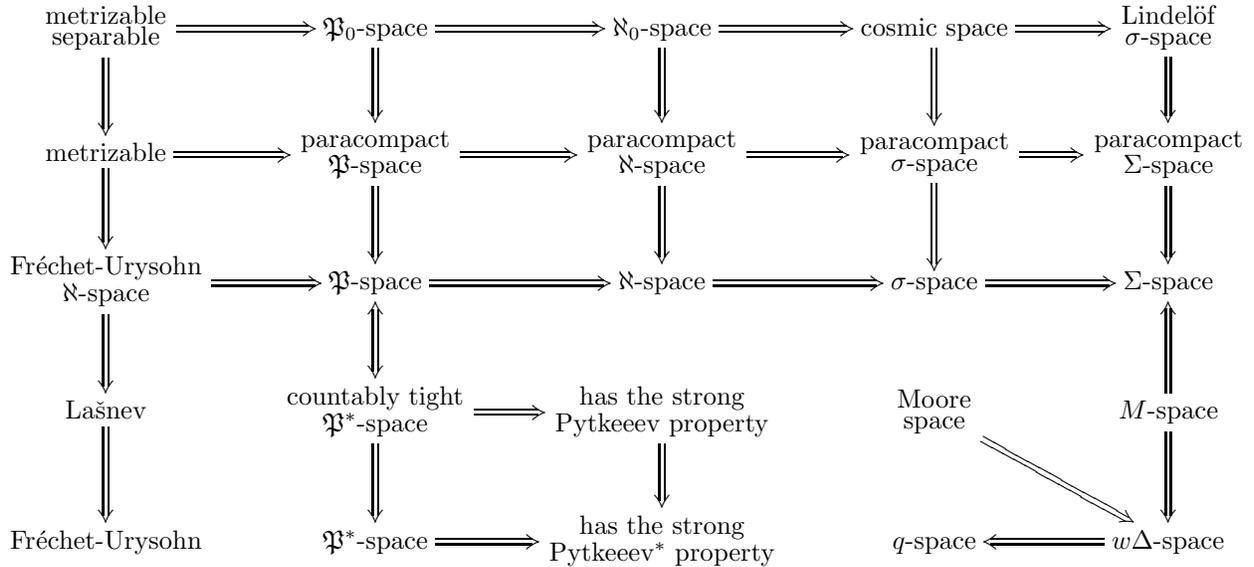
{%\small
 $$
 \xymatrix{
\mbox{metrizable}\atop\mbox{separable}\ar@{=>}[r]\ar@{=>}[d]&\mbox{$\mathfrak P_0$-space}\ar@{=>}[r]\ar@{=>}[d]&\mbox{$\aleph_0$-space}\ar@{=>}[r]\ar@{=>}[d]&\mbox{cosmic space}\ar@{=>}[d]\ar@{=>}[r]&\mbox{Lindel\"of}\atop\mbox{$\sigma$-space}\ar@{=>}[d]\\
\mbox{metrizable}\ar@{=>}[r]\ar@{=>}[d]& \mbox{paracompact}\atop\mbox{$\mathfrak P$-space}\ar@{=>}[r]\ar@{=>}[d]&\mbox{paracompact}\atop\mbox{$\aleph$-space}\ar@{=>}[r]\ar@{=>}[d]&\mbox{paracompact}\atop\mbox{$\sigma$-space}\ar@{=>}[d]\ar@{=>}[r]&\mbox{paracompact}\atop\mbox{$\Sigma$-space}\ar@{=>}[d]\\
\mbox{Fr\'echet-Urysohn}\atop\mbox{$\aleph$-space}\ar@{=>}[r]\ar@{=>}[d] &\mbox{$\mathfrak P$-space}\ar@{=>}[r]&\mbox{$\aleph$-space}\ar@{=>}[r]&\mbox{$\sigma$-space}\ar@{=>}[r]&\mbox{$\Sigma$-space}\\
\mbox{La\v snev}\ar@{=>}[d]&\mbox{countably tight}\atop\mbox{$\mathfrak P^*$-space}\ar@{<=>}[u]\ar@{=>}[d]\ar@{=>}[r]&\mbox{has the strong}\atop\mbox{Pytkeeev property}\ar@{=>}[d]&\mbox{Moore}\atop\mbox{space}\ar@{=>}[rd]
&\mbox{$M$-space}\ar@{=>}[d]\ar@{=>}[u]\\
\mbox{Fr\'echet-Urysohn}&\mbox{$\mathfrak P^*$-space}\ar@{=>}[r]&\mbox{has the strong}\atop\mbox{Pytkeeev$^*$ property}&\mbox{$q$-space}&\mbox{$w\Delta$-space}\ar@{=>}[l]
    }
 $$
 }
More information on these and other classes of generalized metric spaces can be found in \cite{Ban},  \cite{Ban2}, \cite{GK-P}, \cite{Grue}.
\smallskip

The classes of $P$-spaces and $\Sigma$-spaces are orthogonal in the following sense.

\begin{lemma}\label{l:Sigma+P=d} A topological space $X$ is discrete if and only if $X$ is both a $\Sigma$-space and a $P$-space.
\end{lemma}

\begin{proof} The ``only if'' part is trivial. To prove the ``if'' part, assume that a $P$-space $X$ is a $\Sigma$-space. Then $X$ has a $\sigma$-discrete $\C$-network $\N$ for some cover $\C$ of $X$ by closed countably compact sets. Since $X$ is $P$-space, each (closed countably compact) set $C\in\C$ is finite. We claim that $\C\subset\N$. Given a (finite) set $C\in\C$ consider the countable family $\N(C)=\{N\in\N:C\subset N\}$. In each set $N\in\N(C)\setminus\{C\}$ fix a point $x_N\in N\setminus C$. Since $X$ is a $P$-space, the set $U_C=X\setminus\{x_N:N\in\N(C)\setminus\{C\}\}$ is an open neighborhood of $C$ in $X$. Since $\N$ is a $\C$-network, there exists a set $N\in\N$ such that $C\subset N\subset U_C$. It follows that $N\in\N(C)\setminus (\N(C)\setminus\{C\})\subset\{C\}$ and hence $C=N\in\N$. Since the family $\N$ is $\sigma$-discrete, so is its subfamily $\C$. Since each set $C\in\C$ is finite, the space $X$ is $\sigma$-discrete. Being a $P$-space, the $\sigma$-discrete space $X$ is discrete.
\end{proof}

\section{Some countable global properties of topological spaces}

A subset $S$ of a topological space $X$ is called \index{topological space!weakly separated}{\em weakly separated} if each point $x\in S$ has a neighborhood $O_x\subset X$ such that for any distinct points $x,y\in S$ either $x\notin O_y$ or $y\notin O_x$.

We shall say that a topological space $X$
\begin{itemize}\itemsep=2pt
\item is \index{topological space!weakly cosmic}{\em weakly cosmic} if each weakly separated subspace of $X$ is at most countable;
\item is \index{topological space!Lindel\"of}{\em Lindel\"of\/} if each open cover of $X$ has a countable subcover;
\item is \index{topological space!hereditarily Lindel\"of}{\em hereditarily Lindel\"of\/} if each subspace of $X$ is Lindel\"of;
\item is \index{topological space!separable}{\em separable} if $X$ contains a countable dense subset;
\item is \index{topological space!hereditarily separable}{\em hereditarily separable} if each subspaces of $X$ is separable;
\item has \index{topological space!with countable spread}\index{spread}{\em countable spread} if each discrete subspace of $X$ is countable;
\item has \index{topological space!with countable extent}\index{extent}{\em countable extent} if each closed discrete subspace of $X$ is countable;
\item has \index{topological space!with countable cellularity}\index{cellularity}{\em countable cellularity} if each disjoint family of non-empty open subsets is countable;
\item has \index{topological space!with countable discrete cellularity}\index{discrete cellularity}{\em countable discrete cellularity} if each discrete family of non-empty open subsets is countable.
\end{itemize}

By \cite{Tka}, each cosmic space is weakly cosmic and each weakly cosmic space is hereditarily separable and hereditarily Lindel\"of.
The following theorem proved by Todor\v cevi\'c \cite[p.30]{Todo} characterizes cosmic spaces under \index{PFA}PFA, the Proper Forcing Axiom. It is known that PFA is consistent with ZFC and has many nice implications (beyond of Set Theory), see \cite{Baum} or \cite{Moore}.

\begin{theorem}[Todor\v cevi\'c]\label{t:TodoPFA} Under PFA a regular $T_0$-space $X$ is cosmic if and only if all finite powers of $X$ are weakly cosmic.
\end{theorem}

The following diagram describes the relations between some global properties of Hausdorff topological spaces.
$$
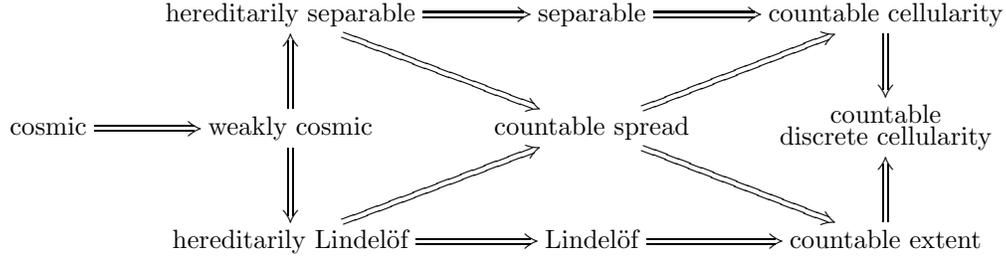
\xymatrix{
&\mbox{hereditarily separable}\ar@{=>}[r]\ar@{=>}[rd]& \mbox{separable}\ar@{=>}[r]&\mbox{countable cellularity}\ar@{=>}[d]\\
\mbox{cosmic}\ar@{=>}[r]&\mbox{weakly cosmic}\ar@{=>}[u]\ar@{=>}[d]&\mbox{countable spread}\ar@{=>}[ru]\ar@{=>}[rd]&\mbox{countable}\atop\mbox{discrete cellularity}\\
&\mbox{hereditarily Lindel\"of}\ar@{=>}[r]\ar@{=>}[ru]&\mbox{Lindel\"of}\ar@{=>}[r]&
\mbox{countable extent}\ar@{=>}[u]
}
$$
\smallskip

\section{The topology of the space $\w^\w$}

On the set $\w^\w$ we consider the topology of the Tychonoff product of countably many  discrete spaces. This topology on $\w^\w$ has a canonical base $\{{\uparrow}\beta\}_{\beta\in\w^{<\w}}$ indexed by the family $\w^{<\w}=\bigcup_{n\in\w}\w^n$ of all finite sequences of finite ordinals. Observe that for every function $\alpha\in\w^\w$ and every $n\in\w$ the restriction $\alpha|n$ belongs to $\w^n\subset\w^{<\w}$.

For any sequence $\beta\in\w^{<\w}$ we put ${\uparrow}\beta=\{\alpha\in\w^\w:\exists n\in\w\;\;(\alpha|n=\beta)\}\subset \w^\w$. It follows that for every $\alpha\in\w^\w$ the family $\{{\uparrow}(\alpha|n)\}_{n\in\w}$ is a decreasing neighborhood base at $\alpha$ in the space $\w^\w$. Consequently, $\{{\uparrow}\beta\}_{\beta\in\w^{<\w}}$ is a base of the topology of the space $\w^\w$.

\section{Analytic and $K$-analytic spaces}\label{ss:Ka}
A topological space $X$ is called \index{topological space!analytic}{\em analytic} if $X$ is a continuous image of a Polish space. If $X$ is non-empty, then this is equivalent to saying that $X$ is a continuous image of the Polish space $\w^\w$.

A topological space $X$ is called \index{topological space!$K$-analytic}{\em $K$-analytic} if
$X=\bigcup_{\alpha\in\w^\w}K_\alpha$ for a family $(K_\alpha)_{\alpha\in\w^\w}$ of compact subsets of $X$, which is \index{family of sets!upper semicontinuous}{\em upper semicontinuous} in the sense that for every open set $U\subset X$ the set $\{\alpha\in\w^\w:K_\alpha\subset U\}$ is open in  $\w^\w$. If each compact set $K_\alpha$ is a singleton $\{f(\alpha)\}$, then the upper semicontinuity of the family $(K_\alpha)_{\alpha\in\w}$ is equivalent to the continuity of the map $f:\w^\w\to X$, meaning that the space $X$ is analytic.

A \index{compact resolution}\index{resolution}\index{topological space!compact resolution of}({\em compact}) {\em resolution} of a topological space $X$ is a family $(K_\alpha)_{\alpha\in\w^\w}$ of (compact) subsets of $X$ such that $X=\bigcup_{\alpha\in\w^\w}K_\alpha$ and $K_\alpha\subset K_\beta$ for every $\alpha\le\beta$ in $\w^\w$.
By Proposition 3.10(i) \cite{kak}, each $K$-analytic space admits a compact resolution. The converse is true for Dieudonn\'e complete spaces, see \cite[Proposition 3.13]{kak}.
Since each submetrizable space is Dieudonn\'e complete \cite[6.10.8]{AT}, we have the following equivalence, which will be used in the proof of Theorems~\ref{t:lP*}, \ref{t:small}, \ref{t:tsmall}, \ref{t:qu-cc}.

\begin{lemma}\label{l:analytic} For a submetrizable space $X$ the following conditions are equivalent:
\begin{enumerate}
\item[\textup{(1)}] $X$ is analytic;
\item[\textup{(2)}] $X$ is $K$-analytic;
\item[\textup{(3)}] $X$ has a compact resolution.
\end{enumerate}
\end{lemma}

More information on compact resolutions can be found in the monograph \cite{kak}.

\section{Analytic and meager filters}

A \index{filter}{\em filter} on a set $X$ is any family $\F$ of non-empty subsets of $X$, which is closed under finite intersections and taking supersets in $X$. Each filter on $X$ is a subset  of the power-set $\mathcal P(X)$.    Identifying each subset of $X$ with its characteristic function, we can identify the power-set $\mathcal P(X)$ with the Cantor cube $\{0,1\}^X$ and thus endow $\mathcal P(X)$ with a compact Hausdorff topology. This topology is metrizable if the set $X$ is countable. This allows us to talk about topological properties of filters.  We shall need the following famous characterization of meager filters due to Talagrand \cite{Tal80}.

\begin{lemma}[Talagrand]\label{l:Tal} A filter $\F$ on a countable set $X$ is a meager subset of the power-set $\mathcal P(X)$ if and only if there exists a finite-to-one map $\varphi:X\to\w$  such that for every set $F\in\F$ the image $\varphi[F]$ has finite complement in $\w$.
\end{lemma}

A function $\varphi:X\to Y$ between two sets is called \index{function!finite-to-one}{\em finite-to-one} if for each point $y\in Y$ the preimage $\varphi^{-1}(y)$ is finite and not empty.

In the proof of Theorem~\ref{t:lP*} we shall also use the following known fact that can be derived from Lusin-Sierpi\'nski Theorem \cite[21.6]{Ke} (saying that analytic subsets of Polish spaces have the Baire property).

\begin{lemma}\label{l:am} Each analytic free filter on a countable set is meager.
\end{lemma}

\chapter{Posets and their (Tukey) reducibility}\label{ch:poset}

In this chapter we recall the necessary information on the (Tukey) reducibility of posets and also detect cardinals and powers of cardinals, which are Tukey dominated by $\w^\w$.

\section{Examples of posets}

By a \index{poset}{\em poset} we understand a non-empty set endowed with a partial preorder (i.e., a reflexive transitive binary relation). A poset $P$ is \index{poset!directed}{\em directed} if for any $x,y\in P$ there exists $z\in P$ such that $x\le z$ and $y\le z$.
Now we consider some standard examples of directed posets (which will appear in this paper).

\begin{example} Each cardinal $\kappa$ will be considered as a poset endowed with the canonical well-order ($x\le y$ iff $x\in y$ or $x=y$).
\end{example}

\begin{example} For any poset $P$ and set $X$ the power $P^X$ is a poset endowed with the partial preorder $\le$ defined by $f\le g$ iff $f(x)\le g(x)$ for all $x\in X$. If the set $X$ is infinite, then the set $P^X$ carries also the preorder $\le^*$ defined by $f\le^* g$ iff $|\{x\in X:f(x)\not\le g(x)\}|<|X|$.
\end{example}

\begin{example} For any set $X$ the set $\IR^X$ of all functions from $X$ to $\IR$ is a poset endowed with the partial order $\le$ defined by $f\le g$ iff $f(x)\le g(x)$ for all $x\in X$.

If $X$ is a topological space, then the poset $\IR^X$ contains the space \index{$C(X)$}$C(X)$ of all continuous real-valued functions on $X$. Endowed with the topology inherited from the Tychonoff product $\IR^X$ of the real lines, the space $C(X)$ is usually denoted by \index{$C_p(X)$} $C_p(X)$.

If $X$ is a uniform space, then the poset $\IR^X$ contains the space \index{$C_u(X)$}$C_u(X)$ of all uniformly continuous real-valued functions on $X$. In this case $C_u(X)\subset C_p(X)=C(X)\subset \IR^X$.
\end{example}

\begin{example}\label{e:filter} Any filter $\F$ on a set $X$ will be considered as a poset endowed with the partial order of reverse inclusion ($F\le E$ iff $E\subset F$). %We recall that a family $\F$ of subsets of a set $X$ is called a {\em filter} on $X$ if $\emptyset\notin\F$, for any sets $E,F\in\F$ we get $F\cap  E\in\F$ and for any sets $F\subset E\subset X$ the inclusion $F\in\F$ implies $E\in\F$.
\end{example}

The following two examples are partial cases of Example~\ref{e:filter}.

\begin{example} For each uniform space $X$ its uniformity \index{$\U_X$}$\U_X$ will be considered as a poset endowed with the partial order of reverse inclusion ($U\le V$ iff $V\subset U$).
\end{example}

\begin{example}\index{$\Tau_x(X)$}
For any point $x\in X$ of a topological space $X$ the family $\Tau_x(X)$ of all neighborhoods of $x$ is a poset, endowed with the partial order of reverse inclusion ($U\le V$ iff $V\subset U$).
\end{example}

A subset $A$ of a poset $P$ is called
\begin{itemize}
\item \index{subset of a poset!cofinal}{\em cofinal} in $P$ if for each $x\in P$ there exists $a\in A$ such that $x\le a$;
\item \index{subset of a poset!bounded}{\em bounded} in $P$ if there exists $x\in P$ such that $a\le x$ for all $a\in A$.
\end{itemize}

We shall say that a subset $A$ of a poset $P$ {\em dominates} a subset $B\subset P$ if for each $b\in B$ there is $a\in A$ with $b\le a$. Observe that a subset $B\subset P$ is bounded in $P$ iff it is dominated by a singleton in $P$. A subset $C\subset P$ is cofinal iff $C$ dominates $P$.

\section{The Tukey reducibility of posets}

A function $f:P\to Q$ between two posets is called
\begin{itemize}
\item \index{function between posets!monotone}{\em monotone} if $f(p)\le f(p')$ for all $p\le p'$ in $P$;
\item \index{function between posets!cofinal}{\em cofinal} if for each $q\in Q$ there exists $p\in P$ such that $q\le f(p)$.
\end{itemize}
Given two partially ordered sets $P,Q$ we shall write $P\succcurlyeq Q$ or $Q\preccurlyeq P$ and say that $Q$ {\em reduces} to $P$ if there exists a monotone cofinal map $f:P\to Q$. In this case we shall also say that the poset $Q$ is \index{poset!$P$-dominated}{\em $P$-dominated}. Also we write $P\cong Q$ if $P\preccurlyeq Q$ and $P\succcurlyeq Q$. It is clear that the relation of reducibility is transitive, i.e., for any posets $P,Q,R$ the relations $P\succcurlyeq Q\succcurlyeq R$ imply $P\succcurlyeq R$.

This notion of reducibility of posets  is a bit stronger than the Tukey reducibility introduced by Tukey \cite{Tuk} and studied e.g. in \cite{DT}, \cite{Frem}, \cite{F-MT}, \cite{GM16}, \cite{Mama}, \cite{Sol15}, \cite{ST}, \cite{Tod85}. Given two posets $P,Q$ we say that $Q$ is \index{poset!Tukey reducible}\index{Tukey reduction}{\em Tukey reducible to} $P$ (and denote this by $Q\le_T P$) if there exists a map $f:P\to Q$ such that for every cofinal subset $C\subset P$ its image $f(C)$ is cofinal in $Q$.  It is known (see \cite[Proposition 1]{Tod85}) that a poset $Q$ is Tukey reducible to a poset $P$ if and only if there exists a function $f:Q\to P$ that maps unbounded subsets of $Q$ to unbounded subsets of $P$.

By \cite{GM16}, for posets $P,Q$ the Tukey relation $Q\le_T P$ is equivalent to $P\succcurlyeq Q$ if each bounded set in $Q$ has a least upper bound. This result can be compared with the following lemma.

\begin{lemma}\label{l:Tuckey} Let $P,Q$ be two posets. If $Q\preccurlyeq P$, then $Q\le_T P$.
If each non-empty subset $A\subset Q$ has a greatest lower bound in $Q$, then $Q\le_T P$ is equivalent to $Q\preccurlyeq P$.
\end{lemma}

\begin{proof} Assuming that $Q\preccurlyeq P$, fix a monotone cofinal map $f:P\to Q$. We claim that for every cofinal subset $C\subset P$ the image $f[C]$ is cofinal in $Q$. Indeed, for every $y\in Q$ the cofinality of $f$ yields an element $x\in P$ with $y\le f(x)$. The cofinality of $C$ yields a point $c\in C$ with $c\ge x$. Then $f(c)\ge f(x)\ge y$, which means that the set $f[C]$ is cofinal in $Q$. Therefore, $Q\le_T P$.

Now assume that each subset $A\subset Q$ has a greatest lower bound in $Q$. Assuming that $Q\le_T P$, we could find a map $f:P\to Q$ that sends cofinal subsets of $P$ to cofinal subsets of $Q$. For every point $x\in P$ consider the upper set ${\uparrow}x=\{p\in P:x\le p\}$ and let $\check f(x)$ be a greatest lower bound of the set $f[{\uparrow} x]$ in the poset $Q$. It is clear that the map $\check f: P\to Q$ is monotone. To see that $\check f$ is cofinal, take any element $y\in Q$ and consider the preimage $f^{-1}[{\uparrow}y]$. We claim that ${\uparrow}x\subset f^{-1}[{\uparrow}y]$ for some $x\in P$. In the opposite case, for every $x\in P$ we can choose a point $g(x)\in{\uparrow}x\setminus f^{-1}[{\uparrow}y]$. It is clear that the set $B=\{g(x):x\in P\}$ is cofinal in $P$. Then its image $f[B]$ is cofinal in $Q$ and hence $f[B]\cap {\uparrow}y\ne\emptyset$, which contradicts the choice  of the points $g(x)\notin f^{-1}[{\uparrow}y]$, $x\in P$. This contradiction shows that ${\uparrow}x\subset f^{-1}[{\uparrow}y]$ and hence $y\le \check f(x)$.
\end{proof}

Lemma~\ref{l:Tuckey} implies that for any filter $\F$ and poset $P$ the reducibility $\F\preccurlyeq P$ is equivalent to the Tukey reducibility $\F\le_T P$. In particular, we get the following two characterizations.

\begin{lemma}\label{l:p-Tukey} Let $P$ be a poset.
\begin{enumerate}
\item The uniformity $\U_X$ of a uniform space $X$ has a $P$-base iff $P\succcurlyeq \U_X$ iff $\U_X\le_T P$.
\item A topological space $X$ has a neighborhood $P$-base at a point $x\in X$ iff $P\succcurlyeq \Tau_x(X)$ iff $\Tau_x(X)\le_T P$.
\end{enumerate}
\end{lemma}

For a topological space $X$ by $\K(X)$ we denote the poset of all compact subsets of  $X$, endowed with the inclusion order ($A\le B$ iff $A\subset B$). The following fundamental reducibility result was proved by Christensen \cite{Chris} (see also \cite[6.1]{kak}). This theorem will be essentially used in the proofs of Theorems~\ref{t:dominat}, \ref{t:Fat}, \ref{t:Lasnev}.

\begin{theorem}[Christensen]\label{t:Chris} A metrizable space $X$ is Polish if and only if $\w^\w\succcurlyeq \K(X)$ if and only if $\K(X)\le_T \w^\w$.
\end{theorem}

%\begin{lemma}\label{l:ww-not-ww1} $\w^\w\not\succcurlyeq\w^{\w_1}$.
%\end{lemma}

%\begin{proof} Assuming that $\w^\w\succcurlyeq \w^{\w_1}$, we can find a cofinal subset $\{\varphi_\alpha\}_{\alpha\in\w^w}\subset \w^{\w_1}$ such that $\varphi_\alpha\le\varphi_\beta$ for any $\alpha\le\beta$ in $\w^\w$. Then the family $(K_\alpha)_{\alpha\in\w^\w}$ of compact subsets $$K_\alpha=\big\{\varphi\in \IR^{\aleph_1}:|\varphi|\le\max\{0,\varphi_\alpha\}\big\},\;\;\alpha\in\w^\w,$$is a compact resolution of the space $\IR^{\w_1}=C_p(\w_1)$ (here the cardinal $\w_1$ is endowed with the discrete topology).
% By Corollary 2.4 \cite{COT}, the function space $C_p(\w_1) $ is $K$-analytic and hence Lindel\"of. By Theorem 9.17 of \cite{kak}, the discrete space $\w_1$ is $\sigma$-compact, which is a desired contradiction, showing that $\w^\w\not\succcurlyeq\w^{\w_1}$.
%\end{proof}

\section{The cofinality and the additivity of a poset}

The \index{poset!cofinality of}{\em cofinality}\index{$\cof(P)$} $\cof(P)$ of a poset $P$ is the smallest cardinality of a cofinal subset $D\subset P$.

The \index{poset!additivity of}{\em additivity}\index{$\add(P)$} $\add(P)$ of an unbounded poset $P$ is defined as the smallest cardinality of an unbounded subset of $P$.
For a bounded poset $P$ it will be convenient to put $\add(P)=1$.

The proof of the following well-known fact can be found in \cite[513C]{F-MT}.

\begin{lemma}\label{l:b-reg} For any unbounded poset $P$ the cardinal $\add(P)$ is regular and $\add(P)\le\cf(\cof(P))\le\cof(P)$.
\end{lemma}

By \cite[513E]{F-MT}, for two posets $P,Q$ the Tukey reducibility $P\le_T Q$ implies $\cof(P)\le\cof(Q)$ and $\add(Q)\le \add(P)$.

\begin{proposition}\label{p:add=cof} For each poset $P$ we have the reductions  $P\succcurlyeq \cof(P)\cong \cf(\cof(P))$ and $P\succcurlyeq \add(P)$. Moreover, if $\add(P)=\cof(P)$, then $P\cong\add(P)=\cof(P)$.
\end{proposition}

\begin{proof}
By the definition of the cardinal $\cof(P)$, the poset $P$ contains a dominating set $\{y_\alpha\}_{\alpha\in\cof(P)}$. The monotone cofinal map $P\to\cof(P)$, $x\mapsto \min\{\alpha\in\cof(P):x\le y_\alpha\}$, witnesses that $P\succcurlyeq \cof(P)$. In particular, $\cof(P)\succcurlyeq \cof(\cof(P))$. By definition, for every non-zero cardinal $\kappa$, there exists a monotone cofinal map $\cof(\kappa)\to\kappa$, witnessing that $\cof(\kappa)\succcurlyeq \kappa$. Combined with $\cof(P)\succcurlyeq\cof(\cof(P))$, this yields $\cof(P)\cong\cof(\cof(P))$.

If the poset $P$ is bounded, then $P\succcurlyeq 1=\add(P)$. If $P$ is unbounded, then it contains an unbounded subset $\{x_\alpha\}_{\alpha\in\add(P)}$ of cardinality $\add(P)$. Then the monotone cofinal map $P\to\add(P)$, $x\mapsto \min\{\alpha\in\cof(P):x_\alpha\not\le x\}$, witnesses that $P\succcurlyeq \add(P)$.

Now assume that $\kappa=\add(P)=\cof(P)$ and fix a cofinal subset $\{x_\alpha\}_{\alpha\in\kappa}$ in $P$. By transfinite induction for every $\alpha\in\kappa$ choose an upper bound $y_\alpha$ of the set $Y_{\alpha}=\{x_\beta\}_{\beta\le\alpha}\cup\{y_\beta\}_{\beta<\alpha}$. Such an upper bound exists as $|Y_\alpha|<\kappa=\add(P)$. The monotone cofinal map $\kappa\to P$, $\alpha\mapsto y_\alpha$, witnesses that $\kappa\succcurlyeq P$.
\end{proof}

The cardinals \index{$\mathfrak d$}$\mathfrak d=\cof(\w^\w,\le^*)=\cof(\w^\w)$ and \index{$\mathfrak b$} $\mathfrak b=\add(\w^\w,\le^*)$ play an important role in the modern Set Theory and its applications to General Topology and Topological Algebra, see \cite{Douwen}, \cite{Vaug}, \cite{Blass}.
It is well-known that $\w_1\le\mathfrak b\le\mathfrak d\le\mathfrak c$ in ZFC and
$\mathfrak b=\mathfrak d=\mathfrak c$ under Martin's Axiom (abbreviated by MA).
So, $\w_1<\mathfrak b=\mathfrak d=\mathfrak c$ under (MA$+\neg$CH).
\smallskip

\begin{lemma}\label{l:b-bound} If a poset $P$ is $\w^\w$-dominated, then $$\add(P),\cof(P)\in\{1,\w\}\cup[\mathfrak b,\mathfrak d].$$
\end{lemma}

\begin{proof}  Let $f:\w^\w\to P$ be a monotone cofinal map. Then the poset $P$ is directed.

To see that $\cof(P)\le\mathfrak d$, take a dominating set $D\subset\w^\w$ of cardinality $|D|=\cof(\w^\w)=\mathfrak d$ and observe that the set $f[D]$ dominates $P$.  Consequently, $\cof(P)\le|f[D]|\le\mathfrak d$.

If $\cof(P)<\mathfrak b$, then we can find a subset set $D\subset\w^\w$ of cardinality $|D|=\cof(P)<\mathfrak b$ whose image $f[D]$ is dominating in $P$. By the definition of the cardinal $\mathfrak b>|D|$, the set $D$ is dominated by a countable subset $B\subset\w^\w$. Then its image $f[B]$ dominates $f[D]$ and hence is dominating in $P$. Consequently, $\cof(P)\le|f[B]|\le\w$. If the (directed) poset $P$ is unbounded, then $\cof[P]\ge\w$. If $P$ is bounded, then $\cof(P)=1$. Consequently,  $\cof(P)\in\{1,\w\}\cup[\mathfrak b,\mathfrak d]$.

Next, we show that $\add(P)\in \{1,\w\}\cup[\mathfrak b,\mathfrak d]$. If the poset $P$ is bounded, then $\add(P)=1$. If $P$ is unbounded, then we can choose an unbounded subset $B\subset P$ of cardinality $|B|=\add(P)$ and choose a set $B'\subset\w^\w$ such that $f|B':B'\to B$ is a bijective map. If $B'$ is finite, then $B'$ is dominated by a singleton $\{x\}$ in $\w^\w$ and then $B=f[B']$ is bounded in $P$, being dominated by the singleton $\{f(x)\}$ in $P$.
This contradiction shows that $B$ is infinite and hence $\w\le|B|=\add(P)$.
If $\add(P)=\w$, then we are done. So, assume that $\add(P)>\w$.

If $|B|=\add(P)<\mathfrak b$, then by the definition of the cardinal $\mathfrak b=\mathfrak b(\w^\w,\le^*)$, the set $B'$ is bounded in the poset $(\w^\w,\le^*)$ and hence is dominated by a countable set $C$ in the poset $\w^\w=(\w^\w,\le)$. By the monotonicity of $f$, the set $B$ dominated by the countable set $f[C]\subset P$. Since $\add(P)>\w$, the countable set $f[C]$ is dominated by a singleton in $P$ and by the transitivity of the order relation, the set $B$ is dominated by a singleton in $P$ and hence in bounded in $P$. This contradiction shows that $\add(P)\ge\mathfrak b$. Combined with the trivial inequality $\add(P)\le\cof(P)$, we complete the proof of the embedding $\add(P)\in\{1,\w\}\cup[\mathfrak b,\mathfrak d]$.
\end{proof}

Proposition~\ref{p:add=cof} and Lemma~\ref{l:b-bound} imply

\begin{corollary}\label{c:P=b=d} Assume that $\mathfrak b=\mathfrak d$. If an $\w^\w$-dominated poset $P$ has $\add(P)>\w$, then $P\cong\mathfrak b=\mathfrak d$.
\end{corollary}

Several times we shall use the following fact:

\begin{lemma}\label{l:loc-bound} For every monotone function $f:\w^\w\to P$ into a poset $P$ with $\cof(P)\le\w$ and every $\alpha\in\w^\w$ there exists $k\in\w$ such that the set $f[{\uparrow}(\alpha|k)]=\{f(\beta):\beta\in{\uparrow}(\alpha|k)\}$ is bounded in $P$.
\end{lemma}

\begin{proof} Fix a countable cofinal set $\{y_n\}_{n\in\w}$ in $P$. Assuming that for every $k\in\w$ the set $f[{\uparrow}(\alpha|k)]$ is unbounded in $P$, we can find a sequence $\beta_k\in\w^\w$ such that $\beta_k|k=\alpha|k$ and $f(\beta_k)\not\le y_k$. Consider the function $\beta\in\w^\w$ defined by $$\beta(i)=\max(\{\alpha(i)\}\cup\{\beta_k(i):k+1\le i\})\mbox{ \  for every $i\in\w$}.$$

We claim that $\beta\ge \beta_n$ for all $n\in\w$. Given any $i\in\w$ we should check that $\beta(i)\ge \beta_n(i)$. If $n\le i+1$, then $\beta(i)=\max\{\beta_k(i):k\le i+1\}\ge \beta_n(i)$ by the definition of $\beta(i)$. If $n>i+1$, then $\beta_n(i)=\alpha(i)\le \beta(i)$.

Therefore, $\beta\ge\beta_n$ and hence $f(\beta)\ge f(\beta_n)$. Taking into account that $f(\beta_n)\not\le y_n$, we conclude that $f(\beta)\not\le y_n$ for all $n\in\w$. But this contradicts the choice of the dominating sequence $(y_n)_{n\in\w}$.
\end{proof}

%\begin{lemma}\label{l:Cu} Let $X$ be a uniform space, $Z$ be a dense subset of $X$ and $D\subset\IR^Z$ be a subset dominating the set $C_u(Z)$ in $\IR^Z$.
%\begin{enumerate}
%\item There exists a monotone map $K:\U(X)^\w\times D\to\K(C_u(X))$ such that $C_u(X)=\bigcup\{K(P):P\in\U(X)^\w\times D\}$.
%\item $\U(X)^\w\times D\succcurlyeq C_u(X)$.
%\end{enumerate}
%\end{lemma}

%\begin{proof} Consider the monotone map $K:\U(X)^\w\times D\to\K(C_u(X))$ assigning to each pair $P=\big((U_n)_{n\in\w},\varphi\big)\in\U(X)^\w\times D$ the compact subset
%$$K(P)=\bigcap_{z\in Z}\{f\in \IR^X:|f(z)|\le|\varphi(z)|\}\cap\bigcap_{n\in\w}\bigcap_{x,y\in U_n}\{f\in \IR^X:|f(x)-f(y)|\le \tfrac1{2^n}\big\}$$of $\IR^X$ and observe that $K(P)\subset C_u(X)$ and moreover $C_u(X)=\bigcup\{K(P):P\in\U(X)^\w\times D\}$.

%Next, for every $P\in\U(X)^\w\times D$ consider the function $\sup K(P):X\to \IR$, $\sup K(P):x\mapsto\sup\{f(x):f\in K(P)\}$ and observe that $\sup K(P)\in C_u(X)$ and moreover, the map $\sup K:\U(X)^\w\times D\to C_u(X)$, $\sup K:P\mapsto \sup K(P)$, is monotone and cofinal, witnessing that $\U(X)^\w\times D\succcurlyeq C_u(X)$.
%\end{proof}

\section{$\w^\w$-Dominated cardinals and their powers}\label{s:Set}

In this section we detect cardinals $\kappa,\lambda$ with  $\w^\w$-dominated power $\kappa^\lambda$. It is clear that the $\w^\w$-dominacy of $\kappa^\lambda$ implies the $\w^\w$-dominacy of the power $\kappa^\mu$ for any cardinal $\mu\le\lambda$. This observation suggests to consider the function $\mathfrak e$ assigning to each cardinal $\kappa$ the  cardinal $\mathfrak e(\kappa)=\min\{\lambda:\w^\w\not\succcurlyeq \kappa^\lambda\}$. The definition of the function $\mathfrak e$ implies that the poset $\kappa^{\mathfrak e(\kappa)}$ is not $\w^\w$-dominated but for any cardinal $\lambda<\mathfrak e(X)$ the power $\kappa^\lambda$ is $\w^\w$-dominated.

The following proposition describes some properties of the function $\mathfrak e$.

\begin{proposition}\label{p:e(X)} Let $\kappa$ be an infinite cardinal. Then
\begin{enumerate}
\item $\mathfrak e(\kappa)=\mathfrak e(\cf(\kappa))$.
\item If $\cf(\kappa)=\w$, then $\mathfrak e(\kappa)=\w_1$.
\item If $\mathfrak e(\kappa)>1$, then $\mathfrak e(\kappa)\ge\cf(\kappa)\in  \{\w\}\cup[\mathfrak b,\mathfrak d]$.
\item $\mathfrak e(\kappa)\in\{1,\w_1\}\cup[\mathfrak b,\mathfrak d]$.
\end{enumerate}
\end{proposition}

\begin{proof} 1. The equality $\mathfrak e(\kappa)=\mathfrak e(\cf(\kappa))$ follows from the (obvious) relation $\kappa\cong\cf(\kappa)$.
\smallskip

2. Assuming that $\cf(\kappa)=\w$ and $\mathfrak e(X)>\w_1$, we conclude that $\w^\w\succcurlyeq\w^{\w_1}$. This allows us to find a cofinal subset $\{\varphi_\alpha\}_{\alpha\in\w^\w}\subset \w^{\w_1}$ such that $\varphi_\alpha\le\varphi_\beta$ for any $\alpha\le\beta$ in $\w^\w$. Then the family $(K_\alpha)_{\alpha\in\w^\w}$ of compact subsets $$K_\alpha=\big\{\varphi\in \IR^{\w_1}:|\varphi|\le\max\{0,\varphi_\alpha\}\big\},\;\;\alpha\in\w^\w,$$is a compact resolution of the space $\IR^{\w_1}=C_p(\w_1)$ (here the cardinal $\w_1$ is endowed with the discrete topology).
 By Corollary 2.4 \cite{COT}, the function space $C_p(\w_1) $ is $K$-analytic and hence Lindel\"of. By Theorem 9.17 of \cite{kak}, the discrete space $\w_1$ is $\sigma$-compact, which is a desired contradiction, showing that $\w^\w\not\succcurlyeq\w^{\w_1}$.
\smallskip

3. Assuming that $\mathfrak e(\kappa)>1$, we conclude that
$\w^\w\succcurlyeq \kappa^1\cong\cf(\kappa)$ and then $\cof(\kappa)\in\{\w\}\cup[\mathfrak b,\mathfrak d]$ by Lemma~\ref{l:b-bound}.
The regularity of the cardinal $\cf(\kappa)$ implies that $\cf(\kappa)^\lambda\cong\cf(\kappa)$ for every cardinal $\lambda<\cf(\kappa)$ and hence $\mathfrak e(\kappa)=\mathfrak e(\cf(\kappa))\ge\cf(\kappa)$.
\vskip3pt

4. If $\mathfrak e(\kappa)>1$, then $\w^\w\succcurlyeq \kappa^1$ and hence $\w^\w\cong(\w^\w)^\w\succcurlyeq \kappa^\w$. This implies that $\mathfrak e(\kappa)\ge\w_1$.

If $\mathfrak e(\kappa)>\w_1$, then $\w^\w\succcurlyeq \kappa^{\w_1}\succcurlyeq \kappa$ and hence $\cof(\kappa)\in\{\w\}\cup[\mathfrak b,\mathfrak d]$ by  Lemma~\ref{l:b-bound}.
Taking into account that $\mathfrak e(\w)=\w_1<\mathfrak e(\kappa)=\mathfrak e(\cf(\kappa))$, we conclude that $\cf(\kappa)\ne\w$ and hence $\cf(\kappa)\in[\mathfrak b,\mathfrak d]$. By the preceding statement, $\mathfrak e(\kappa)\ge\cf(\kappa)\ge\mathfrak b$.

 Finally, we prove that $\mathfrak e(\kappa)\le \mathfrak d$. This inequality will follow as soon as we check that $\w^\w\not\succcurlyeq \kappa^{\mathfrak d}$.
Find a dominating subset $\{x_\alpha\}_{\alpha\in\mathfrak d}\subset \w^\w$. For every $\alpha\in\mathfrak d$, consider the projection $\pr_\alpha:\kappa^{\mathfrak d}\to\kappa$, $\pr_\alpha:\varphi\mapsto\varphi(\alpha)$.
Assuming that $\mathfrak e(\kappa)>\mathfrak d$, we can find a cofinal monotone map $f:\w^\w\to\kappa^{\mathfrak d}$. Choose a function $y\in\kappa^{\mathfrak d}$ such that $y(\alpha)>\pr_\alpha\circ f(\varphi_\alpha)$ for every $\alpha\in\mathfrak d$. By the cofinality of $f$, there exists a function $x\in\w^\w$ such that $f(x)\ge y$.
By the choice of the dominating set $\{x_\alpha\}_{\alpha\in\mathfrak d}$, there exists an ordinal $\alpha\in\kappa$ such that $x\le x_\alpha$. Then $f(x_\alpha)\ge f(x)\ge y$ and hence $\pr_\alpha(f(x_\alpha))\ge y(\alpha)>\pr_\alpha(f(x_\alpha))$, which is a contradiction completing the proof of the inequality $\mathfrak e(\kappa)\le\mathfrak d$.
\end{proof}

Now we give some examples of uncountable cardinals $\kappa$ with $\mathfrak e(\kappa)>1$. This inequality holds if and only if the cardinal $\kappa$ is $\w^\w$-dominated. Taking into account that $\kappa\cong\cf(\kappa)$, we get the following lemma reducing the problem of detecting $\w^\w$-dominated cardinal to the case of regular uncountable cardinals.

\begin{lemma}\label{l:cf} A cardinal $\kappa$ is $\w^\w$-dominated if and only if its cofinality $\cf(\kappa)$ is $\w^\w$-dominated.
\end{lemma}

Given a filter $\F$ on $\w$, consider the partial preorder $\le_\F$ on $\w^\w$ defined by $f\le_\F g$ iff $\{n\in\w:f(n)\le g(n)\}\in\F$ for $f,g\in\w^\w$. Taking into account that the identity map $\w^\w\to(\w^\w,\le_\F)$ is monotone and cofinal, we can apply Proposition~\ref{p:add=cof} and get

\begin{corollary}\label{c:bd} For any filter $\F$ on $\w$ the cardinals $\add(\w^\w,\le_\F)$, $\cof(\w^\w,\le_\F)$, and $\cf(\cof(\w^\w,\le_\F))$ are $\w^\w$-dominated.
\end{corollary}

Taking into account that $\le^*=\le_{\F}$ for the \index{filter!Fr\'echet}{\em Frech\'et filter} $\F=\{A\subset \w:|\w\setminus A|<\w\}$ on $\w$, we conclude that Corollary~\ref{c:bd} imply the following known fact (see \cite{LPT}, \cite{Mama}).

\begin{corollary}\label{c:bdcf(d)} The cardinals $\mathfrak b$, $\mathfrak d$, $\cf(\mathfrak d)$ are $\w^\w$-dominated.
\end{corollary}

\begin{remark} If $\F$ is an ultrafilter on $\w$, then any two elements $x,y\in\w^\w$ are comparable in the poset $(\w^\w,\le_{\F})$, which implies that $\add(\w^\w,\le_\F)=\cof(\w^\w,\le_\F)$. By \cite{Canjar}, there exists a free ultrafilter $\F$ on $\w$ such that  $\add(\w^\w,\le_\F)=\cof(\w^\w,\le_\F)=\cf(\mathfrak d)$. More information on the cardinals $\mathfrak b(\w^\w,\le_\F)$ and $\cof(\w^\w,\le_\F)$ for various filters $\F$ can be found in \cite{BM} and  \cite[Ch.3]{BZ}.
\end{remark}

For the cardinal $\kappa=\w$ the definition of the cardinal $\mathfrak e(\kappa)$ ensures that $\mathfrak e(\kappa)>\kappa$. It is natural to ask if $\w$ is a unique infinite regular cardinal $\kappa$ such that $\mathfrak e(\kappa)>\kappa$. To answer this question, let us consider the cardinal
$$\mathfrak e^\sharp=\sup\{\kappa^+:\w\le \cf(\kappa)=\kappa<\mathfrak e(\kappa)\}.$$

\begin{proposition}\label{p:e-bound} $\mathfrak e^\sharp\in\{\w_1\}\cup(\mathfrak b,\mathfrak d]$.
\end{proposition}

\begin{proof} If $\kappa$ is a regular infinite cardinal with $\w^\w\succcurlyeq \kappa^\kappa$, then $\w^\w\succcurlyeq \kappa$ and by Lemma~\ref{l:b-bound}, $\kappa=\cf(\kappa)\in\{\w\}\cup[\mathfrak b,\mathfrak d]$. Combining this fact with (the proof of) Proposition~\ref{p:e(X)}(4), we conclude that $\kappa\in\{\w\}\cup[\mathfrak b,\mathfrak d)$ and hence $\kappa^+\in\{\w_1\}\cup(\mathfrak b,\mathfrak d]$. Then $\mathfrak e^\sharp\in\{\w_1\}\cup(\mathfrak b,\mathfrak d]$, too.
\end{proof}

Proposition~\ref{p:e-bound} implies:

\begin{corollary}\label{c:b=d} If $\mathfrak b=\mathfrak d$, then $\mathfrak e^\sharp=\w_1$.
\end{corollary}

It is natural to ask if $\mathfrak e^\sharp=\w_1$ implies $\mathfrak b=\mathfrak d$. The answer to this question is negative and will be derived from the following deep result of Cummings and Shelah \cite{CS} (see also \cite[2.28]{Monk}).

\begin{theorem}[Cummings-Shelah]\label{t:CS} Let $F$ be a class function assigning to each regular cardinal $\kappa$ a triple $(\beta(\kappa),\delta(\kappa),\lambda(\kappa))$ of cardinals such that $$\kappa<\cf(\lambda(\kappa)),\;\;\kappa<\beta(\kappa)=\cf(\beta(\kappa))\le\cf(\delta(\kappa)\le \delta(\kappa)\le\lambda(\kappa)\mbox{ \; and \; }\lambda(\kappa)\le\lambda(\kappa')$$ for any regular cardinals $\kappa<\kappa'$. Then there is a class forcing poset $\mathbb P$ preserving all cardinals and cofinalities such that  in the generic extension $$\add(\kappa^\kappa,\le^*)=\beta(\kappa),\;\;
\cof(\kappa^\kappa,\le^*)=\delta(\kappa)\mbox{ \ and \ }2^\kappa=\lambda(\kappa)$$ for every regular infinite cardinal $\kappa$.
\end{theorem}

In this theorem by $\le^*$ we denote the preorder on $\kappa^\kappa$ defined by $f\le^* g$ iff $|\{n\in\kappa:f(n)\not\le g(n)\}|<\kappa$.

\begin{proposition} There exists a model of ZFC such that $\w_1=\mathfrak b<\mathfrak d=\w_2$ and $\mathfrak e^\sharp=\w_1$.
\end{proposition}

\begin{proof} By Theorem~\ref{t:CS}, there exists a model of ZFC such that $\mathfrak b=\w_1$, $\mathfrak d=\w_2$ and $\cof(\kappa^\kappa,\le^*)=\w_3$ for the cardinal $\kappa=\w_1$. We claim that $\mathfrak e^\sharp=\w_1$ in this model. Assuming that $\mathfrak e^\sharp>\w_1$, we could find an uncountable regular cardinal $\kappa$ such that $\w^\w\succcurlyeq \kappa^\kappa$, which implies $\cof(\kappa^\kappa,\le^*)\le \cof(\kappa^\kappa)\le\cof(\w^\w)=\mathfrak d$. Lemma~\ref{l:b-bound} and Proposition~\ref{p:e(X)} imply that $\kappa=\cf(\kappa)\in[\mathfrak b,\mathfrak d)=\{\w_1\}$. Then $\kappa=\w_1$ and $\cof(\kappa^\kappa,\le^*)=\w_3>\w_2=\mathfrak d$, which is a desired contradiction.
\end{proof}

Now we shall describe models of ZFC in which $\mathfrak e^\sharp>\w_1$.

\begin{definition} Let $\kappa$ be an infinite cardinal. A subset $L\subset \w^\w$ will be called \index{$\kappa$-Lusin set}{\em $\kappa$-Lusin} if for every $x\in\w^\w$ the set $L_x=\{y\in L:y\le x\}$ has cardinality $|L_x|<\kappa$.
\end{definition}

%The paper \cite{BH} describes a model of ZFC in which the space $\w^\w$ contains an $\w_1$-Lusin set of cardinality $\mathfrak c=2^{\w_1}$.

The following observation is due to Lyubomyr Zdomskyy\footnote{http://mathoverflow.net/questions/243365/cofinal-monotone-maps-from-omega-omega-to-kappa-kappa}.

\begin{proposition}[Zdomskyy]\label{p:zdomskyy} Let $\kappa$ be a regular infinite cardinal and $\lambda$ be an infinite cardinal. If $L\subset \w^\w$ is a $\kappa$-Lusin set of cardinality $|L|\ge\cof(\kappa^\lambda)$, then $\w^\w\succcurlyeq \kappa^\lambda$.
\end{proposition}

\begin{proof} Fix a dominating set $D\subset \kappa^\lambda$ of cardinality $|D|=\cof(\kappa^\lambda)$ and let $g:L\to D$ be a surjective map. For every $x\in\w^\w$ the set $L_x=\{z\in L:z\le x\}$ has cardinality $|L_x|<\kappa$ and hence its image $g(L_x)$ has supremum in the poset $\kappa^\lambda$. Then the formula $f(x)=\sup g(L_x)$ determines a well-defined monotone cofinal function $f:\w^\w\to\kappa^\lambda$, witnessing that $\w^\w\succcurlyeq \kappa^\lambda$.
\end{proof}

\begin{theorem}\label{t:e>w1} In some model of ZFC  the space $\w^\w$ contains an $\w_1$-Lusin set of cardinality $\mathfrak c=\w_2=2^{\w_1}$. In this model $\w^\w\succcurlyeq \w_1^{\w_1}$ and $\mathfrak e^\sharp=\w_2=\mathfrak d>\mathfrak b=\w_1$.
\end{theorem}

\begin{proof} The paper \cite{BH} (see also \cite[Ch.27]{Halb}) describes a model of ZFC in which the space $\w^\w$ contains an $\w_1$-Lusin set of cardinality $\mathfrak c=\w_2$. By \cite[VII.5.13]{Kunen} in this model $2^{\w_1}=\mathfrak c$. By  Proposition~\ref{p:zdomskyy}, $\w^\w\succcurlyeq (\w_1)^{\w_1}$ and hence $\w_1<\mathfrak e^\sharp=\w_2=\mathfrak c=2^{\w_1}$. By Corollary~\ref{c:b=d}, the strict inequality $\w_1<\mathfrak e^\sharp$ implies $\mathfrak b<\mathfrak d$. Taking into account that $\mathfrak c=\w_2$, we conclude that $\mathfrak b=\w_1$ and $\mathfrak d=\w_2$.
\end{proof}

\chapter{Preuniform and uniform spaces}\label{Ch:pu}

In this chapter we recall the necessary information related to preuniform spaces.

\section{Entourages}

By an \index{entourage}{\em entourage} on a set $X$ we understand any subset $U\subset X\times X$ containing the diagonal $\Delta_X=\{(x,y)\in X\times X:x=y\}$ of the square $X\times X$. For an entourage $U\subset X\times X$ and a point $x\in X$ the set $U[x]=\{y\in X:(x,y)\in U\}$ is called the \index{$U$-ball}\index{$U[x]$}{\em $U$-ball} centered at $x$. For a subset $A\subset X$ the set $U[A]=\bigcup_{a\in A}U[a]$ is the {\em $U$-neighborhood} of $A$.

For two entourages $U,V$ on $X$ let $UV=\{(x,z):\exists y\in X,\;(x,y)\in U,\;(y,z)\in V\}$ be their composition and $U^{-1}=\{(y,x):(x,y)\in U\}$ be the inverse entourage to $U$.

For an entourage $U$ on $X$ its powers $U^n$ an $U^{-n}$ are defined by induction: $U^{0}=\Delta_X$ and $U^{(n+1)}=UU^n$, $U^{-(n+1)}=U^{-1}U^{-n}$ for $n\in\w$. The alternating powers $U^{\pm n}$ and $U^{\mp n}$ of $U$ also are defined by induction: $U^{\pm 0}=U^{\mp 0}=\Delta_X$ and $U^{\pm (n+1)}=U U^{\mp n}$ and $U^{\mp(n+1)}=U^{-1}U^{\pm n}$ for $n\in\w$.

So, $U^{\pm 1}=U$, $U^{\mp 1}=U^{-1}$, $U^{\pm 2}=UU^{-1}$, $U^{\mp 2}=U^{-1}U$, $U^{\pm 3}=UU^{-1}U$, etc.

An entourage $U$ on a topological space $X$ is called a \index{neighborhood assignment}{\em neighborhood assignment} if for every $x\in X$ the $U$-ball $U[x]$ is a neighborhood of $x$.

\section{Uniformities, quasi-uniformities, preuniformities}

A \index{preuniformity}{\em preuniformity} on a set $X$ is any filter of entourages on $X$, i.e., a family $\U$ of entourages on $X$, satisfying two axioms:
\begin{itemize}
\item[(U1)] for any $U,V\in\U$ the intersection $U\cap V$ belongs to $\U$;
\item[(U2)] for any entourages $U\subset V$ on $X$ the inclusion $U\in\U$ implies $V\in\U$.
\end{itemize}
A subfamily $\mathcal B\subset \U$ is called a \index{preuniformity!base of}\index{base of a preuniformity}{\em base} of a preuniformity $\U$ if for each $U\in\U$ there exists $B\in\mathcal B$ such that $B\subset U$. It is clear that each base of a preuniformity satisfies the axiom (U1), and each family $\mathcal B$ of entourages on a set $X$ satisfying the axiom (U1) is a base of the unique preuniformity ${\uparrow}\mathcal B=\{U\subset X\times X:\exists B\in\mathcal B\;\;B\subset U\}$.

 A preuniformity $\U$ on a set $X$ is called
\begin{itemize}
\item \index{preuniformity!quasi-uniform}{\em quasi-uniform} (or else a {\em quasi-uniformity}) if for each $U\in\U$ there is $V\in\U$ such that $VV\in\U$;
\item \index{preuniformity!uniform}{\em uniform} (or else a {\em uniformity}) if for each $U\in\U$ there is $V\in\U$ such that $VV\subset U$ and $V^{-1}\subset U$.
\end{itemize}

A \index{preuniform space}{\em preuniform space} is a pair $(X,\U_X)$ consisting of a set $X$ and a preuniformity $\U_X$ on $X$. A preuniform space $(X,\U_X)$ is called a \index{quasi-uniform space}\index{uniform space}({\em quasi-}){\em uniform space} if its preuniformity $\U_X$ is  (quasi-)uniform. In the sequel the preuniformity of a preuniform space $X$ will be denoted by $\U_X$.

Each subset $A\subset X$ of a preuniform space $(X,\U_X)$ carries the induced preuniformity $\U_A=\{U\cap(A\times A):U\in\U_X\}$. The preuniform space $(A,\U_A)$ is called a \index{preuniform space!subspace of}{\em  subspace} of the preuniform space $(X,\U_X)$.

Each preuniformity $\U$ on a set $X$ generates a topology on $X$, which consists of all subsets $W\subset X$ such that for every $x\in W$ there is $U\in\U$ such that $U[x]\subset W$.

\begin{definition} A preuniformity $\U$ on a set $X$ is called \index{preuniformity!topological}{\em topological} if for each point $x\in X$ the family $(U[x])_{x\in X}$ is a neighborhood base at $x$ in the topology generated by the preuniformity.
\end{definition}

  It is easy to see that each quasi-uniform preuniformity is topological.

\begin{example} For a topological space $X$ the family $p\U$ of all neighborhood assignments is a topological preuniformity, called the \index{preuniformity!universal}\index{universal preuniformity}\index{topological space!universal preuniformity of}{\em universal preuniformity} of the topological space $X$.
\end{example}

\begin{proposition}\label{p:qut} Let $(A,\U_A)$ be a subspace of a topological preuniform space $(X,\U_X)$. Then the topology $\tau_A$ on $A$ generated by the preuniformity $\U_A$ coincides with the subspace topology $\tau_X|A=\{U\cap A:U\in\tau_X\}$ induced from the topology $\tau_X$, generated by the preuniformity $\U_X$. Consequently, the preuniform space $(A,\U_A)$ is topological.
\end{proposition}

\begin{proof} It is clear that $\tau_X|A\subset\tau_A$. To prove the inclusion $\tau_A\subset \tau_X|A$, fix any open set $W\in\tau_A$ and for every point $x\in W$ find an entourage $V_x\in\U_A$ such that $V_x[x]\subset W$. By the definition of the preuniformity $\U_A$, there exists an entourage $U_x\in\U_X$ such that $V_x=U_x\cap(A\times A)$. Since the preuniformity $\U_X$ is topological, the interior $U_x[x]^\circ$ of the $U_x$-ball $U_x[x]$ in $X$ contains $x$ and hence $\widetilde{W}=\bigcup_{x\in W}U_x[x]^\circ$ is an open subset of $X$ such that $\widetilde{W}\cap A=W$.
\end{proof}

%The subfamilies
% $$
% \begin{aligned}
% q\U_X&=\bigcup\{\U:\U\subset p\U\mbox{ is a quasi-uniformity}\} \mbox{ and}\\
%\U_X&=\bigcup\{\U:\U\subset p\U\mbox{ is a uniformity}\}
%\end{aligned}
%$$are called the {\em universal quasi-uniformity} and the {\em universal uniformity} of $X$.
%It is clear that $p\U_X$, $q\U_X$, $\U_X$ are topological %preuniformities. The preuniformities $p\U_X$, $q\U_X$ generate the %topology of the space $X$ and the universal uniformity $\U_X$ generates the topology of $X$ if and only if $X$ is {\em $\IR$-regular} in the sense that for any point $x\in X$ and a neighborhood $O_x\subset X$ there exists a continuous function $f:X\to[0,1]$ such that $f(x)=0$ and $f^{-1}([0,1))\subset O_x$.

A preuniformity $\U$ on a topological space $X$ is called \index{preuniformity!basic}{\em basic} if $\U$ is topological and generates the topology of $X$. This happens if and only if for every $x\in X$ the family $\{U[x]:U\in\U\}$ is a neighborhood base at $x$. It is easy to see that each basic preuniformity on a topological space $X$ is contained in its universal preuniformity $p\U_X$.

A preuniformity $\U$ on a set $X$ is called \index{preuniformity!normal}{\em normal} if $\overline{A}\subset \overline{U[A]}^\circ$ for any subset $A\subset X$ and any entourage $U\in\U$. Here the closure and the interior is taken in the topology generated by the preuniformity. By \cite{BR17}, each uniformity is normal. In \cite{BR16q} normal preuniformities are called {\em set-rotund}.

The following theorem was proved by Banakh and Ravsky \cite{BR17}.

\begin{theorem}[Banakh, Ravsky]\label{t:BR17} Assume that the topology of a topological space $X$ is generated by a normal quasi-uniformity $\U$. Then for every set $A\subset X$ and entourage $U\in\U$ there exists a continuous function $f:X\to[0,1]$ such that $$A\subset f^{-1}(0)\subset f^{-1}\big[[0,1)\big]\subset \overline{U[A]}^\circ.$$ This implies that the space $X$ is
\begin{enumerate}
\item Hausdorff (at a point $x\in X$) iff $X$ is semi-Hausdorff (at $x$) iff $X$ is functionally Hausdorff (at $x$);
\item regular (at a point $x\in X$) iff $X$ is semi-regular (at $x$) iff $X$ is completely regular (at $x$).
\end{enumerate}
\end{theorem}

Each preuniformity $\U$ will be considered as a poset endowed with the partial order $\le$ of reverse inclusion defined by $U\le V$ iff $V\subset U$.

Given a poset $P$ we shall say that a preuniformity $\U$ has a \index{preuniformity!$P$-base of}\index{$P$-base!of a preuniformity}{\em $P$-base} if $\U$ has a base $\{U_\alpha\}_{\alpha\in P}$ such that $U_\beta\subset U_\alpha$ for all $\alpha\le\beta$ in $P$. Lemma~\ref{l:Tuckey} implies that a preuniformity $\U$ has a $P$-base iff $P\succcurlyeq \U$ iff $\U\le_T P$.

%\begin{example} Let $P$ be a directed posed. Any $P$-base $\mathcal B=\{U_\alpha\}_{\alpha\in P}$ for a topological space $X$ is a $P$-base of the (unique) topological preuniformity $${\uparrow}\mathcal B=\{U\subset X\times X:\exists \alpha\in P\;\;\;(U_\alpha\subset U)\}.$$
%\end{example}

\section{The canonical (quasi-)uniformity on a preuniform space}\label{s:cu}

For a preuniform space $(X,\U_X)$ its \index{preuniform space!canonical uniformity of}\index{preuniform space!canonical quasi-uniformity of}{\em canonical} ({\em quasi-}){\em uniformity} is the largest (quasi-)uniformity, contained in $\U_X$. This (quasi-)uniformity admits the following constructive description.

For a sequence $(U_n)_{n\in\w}$ of entourages
on $X$ define their \index{$\otimes_{n\in\w}U_n$}{\em permutation product} $\bigotimes_{n\in\w}U_n$
by the formula
$$\textstyle{\bigotimes}_{n\in\w}U_n=\bigcup_{n\in\w}\bigcup_{\sigma\in S_n}U_{\sigma(0)}\cdots U_{\sigma(n-1)},$$
where $S_n$ is the group of bijections of the ordinal $n=\{0,\dots,n-1\}$.

\begin{definition}\label{d:cqu+cu} Let $(X,\U_X)$ be a preuniform space.
\begin{itemize}
\item The preuniformity $\U_X^{+\w}$ on $X$, generated by the base $\big\{\bigotimes_{n\in\w}U_n:(U_n)_{n\in\w}\in(\U_X)^\w\big\}$, is called the \index{preuniform space!canonical quasi-uniformity of}\index{canonical quasi-uniformity}{\em canonical quasi-uniformity} of the preuniform space $(X,\U_X)$.
\item The preuniformity $\U_X^{\pm\w}$ on $X$, generated by the base $\big\{\bigotimes_{n\in\w}(U\cup U^{-1}):(U_n)_{n\in\w}\in(\U_X)^\w\big\}$, is called the \index{preuniform space!canonical uniformity of}\index{canonical uniformity} {\em canonical uniformity} of the preuniform space $(X,\U_X)$.
\end{itemize}
\end{definition}

It can be shown that the canonical (quasi-)uniformity of a preuniform space $X$ coincides with the largest (quasi-)uniformity, contained in the preuniformity of $X$.

Definition~\ref{d:cqu+cu} implies the following proposition.

\begin{proposition}\label{p:cu-Pb} Let $P$ be a poset. If the preuniformity $\U_X$ of a preuniform space $X$ has a $P$-base, then the canonical quasi-uniformity $\U_X^{+\w}$ and the canonical uniformity $\U_X^{\pm\w}$ of $X$ both have $P^\w$-bases.
\end{proposition}

The canonical uniformity of a preuniform space can be equivalently defined using uniform pseudometrics. A pseudometric $d:X\times X\to\IR$ on a preuniform space $(X,\U_X)$ is called \index{preuniform space!uniform}\index{uniform space}{\em uniform} if for every $\e>0$ the entourage $[d]_{<\e}=\{(x,y)\in X\times X:d(x,y)<\e\}$ belongs to the preuniformity $\U_X$. By $\PM_u(X)$ we shall denote the family of uniform pseudometrics on $X$. Using Theorem 8.1.10 \cite{Eng}, it can be shown that for a preuniform space $X$ the family
$\{[d]_{<1}:d\in \PM_u(X)\}$ is a base of the canonical uniformity $\U_X^{\pm \w}$ of $X$.

Recall that for a topological space $X$ the family $p\U_X$ of all neighborhood assignments on $X$ is called the \index{preuniformity!universal}\index{universal preuniformity}\index{topological space!universal preuniformity of}\index{$p\U_X$}{\em universal preuniformity} of $X$. The canonical (quasi-)uniformity $p\U_X^{\pm\w}$ (resp. $p\U_X^{+\w}$)\index{quasi-uniformity!universal}\index{universal quasi-uniformity}\index{topological space!universal quasi-uniformity of}\index{$q\U_X$}
 \index{uniformity!universal}\index{universal uniformity}\index{topological space!universal uniformity of}\index{$\U_X$}of the preuniform space $(X,p\U_X)$ is called the {\em universal} ({\em quasi-}){\em uniformity} on $X$ and is denoted by $\U_X$ (resp. $q\U_X$), see \cite[\S3]{BR16}.

\begin{proposition}\label{p:u-sub} Let $Z$ be a subspace of a topological space $X$. Then
\begin{enumerate}
\item[\textup{(1)}] $p\U_Z=\{U\cap (Z\times Z):U\in p\U_X\}$ and hence $p\U_X\succcurlyeq p\U_Z$;
\item[\textup{(2)}] If $Z$ is closed in $X$, then $q\U_Z=\{U\cap (Z\times Z):U\in q\U_X\}$ and hence $q\U_X\succcurlyeq q\U_Z$;
\item[\textup{(3)}] If $Z$ is closed in $X$ and $X$ is paracompact, then $\U_Z=\{U\cap (Z\times Z):U\in \U_X\}$ and hence $\U_X\succcurlyeq \U_Z$.
\end{enumerate}
\end{proposition}

\begin{proof} 1. The equality  $p\U_Z=\{U\cap (Z\times Z):U\in p\U_X\}$ follows from the observation that for every neighborhood assignment $V\in p\U_Z$ there exists a neighborhood assignment $U\in p\U_X$ such that $U\cap (Z\times Z)=V$. This implies that the monotone map $p\U_X\to p\U_Z$, $U\mapsto U\cap(Z\times Z)$, is cofinal, witnessing that $p\U_X\succcurlyeq p\U_Z$.
\smallskip

2. Assume that the subspace $Z$ is closed in $X$. To prove that $q\U_Z=\{U\cap (Z\times Z):U\in q\U_X\}$, it suffices to find for every entourage $V\in q\U_Z$ an entourage $U\in q\U_X$ such that $U\cap(Z\times Z)\subset U$. By the definition of the universal quasi-uniformity $q\U_X$, there exists a sequence of entourages $(V_n)_{n\in\w}\in p\U_Z$ such that $\bigotimes_{n\in\w}V_n\subset V$. For every $n\in\w$ find a neighborhood assignment $U_n\in p\U_X$ such that $U_n[z]\cap Z=V_n[z]$ for every $z\in Z$ and $U_n[z]\subset X\setminus Z$ for every $x\in X\setminus Z$. Consider the neighborhood assignment $U=\bigotimes_{n\in\w}U_n\in q\U_X$ and observe that $U\cap(Z\times Z)=\bigotimes_{n\in\w}V_n\subset V$.

The monotone cofinal map $q\U_X\to q\U_Z$, $U\mapsto U\cap(Z\times Z)$, witnesses that $p\U_X\succcurlyeq p\U_Z$.
\smallskip

3. Assume that the space $X$ is paracompact and the subspace $Z\subset X$ is closed. To prove that $\U_Z=\{U\cap (Z\times Z):U\in \U_X\}$, take any entourage $E\in\U_Z$ and by the paracompactness of $Z$, find an open cover $\V$ of $Z$ such that $V\times V\subset E$ for all $V\in\V$. Next, find an open cover $\U$ of $X$ such that for every $U\in\U$ the set $U\cap Z$ is either empty or belongs to the cover $\V$. By the paracompactness of $X$, the entourage $W=\bigcup_{U\in\U}U\times U$ belongs to the universal uniformity $\U_X$ of $X$. The choice of the cover $\U$ guarantees that $W\cap(Z\cap Z)=\bigcup_{V\in\V}V\times V\subset E$. This implies that $\{U\cap(Z\times Z):U\in\U_X\}\subset \U_Z$. The reverse inclusion is obvious. The monotone cofinal map $\U_X\to \U_Z$, $U\mapsto U\cap(Z\times Z)$, witnesses that $\U_X\succcurlyeq \U_Z$.
\end{proof}

\section{Locally (quasi-)uniform preuniformities}

\begin{definition}\label{d:lu+lqu} A preuniformity $\U$ on a set $X$ is called
\begin{itemize}
\item \index{preuniformity!locally quasi-uniform}{\em a locally quasi-uniform} if for every point $x\in X$ and neighborhood $O_x\subset X$ of $x$ there is an entourage $U\in\U$ such that $UU[x]\subset O_x$;
\item \index{preuniformity!locally uniform}{\em locally uniform} if for every point $x\in X$ and neighborhood $O_x\subset X$ of $x$ there is an entourage $U\in\U$ such that $UU^{-1}U[x]\subset O_x$;
\item \index{preuniformity!locally $\infty$-quasi-uniform}{\em locally $\infty$-quasi-uniform} if for any point $x\in X$ and neighborhood $O_x\subset X$ there exists a sequence of entourages $(U_n)_{n\in\w}\in\U^\w$ such that $\bigcup_{n\in\w}\bigcup_{\sigma\in S_n}U_{\sigma(0)}\cdots U_{\sigma(n-1)}[x]\subset O_x$;
\item \index{preuniformity!locally $\infty$-uniform}{\em locally $\infty$-uniform} if for any point $x\in X$ and neighborhood $O_x\subset X$ there exists a sequence of entourages $(U_n)_{n\in\w}\in\U^\w$ such that $\bigcup_{n\in\w}\bigcup_{\e\in\{-1,1\}^n}\bigcup_{\sigma\in S_n}U^{\e(0)}_{\sigma(0)}\cdots U_{\sigma(n-1)}^{\e(n-1)}[x]\subset O_x$.
\end{itemize}
Here $X$ is endowed with the topology generated by the preuniformity $\U$.
\end{definition}

It is clear that for any preuniformity we have the implications:
$$
\xymatrix{
\mbox{uniform}\ar@{=>}[r]\ar@{=>}[d]&\mbox{locally $\infty$-uniform}\ar@{=>}[r]\ar@{=>}[d]&\mbox{locally uniform}\ar@{=>}[d]\\
\mbox{quasi-uniform}\ar@{=>}[r]&\mbox{locally $\infty$-quasi-uniform}\ar@{=>}[r]& \mbox{locally~quasi-uniform}.
}
$$

\begin{definition}\label{d:lqu} A preuniform space $X$ is called
\index{preuniform space!locally uniform}\index{preuniform space!locally quasi-uniform}{\em  locally} ({\em quasi-}){\em uniform} if its preuniformity is topological and locally (quasi-)uniform.
\end{definition}

The following simple (but useful) proposition can be derived from Definition~\ref{d:lu+lqu} by induction on $n$.

\begin{proposition}\label{p:p-lqu+qu} A topological preuniformity $\U$ on a set $X$ is
\begin{enumerate}
\item[\textup{(1)}] locally quasi-uniform if and only if for every $n\in\IN$, point $x\in X$ and neighborhood $O_x\subset X$ of $x$ there exists an entourage $U\in\U$ such that $U^n[x]\subset O_x$;
\item[\textup(2)] locally uniform if and only if
for every integer $n\ge 3$, point $x\in X$ and neighborhood $O_x\subset X$ of $x$ there exists an entourage $U\in\U$ such that $U^{\pm n}[x]\subset O_x$;
\item[\textup(3)] locally $\infty$-quasi-uniform if and only if the canonical quasi-uniformity $\U^{+\w}$ generates the topology of $X$;
\item[\textup(4)] locally $\infty$-uniform if and only if the canonical uniformity $\U^{\pm\w}$ generates the topology of $X$.
\end{enumerate}
\end{proposition}

The following proposition can be easily derived from Definition~\ref{d:lqu} and Proposition~\ref{p:qut}.

\begin{proposition}\label{p:lqu-subspace} Each subspace $A$ of a locally (quasi-)uniform space $X$ is locally (quasi-)uniform.
\end{proposition}

\begin{theorem}\label{t:c-lqu=lu} Assume that $(X,\U_X)$ is a preuniform space whose preuniformity is topological and generates a compact Hausdorff topology on $X$. The preuniformity $\U_X$ is locally uniform if and only if it is locally quasi-uniform.
\end{theorem}

\begin{proof} The ``only if'' part is trivial. To prove the ``if'' part, assume that the preuniformity $\U_X$ is locally quasi-uniform. To prove that $\U_X$ is locally uniform, fix any point $x\in X$ and an open neighborhood $O_x\subset X$ of $x$. Since the preuniformity $\U_X$ is locally quasi-uniform, there exists an entourage $V\in\U_X$ such that $V^2[x]\subset O_x$. For every $y\in X\setminus V[x]^\circ$ use the Hausdorff property of the space $X$ and the local quasi-uniformity of the preuniformity $\U_X$ to find an entourage $U_y\in\U_X$ such that $U_y^2[y]\cap U_y[x]=\emptyset$. By the compactness of $X\setminus V[x]^\circ$ there is a finite subset $F\subset X\setminus V[x]^\circ$ such that $X\setminus V[x]^\circ\subset \bigcup_{y\in F}U_y[y]$. Then for the entourage $U=V\cap \bigcap_{y\in F}U_y\in\U_X$ we get $U[x]\cap U[X\setminus V[x]^\circ]=\emptyset$ and hence $U^{-1}U[x]\subset V[x]^\circ\subset V[x]$. Then $UU^{-1}U[x]\subset UV[x]\subset V^2[x]\subset O_x$, witnessing that the preuniformity $\U_X$ is  locally uniform.
\end{proof}

A quasi-uniformity on a compact Hausdorff space needs not be uniform.

\begin{example} For the convergent sequence $X=\{0\}\cup\{2^{-n}:n\in\w\}$ the universal quasi-uniformity $q\U_X$ fails to be a uniformity, see \cite[3.1]{BR16}.
\end{example}

The following example shows that a topological preuniformity generating a compact Hausdorff topology is not necessarily locally (quasi-)uniform.

\begin{example} On any infinite $T_1$-space $X$ consider the preuniformity $\U_X$ consisting of neighborhood assignments $U\subset X\times X$ such that $U[x]=X$ for all but finitely many points $x\in X$. It is easy to see that the preuniformity $\U_X$ is topological and generates the topology of $X$. Since $U^{-1}[x]=X$ for any $U\in\U_X$ and $x\in X$, the preuniformity $\U_X$ is not locally uniform. If $X$ is not discrete, then $\U_X$ is not locally quasi-uniform.
\end{example}

\section{Entourage bases for topological spaces}

\begin{definition} A family $\mathcal B$ of entourages on a topological space $X$ is called an \index{topological space!entourage base}{\em entourage base} for $X$ if $\mathcal B$ is a base of a topological preuniformity generating the topology of $X$.

A \index{based space}{\em based space} is a pair $(X,\Bas_X)$ consisting of a topological space $X$ and an entourage base $\Bas_X$ for $X$.
\end{definition}

Entourage bases for topological spaces can be characterized as follows.

\begin{proposition}  A family $\mathcal B$ of entourages on a topological space $X$ is an entourage base for $X$ if and only if has the following two properties:
\begin{enumerate}
\item for any entourages $U,V\in\mathcal B$ there exists an entourage $W\in\mathcal B$ such that $W\subset U\cap V$;
\item for every $x\in X$ the family $\mathcal B[x]:=\{B[x]:B\in\mathcal B\}$ is a neighborhood base at $x$ in the topological space $X$.
\end{enumerate}
\end{proposition}

An entourage base $\mathcal B$ for a topological space $X$ is called
\begin{itemize}
\item a \index{entourage base!symmetric}\index{base!symmetric}\index{symmetric base}{\em symmetric base} if $B=B^{-1}$ for each entourage $B\in\mathcal B$;
\item a \index{entourage base!symmetrizable}\index{base!symmetrizable}\index{symmetrizable base}{\em symmetrizable} if for every $x\in X$ the set $B^{-1}[x]$ is a neighborhood of $x$ in $X$;
\item  \index{entourage base!uniform}\index{base!uniform}\index{uniform base}
\index{entourage base!quasi-uniform}\index{base!quasi-uniform}\index{quasi-uniform base}
\index{entourage base!locally quasi-uniform}\index{base!locally quasi-uniform}\index{locally quasi-uniform base}
\index{entourage base!locally uniform}\index{base!locally uniform}\index{locally uniform base}
\index{entourage base!locally $\infty$-quasi-uniform}\index{base!locally $\infty$-quasi-uniform}\index{locally $\infty$-quasi-uniform base}
\index{entourage base!locally $\infty$-uniform}\index{base!locally $\infty$-uniform}\index{locally $\infty$-uniform base}
{\em uniform} (resp. {\em quasi-uniform, locally uniform, locally quasi-uniform}, {\em locally $\infty$-uniform}, {\em locally $\infty$-quasi-uniform}) if so is the preuniformity $\U=\{U\subset X\times X:\exists B\in\mathcal B\;\;B\subset U\}$ generated by the base $\mathcal B$.
\end{itemize}
It is easy to see that an entourage base is locally uniform if it is symmetric and locally quasi-uniform. Observe also that each symmetric base is symmetrizable. In Chapter~\ref{Ch:ww-base} we shall use the following simple fact.

\begin{proposition}\label{p:sym+lqu=lu} If a  base $\Bas$ for a topological space $X$ is symmetrizable, then the family $\overleftrightarrow\Bas:=\{B\cap B^{-1}:B\in\Bas\}$ is a symmetric base for $X$. The base $\overleftrightarrow\Bas$ is (locally) uniform if the base $\Bas$ is (locally) quasi-uniform.
\end{proposition}

\begin{theorem}\label{t:metr-base} For a $T_0$-space $X$ the following conditions are equivalent:
\begin{enumerate}
\item $X$ is metrizable;
\item $X$ has a countable uniform base;
\item $X$ has a countable locally uniform base;
\item $X$ has a countable symmetrizable locally quasi-uniform base.
\end{enumerate}
\end{theorem}

\begin{proof} We shall prove the implications $(1)\Ra(2)\Ra(4)\Ra(3)\Ra(1)$. In fact, the first two implications are trivial.

$(4)\Ra(3)$ If $\Bas$ is a countable symmetrizable locally quasi-uniform base for $X$, then the family $\{B\cap B^{-1}:B\in\Bas\}$ is a countable (symmetric) locally uniform base for $X$.

$(3)\Ra(1)$. Assume that the $T_0$-space $X$ has a countable locally uniform base $\Bas=\{B_n\}_{n\in\w}$. Replacing each entourage $B_n$ by the intersection $\bigcap_{i\le n}B_i$, we can assume that the sequence $(B_n)_{n\in\w}$ is decreasing.  To prove that the space $X$ is metrizable, we shall apply the Metrization Theorem of Moore \cite[5.4.2]{Eng} (see also \cite[1.4]{Grue}). According to this theorem, the metrizability of $X$ will be proved as soon as we construct a sequence $(\U_n)_{n\in\w}$ of open covers of $X$ such that for each point $x\in X$ and neighborhood $O_x\subset X$ of $x$ there exist a neighborhood $V_x$ of $x$ and a number $n\in\w$ such that $\St(V_x;\U_n)\subset O_x$.

For every $n\in\w$ consider the open cover $\U_n=\{B_n[x]^\circ\}_{x\in X}$ by interiors of the $B_n$-balls. We claim the sequence of covers $(\U_n)_{n\in\w}$ satisfies the requirement of the Moore Metrization Theorem. Indeed, for every point
$x\in X$ and every neighborhood $O_x\subset X$ of $x$, the local uniformity of the base $\Bas$ yields a number $n\in\w$ such that $B_n^{\pm3}[x]\subset O_x$. Then for the neighborhood $V_x=B_n[x]$ we get $\St(V_x;\U_n)\subset B_nB_n^{-1}[V_x]=B_n^{\pm3}[x]\subset O_x$. Now we can apply Moore Metrization Theorem \cite[5.4.2]{Eng} and conclude that the space $X$ is metrizable.
\end{proof}

%Now on each (regular) $R_0$-space $X$, we construct a canonical (locally uniform) symmetric base $s\Bas_X$.
We recall that a topological space $X$ is an\index{$R_0$-space} {\em $R_0$-space} if for each point $x\in X$ any neighborhood $O_x\subset X$ of $x$ contains the closure $\overline{\{x\}}$ of the singleton $\{x\}$.

\begin{proposition}\label{p:exist-lubase} Let $X$ be a topological space and $\cov_{<\w}(X)$ be the family of open finite covers of $X$. If $X$ is a $R_0$-space, then the family $s\mathcal B_X=\{B_\U:\U\in\cov_{<\w}(X)\}$ of the entourages $B_\U=\bigcup_{U\in\U}U\times U$ is a symmetric base for $X$. If $X$ is a regular space, then $s\mathcal B_X$ is a locally uniform symmetric base for $X$.
\end{proposition}

\begin{proof} It is clear that the family $s\mathcal B_X$ is closed under finite intersections and consists of symmetric entourages. To see that $s\mathcal B_X$ is a symmetric base for $X$, it suffices to check that for every $x\in X$ the family $\{B_\U[x]:\U\in\cov_{<\w}(X)\}$ is a neighborhood base at $x$. Given any open neighborhood $O_x$ of $x$, consider the open cover $\U=\{O_x,X\setminus\overline{\{x\}}\}$ of $X$ and observe that $B_\U[x]=O_x$.

Now assuming that the space $X$ is regular, we shall prove that the base $s\mathcal B_X$ is locally uniform. Given a point $x\in X$ and an open neighborhood $O_x\subset X$ of $x$, use the regularity of the space $X$ to find open sets $U_1,U_2,U_3\subset X$ such that
$$x\in U_1\subset\bar U_1\subset U_2\subset\bar U_2\subset U_3\subset\bar U_3\subset O_x.$$
 Put $U_0=\emptyset$, $U_4=O_x$, and $U_5=X$. Consider the open cover $\U=\{U_{i+1}\setminus\bar U_{i-1}:1\le i\le 4\}$ of $X$ and the corresponding entourage $B_\U\in s\mathcal B_X$. Observe that $B_\U^{-1}=B_\U$ and $$B_\U^{\pm3}[x]=B^3_\U[x]=B^2_\U[B_\U[x]]=B^2_\U[U_2]\subset B_\U[U_3]\subset U_4=O_x,$$which means that the base $s\mathcal B_X$ is locally  uniform.
\end{proof}

Now we investigate the separation properties of topological spaces admitting entourage bases with some specific properties.

\begin{proposition} If a $T_0$-space $X$ admits a symmetrizable entourage base $\mathcal B$, then $X$ is a $T_1$-space.
\end{proposition}

\begin{proof} Given two distinct points $x,y\in X$ we need to find a neighborhood $U_x$ of $x$, which does not contain $y$. Since $X$ is a $T_0$-space, there exists an open set $V\subset X$ such that $V\cap \{x,y\}$ is a singleton. If this singleton is $\{x\}$, then $U_x:=V$ is a required neighborhood of $x$, not containing  $y$. So, we assume that $V\cap\{x,y\}=\{y\}$. Since the family $\{B[y]:B\in\Bas\}$ is an entourage base at $y$, there exists an entourage $B\in\Bas$ such that $B[y]\subset V$. Since the base $\Bas$ is symmetrizable, $B[y]\cap B^{-1}[y]\subset V$ is a neighborhood of $y$, not containing $x$. Then $U_x:=B[x]\cap B^{-1}[x]$ is a required neighborhood of $x$, not containing $y$.
\end{proof}

%We shall say that a topological space $X$ is \index{topological space!regular}{\em regular} if each neighborhood of any point $x$ of $X$ contains a closed neighborhood of $x$.
 %A topological space $X$ is \index{topological space!regular}{\em regular} if $X$ is a $T_0$-space and for every point $x\in X$ any neighborhood of $x$ contains a closed neighborhood of $x$.

\begin{proposition}\label{p:lu=>regular} For a topological space $X$ the following conditions are equivalent:
\begin{enumerate}
\item $X$ is regular;
\item $X$ admits a symmetric locally uniform base;
\item $X$ admits a locally uniform base;
\item $X$ admits a symmetrizable locally quasi-uniform base;
\end{enumerate}
\end{proposition}

\begin{proof} The implication $(1)\Ra(2)$ was proved in Proposition~\ref{p:exist-lubase} and $(2)\Ra(4)$ is trivial.
\smallskip

$(4)\Ra(3)$ If $\Bas$ is a symmetrizable locally quasi-uniform base for $X$, then the family $\{B\cap B^{-1}:B\in\Bas\}$ is a (symmetric) locally uniform base for $X$.
\smallskip

$(3)\Ra(1)$ Assume that $X$ admits a locally uniform base $\Bas$. Given a point $x\in X$ and a neighborhood $O_x\subset X$ of $x$, find an entourage  $B\in\mathcal B$ such that $B^{-1}\!B[x]\subset O_x$. Then  $\overline{B[x]}\subset B^{-1}\!B[x]\subset O_x$ and $\overline{B[x]}$ is a closed neighborhood of $x$, contained in $O_x$.
\end{proof}

\begin{proposition}\label{p:u=>creg} For a topological space $X$ the following conditions are equivalent:
\begin{enumerate}
\item $X$ is completely regular;
\item $X$ admits a uniform base;
\item $X$ admits a normal quasi-uniform base;
\item $X$ admits a symmetrizable quasi-uniform base;
\item $X$ admits a locally $\infty$-uniform base.
\end{enumerate}
\end{proposition}

\begin{proof} The implications $(1)\Ra(2)\Ra(3,4,5)$ are well-known or trivial. The implication $(3)\Ra(1)$ and $(5)\Ra(2)$ follow from Theorem~\ref{t:BR17} and Proposition~\ref{p:p-lqu+qu}(4). The implication $(4)\Ra(2)$ follows from the observation that for any symmetrizable quasi-uniform base $\Bas$ for $X$ the family $\U=\{B\cap B^{-1}:B\in\Bas\}$ is a uniform base for $X$.
\end{proof}

\section{Uniformly continuous and $\kappa$-continuous maps between preuniform spaces}

\begin{definition}
A function $f:X\to Y$ between preuniform spaces is called \index{function!uniformly continuous}{\em uniformly continuous} if for every entourage $U_Y\in \U_Y$ there exists an entourage $U_X\in \U_X$ such that $\{(f(x),f(x')):(x,x')\in U_X\}\subset U_Y$. %It is clear that each uniformly continuous function between based topological spaces is continuous.
\end{definition}

Inserting a cardinal parameter $\kappa$ into the definition of a uniformly continuous map, we obtain the definition of a $\kappa$-continuous map.

\begin{definition} Let $\kappa$ be a non-zero cardinal.
A function $f:X\to Y$ between preuniform spaces is called \index{function!$\kappa$-continuous}{\em $\kappa$-continuous} if for every entourage $U\in \U_Y$ there exists a subfamily $\V\subset\U_X$ of cardinality $|\V|\le\kappa$ such that for every $x\in X$ there exists an entourage $V\in\V$ such that $f(V[x])\subset U[f(x)]$.
\end{definition}

Observe that a function $f:X\to Y$ between preuniform spaces is uniformly continuous if and only if it is $1$-continuous.
On the other hand, a function $f:X\to Y$ from a preuniform space $X$ to a (topological) preuniform space $Y$ is continuous if (and only if) $f$ is $\kappa$-continuous for the cardinal $\kappa=\cof(\U_X)$.

For a cardinal $\kappa$ and a preuniform space $X$ by $C_\kappa(X)$ we denote the space consisting of $\kappa$-continuous real-valued functions on $X$. The space $C_\kappa(X)$ is endowed with the topology of pointwise convergence, inherited from the Tychonoff power $\IR^X$. The set $C_\kappa(X)$ will be considered as a poset endowed with the partial order, inherited from $\IR^X$.

It will be convenient to denote the function space $C_1(X)$ by $C_u(X)$. On the other hand, for the cardinal $\kappa=|\Bas_X|$ the function space $C_\kappa(X)$ coincides with the space $C_p(X)$ of all continuous real-valued functions on $X$. It follows that $$C_u(X)\le C_\kappa(X)\subset C_p(X)\subset \IR^X$$for every non-zero cardinal $\kappa$.

\section{$\IR$-separated and $\IR$-regular preuniform spaces}

Given a preuniform space $X$, consider the canonical map $\delta:X\to \IR^{C_u(X)}$ assigning to each point $x\in X$ the Dirac measure $\delta_x:C_u(X)\to \IR$, $\delta_x:f\mapsto f(x)$, on the space $C_u(X)$ of all uniformly continuous real-valued functions on $X$. The continuity of uniformly continuous functions on $X$ guarantees that the Dirac measure $\delta_x$ is a continuous function on $C_u(X)$, which yields the inclusion $\delta(X)\subset C_p(C_u(X))\subset \IR^{C_u(X)}$.

\begin{definition} A preuniform space $X$ is called
\begin{itemize}
\item \index{preuniform space!$\IR$-separated}{\em $\IR$-separated} if the canonical map $\delta:X\to \IR^{C_u(X)}$ is injective;
\item\index{preuniform space!$\IR$-regular}{\em $\IR$-regular} if the canonical map $\delta:X\to \IR^{C_u(X)}$ is a topological embedding.
\end{itemize}
\end{definition}

It follows that each $\IR$-separated preuniform space is functionally Hausdorff and each $\IR$-regular  preuniform space is Tychonoff. Conversely, each Tychonoff space $X$ endowed with its universal (pre- or quasi-)uniformity is an $\IR$-regular preuniform space.

\begin{proposition}\label{p:Rr-u} A preuniform space $X$ is $\IR$-regular if and only if $X$ is a $T_0$-space and the topology of $X$ is generated by the canonical uniformity $\U_X^{\pm\w}$ of $X$.
\end{proposition}

\begin{proof} If $X$ is $\IR$-regular, then the canonical map $\delta:X\to \IR^{C_u(X)}$ is a topological embedding, which implies that $X$ is a $T_0$-space. Let $\tau_u$ be the topology on $X$, generated by the canonical uniformity $\U_X^{\pm\w}$.

The definition of the canonical uniformity $\U_X^{\pm\w}$ implies that the map $\delta:(X,\tau_u)\to \IR^{C_u(X)}$ is continuous. Taking into account that the identity map $\id:X\to (X,\tau_u)$ is continuous and $\delta:X\to\IR^{C_u(X)}$ is a topological embedding, we conclude that the identity map $X\to(X,\tau_u)$ is a homeomorphism, which implies that the topology $\tau_u$ generated by the canonical uniformity $\U_X^{\pm\w}$ coincides with the topology of the space $X$.

Now assume that $X$ is a $T_0$-space and the topology of $X$ is generated by the canonical uniformity $\U_X^{\pm\w}$. To show that $\delta:X\to\IR^{C_u(X)}$ is a topological embedding, fix any point $x\in X$ and any neighborhood $O_x\subset X$. We need to find an open set $W\subset \IR^{C_u(X)}$ such that $\delta^{-1}(W)\subset O_x$. Find an entourage $U\in\U_X^{\pm\w}$ such that $U[x]\subset O_x$. By Theorem 8.1.10 \cite{Eng}, there exists a pseudometric $d:X\times X\to[0,1]$ such that $U\supset [d]_{<1}\in\U_X$. It follows that the function $d_x:X\to\IR$, $d_x:y\mapsto d(x,y)$, is uniformly continuous and hence belongs to the function space $C_u(X)$. Consider the open set $W=\{\mu\in \IR^{C_u(X)}:\mu(d_x)<1\}\subset \IR^{C_u(X)}$ and observe that $$\delta^{-1}(W)=\{y\in X:\delta_y(d_x)<1\}=\{y\in X:d(x,y)<1\}\subset U[x]\subset O_x.$$
\end{proof}

An entourage base $\Bas$ for a topological space $X$ is defined to be \index{base!$\IR$-regular}\index{base!$\IR$-separated}{\em $\IR$-regular} (resp. {\em $\IR$-separated}) if so is the preuniformity $\U=\{U\subset X\times X:\exists B\in\Bas\;(B\subset U)\}$ generated by the base $\Bas$.
Proposition~\ref{p:Rr-u} implies the following characterization.

\begin{corollary}\label{c:Rreg} A base $\Bas_X$ for a topological space $X$ is $\IR$-regular if and only if $X$ is a $T_0$-space and $\Bas_X$ is locally $\infty$-uniform.
\end{corollary}

\section{Universal preuniform spaces}

\begin{definition} A preuniform space $X$ is called \index{preuniform space!universal}{\em universal} if each continuous map $f:X\to Y$ to a metric space $Y$ is uniformly continuous. %In this case the preuniformity of $X$ is also called {\em universal}.
\end{definition}

\begin{example} Any topological space $X$ endowed with its universal (pre- or quasi-) uniformity is a universal preuniform space.
\end{example}

Inserting a cardinal parameter $\kappa$ into the definition of a universal preuniform space, we get the definition of a $\kappa$-universal preuniform space.

\begin{definition} Let $\kappa$ be a cardinal. A preuniform space $X$ is called \index{preuniform space!$\kappa$-universal}{\em $\kappa$-universal} if each $\kappa$-continuous map $f:X\to Y$ to a metric space $Y$ of density $d(Y)\le \kappa$ is uniformly continuous. %In this case the preuniform space $X$ is called {\em $\kappa$-universal}.
\end{definition}

It is clear that a preuniform space is universal if and only if it is $\kappa$-universal for any cardinal $\kappa$.

\begin{definition} A preuniform space $X$ is called \index{preuniform space!$\IR$-universal}{\em $\IR$-universal} if each $\w$-continuous map $f:X\to \IR$ is uniformly continuous. This is equivalent to $C_\w(X)=C_u(X)$.
\end{definition}

It is clear that for each preuniform space $X$ we have the implications
$$\mbox{universal $\Ra$ $\w$-universal $\Ra$ $\IR$-universal.}$$

\begin{definition}\index{base!$\IR$-universal}\index{base!universal}\index{base!$\IR$-regular} An entourage base $\Bas$ for a topological space $X$ is called {\em universal} (resp. {\em $\IR$-universal}, {\em $\IR$-regular}) if so is the preuniform space $(X,\U_X)$ endowed with the preuniformity $\U_X=\{U\subset X\times X:\exists B\in\Bas\;(B\subset U)\}$ generated by the base $\Bas$.
\end{definition}

\begin{lemma}\label{l:u=>Ru+Rr} Each universal base $\Bas$ for a Tychonoff space $X$ is $\IR$-universal and $\IR$-regular.
\end{lemma}

\begin{proof} Let $C_u(X)$ be the space of uniformly continuous real-valued functions on the based space $(X,\Bas)$. The universality of the base $\Bas$ implies that $C_u(X)=C(X)$ and hence $C_u(X)=C_\w(X)$, which means that the base $\Bas$ is $\IR$-universal. Since the space $X$ is Tychonoff, the canonical map $\delta:X\to\IR^{C(X)}=\IR^{C_u(X)}$ is a topological embedding, which means that the base $\Bas$ is $\IR$-regular.
\end{proof}

\section{Precompact and $\w$-narrow subsets in preuniform spaces}

 A subset $B\subset X$ of a preuniform space $X$ is called
\begin{itemize}
\item \index{subset of a preuniform space!precompact}{\em precompact} if for each uniformly continuous map $f:X\to Y$ to a complete metric space $Y$ the image $f(B)$ has compact closure in $Y$;
\item \index{subset of a preuniform space!$\w$-narrow}{\em $\kappa$-narrow} for some cardinal $\kappa$ if for each uniformly continuous map $f:X\to Y$ to a metric space $Y$ the image $f(B)$ has density $\le\kappa$.
\end{itemize}
A preuniform space $X$ is called {\em precompact} (resp. {\em $\w$-narrow}) if so is the set $X$ in $X$.

Precompact and $\kappa$-narrow preuniform spaces can be equivalently defined with the help of the (sharp) covering number.

For sets $A\subset X$ and an entourage $E\subset X\times X$ let $$\cov(A;E)=\min\{|D|:D\subset X,\;A\subset E[D]\}$$and let $\cov(A;E)^+$ be the successor cardinal of the cardinal $\cov(A;E)$.
For a family $\mathcal E$ of entourages on $X$ we put\index{$\cov(A;\mathcal E)$}\index{$\cov^\sharp(A;\mathcal E)$}
$$\cov(A;\mathcal E)=\sup_{E\in\mathcal E}\cov(A;E)\mbox{ \ and \ }\cov^\sharp(A;\mathcal E)=\sup_{E\in\mathcal E}\cov(A;E)^+.$$

\begin{proposition}\label{p:pc=cov} A subset $A$ of a preuniform space $(X,\U_X)$ is precompact if and only if $\cov^\sharp(A;\U_X^{\pm\w})\le\w$.
\end{proposition}

\begin{proof} To prove that ``if'' part, assume that $\cov^\sharp(A;\U_X^{\pm\w})\le\w$ and take any uniformly continuous map $f:X\to Y$ to a complete metric space $(Y,d)$. Let $\U_d$ be the uniformity on $Y$ generated by the metric $d$. Taking into account that the map $f$ remains uniformly continuous with respect to the uniformity $\U_X^{\pm\w}$, we conclude that $\cov^\sharp(f(A);\U_d)\le\cov^\sharp(A;\U_X^{\pm\w})\le\w$, which implies that the set $f(A)$ is totally bounded in the metric space $Y$ and hence has compact closure in $Y$, see \cite[8.3.17]{Eng}.

To prove the ``only if'' part, assume that $\cov^\sharp(A,\U_X^{\pm\w})>\w$. Then we can find an entourage $E\in\U_X^{\pm\w}$ such that $\cov(A;E)\ge\w$. Find an entourages $U\in\U_X^{\pm\w}$ such that $U^{-1}U\subset E$ and choose a maximal subset $D\subset A$, which is $U$-separated in the sense that $U[x]\cap U[y]=\emptyset$ for any distinct points $x,y\in D$. By the maximality of $D$, for any point $x\in A$ there exists a point $y\in D$ such that $U[x]\cap U[y]=\emptyset$ and hence $x\in U^{-1}U[y]\subset E[y]\subset E[D]$. So, $A\subset E[D]$ and the choice of the entourage $E$ guarantees that $|D|\ge\w$. Choose an entourage $V\in\U_X^{\pm\w}$ such that $V^{-1}V\subset U$ and observe that the family of $V$-balls $\{V[x]\}_{x\in D}$ is discrete in $X$.

By \cite[8.1.10]{Eng}, there exists a $\U_X^{\pm\w}$-uniform pseudometric $\rho:X\times X\to[0,1]$ on $X$ such that $\{(x,y)\in X\times X:\rho(x,y)<1\}\subset V$. Consider the Hilbert space $\ell_2(D)=\{(x_\alpha)_{\alpha\in D}\in\IR^D:\sum_{\alpha\in D}|x_\alpha|^2<\infty\}$  endowed with the norm $\|(x_\alpha)_{\alpha\in D}\|=\sqrt{\sum_{\alpha\in D}|x_\alpha|^2}$. The family $(\delta_\alpha)_{\alpha\in D}$ of characteristic functions $\delta_\alpha:X\to\{0,1\}$ of the singletons $\{\alpha\}\subset D$ is an orthonormal basis in $\ell_2(D)$.

Observe that the function $f:X\to\ell_2(D)$ defined by the formula
$$f(x)=\begin{cases}
(1-\rho(x,z))\cdot\delta_z,&\mbox{if $x\in V[z]$ for some $z\in D$},\\
0,&\mbox{otherwise}
\end{cases}
$$
is uniformly continuous and its image $f(A)\subset\ell_2(\kappa)$ contains the closed discrete subspace $\{\delta_z\}_{z\in D}$ of cardinality $|D|$ and hence does not have compact closure in $\ell_2(D)$.
This implies that the set $A$ is not precompact in $(X,\U_X)$.
\end{proof}

By analogy we can prove a characterization of $\kappa$-narrow sets in preuniform spaces.

\begin{proposition}\label{p:nar=cov} A subset $A$ of a preuniform space $(X,\U_X)$ is $\kappa$-narrow for some cardinal $\kappa$ if and only if $\cov(X;\U_X^{\pm\w})\le\kappa$.
\end{proposition}

\begin{lemma}\label{l:C=Cw} If a preuniform space $(X,\U_X)$ is locally quasi-uniform and has $\cov(X;\U_X)\le\w$, then $C(X)=C_\w(X)$.
\end{lemma}

\begin{proof} Given any continuous function $f:X\to \IR$, for every $n\in\w$ and $z\in X$ find a neighborhood $O_{n,z}\subset X$ such that $\diam f[O_{n,z}]<2^{-n}$. By the local quasi-uniformity of the base $\Bas$, there exists an entourage $E_{n,z}\in\Bas$ such that $E_{n,z}^2[z]\subset O_z$. Since $\cov(X;\U_X)\le\w$, for every $n\in\w$ there exists a countable subset $Z_n\subset X$ such that $X=\bigcup_{z\in Z_n}E_{n,z}[z]$. We claim that the countable family of entourages $\V=\bigcup_{n\in\w}\{E_{n,z}:z\in Z_n\}$ witnesses that the function $f$ is $\w$-continuous. Given any point $x\in X$ and $\e>0$, find $n\in\w$ with $2^{-n}<\e$ and then find $z\in Z_n$ such that $x\in E_{n,z}[z]$. Then $x\in E_{n,z}[x]\subset E_{n,z}^2[z]\subset O_{n,z}$ and hence $\diam f(E_{n,z}[x])\le\diam f(E_{n,z}^2[z])\le\diam f[O_{n,z}]<2^{-n}<\e$.
\end{proof}

\section{Functionally bounded subsets of preuniform spaces}

A subset $B$ of a preuniform space $X$ is called
\index{subset of a preuniform space!functionally bounded}\index{functionally bounded subset}{\em functionally bounded\/} if for every uniformly continuous function $f:X\to\IR$ the set $f[B]$ is bounded in the real line.

\begin{proposition}\label{p:fb=pc} A subset $B$ of an (~$\IR$-universal~) preuniform space $(X,\U_X)$ is functionally bounded if (and only if) $B$ is precompact in $(X,\U_X)$.
\end{proposition}

\begin{proof} The ``if'' part follows immediately from the definitions. To prove the ``only if'' part, assume that the preuniform space $(X,\U_X)$ is $\IR$-universal.

Assuming that a set $B\subset X$ is not precompact, we could find a uniformly continuous function $f:X\to Y$ to a complete metric space $(Y,d)$ such that the closure of the set $f(B)$ in $Y$ is not compact. Then we can use the normality of the metric space $Y$ and construct a continuous function $g:Y\to\IR$ such that the set $g(f(B))$ is unbounded in the real line $\IR$. It follows that the function $g$ is $\w$-continuous and so is the function $g\circ f:X\to\IR$. The $\IR$-universality of the preuniform space $(X,\U_X)$ guarantees that the $\w$-continuous function $g\circ f:X\to\IR$ is uniformly continuous. Since the image $g\circ f(B)$ is not bounded in $\IR$, the set $B$ is not functionally bounded in $(X,\U_X)$.
\end{proof}

A subset $A$ of a topological space $X$ is {\em $\w$-Urysohn} if each infinite closed discrete subset $B\subset A$ of $X$ contains an infinite strongly discrete subset $C\subset B$ of $X$.
It is clear that each subset of an $\w$-Urysohn space is $\w$-Urysohn.

\begin{proposition}\label{p:fb+wU=>cc} For an $\w$-Urysohn subset $B$ of an $\IR$-regular $\IR$-universal preuniform space $X$ the following conditions are equivalent:
\begin{enumerate}
\item $B$ is functionally bounded in $X$;
\item $B$ is precompact in $X$;
\item $B$ is countably compact in $X$.
\end{enumerate}
\end{proposition}

\begin{proof} The equivalence $(1)\Leftrightarrow(2)$ was proved in Proposition~\ref{p:fb=pc} and $(3)\Ra(1)$ is trivial. To prove that $(1)\Ra(3)$, assume that the set $B$ is $\w$-Urysohn but not countably compact in $X$. Then we can find a countable infinite subset $D\subset B$ that has no accumulation points in $X$. Using the $\w$-Urysohn property of $B$, we can replace $D$ by a smaller infinite subset and assume that the set $D$ is strongly discrete in $X$.
So, each point $x\in D$ has a neighborhood $O_x\subset X$ such that the family $(O_x)_{x\in D}$ is discrete in $X$. Using the $\IR$-regularity of the preuniform space $X$, for every $x\in X$ we can choose a uniformly continuous function $f_x\colon X\to [0,1]$ such that $f_x(x)=1$ and $f(X\setminus O_x)\subset\{0\}$. Choose a sequence $(x_n)_{n\in\w}$ of pairwise distinct points in $D$ and observe that the function $f=\sum_{n\in\w}n\cdot f_{x_n}\colon X\to\IR$ is $\w$-continuous and $f(B)\supset f(D)\supset\IN$ is unbounded in $\IR$. By the $\IR$-universality of $X$, the $\w$-continuous function $f$ is uniformly continuous and hence the set $B$ is not functionally bounded in $X$.
\end{proof}

\begin{corollary} Each functionally bounded hereditarily Lindel\"of $\bar G_\delta$-subset $B$ of an $\IR$-regular $\IR$-universal preuniform space $X$ is compact.
\end{corollary}

\begin{proof} By Lemma~\ref{l:hL+bG=>wU}, the set $B$ is $\w$-Urysohn and by Proposition~\ref{p:fb+wU=>cc}, $B$ is countably compact. Being Lindel\"of, the countably compact space $B$ is compact.
\end{proof}

In the proof of Theorem~\ref{t:universal=>s-bound} we shall need the following lemma.

\begin{lemma}\label{l:funbound} Let $X$ be a preuniform space and $\{B_n\}_{n\in\w}$ be a countable family of subsets of $X$ which are not functionally bounded in $X$. Then there exists an $\w$-continuous map $f:X\to\IR$ such that for every $n\in\w$ the set $f(B_n)$ is unbounded in $\IR$.
\end{lemma}

\begin{proof} For every $n\in\w$ choose a uniformly continuous function $g_n:X\to\IR$ such that $g_n(B_n)$ is not bounded in $\IR$. The functions $g_n$, $n\in\w$, compose a uniformly continuous function $g:X\to\IR^\w$, $g:x\mapsto (g_n(x))_{n\in\w}$.

Let $\xi:\w\to\w$ be a map such that for every $k\in\w$ the preimage $\xi^{-1}(k)$ is infinite. For every $n\in\w$ choose inductively a point $x_n\in B_{\xi(n)}$ such that $g_{\xi(n)}(x_n)\ge 1+\max\{g_{\xi(n)}(x_k):k<n\}$. It follows that $\lim_{n\to\infty}g_k(x_n)=+\infty$ for every $k\in\w$, which implies that the set $D=\{g(x_n)\}_{n\in\w}$ is closed and discrete in $\IR^\w$. By the normality of the metrizable space $\IR^\w$, there exists a continuous function $h:\IR^\w\to\IR$ such that $h\circ g(x_n)=n$ for all $n\in\w$. It follows that the function $f:=h\circ g:X\to\IR$ is $\w$-continuous and for every $k\in\w$ the set $f[B_k]\supset\{f(x_n):n\in\xi^{-1}(k)\}=\xi^{-1}(k)$ is unbounded in $\IR$.
\end{proof}

\section{Completeness of preuniform spaces}

In this section we discuss some completeness properties of preuniform spaces.
% We recall that a family $\F$ of non-empty subsets of a set $X$ is called a {\em filter} on $X$ if $\F$ is closed under finite intersections and   taking supersets in $X$.

\begin{definition}
A filter $\F$ on a preuniform space $X$ is called
\begin{itemize}
\item \index{filter!Cauchy}{\em Cauchy\/} if for any uniform pseudometric $\rho$ on $X$ there is a set $F\in\F$ such that $\diam_\rho(F)=\sup\{\rho(x,y):x,y\in F\}<1$;
\item \index{filter!$\IR$-Cauchy}{\em $\IR$-Cauchy} if for any uniformly continuous function $\varphi:X\to\IR$ and any $\e>0$ there exists $F\in \F$ such $\diam\, \varphi(F)=\sup\{|x-y|:x,y\in\varphi(F)\}<\e$.
\end{itemize}
\end{definition}

It is clear that each Cauchy filter $\F$ on a preuniform space $X$ is $\IR$-Cauchy. Moreover, $\F$ is Cauchy if and only if $\F$ is Cauchy with respect to the canonical uniformity $\U_X^{\pm\w}$ of $X$.

\begin{definition}
A preuniform space $X$ is called \index{preuniform space!complete}{\em complete} (resp. \index{preuniform space!$\IR$-complete}{\em $\IR$-complete}) if $X$  each Cauchy (resp. $\IR$-Cauchy) filter $\F$ on $X$ \index{filter!convergent}{\em converges} to a unique point $x\in X$ in the sense that $x$ is the unique point of the intersection $\bigcap_{F\in\F}\overline{F}$.
\end{definition}

Observe that for uniform spaces our definition of the completeness coincides with the standard one, see \cite[\S8.3]{Eng}.

\begin{proposition}\label{p:pU-comp} A preuniform space $(X,\U_X)$ is complete if and only if $X$ is $\IR$-regular and the uniform space $(X,\U_X^{\pm\w})$ is complete.
\end{proposition}

\begin{proof} The ``if'' part follows from Proposition~\ref{p:Rr-u}. To prove the ``only if'' part, assume that the preuniform space $X$ is complete. First we prove that the canonical map $\delta:X\to \IR^{C_u(X)}$ is injective. In the opposite case, we could find two distinct points $x,y\in X$ with common image $z=\delta(x)=\delta(y)$. Consider the filter $\F=\{\delta^{-1}(U):U\in\Tau_{z}(\IR^{C_u(X)})\}$ on $X$ and observe that it is Cauchy and $x,y\in\bigcap\F\subset \bigcap_{F\in\F}\bar F$, which implies that the preuniform space $X$ is not complete. By the injectivity of the canonical map $\delta$, the space $X$ is Hausdorff.

Next, we prove that the topology $\tau_u$ on $X$, generated by the canonical uniformity $\U_X^{\pm\w}$, coincides with the topology of the space $X$. In the opposite case we can find a point $x\in X$ and a neighborhood $O_x$ of $x$ in $X$, which is not a neighborhood of $x$ in the topology $\tau_u$. Then the filter $\F=\{F\subset X:\exists U\in\U_X^{\pm\w}\;\;U[x]\setminus O_x\subset F\}$ is Cauchy and by the completeness of $X$ converges to a unique point $x'\in X$. Then  injectivity of the canonical map $\delta$ guarantees that $x'=x$, which is not possible as $X\setminus O_x\in\F$. This contradiction shows that the topology of the Hausdorff space $X$ is generated by the uniformity $\U_X^{\pm \w}$. The the completeness of $X$ implies the completeness of the uniform space $(X,\U_X^{\pm\w})$. By Proposition~\ref{p:Rr-u}, the preuniform space $X$ is $\IR$-regular.
\end{proof}

Since each Cauchy filter on a preuniform space is $\IR$-Cauchy, each $\IR$-complete preuniform space is complete. The following characterization shows that the $\IR$-completeness can be considered as a preuniform counterpart of the Hewitt completeness of topological spaces.

\begin{proposition} A preuniform space $X$ is $\IR$-complete if and only if the canonical map $\delta:X\to \IR^{C_u(X)}$, $\delta:x\mapsto (\varphi(x))_{\varphi\in C_u(X)}$, is a closed topological embedding.
\end{proposition}

\begin{proof} If $X$ is $\IR$-complete, then $X$ is complete and Proposition~\ref{p:pU-comp}, $X$ is $\IR$-regular, which means that the canonical map $\delta:X\to \IR^{C_u(X)}$ is a topological embedding. Given any point $y\in \overline{\delta(X)}$, consider the  $\IR$-Cauchy filter $\F=\{\delta^{-1}(V):V\in\Tau_y(\IR^{C_u(X)})\}$ on $X$. By the $\IR$-completeness of $X$, the set $\bigcap_{F\in\F}\bar F$ contains some point $x\in X$. The continuity and injectivity of the map $\delta$ guarantees that $y=\delta(x)\in\delta(X)$, witnessing that $\overline{\delta(X)}=\delta(X)$ and $\delta:X\to\IR^{C_u(X)}$ is a closed topological embedding.
\smallskip

Now assume conversely that the canonical map $\delta:X\to\IR^{C_u(X)}$ is a closed topological embedding. Then $\delta$ is a topological embedding and hence the based space $X$ is $\IR$-regular. To show that $X$ is $\IR$-complete, fix an $\IR$-Cauchy filter $\F$ on $X$. Then its image $\delta(\F)=\{E\subset\IR^{C_u(X)}:\delta^{-1}(E)\in\F\}$ is a Cauchy filter in the space $\IR^{C_u(X)}$ endowed with the standard product uniformity. The completeness of the uniform space $\IR^{C_u(X)}$ implies the existence of a point $y\in\bigcap_{E\in\delta(\F)}\bar E\in\IR^{C_u(X)}$. It follows that $y\in\overline{\delta(X)}=\delta(X)$ and hence $y=\delta(x)$ for some $x\in X$. Since $\delta$ is a topological embedding, the inclusion $\delta(x)\in\bigcap_{E\in\delta(\F)}\bar E\subset \bigcap_{F\in\F}\overline{\delta(F)}$ is equivalent to the inclusion $x\in \bigcap_{F\in\F}\overline{F}$, witnessing that $x$ is a limit point of the $\IR$-Cauchy filter. The injectivity of the canonical map $\delta$ implies that $x$ is a unique limit point of the filter $\F$.
\end{proof}

A preuniform space $X$ is defined to be \index{preuniform space!functionally bounded\/}{\em functionally bounded} if each uniformly continuous function $f:X\to\IR$ is bounded. An example of a functionally bounded uniform space is any bounded convex subset of a Banach space. The following characterization can be easily derived from the definitions.

\begin{proposition} An $\IR$-regular preuniform space $X$ is compact if and only if it is functionally bounded and $\IR$-complete.
\end{proposition}

\begin{proposition}\label{p:Rcomplete}  Let $X$ be an $\IR$-universal $\w$-narrow preuniform space. Then
\begin{enumerate}
\item[\textup{(1)}] Each $\IR$-Cauchy filter on $X$ is Cauchy;
\item[\textup{(2)}] $X$ is complete if and only if $X$ is $\IR$-complete.
\end{enumerate}
\end{proposition}

\begin{proof} 1. Let $\F$ be an $\IR$-Cauchy filter on $X$. To show that $\F$ is Cauchy, we need to prove that for any uniform pseudometric $d$ on $X$ the filter $\F$ is Cauchy in the pseudometric space $(X,d)$. The pseudometric $d$ induces a metric $\tilde d$ on the quotient space $Y=X/_\sim$ of $X$ by the equivalence relation $\sim$ defined by $x\sim y$ iff $d(x,y)=0$. The $\w$-narrowness of the uniform space $X$ implies the separability of the metric space $Y$. Consequently, $C_\w(Y)=C(Y)$. By \cite[3.11.12]{Eng}, the metrizable separable space $Y$ is Hewitt complete, which means that $X$ is complete in the smallest uniformity making all continuous functions $f:Y\to\IR$ uniformly continuous. This uniformity on $Y$ will be called the \index{uniformity!Hewitt}{\em Hewitt uniformity}. Let $q:X\to X/_\sim=Y$ be the quotient map and observe that for every continuous map $\varphi:Y\to\IR$ the function $\varphi\circ q:X\to \IR$ is $\w$-continuous and hence uniformly continuous (by the $\IR$-universality of $X$). Since $\F$ is an $\IR$-Cauchy filter on $X$ and $C_\w(X)=C_u(X)$, the image $q(\F)$ is $\IR$-Cauchy in $Y$ and by the Hewitt completeness of $Y$ converges to some point $y\in Y$. It follows that for every $\e>0$ the $\e$-ball $B(y;\e)=\{y'\in Y:\tilde d(y,y')<\e\}$ belongs to the filter $q(\F)$. Then the set $F=q^{-1}(B(y;\e))\in\F$ has $d$-diameter $<2\e$, which means that the filter $\F$ is Cauchy.
\smallskip

2. The equivalence of the completeness and $\IR$-completeness of $X$ follows from the equivalence of the Cauchy and $\IR$-Cauchy properties for filters on $X$, which was proved in the first statement.
\end{proof}

\begin{remark} The requirement of the $\IR$-universality of the preuniform space $X$ is essential in Proposition~\ref{p:Rcomplete}. Indeed, take any non-compact closed convex bounded subset $X$ of a Banach space and observe that the uniform space $X$ is complete but not $\IR$-complete (being functionally bounded and not compact).
\end{remark}

\chapter{Baseportators and portators}\label{ch:portator}

In this chapter we introduce and study so-called baseportators, which are pointed topological spaces endowed with a structure that allows to transport neighborhoods of the distinguished point to neighborhoods of an arbitrary point of the baseportator. A special kind of baseportators are portators (and netportators). Their transport structure is determined by (finite-to-finite) functions that act on points of the space. The algebraic structure of portators is studied in Section~\ref{s:portator}. Some known examples of (net)portators (like topological groups or rectifiable spaces) are discussed in the last Section~\ref{s:known-portator}.

\section{Baseportators}

A \index{pointed topological space}\index{topological space!pointed}{\em pointed topological space} is a topological space $X$ with a distinguished point $e\in X$ called the \index{unit}\index{pointed topological space!unit of}{\em unit}. For a point $x$ of a topological space $X$ by $\Tau_x(X)$ we denote the poset of all (not necessarily open) neighborhoods of $x$ in $X$, endowed with the partial order of the reverse inclusion ($U\le V$ iff $V\subset U$).

\begin{definition}
A \index{baseportator}{\em baseportator} is a pair $(X,t_X)$ consisting of a pointed topological space $X$ with a distinguished point $e$ and an indexed family $t_X=(t_x)_{x\in X}$ of monotone cofinal functions $t_x:\Tau_e(X)\to\Tau_x(X)$ defined for all points $x\in X$. The family $t_X$ is called the \index{transport structure}\index{baseportator!transport structure of}{\em transport structure} of the baseportator $(X,t_X)$.
\end{definition}

Sometimes we shall identify a baseportator $(X,t_X)$ with its underlying topological space $X$ assuming that the transport structure $t_X$ is clear from the context.

The transport structure $t_X$ of a baseportator $X$ can be encoded by the set-valued binary operation
$$\mathbf{xV}:X\times\Tau_e(X)\multimap X,\;\;\mathbf{xV}:(x,V)\mapsto xV:=t_x(V),$$ assigning to each pair $(x,V)\in X\times\Tau_e(X)$ the set $t_x(V)\subset X$ denoted by $xV$. The binary operation $\mathbf{xV}$ is called the \index{baseportator!multiplication of}{\em multiplication} of the baseportator $X$. It will be convenient to extend the operation of multiplication letting $AV:=\bigcup_{a\in A}aV$ for any set $A\subset X$ and any neighborhood $V\in\Tau_e(X)$ of $e$.
\smallskip

Observe that a set-valued function $\mathbf{xv}:X\times\Tau_e(X)\multimap X$, $\mathbf{xv}:(x,V)\mapsto xV$, coincides with the multiplication of some baseportator on $X$ if and only if for every $x\in X$ it satisfies the following three conditions:
\begin{enumerate}
\item for every $V\in\Tau_e(X)$ the set $xV$ is a neighborhood of $x$;
\item for any $V,U\in\Tau_e(X)$ with $V\subset U$ we get $xV\subset xU$;
\item $\forall O_x\in\Tau_x(X)\;\;\exists V\in\Tau_e(X)\;\;xV\subset O_x$.
\end{enumerate}

For every baseportator $X$ its multiplication $\mathbf{xV}$ determines another set-valued binary operation $$\mathbf{xV}^{-1}:X\times\Tau_e(X)\multimap X,\;\;\mathbf{xV}^{-1}:(x,V)\mapsto xV^{-1}:=\{y\in X:x\in yV\},$$
called the \index{baseportator!division}{\em division} operation of $X$.

For a subset $A\subset X$ and a neighborhood $V\in\Tau_e(X)$ put  $AV^{-1}:=\bigcup_{a\in A}aV^{-1}$.
\smallskip

For any neighborhood $B\in\Tau_e(X)$ of the unit $e$ of a baseportator $X$ the set
$$\vec B=\{(x,y)\in X\times X:y\in xB\}$$ is a neighborhood assignment on $X$.  The family
 $$\vec\Tau_X=\{U\subset X\times X:\exists B\in\Tau_e(X)\;\;\vec B\subset U\}$$ is a preuniformity on $X$ called the \index{baseportator!canonical preuniformity of}{\em canonical preuniformity} of the baseportator $X$.  The following (trivial) proposition shows that this preuniformity is topological and generates the topology of $X$.

\begin{proposition}\label{p:trans-base} For any neighborhood base $\Bas$ at the unit $e$ of a baseportator $X$ the family $\vec\Bas=\{\vec B:B\in\Bas\}$ is an entourage base for $X$ and $\vec\Bas$ is a base of the canonical preuniformity $\vec\Tau_X$.
\end{proposition}

\begin{proof} The cofinality of the base $\Bas$ in $\Tau_e(X)$ and the cofinality of the transport maps $t_x:\Tau_e(X)\to\Tau_x(X)$ imply that for every $x\in X$ the family $\{\vec B[x]\}_{B\in\Bas}=\{t_x(B)\}_{B\in\Bas}$ is cofinal in the poset $\Tau_x(X)$ and hence is a neighborhood base for $x$. This means that the family of entourages $\vec\Bas$ is a base for $X$. It is clear that $\vec \Bas$ is a base of the preuniformity $\vec\Tau_X$.
\end{proof}

\begin{definition}\label{d:trans-lqu}
\index{baseportator!symmetrizable}
\index{baseportator!uniform} \index{baseportator!quasi-uniform}
 \index{baseportator!locally quasi-uniform}\index{baseportator!locally uniform}
\index{symmetrizable baseportator}
\index{uniform baseportator}\index{quasi-uniform baseportator}
 \index{locally quasi-uniform baseportator}\index{locally uniform baseportator}A baseportator $X$ is called
{\em uniform} (resp. {\em quasi-uniform}, {\em locally uniform}, {\em locally quasi-uniform}, {\em symmetrizable}) if so it its canonical preuniformity $\vec\Tau_X$. A baseportator $X$ is called \index{baseportator!normally preuniform} {\em normally preuniform} (resp. \index{baseportator!normally quasi-uniform}{\em normally quasi-uniform}) if its canonical preuniformity $\vec\Tau_X$ is normal (resp. normal and quasi-uniform).
\end{definition}

%For any baseportator these notions relate as follows:
%$$
%\xymatrix{
%\mbox{symmetrizable}&\mbox{uniform}\ar@{=>}[l]\ar@{=>}[d]\ar@{=>}[rr]&&\mbox{locally uniform}\ar@{=>}[d]\\
%\mbox{normally preuniform}&\mbox{normally quasi-uniform}\ar@{=>}[r]\ar@{=>}[l]&\mbox{quasi-uniform}\ar@{=>}[r]&\mbox{locally~quasi-uniform.}
%}
%$$

In Propositions~\ref{p:lqu-bport} and \ref{p:lu-bport} we shall characterize locally (quasi-)uniform baseportators in the terms of continuity of  their multiplication and division operations.

We shall say that the multiplication $\mathbf{xV}$ (resp. division $\mathbf{xV}^{-1}$) of a baseportator $X$ is {\em continuous} at a point $(x,e)\in X\times \{e\}$ if for any neighborhood $O_x\in\Tau_x(X)$ there are neighborhoods $U_x\in\Tau_x(X)$ and $V_e\in\Tau_e(X)$ such that $U_xV_e\subset O_x$ (resp. $U_xV_e^{-1}\subset O_x$). %We shall say that the multiplication $\mathbf{xV}$ (resp. division $\mathbf{xV}^{-1}$) is {\em continuous at} $X\times\{e\}$ if it is continuous at each point $(x,e)\in X\times\{e\}$.%The multiplication $\mathbf{xV}$ on a baseporter $X$ is defined to be  \index{baseporter!weakly associative}{\em weakly associative} if $(xV)V\subset x(VV)$ for any $x\in X$ and $V\in\Tau_e(X)$.

\begin{proposition}\label{p:lqu-bport} For a baseportator $X$ the following conditions are equivalent:
\begin{enumerate}
\item $X$ is locally quasi-uniform;
\item the multiplication $\mathbf{xV}$ is continuous at each point $(x,e)\in X\times\{e\}$.
\end{enumerate}
If $(xV)V\subset x(VV)$ for any $x\in X$ and $V\in\Tau_e(X)$, then the conditions $(1),(2)$ are equivalent to each of the following conditions:
\begin{enumerate}
\item[(3)] $X$ is quasi-uniform;
\item[(4)] the multiplication  $\mathbf{xV}$ is continuous at $(e,e)$.
\end{enumerate}
\end{proposition}

\begin{proof} $(1)\Ra(2)$ Assume that the baseportator $X$ is locally quasi-uniform. To prove that the multiplication $\mathbf{xV}$ is continuous at $X\times\{e\}$, fix any point $x\in X$ and a neighborhood $O_x\in\Tau_x(X)$ of $x$. We should find neighborhoods $U_x\in\Tau_x(X)$ and $V_e\in\Tau_e(X)$ such that $U_xV_e\subset O_x$. Since the preuniformity $\vec\Tau_X$ is locally quasi-uniform and is generated by the base $\{\vec V:V\in\Tau_e(X)\}$, there exists a neighborhood $V\in\Tau_e(X)$ such that $\vec V^2[x]\subset O_x$. Then $U_x:=xV=\vec V[x]$ and $V_e:=V$ are required neighborhoods of $x$ and $e$ as $U_xV_e=\bigcup_{y\in U_x}yV=\bigcup_{y\in \vec V[x]}\vec V[y]=\vec V^2[x]\subset O_x$.
\smallskip

$(2)\Ra(1)$ Now assume that the multiplication $\mathbf{xV}$ is continuous at $X\times\{e\}$. To show that the baseportator $X$ is locally quasi-uniform, fix any point $x\in X$ and  neighborhood $O_x\in\Tau_x(X)$ of $x$. By the continuity of the multiplication $\mathbf{xV}$ at $(x,e)$, there exist neighborhoods $U_x\in\Tau_x(X)$ and $V_e\in\Tau_e(X)$ such that $U_xV_e\subset O_x$. By the cofinality of the transport map $t_x:\Tau_e(X)\to\Tau_x(X)$, there exists a neighborhood $V\subset V_e$ of $e$ such that $xV=t_x(V)\subset U_x$. Then for the entourage $\vec V\in\vec\Tau_X$ we get the required inclusion
$\vec V^2[x]=\bigcup_{y\in\vec V[x]}\vec V[y]=\bigcup_{y\in xV}yV=(xV)V\subset U_xV_e\subset O_x$.
\smallskip

It is clear that $(3)\Ra(1)\Leftrightarrow (2)\Ra(4)$.
Assuming that $(xV)V\subset x(VV)$ for all $x\in X$ and $V\in\tau_e(X)$, we shall prove that $(4)\Ra(3)$. Suppose that the multiplication $\mathbf{xV}$ is continuous at $(e,e)$.

To prove that $\vec\Tau_X$ is  quasi-uniform, for every neighborhood $U\in\Tau_e(X)$ we should find a neighborhood $V\in\Tau_e(X)$ such that $\vec V^{2}\subset\vec U$. By the continuity of the multiplication $\mathbf{xV}$ at $(e,e)$, there exists a neighborhood $V$ of $e$ such that $VV\subset U$. Then for every $x\in X$ we get $\vec V^2[x]=(xV)V\subset x(VV)\subset xU=\vec U[x]$ and hence $\vec V^2\subset\vec U$, which means that the preuniformity $\vec\Tau_X$ is  quasi-uniform and so is the baseportator $X$.
\end{proof}

By analogy we can prove a characterization of locally uniform baseportators.

\begin{proposition}\label{p:lu-bport}  For a baseportator $X$ the following conditions are equivalent:
\begin{enumerate}
\item $X$ is locally uniform;
\item the operations $\mathbf{xV}$ and $\mathbf{xV}^{-1}$ are continuous at each point $(x,e)\in X\times\{e\}$.
\end{enumerate}
If $(xV)V\subset x(VV)$ and $(xV)V^{-1}\subset x(VV^{-1})$ for any $x\in X$ and $V\in\Tau_e(X)$, then the conditions \textup{(1),(2)} are equivalent to each of the following conditions:
\begin{enumerate}
\item[(3)] $X$ is uniform;
\item[(4)] the multiplication $\mathbf{xV}$ and the division $\mathbf{xV}^{-1}$ are continuous at $(e,e)$.
\end{enumerate}
\end{proposition}

\section{Portators}\label{s:portator}

In this section we study a special kind of baseportators, called portators. These are baseportators $X$ whose multiplication operation $\mathbf{xV}$ is generated by a suitable binary set-valued operation $\mathbf{xy}:X\times X\multimap X$. A more precise definition follows.

\begin{definition} A \index{portator}{\em portator} is a pair $(X,\mathbf{xy})$ consisting of a pointed topological space $X$ whose distinguished point $e$ is called the \index{portator!unit of}{\em unit} of $X$, and a set-valued binary operation $\mathbf{xy}:X\times X\multimap X$, $\mathbf{xy}:(x,y)\mapsto xy\subset X$, called the \index{portator!multiplication of}{\em multiplication} of $X$, such that for every $x\in X$ the following conditions are satisfied:
\begin{itemize}
\item $xe=\{x\}$;
\item for every $V\in\Tau_e(X)$ the set $xV=\bigcup_{y\in V}xy$ is a neighborhood of $x$ in $X$;
\item $\forall U_x\in\Tau_x(X)\;\;\exists V\in\Tau_e(X)$ such that $xV\subset U_x$.
\end{itemize}
In this case the map $t_x:\Tau_e(X)\to\Tau_x(X)$, $t_x:V\mapsto xV$, is a well-defined monotone cofinal function between the posets $\Tau_e(X)$ and $\Tau_x(X)$ and the family $(t_x)_{x\in X}$ turns $X$ into a baseportator.
\end{definition}

Sometimes we shall identify a portator $(X,\mathbf{xy})$ with its underlying topological space $X$ assuming that the multiplication $\mathbf{xy}$ is clear from the context. It will be convenient to extend the multiplication operation to subset $A,B\subset X$ letting $AB:=\bigcup_{a\in A}aB=\bigcup_{b\in B}Ab$ where $aB:=\bigcup_{b\in B}ab$ and $Ab:=\bigcup_{a\in A}ab$.

The multiplication map $\mathbf{xy}$ of a portator $X$ is called
\begin{itemize}
\item \index{multiplication!asociative}{\em associative} if $x(yz)=(xy)z$ for all $x,y,z\in X$;
\item \index{multiplication!locally associative}{\em locally associative} if there exists a neighborhood $V\in\Tau_e(X)$ of $e$ such that $x(uv)=(xu)v$ for all $x\in X$ and $u,v\in V$.
\end{itemize}
\index{portator!locally associative}\index{portator!associative}A portator $X$ is defined to be ({\em locally}) {\em associative} if so is its multiplication $\mathbf{xy}$.
\smallskip

The multiplication $\mathbf{xy}$ of a portator $X$ induces two binary set-valued operations
$$
\mathbf{x}^{-1}\mathbf{y}:X\times X\multimap X,\;\;\mathbf{x}^{-1}\mathbf{y}:(x,y)\mapsto x^{-1}y:=\{z\in X:y\in xz\},$$
and
$$\mathbf{xy}^{-1}:X\times X\multimap X,\;\;\mathbf{xy}^{-1}:(x,y)\mapsto xy^{-1}:=\{z\in X:x\in zy\},$$
called the \index{portator!division operation}{\em left division} and the {\em right division} on $X$, respectively.

Observe that $xe^{-1}=\{z\in X:x\in ze\}=\{z\in X:x\in\{z\}\}=\{x\}$ for every $x\in X$.

The binary operations $\mathbf{x}^{-1}\mathbf{y}$ and $\mathbf{xy}^{-1}$ induce two unary set-valued operations
$$\mathbf{x}^{-1}e:X\multimap X,\;\;\mathbf{x}^{-1}e:x\mapsto x^{-1}e:=\{z\in X:e\in xz\},$$
and
$$e\mathbf{x}^{-1}:X\multimap X,\;\;e\mathbf{x}^{-1}:x\mapsto ex^{-1}:=\{z\in X:e\in zx\},$$
called the \index{portator!inversion operation} {\em left inversion} and the {\em right inversion} of $X$, respectively.

%It follows that $xe=\{x\}$ and $x(x^{-1}y)\subset\{y\}$ for all $x,y\in X$.

As in case of baseportators, the \index{portator!canonical preuniformity of}{\em canonical preuniformity} $\vec \Tau_X$ on a portator $(X,\mathbf{xy})$ is generated by the base $\{\vec V:V\in\Tau_e(X)\}$ consisting of the entourages $\vec V=\{(x,y)\in X\times X:y\in xV\}$.

\index{portator!symmetrizable}
\index{portator!uniform} \index{portator!quasi-uniform}
 \index{portator!locally quasi-uniform}\index{portator!locally uniform}
\index{symmetrizable portator}
\index{uniform portator}\index{quasi-uniform portator}
 \index{locally quasi-uniform portator}\index{locally uniform portator}A portator $X$ is called
{\em uniform} (resp. {\em quasi-uniform}, {\em locally uniform}, {\em locally quasi-uniform}, {\em symmetrizable}) if so it its canonical preuniformity $\vec\Tau_X$. A portator $X$ is called \index{portator!normally preuniform} {\em normally preuniform} (resp. \index{portator!normally quasi-uniform}{\em normally quasi-uniform}) if its canonical preuniformity $\vec\Tau_X$ is normal (resp. normal and quasi-uniform).

Proposition~\ref{p:lqu-bport}, \ref{p:lu-bport} characterizing locally (quasi-)uniform baseportators imply the corresponding characterizations of locally (quasi-)uniform portators via the continuity properties of the binary operations $\mathbf{xy}$ and $\mathbf{xy}^{-1}$. We shall need two kinds of continuity of set-valued maps.

\index{set-valued map!continuous}\index{set-valued map!semicontinuous} We shall say that a set-valued map $F:X\multimap Y$ between topological spaces is
\begin{itemize}
\item \index{set-valued map!semicontinuous}{\em semicontinuous} at a point $x\in X$ if for any open set $U\subset Y$ containing $F(x)$ the set $\{x'\in X:F(x')\subset U\}$ is a neighborhood of $x$ in $X$;
\item \index{set-valued map!continuous}{\em continuous} at $x\in X$ if $F$ is semicontinuous at $x$ and for every open set $U\subset Y$ with $U\cap F(x)\ne\emptyset$ the set $\{x'\in X:F(x')\cap U\ne\emptyset\}$ is a neighborhood of $x$.
\end{itemize}
Observe that $f$ is semicontinuous (resp. continuous) at $x$ if and only if it is continuous as a  function $f:X\to\mathcal P(Y)$ to the power-set of $Y$, endowed with the topology generated by the subbase consisting of the sets $\langle U\rangle_\subset=\{Z\in \mathcal P(X):Z\subset U\}$ (and $\langle U\rangle_\cap=\{Z\subset X:Z\cap U\ne\emptyset\}$) where $U$ runs over the topology of $X$. The topology on $\mathcal P(Y)$ generated by the subbase $\{\langle U\rangle_\subset,\langle U\rangle_\cap:U$ is open in $Y\}$ is well-known as the {\em Vietoris topology} on the power-set $\mathcal P(Y)$, see \cite[2.7.20]{Eng}.

These continuity notions will be applied to the set-valued operations $\mathbf{xy},\mathbf{xy}^{-1},\mathbf{x}^{-1}\mathbf{y}:X\times X\multimap X$.

The following characterization can be proved by analogy with Proposition~\ref{p:lqu-bport}.

\begin{proposition}\label{p:lqu-portator} For a portator $X$ the following conditions are equivalent:
\begin{enumerate}
\item $X$ is locally quasi-uniform;
\item the multiplication $\mathbf{xy}$ is semicontinuous at each point  $(x,e)\in X\times\{e\}$.
\end{enumerate}
If the multiplication $\mathbf{xy}$ is locally associative, then the conditions $(1),(2)$ are equivalent to each of the following conditions:
\begin{enumerate}
\item[(3)] $X$ is quasi-uniform;
\item[(4)] the multiplication  $\mathbf{xy}$ is semicontinuous at $(e,e)$.
\end{enumerate}
\end{proposition}

\begin{proposition}\label{p:lu-portator}  A portator $X$ is locally uniform if and only if  the operations $\mathbf{xy}$ and $\mathbf{xy}^{-1}$ are semicontinuous at each point $x\in X$.
\end{proposition}

\begin{proof} To prove the ``if'' part, assume that for every $x\in X$ the operations $\mathbf{xy}$ and $\mathbf{xy}^{-1}$ are semicontinuous at $(x,e)$. To show that the preuniformity $\vec\Tau_X$ is locally uniform, take any point $x\in X$ and neighborhood $O_x\in\Tau_x(X)$. Since $xe=\{x\}\subset O_x$, the semicontinuity of the multiplication $\mathbf{xy}$ at $(x,e)$ yields neighborhoods $U_x\in\Tau_x(X)$ and $U_e\in\Tau_e(X)$ such that $U_xU_e\subset O_x$. Since $xe^{-1}=\{x\}\subset U_x$, the semicontinuity of the operation $\mathbf{xy}^{-1}$ at $(x,e)$ yields neighborhoods $V_x\in\Tau_x(X)$ and $V_e\in\Tau_e(X)$ such that $V_xV_e^{-1}\subset U_x$. Finally, using the semicontinuity of the multiplication $\mathbf{xy}$ at $(x,e)$, find a neighborhood $V\subset V_e\cap U_e$ of $e$ such that $xV\subset V_x$. We claim that $\vec V^{\mp2}[x]\subset U_x$. Given any point $y\in \vec V^{\mp2}[x]$, find a point $z\in \vec V[y]\cap\vec V[x]=yV\cap xV$. Then there is a point $v\in V$ such that $z\in yv$ and hence  $y\in zV^{-1}\subset (xV)V^{-1}\subset V_xV_e^{-1}\subset U_x$. So, $\vec V^{\mp2}[x]\subset U_x$ and $\vec V^{\pm3}[x]\subset (\vec V^{\mp2}[x])V\subset U_xU_e\subset O_x$, witnessing that the portator $X$ is locally uniform.
\smallskip

Now we prove the ``only if'' part. Assume that the portator $X$ is locally uniform.
By Proposition~\ref{p:lqu-portator}, the multiplication map $\mathbf{xy}$ is semicontinuous at each point $(x,e)\in X\times \{e\}$. To prove that the map $\mathbf{xy}^{-1}$ is semicontinuous at any point $(x,e)\in X\times \{e\}$, fix a point $x\in X$ and an open set $W\subset X$ containing the set $xe^{-1}=\{x\}$. We need to find neighborhoods $U_x\in\Tau_x(X)$ and $V_e\in\Tau_e(X)$ such that $U_xV_e^{-1}\subset W$. Since the preuniformity $\vec\Tau_X$ is locally uniform, there exists a neighborhood $V\in\tau_e(X)$ such that $\vec V^{\mp2}[x]\subset W$. Then the neighborhoods $U_x:=xV$ and $V_e:=V$ has the required property: $U_xV_e^{-1}=\bigcup_{y\in xV}yV^{-1}=\bigcup_{y\in \vec V[x]}\vec V^{-1}[y]=\vec V^{\mp2}[x]\subset W$.
\end{proof}

%For a point $x$ of a topological space $X$ by $\ddot x$ we denote the intersection $\bigcap\Tau_x(X)$ of all neighborhoods of $x$ in $X$. The set $\ddot x$ coincides with the singleton $\{x\}$ if and only if the space $X$ is $T_1$ at $x$.

\begin{proposition}\label{p:u-port} A locally associative portator $X$ is uniform if the following conditions are satisfied:
\begin{enumerate}
\item the multiplication $\mathbf{xy}$ is semicontinuous at each point $(x,e)\in X\times X$;
\item  the left inversion $\mathbf{x}^{-1}e:X\to X$, $\mathbf{x}^{-1}e:x\mapsto x^{-1}e$, is continuous at $e$;
\item $e^{-1}e=\{e\}$.
\end{enumerate}
\end{proposition}

\begin{proof} Assume that the conditions (1)--(3) are satisfied. Since $X$ is locally associative, there is a neighborhood $B\in\Tau_e(X)$ such that $(xu)v=x(uv)$ for all $x\in X$ and $u,v\in B$.

To show that the portator $(X,\Tau_X)$ is uniform, fix any neighborhood $W\in\Tau_e(X)$. We need to find a neighborhood $V\in\Tau_e(X)$ such that $\vec V^{\pm3}\subset \vec W$. By (1) and Proposition~\ref{p:lqu-portator}, there exists a neighborhood $U\subset B$ of $e$ such that $\vec U^3\subset\vec W$. Since $e^{-1}e=\{e\}\subset U$, the continuity of the inversion $\mathbf{x}^{-1}e$ at $e$ yields a neighborhood $V\subset U$ of $e$ such that $V^{-1}e\subset U$ and $v^{-1}e\ne\emptyset$ for all $v\in V$. We claim that $\vec V^{\mp2}[x]\subset \vec U^2[x]$ for all $x\in X$. Given any point $y\in \vec V^{\mp2}[x]$, find a point $z\in \vec V[y]\cap \vec V[x]=yV\cap xV$.  Then $z\in yv$ for some $v\in V\subset B$. Choose a point $u\in v^{-1}e$ and observe that $u\in V^{-1}e\subset U\subset B$. By the local associativity of the multiplication, $y\in ye\subset y(vu)=(yv)u=zu\subset \vec U[\vec V[x]]\subset \vec U^2[x]$. So, $\vec V^{\mp2}[x]\subset \vec U^2[x]$ for all $x\in X$, which means that $\vec V^{\mp2}\subset \vec U^2$ and finally $\vec V^{\pm3}=\vec V\vec V^{\mp2}\subset \vec V\vec U^2\subset \vec U^3\subset \vec W$.
\end{proof}

Next, we detect symmetrizable and normally preuniform portators.

\begin{proposition}\label{p:sym-portator} A portator $X$ is symmetrizable if for every $y\in X$ the map ${\mathbf x}^{-1}y:X\multimap X$, $x\mapsto x^{-1}y$, is continuous at $y$ and $y^{-1}y=\{e\}$.
\end{proposition}

\begin{proof} Assume that for every $y\in X$ the map ${\mathbf x}^{-1}y:X\multimap X$, $x\mapsto x^{-1}y$, is continuous at $y$ and $y^{-1}y=\{e\}$. Given any neighborhood $U\subset X$ of $e$ we need to show that the set $\vec U^{-1}[y]$ is a neighborhood of $y$ in $X$. The assumption $y^{-1}y=\{e\}\subset U$ and the continuity of the map $\mathbf{x}^{-1}y$ at $y$ yields a neighborhood $V\in\Tau_y(X)$ of $y$ such that for every $x\in V$ we get $\emptyset\ne x^{-1}y\subset U$. We claim that $V\subset \vec U^{-1}[y]$. Choose any point $x\in V$. The choice of $V$ guarantees that the set $x^{-1}y$ is not empty and hence contains some point $u\in x^{-1}y\subset U$. Then $y\in xu\subset xU=\vec U[x]$ and  $x\in \vec U^{-1}[y]$. Now we see that $\vec U^{-1}[y]\supset V$ is a neighborhood of $y$ in $X$.
%To prove the ``only if'' part, assume that the sportator $(X,t_X)$ is symmetrizable. Then for any $y\in X$  the set $\vec X^{-1}[y]=yX^{-1}$ is a neighborhood of $y$. Observe that for every $x\in yX^{-1}$ there exists $v\in X$ such that $\{y\}=xv$, which implies that $y\in\ran(t_x)$ and hence $\{x\in X:y\in\ran(t_x)\}\supset yX^{-1}$ is a neighborhood of $y$. To show that the map $\mathbf{x}^{-1}y$ is continuous at $y$, fix any neighborhood $U\subset X$ of $\{e\}=y^{-1}y$. Then $V:=\vec U^{-1}[y]$ is a neighborhood of $y$. We claim that $V^{-1}y\subset U$. Given any point $x\in V^{-1}y$, find a point $z\in V$ such that $zx=\{y\}$. Taking into account that $z\in V=\vec U^{-1}[y]$, we conclude that $y\in \vec U[z]=zU$. Then there exists $u\in U$ such that $\{y\}=zu$. The injectivity of the translation $t_z$ and the equality $zx=\{y\}=zu$ imply that $x=u\in U$. This completes the proof of the inclusion $V^{-1}y\subset U$, which yields the continuity of the map ${\mathbf x}^{-1}y$ at $y$.
\end{proof}

\begin{proposition}\label{p:norm-portator} A portator $X$ is normally preuniform if there exists a neighborhood $U_0\subset X$ of $e$ such that for every $u\in U_0$ the map $\mathbf{x}u:X\multimap X$, $\mathbf{x}u:x\mapsto xu$, is continuous.
\end{proposition}

\begin{proof} Given a set $A\subset X$ and an entourage $E\in\vec\Tau_X$, we should prove that $\bar A\subset \overline{E[A]}^\circ$. Choose a neighborhood $U\subset U_0$ of $e$ such that $\vec U\subset E$. We claim that $\bar Au\subset \overline{Au}$ for every $u\in U$. Given any point $x\in\bar Au$ and any neighborhood $B\subset X$ of $e$, we should prove that $\vec B[x]\cap Au\ne\emptyset$.

Since $x\in\bar Au$, there exists a point $y\in\bar A$ such that $x\in yu$. By the continuity of the map $\mathbf xu$ at the point $y$, there exists a neighborhood $O_y\subset X$ such that $O_yu\subset  \vec B[yu]=\vec B[x]=xB$. Moreover, since the set $yu\ni x$ is not empty, we can additionally assume that $au\ne\emptyset$ for any point $a\in O_y$. By the definition of a portator, there exists a neighborhood
$ V\subset B\cap U\subset U_0$ of $e$ such that $yV\subset O_y$. Since $yV$ is a neighborhood of the point $y\in\bar A$, there is a point $a\in yV\cap A\subset O_y$. It follows that $au\ne\emptyset$ and $a\in yv$ for some $v\in V$. Then $ au\subset(yv)u\subset (yV)u\subset O_yu\subset \vec B[yu]=\vec B[x]$ and $\emptyset\ne au\subset \vec B[x]\cap Au$, which implies $x\in \overline{Au}$, $\bar Au\subset \overline{Au}\subset \overline{AU}$, and finally $\bar AU\subset \overline{AU}$. Taking into account that $AU=\vec U[A]\subset E[A]$ and $\bar AU$ is a neighborhood of the set $\bar A$ in $X$, we conclude that $\bar A\subset\overline{AU}^\circ\subset\overline{E[A]}^\circ$.
\end{proof}

Propositions~\ref{p:lqu-portator}, \ref{p:norm-portator}, \ref{t:BR17} imply the following corollary detecting normally quasi-uniform portators.

\begin{corollary} Assume that a locally associative portator $X$ has a neighborhood $U_0$ of the unit $e$ such that  for every $u\in U_0$ the set-valued map $\mathbf{x}u:X\multimap X$, $\mathbf{x}u:x\mapsto xu$, is continuous. If the multiplication $\mathbf{xy}$ is semicontinuous at $(e,e)$, then the portator $X$ has the following properties:
\begin{enumerate}
\item $X$ is normally quasi-uniform;
\item For any $A\subset X$ and $U\in\Tau_e(X)$ there exists a continuous function $f:X\to[0,1]$ such that\newline $A\subset f^{-1}(0)\subset f^{-1}\big[[0,1)\big]\subset \overline{AU}^\circ$;
\item $X$ is Hausdorff (at a point $x\in X$) iff $X$ is semi-Hausdorff (at $x$) iff $X$ is functionally Hausdorff (at $x$);
\item $X$ is regular (at a point $x\in X$) iff $X$ is semi-regular (at $x$) iff $X$ is completely regular (at $x$).
\end{enumerate}
\end{corollary}

In Chapter~\ref{ch:netbase} we shall need special kinds of portators called netportators.

\begin{definition} A portator $X$ is called a \index{netportator}\index{portator!netportator}{\em netportator} if for every $x,y\in X$ the sets $xy$, $x^{-1}y$ are finite.
\end{definition}

\section{Some examples of netportators}\label{s:known-portator}

In this section we define several types of netportators generalizing some well-known structures of Topological Algebra (see e.g. \cite[\S1.2]{AT}). %We recall that a portator $(X,(t_x)_{x\in X}$ is {\em full} if $\dom(t_x)=X$ for all $x\in X$. In this case the multiplication $\mathbf{xy}:X\times X\to X$, $\mathbf{xy}:(x,y)\mapsto xy:=t_x(y)$, is a well-defined binary operation on $X$ with $xe=x$ for all $x\in X$.
A portator $X$ will be called \index{portator!faithful}{\em faithful} if $x^{-1}x=\{e\}$ for all $x\in X$.

\begin{definition}\label{d:trans-algebra} A faithful portator $X$ with multiplication $\mathbf{xy}$ is called
\begin{itemize}
\item  \index{portator!left-topological}{\em left-topological} if for every $x\in X$ the map $x^{-1}\mathbf{y}:X\multimap X$, $x^{-1}\mathbf{y}:y\mapsto x^{-1}y$, is continuous;
\item  \index{portator!right-topological}{\em right-topological} if for every $y\in X$ the right shift $\mathbf{x}y:X\multimap X$, $\mathbf{x}y:x\mapsto xy$, is continuous;
\item  \index{portator!semi-topological}{\em semi-topological} if  $X$ is both left-topological and right-topological;
\item a  \index{portator!quasi-topological}{\em quasi-topological} if semi-topological and for every $y\in Y$  the operation $\mathbf{x}^{-1}y:X\to X$, $\mathbf{x}^{-1}y:x\mapsto x^{-1}y$, is continuous;
\item   \index{portator!para-topological}{\em para-topological} if the multiplication $\mathbf{xy}$ is continuous;
\item   \index{portator!leftpara-topological}{\em leftpara-topological} if $X$ is both  left-topological and para-topological;
\item   \index{portator!quasipara-topological}{\em quasipara-topological} if $X$ is both quasi-topological and para-topological;
\item   \index{portator!invpara-topological}{\em invpara-topological} if $X$ is para-topological and the left inversion $\mathbf x^{-1}e:X\multimap X$, $\mathbf{x}^{-1}e:x\mapsto x^{-1}e$ is continuous;
\item   \index{portator!divpara-topological}{\em divpara-topological} if $X$ and the multiplication  $\mathbf{xy}$ and the left division $\mathbf{x}^{-1}\mathbf y$ are continuous;
\item   \index{portator!paradiv-topological}{\em paradiv-topological} if $X$ and the multiplication  $\mathbf{xy}$ and the right division $\mathbf{xy}^{-1}$ are continuous;
\item   \index{portator!topological}{\em topological} if the operations $\mathbf{xy}$, $\mathbf{xy}^{-1}$, $\mathbf{x}^{-1}\mathbf{y}$ are continuous.
\end{itemize}
\end{definition}

Definition~\ref{d:trans-algebra}  and Propositions~\ref{p:lqu-portator} --  \ref{p:norm-portator} imply the following corollary.

%\label{c:TA=>lqu}

\begin{corollary}
\label{c:TA=>lqu}{\hskip-8pt\tiny.}
\begin{enumerate}
\item Each right-topological portator is normally preuniform.
\item Each quasi-topological portator is symmetrizable.
\item Each para-topological portator is locally quasi-uniform.
\item Each quasipara-topological portator is symmetrizable and locally quasi-uniform.
\item Each paradiv-topological portator is locally uniform.
\item Each locally associative para-topological portator is normally quasi-uniform.
\item Each locally associative invpara-topological portator is uniform.

\end{enumerate}
\end{corollary}
Thus for every faithful portator the following implications hold:
$$
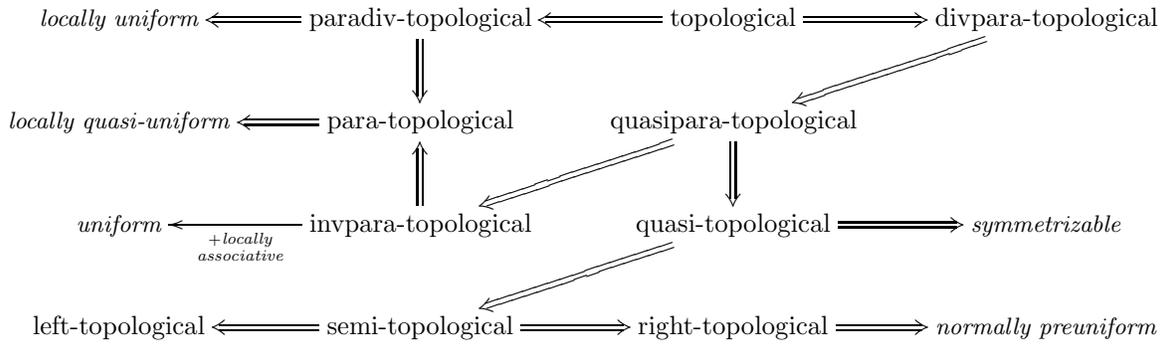
\xymatrix{
\mbox{\small\em locally uniform}&\mbox{paradiv-topological}\ar@{=>}[d]\ar@{=>}[l]&\mbox{topological}\ar@{=>}[l]\ar@{=>}[r]&\mbox{divpara-topological}\ar@{=>}[dl]\\
\mbox{\small\em locally quasi-uniform}&\mbox{para-topological}\ar@{=>}[l]&\mbox{quasipara-topological}\ar@{=>}[d]\ar@{=>}[dl]\\
\mbox{\small\em uniform}&\mbox{invpara-topological}\ar@{=>}[u]\ar^>>>>>>>>>>{+\mbox{\tiny\em locally}\atop\mbox{\tiny\em associative}}[l]&\mbox{quasi-topological}\ar@{=>}[ld]\ar@{=>}[r]&\mbox{\small\em symmetrizable}\\
\mbox{left-topological}&\mbox{semi-topological}\ar@{=>}[l]\ar@{=>}[r]&\mbox{right-topological}\ar@{=>}[r]&\mbox{\small\em normally preuniform}
}
$$

Now we recall the definitions of some well-known objects of Topological Algebra.

\begin{definition}\label{d:lops} A pointed topological space $X$ with a distinguished point $e\in X$ and a binary operation $\mathbf{xy}:X\times X\to X$, $\mathbf{xy}:(x,y)\mapsto xy$, is called
\begin{enumerate}
\item {\em unital} if $xe=x=ex$ for all $x\in X$;
\item a \index{topological loop}{\em topological loop} if $X$ is unital and the maps $X\times X\to X\times X\times X$, $(x,y)\mapsto (x,xy)$, and $X\times X\to X\times X$, $(x,y)\mapsto (xy,y)$, are homeomorphisms;
\item a \index{topological lop}{\em topological lop} if $X$ is unital and the map $X\times X\to X\times X$, $(x,y)\mapsto (x,xy)$, is a homeomorphism;
\item a \index{left-topological lop}{\em left-topological lop} if $X$ is unital and for every $x\in X$ the left shift $X\to X$, $y\mapsto xy$, is a homeomorphism of $X$;
\item a \index{semitopological lop}{\em semitopological lop} if $X$ is a left-topological lop and for every $y\in X$ the right shift $X\to X$, $x\mapsto xy$, is continuous;
\item a \index{paratopological lop}{\em paratopological lop} if $X$ is a left-topological lop with continuous multiplication $\mathbf{xy}$;
\item a \index{topological group}{\em topological group} if $X$ is a topological lop with associative multiplication $\mathbf{xy}$;
\item a \index{paratopological group}{\em paratopological group} if $X$ is a paratopological lop with associative multiplication $\mathbf{xy}$.
\end{enumerate}
\end{definition}

More information on (para)topological groups can be found in \cite{AT} and on  topological loops and topological lops can be found in \cite{HS}, \cite{HM},  \cite{BanRep}, \cite{Ban}.

Corollary~\ref{c:TA=>lqu} implies the following diagram.
$$
\xymatrix{
\mbox{\small\em uniform}\ar@{=>}[rr]&&{\mbox{\small\em normally}\atop\mbox{\small\em quasi-uniform}}\\
{\mbox{\small associative}\atop\mbox{\small parainv-topological}}\atop\mbox{\small  netportator}\ar@{=>}[u]\ar@{=>}[rr]&&
{\mbox{\small associative}\atop\mbox{\small para-topological}}\atop\mbox{\small netportator}\ar@{=>}[u]\\
\mbox{topological group}\ar@{=>}[rr]\ar@{=>}[u]\ar@{=>}[d]&&\mbox{ paratopological group}\ar@{=>}[u]\ar@{=>}[d]\\
\mbox{topological loop}\ar@{=>}[d]\ar@{=>}[r]&
\mbox{topological lop}\ar@{=>}[d]\ar@{=>}[r]&
\mbox{paratopological lop}\ar@{=>}[r]\ar@{=>}[d]&\mbox{semitopological lop}\ar@{=>}[d]\\
\mbox{\small topological}\atop\mbox{\small netportator}\ar@{=>}[r]\ar@{=>}[d]&
{\mbox{\small divpara-topological}}\atop\mbox{\small netportator}\ar@{=>}[r]\ar@{=>}[d]&
{\mbox{\small para-topological}}\atop\mbox{\small netportator}\ar@{=>}[r]\ar@{=>}[d]&
{\mbox{\small right-topological}}\atop\mbox{\small netportator}\ar@{=>}[d]\\
\mbox{\small\em locally}\atop\mbox{\small\em uniform}\ar@{=>}[r]&\mbox{\small\em symmetrizable}\atop\mbox{\small\em locally quasi-uniform}\ar@{=>}[r]&\mbox{\small\em locally}\atop\mbox{\small\em quasi-uniform}&\mbox{\small\em normally}\atop\mbox{\small\em preuniform}
}
$$

\begin{example} The real line $\IR$ endowed with the multiplication $\mathbf{xy}:\IR\times\IR\to\IR$, $\mathbf{xy}:(x,y)\mapsto x+(x^2+1)y$, is a topological lop, which is not locally uniform (because the right division map $\mathbf{xy}^{-1}$ is not continuous at $(0,0)$).
\end{example}

A topological space $X$ is called \index{topological space!rectifiable}\index{rectifiable space}{\em rectifiable} if for some point $e\in X$ there is a homeomorphism $h:X\times X\to X\times X$ such that $h(x,e)=(x,x)$ and $h[\{x\}\times X]=\{x\}\times X$ for all $x\in X$. By \cite[3.2]{BanRep}, a topological space is rectifiable if and only if it is homeomorphic to a topological lop.
Rectifiable spaces were introduced by Arhangelskii \cite{Ar02} and studied by A.Gulko \cite{Gul}. In particular, Gul'ko proved that {\em a rectifiable space is metrizable if and only if it is a first-countable $T_0$-space}.
Since each topological lop is a quasipara-topological portator, this  theorem of Gul'ko can be deduced from the following metrization theorem for quasipara-topological  portators.

\begin{theorem} A quasipara-topological portator $X$ is metrizable if and only if $X$ is a first-countable $T_0$-space.
\end{theorem}

\begin{proof} The ``only if'' part is trivial. To prove the ``if'' part, assume that $X$ is a first-countable $T_0$-space. Fix a countable neighborhood base $\Bas$ at the unit $e$ of $X$. Then $\vec\Bas=\{\vec B:B\in\Bas\}$ is a countable entourage base for $X$. By Corollary~\ref{c:TA=>lqu}(2), this base is symmetrizable and locally quasi-uniform. Then the family $\{\vec B\cap \vec B^{-1}:B\in\Bas\}$ is a countable symmetric locally unform  base for $X$. By Theorem~\ref{t:metr-base}, the space $X$ is metrizable.
\end{proof}

\chapter{Countable netbases for topological spaces}\label{ch:netbase}

In this chapter we introduce the notion of a netbase on a topological space and study properties of topological spaces possessing certain special netbases.

\section{Netbases on topological spaces}

We recall that a family $\mathcal B$ of entourages on a topological space $X$ is called an \index{entourage base}{\em entourage base} for $X$ if (i) for any entourages $B_1,B_2\in\mathcal B$ there exists an entourage $B_3\in\mathcal B$ such that $B_3\subset B_1\cap B_2$ and (ii) for every $x\in X$ the family $\mathcal B[x]=\{B[x]:B\in\mathcal B\}$ is a neighborhood base at $x$ in the topological space $X$.

Also recall that a family $\mathcal N$ of subsets of a topological space $X$ is a \index{$\C^*$-network}{\em $\C^*$-network at a point} $x\in X$ for some family $\C$ of subsets of $X$ if for any neighborhood $O_x\subset X$ of $x$ and any set $C\in\C$ accumulating at $x$ there exists a set $N\in\mathcal N$ such that $x\in N\subset O_x$ and $N\cap C$ is infinite.

\begin{definition}\label{d:enet} Let $\C$ be a family of subsets of a topological space. A \index{$\C^*$-netbase}\index{netbase}{\em $\C^*$-netbase} for a topological space $X$ is a pair $(\mathcal N,\mathcal B)$ consisting of a family $\N$ of entourages on $X$ and an entourage base $\Bas$ for $X$ such that for every entourage $B\in\mathcal B$ the family $\{N[x]:N\in\mathcal N,\;N\subset B\}$ is a $\C^*$-network at $x$.

A $\C^*$-netbase $(\mathcal N,\mathcal B)$  is called
 \begin{itemize}
 \item \index{netbase!countable}{\em countable} if so is the family $\mathcal N$;
 \item \index{netbase!uniform}\index{netbase!quasi-uniform}\index{netbase!locally uniform}\index{netbase!locally uniform}\index{netbase!locally quasi-uniform}
 {\em uniform} (resp. {\em quasi-uniform, locally uniform, locally quasi-uniform}) if so is the entourage base $\mathcal B$.
 \end{itemize}
\end{definition}

Now we describe an interplay between netbases and networks.

\begin{proposition}\label{p:enet<->network} Let $\C$ be a family of subsets in a topological space $X$. \begin{enumerate}
\item If $(\mathcal N,\mathcal B)$ is a $\C^*$-netbase for $X$, then for every $x\in X$ the family $\N[x]=\{N[x]:N\in\mathcal N\}$ is a $\C^*$-network at $x$.
\item If $\mathcal N$ is a $\C^*$-network for $X$, then for the family $\E_{\N}=\{(N\times  N)\cup\Delta_X:N\in\mathcal N\}$ and for any locally uniform base $\mathcal B$ for $X$ the pair $(\E_\N,\mathcal B^{\pm2})$ is a locally uniform $\C^*$-netbase for $X$.
\end{enumerate}
\end{proposition}

\begin{proof} The first statement follows immediately from Definition~\ref{d:enet}.

To prove the second statement, fix an $\C^*$-network $\mathcal N$ for $X$ and a locally uniform  base $\mathcal B$ for $X$. By Proposition~\ref{p:p-lqu+qu}(2), the family $\mathcal B^{\pm2}=\{B^{\pm2}:B\in\Bas\}$ is a locally uniform base for $X$. Consider the family $\E_\N=\{(N\times N)\cup\Delta_X:N\in\mathcal N\}$. To show that the pair $(\E_\N,\mathcal B^{\pm2})$ is a $\C^*$-netbase for $X$, we need to check that for any entourage $B\in\mathcal B$ and any $x\in X$ the family $\{E[x]:E\in\E_\N,\;E\subset B^{\pm2}\}$ is a $\C^*$-network at $x$. Given any neighborhood $O_x\subset X$ of $x$ and a set $C\in\C$ accumulating at $x$, we need to find an entourage $E\in\E$ such that $E\subset B^{\pm2}$, $E[x]\subset O_x$ and $E[x]\cap C$ is infinite.

Since $\mathcal B$ is a locally uniform base for $X$, there exist an entourage $V\in\mathcal B$ such that $V\subset B$ and $V^{\pm2}[x]\subset O_x$.

By definition, the $\C^*$-network $\mathcal N$ contains a set $N\in\mathcal N$ such that $x\in N\subset V[x]$ and $N\cap C$ is infinite. It follows that the entourage $E=(N\times N)\cup\Delta_X$ belongs to $\E_\N$ and $E[x]=N\subset V[x]\subset V^{\pm2}[x]\subset O_x$. It follows that  the intersection $E[x]\cap C=N\cap C$ is infinite, and for every points $y,z\in N\subset V[x]$ we get $z\in V[x]$ and hence $(y,z)\in V[V^{-1}[z]]\subset B^{\pm2}[z]$. So, $E=(N\times N)\cup\Delta_X\subset B^{\pm2}$, which completes the proof.
\end{proof}

We shall also need a weaker ``point-free'' modification of a $\C^*$-netbase.

\begin{definition} Let $\C$ be a family of subsets of a topological space $X$. A \index{$\C^{**}$-netbase}{\em $\C^{**}$-netbase} for $X$ is a pair $(\E,\mathcal B)$ consisting of a family $\E$ of entourages on $X$ and a base $\Bas$ for $X$ such that for any entourage $B\in\mathcal B$ and any set $C\in\C$ accumulating at some point of $X$, the family $\E$ contains an entourage $E\subset B$ such that the intersection $E[x]\cap C$ is infinite for some point $x\in X$.

A $\C^{**}$-netbase $(\E,\Bas)$ is \index{netbase!countable}{\em countable} if so is the family $\E$ and {\em locally} ({\em quasi-}){\em uniform} if so is the entourage base $\Bas$.\index{netbase!uniform}\index{netbase!locally uniform}\index{netbase!locally quasi-uniform}\index{netbase!quasi-uniform}
\end{definition}

For any family $\C$ of subsets of a regular space $X$ we get the following implications:
$$
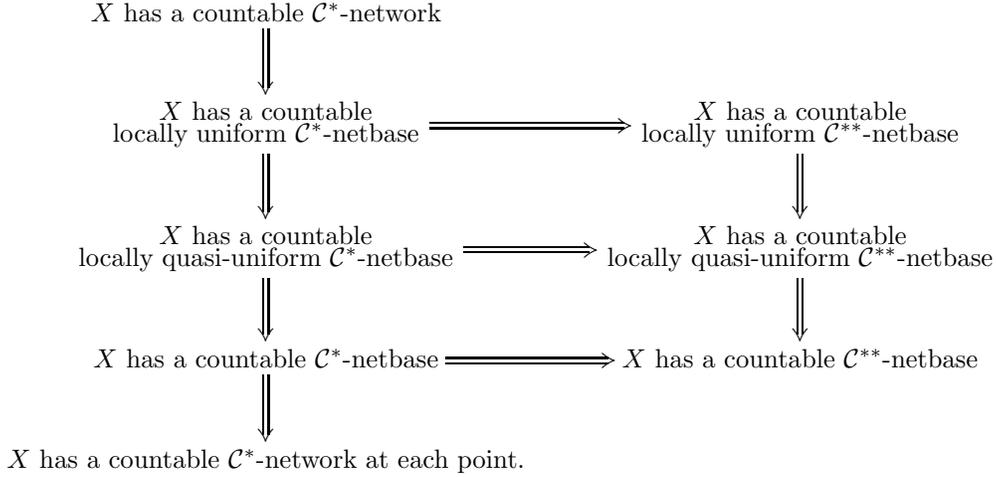
\xymatrix{
\mbox{$X$ has a countable $\C^*$-network}\ar@{=>}[d]\\
\mbox{$X$ has a countable}\atop\mbox{locally uniform $\C^*$-netbase}\ar@{=>}[r]\ar@{=>}[d]
&\mbox{$X$ has a countable}\atop\mbox{locally uniform $\C^{**}$-netbase}\ar@{=>}[d]\\
\mbox{$X$ has a countable}\atop\mbox{locally quasi-uniform $\C^*$-netbase}\ar@{=>}[d]\ar@{=>}[r]
&\mbox{$X$ has a countable}\atop\mbox{locally quasi-uniform $\C^{**}$-netbase}\ar@{=>}[d]\\
\mbox{$X$ has a countable $\C^{*}$-netbase}\ar@{=>}[d]\ar@{=>}[r]
&\mbox{$X$ has a countable $\C^{**}$-netbase}\\
\mbox{$X$ has a countable $\C^*$-network at each point.}
}
$$

Next we study netbases for subspaces of topological spaces. For a subspace $Z$ of a topological space $X$ and a family $\C$ of subsets of $X$ put $\C_Z=\{C\in\C:C\subset Z\}$. For a family $\E$ of entourages on $X$ let $\E|_Z:=\{(Z\times Z)\cap E:E\in\E\}$ be the induced family of entourages on $Z$. Also put $\E^{\pm2}:=\{E^{\pm2}:E\in\E\}$.

\begin{proposition}\label{p:netbase-her} Let $Z$ be a subspace of a topological space $X$ and $\C$ be a family of subsets of $X$.
\begin{enumerate}
\item If $(\E,\Bas)$ is a $\C^*$-netbase for $X$, then $(\E|_Z,\Bas|_Z)$  is a $\C^*_Z$-netbase for the space $Z$.
\item If $(\E,\Bas)$ is a locally uniform $\C^{**}$-netbase for $X$, then $(\E^{\pm2}|_Z,\Bas^{\pm2}|_Z)$ is a locally uniform $\C^{**}_Z$-netbase for the space $Z$.
\end{enumerate}
\end{proposition}

\begin{proof} The first statement follows from the definition of a netbase.

To prove the second statement, assume that $(\E,\Bas)$ is a locally uniform $\C^{**}$-netbase. Proposition~\ref{p:p-lqu+qu} implies that $\Bas^{\pm2}$ is a locally uniform base for $X$ and hence $\Bas^{\pm2}|_Z$ is a locally uniform base for $Z$. To prove that $(\E^{\pm2}|_Z,\Bas^{\pm2}|_Z)$ is a $\C^{**}_Z$-netbase for $Z$, it suffices for every set $C\in\C_Z$ accumulating at some point of $Z$ and every entourage $B\in\Bas$ to find an entourage $E\in\E$ and a point $z\in Z$ such that $E\subset B$ and the set $E^{\pm2}[z]\cap C$ is infinite. Since $(\E,\Bas)$ is a $\C^{**}$-netbase, there exists an entourage $E\in\E$ and a point $x\in X$ such that $E\subset B$ and the set $E[x]\cap C$ is infinite.
Choose any point $z\in E[x]\cap C$ and observe that $x\in E^{-1}[z]$ and hence the intersection $E^{\pm2}[z]\cap C\supset E[x]\cap C$ is infinite.
\end{proof}

In the role of the family $\C$ we shall consider the following four families of subsets of a topological space $X$:
\begin{itemize}
\item the family $\as$ of all subsets of $X$;
\item the family $\css$ of all countable subsets of $X$;
\item the family $\cs$ of convergent sequences in $X$;
\item the family $\ccs$ of countable sets with countably compact closure in $X$.
\end{itemize}

For any pair $(\E,\Bas)$ of families of entourages on a Hausdorff space we get the following implications:
$$
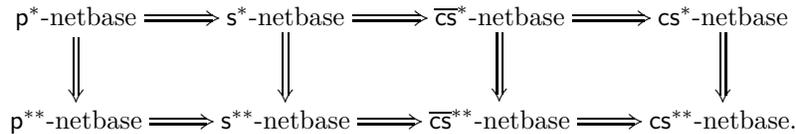
\xymatrix{
\mbox{$\as^*$-netbase}\ar@{=>}[r]\ar@{=>}[d]
&\mbox{$\css^*$-netbase}\ar@{=>}[r]\ar@{=>}[d]
&\mbox{$\ccs^*$-netbase}\ar@{=>}[r]\ar@{=>}[d]
&\mbox{$\cs^*$-netbase}\ar@{=>}[d]\\
\mbox{$\as^{**}$-netbase}\ar@{=>}[r]
&\mbox{$\css^{**}$-netbase} \ar@{=>}[r]
&\mbox{$\ccs^{**}$-netbase}\ar@{=>}[r]
&\mbox{$\cs^{**}$-netbase}.
}
$$

Under some conditions on a topological space $X$ the horizontal arrows can be reversed. For an entourage $E\subset X\times X$ on a topological space $X$ let $\bar E=\bigcup_{x\in X}\{x\}\times \overline{E[x]}$. For a family $\E$ of entourages on $X$ let $\bar\E^{\pm2}:=\{\bar E^{\pm2}:E\in\E\}$ and $\mathcal E^{\mp4}=\{E^{\mp4}:E\in\E\}$.

\begin{proposition}\label{p:c*<->c**} Let $X$ be a topological space.
\begin{enumerate}
\item If $(\E,\Bas)$ is a locally uniform $\cs^{**}$-netbase for $X$, then $(\bar\E^{\pm2},\Bas^{\mp4})$ is a locally uniform $\cs^*$-netbase for $X$.
\item If each countably compact closed subspace of $X$ is sequentially compact, then each $\cs^{**}$-netbase for $X$ is a $\cccs^{**}$-netbase for $X$.
\item If the space $X$ is countably tight, then each $\css^*$-netbase for $X$ is a $\as^*$-netbase for $X$ and each $\css^{**}$-netbase for $X$ is a $\as^{**}$-netbase for $X$.
\item If the space $X$ is countably compact, then each $\ccs^*$-netbase (resp. $\ccs^{**}$-netbase) for $X$ is an $\css^*$-netbase (resp. $\css^{**}$-netbase) for $X$.
\end{enumerate}
\end{proposition}

\begin{proof} 1. Let $(\E,\Bas)$ be a locally uniform $\cs^{**}$-netbase for the space $X$. Proposition~\ref{p:p-lqu+qu} implies that $\Bas^{\mp4}$ is a locally uniform base for $X$. To show that $(\bar\E^{\pm2},\Bas^{\mp4})$ is a $\cs^*$-netbase for $X$, we need to prove that for every $x\in X$ and $B\in\Bas$ the family $\mathcal N=\{E[x]:E\in\bar \E^{\pm2},\;E\subset B^{\mp4}\}$ is a $\cs^*$-network at $x$.

Given a neighborhood $O_x\subset X$ of $x$ and a sequence $S\subset X$ convergent to a point $x$, we need to find a set $N\in\mathcal N$ such that $x\in N\subset O_x$ and $N$ has infinite intersection with $S$. By Proposition~\ref{p:p-lqu+qu}, the locally uniform base $\Bas$ contain an entourage $U\subset B$ such that $U^{\mp4}[x]\subset O_x$. Since $(\E,\Bas)$ is a $\cs^{**}$-netbase for $X$, the family $\E$ contains an entourage $E\subset U$ such that for some point $z\in X$ the intersection $E[z]\cap S$ is infinite. Then $x\in\overline{E[z]\cap S}\subset \bar E[z]$ and $z\in \bar E^{-1}[x]$. It follows that the entourage $\bar E^{\pm2}$ belongs to the family $\bar\E^{\pm2}$ and the set $\bar E^{\pm2}[x]\cap S\supset E[\bar E^{-1}[x]]\cap S\supset E[z]\cap S$ is infinite. Observe that for every point $y\in X$ we get $\bar E[y]\subset \bar U[y]\subset U^{\mp2}[y]$. Consequently, $\bar E\subset U^{\mp2}$ and $\bar E^{\pm2}\subset U^{\mp4}\subset B^{\mp4}$, which implies that the ball $N:=\bar E^{\pm2}[x]$ belongs to the family $\N$. Observing that $x\in N[x]\subset U^{\mp4}[x]\subset O_x$ and $N\cap S\supset E[z]\cap S$ is infinite, we complete the proof of the first statement.
\smallskip

The statements (2)--(4) easily follow from the corresponding definitions.
\end{proof}

\section{Characterizing first-countable spaces}

In this section we characterize first-countable spaces in terms of bases and netbases.

\begin{proposition}\label{p:1-ebase} A topological space $X$ is first-countable if and only if it has a countable entourage base.
\end{proposition}

\begin{proof} The ``if'' part is trivially follows from the definition of an entourage base. To prove the ``only if'' part, assume that the space $X$ is first-countable. At each point $x\in X$ fix a decreasing neighborhood base $(B_n[x])_{n\in\w}$, and for every $n\in\w$ consider the entourage $B_n=\bigcup_{x\in X}\{x\}\times B_n[x]$. It follows that the countable family $\Bas=\{B_n:n\in\w\}$ is an entourage base for $X$.
\end{proof}

In the following theorem we use locally uniform netbases to detect first-countable spaces.

\begin{theorem}\label{t:1-enet} If a topological space $X$ has a countable locally uniform $\cccs^{**}$-netbase, then $X$ is first-countable at a point $x\in X$ if and only if
$X$ is a $q$-space at $x$.
\end{theorem}

\begin{proof} The ``only if'' part is trivial. To prove the ``if'' part, assume that $X$ is a $q$-space at $x$. Let $(\E,\Bas)$ be a countable locally uniform $\cccs^{**}$-netbase. Without lost of generality, we can assume that the family $\E$ is closed under finite unions. By Proposition~\ref{p:lu=>regular}, the space $X$ is regular. Then we can find a decreasing sequence $(V_n)_{n\in\w}$ of open neighborhoods of $x$ such that each sequence $(x_n)_{n\in\w}\in\prod_{n\in\w}\bar V_n$ accumulates at some point of $X$.

For every $k\in\w$ and $E\in\E$ choose a subset $D_{k,E}\subset V_k$ of smallest possible cardinality such that $V_k\subset E^{\pm2}[D_{k,E}]$.

\begin{claim}\label{cl:l-cov} For every entourage $B\in\Bas$ there exists $k\in\w$ and $E\in\E$ such that $E\subset B$ and the set $D_{k,E}$ is finite.
\end{claim}

\begin{proof}
Since the family $\E_B:=\{E\in\E:E\subset B\}$ is countable and closed under finite unions, there exists an increasing sequence of entourages $\{E_k\}_{k\in\w}\subset\E_B$ such that every entourage $E\in\E_B$ is contained in some $E_k$. To derive a contradiction, assume that for every $k\in\w$ and $E\in\E_B$ the set $D_{k,E}$ is infinite.
Then we can inductively choose a sequence of points $(x_k)_{k\in\w}$ such that
$x_k\in V_k\setminus\bigcup_{i<k}E^{\pm2}_k[x_i]$ for all $k\in\w$. The choice of the sequence $(V_k)_{k\in\w}$ guarantees that the set $\{x_k\}_{k\in\w}$ is countably compact in $X$ and the intersection $K=\bigcap_{k\in\w}\bar V_k$ is a closed countably compact subset of $X$.

By the regularity of $X$, for any $x\in X$ the closure  $\overline{\{x\}}$ of the singleton $\{x\}\subset X$ is contained in any neighborhood of $x$.
This implies that for every $k\in\w$ the closure $\overline{\{x_k\}}$ of the singleton $\{x_k\}$ is compact and the set $F=K\cup\bigcup_{k\in\w}\overline{\{x_k\}}$ is countably compact in $X$.
We claim that the set $F$ is closed in $X$. Given any point $x\in X\setminus F\subset X\setminus K$, find $k\in\w$ such that $z\notin\bar V_k$ and observe that the open neighborhood $O_z=X\setminus (\bar V_k\cup\bigcup_{i<k}\overline{\{x_i\}})$ of $x$ does not intersect the set $F$, which implies that the countably compact set $F$ is closed in $X$. Since $\{x_k\}_{k\in\w}\subset F$, the closure of the set $\{x_k\}_{k\in\w}$ in $X$ is countably compact.

 Since $(\E,\Bas)$ is a $\cccs^{**}$-netbase, there exists an entourage $E\in\E_B$ such that for some point $x\in X$ the $E$-ball $E[x]$ contains infinitely many points $x_k$, $k\in\w$. Choose a number $k\in\w$ such that $E\subset E_k$ and then choose numbers $m>n>k$ such that $x_n,x_m\in E[x]\subset E_k[x]$. Then $$x\in E^{-1}_{k}[x_n]\cap E^{-1}_{k}[x_m]\subset E_{m}^{-1}[x_n]\cap E^{-1}_{m}[x_m]$$ and hence $x_m\in E_m^{\pm2}[x_n]$, which contradicts the choice of the point $x_m$.
\end{proof}

Now for every $k\in\w$ and $E\in\E$ we shall construct an open neighborhood $W_{k,E}$ of $x$ in the following way. If the set $D_{k,E}$ is infinite, then put $W_{k,E}=V_k$. If  $D_{k,E}$ is finite, then put
 $$W_{k,E}=V_{k}\setminus\textstyle{\bigcup}\big\{\overline{E^{\pm2}[z]}:z\in D_{k,E},\;x\notin  \overline{E^{\pm2}[z]}\big\}.$$

We claim that $\{W_{k,E}:k\in\w,\;E\in\E\}$ is a countable neighborhood base at $x$. Given a neighborhood $O_x\subset X$ of $x$, we should find $k\in\w$ and $E\in\E$ such that $W_{k,E}\subset O_x$. By Proposition~\ref{p:p-lqu+qu}, there exists an entourage $B\in\Bas$ such that $B^{\pm5}[x]\subset O_x$. By Claim~\ref{cl:l-cov}, there exist $k\in\w$ and $E\in\E_B$ such that the set $D_{k,E}$ is finite. We claim that $W_{k,E}\subset O_x$. Given any point $y\in W_{k,E}\subset V_k\subset E^{\pm2}[D_{k,E}]$, find a point $z\in D_{k,E}$ such that $y\in E^{\pm2}[z]$.
The definition of the set $W_{k,E}\ni y$ guarantees that $x\in \overline{E^{\pm2}[z]}\subset \overline{B^{\pm2}[z]}\subset B^{\mp3}[z]$ and hence $z\in B^{\pm3}[x]$.
Finally,
$y\in E^{\pm2}[z]\subset B^{\pm2}[z]\subset B^{\pm2}B^{\pm3}[x]=B^{\pm5}[x]\subset O_x$.
\end{proof}

\section{Countable locally uniform netbases in $w\Delta$-spaces}\label{s:wD-Gd}

%The following theorem also gives an alternative proof of the metrization theorem of Cascales and Orihuela \cite{CO}.

We recall that a topological space $X$ is  a \index{$w\Delta$-space}\index{topological space!$w\Delta$-space}{\em $w\Delta$-space} if $X$ is a regular $T_0$-space admitting a sequence $(\V_n)_{n\in\w}$ of open covers such that for every point $x\in X$ every sequence $(x_n)_{n\in\w}\in\prod_{n\in\w}\St(x;\V_n)$ accumulates at some point of $X$.

 \begin{theorem}\label{t:wD=>Gd} If $X$ is a $w\Delta$-space with a countable locally uniform $\cccs^{**}$-netbase, then $X$ is first-countable and has a $G_\delta$-diagonal.
 \end{theorem}

  \begin{proof} Let $(\E,\Bas)$ be a countable locally uniform $\cccs^{**}$-netbase for $X$. We lose no generality assuming that the family $\E$ is closed under finite unions.

   Since $X$ is a $w\Delta$-space, there exists a  sequence $(\V_n)_{n\in\w}$ of open covers of $X$ such that for every $x\in X$, any sequence $(x_n)_{n\in\w}\in\prod_{n\in\w}\St(x,\V_n)$ has an accumulation point in $X$. This implies that $X$ is a $q$-space and by Theorem~\ref{t:1-enet}, the space $X$ is first countable.

For every $k\in\w$ and $E\in\E$ we shall construct an open cover $\W_{k,E}=\{W_{k,E,x}:x\in X\}$ of $X$ in the following way.
For every $x\in X$ choose a set $V_{k,x}\in \V_k$ containing $x$. Let $D_{k,E,x}\subset X$ be a subset of smallest possible cardinality such that $V_{k,x}\subset E^{\pm2}[D_{k,E,x}]$. If $D_{k,E,x}$ is infinite, then put $W_{k,E,x}:=V_{k,x}$. If  $D_{k,E,x}$ is finite, then put
 $$W_{k,E,x}:=V_{k,x}\setminus\bigcup
 \big\{\overline{E^{\pm2}[z]}:z\in D_{k,E,x},\;x\notin  \overline{E^{\pm2}[z]}\big\}.$$
 Finally, consider the open cover $\W_{k,E}=\{W_{k,E,x}:k\in\w,\;E\in\E,\;x\in X\}$ of $X$.

\begin{claim}\label{cl:inter}  For every $x\in X$ we get the equality $$\{x\}= \bigcap_{(k,E)\in\w\times\E} \St(x;\W_{k,E}).$$
\end{claim}

 \begin{proof} Assume that the intersection in the claim contains some point $y\ne x$. By the $T_1$-property of $X$ and Proposition~\ref{p:p-lqu+qu}, there is  $U\in\U$ such that  $x\notin U^{\mp 5}[y]$.  Using the first-countability of $X$ and repeating the argument of  Claim~\ref{cl:l-cov}, we can prove that there are $k\in\w$, $E\in\E$, and a finite subset $D\subset X$ such that $E\subset U$ and $\St(x;\V_k)\subset E^{\pm2}[D]$.

By our assumption,  $y\in\St(x;\W_{k,E})$. Then we can find a point $z\in X$ such that $\{x,y\}\subset W_{k,E,z}\subset V_{k,z}\subset \St(x;\V_k)$. It follows that $|D_{k,E,z}|\le|D|<\w$ and hence $y\in E^{\pm2}[s]$ for some $s\in D_{k,E,z}$. We claim that $x\notin \overline{E^{\pm2}[s]}$. In the opposite case $x\in \overline{E^{\pm2}[s]}\subset \overline{U^{\pm2}[s]}\subset U^{\mp3}[s]\subset U^{\mp3}[E^{\pm2}[y]]=U^{\mp5}[y]$, which contradicts the choice of the entourage $U$.
\end{proof}
Now consider the countable family $\{\W_{k,E}\}_{(k,E)\in\w\times\E}$ of open covers of $X$ and observe that Claim~\ref{cl:inter} implies that
$$\Delta_X=\bigcap_{(k,E)\in\w\times\E}\textstyle{\bigcup}\{W\times W:W\in\W_{k,E}\},$$
which means that $X$ has a $G_\delta$-diagonal.
\end{proof}

\section{Countable locally uniform netbases in $\Sigma$-spaces}\label{s:lu-Sigma}

The main result of this section is the following theorem.

%Now we prove a far generalization of the  metrization theorem of Cascales and Orihuela \cite{CO}.
\index{$\Sigma$-space}\index{topological space!$\Sigma$-space}
\begin{theorem}\label{t:Sigma} Each $\Sigma$-space with a countable locally uniform $\cccs^{**}$-netbase is a $\sigma$-space.
\end{theorem}

\begin{proof}  Let $\F$ be a $\sigma$-discrete $\C$-network for some cover $\C$ of $X$ by non-empty closed countably compact subspaces. Without loss of generality we can assume that the family $\F$ is closed under finite intersections.

Let $(\E,\Bas)$ be a countable locally uniform $\cccs^{**}$-netbase for $X$. Without loss of generality, the family $\E$ is closed under finite unions. For every entourage $B\in\Bas$ let $\E_B:=\{E\in\E:E\subset B\}$.

For every entourage $E\in\E$ and set $F\in\F$, fix a subset $D_{E,F}\subset F$ of smallest possible cardinality such that $F\subset E^{\pm2}[D_{E,F}]$.

 \begin{claim}\label{cl3.3} For any set $C\in\C$ and entourage $B\in\Bas$ there exist an entourage $E\in\E_B$ and a set $F\in\F$ such that $C\subset F$ and the set $D_{E,F}$ is finite.
 \end{claim}

\begin{proof} Taking into account that the family $\E_B$ is countable and closed under finite unions, we can choose an increasing sequence of entourages $\{E_k\}_{k\in\w}\subset \E_B$ such that each $E\in\E_B$ is contained in some $E_k$.

The $\sigma$-discreteness of $\F$ implies that the family $\F(C)=\{F\in\F:C\subset F\}$ is countable and not empty. So, we can enumerate it as $\{F_n\}_{n\in\w}$. Since $\F$ is closed under finite intersections, for every $n\in\w$ the set $\hat F_n=\bigcap_{k\le n}F_k$ belongs to the family $\F(C)$. It is standard to show that $\bigcap_{n\in\w}\hat F_n=C$.

We claim that for some $k\in\w$ the set $D_{E_k,\hat F_k}$ is finite. In the opposite case we can inductively choose a sequence $(x_k)_{k\in\w}$ of points such that $x_k\in \hat F_k\setminus \bigcup_{i<k}E_k^{\pm2}[x_i]$ for all $k\in\w$.

\begin{claim}\label{cl:Sigma-cc} The set $\ddot C=C\cup \{x_k\}_{k\in\w}$ is closed and countably compact in $X$.
\end{claim}

\begin{proof} To see that the set $\ddot C$ is closed, take any point $x\in X\setminus \ddot C$. Since $\bigcap_{k\in\w}\hat F_k=C\subset \ddot C$, there is $k\in\w$ such that $x\notin\hat F_k$. Then $O_x=X\setminus(\hat F_k\cup\{x_i\}_{i<k})$ is an open neighborhood of $O_x$, disjoint with the set $\ddot C$. So, the set $\ddot C$ is closed in $X$. Assuming that this set is not countably compact and taking into account that the set $C$ is countable compact, we can find an increasing number sequence $(k_i)_{i\in\w}$ such that the subsequence $(x_{k_i})_{i\in\w}$ is not contained in $C$ and has no accumulation points in $X$. This implies that the set $\{x_{k_i}\}_{i\in\w}$ is closed in $X$ and hence $X\setminus\{x_{k_i}\}_{i\in\w}$ is an open neighborhood of $C$ in $X$. It follows that the family $\F(C)$ contains a set $F\in\F(C)$ which is disjoint with the closed set $\{x_{k_i}\}_{i\in\w}$.
Find $k\in\w$ such that $F=F_k$ and choose a number $i\in\w$ with $k_i>k$. Then $x_{k_i}\in \hat F_{k_i}\subset F_k\subset X\setminus\{x_{k_i}\}$ and this is a desired contradiction, showing that the set $\ddot C$ is countably compact.
\end{proof}

Claim~\ref{cl:Sigma-cc} implies that the set $\{x_k\}_{k\in\w}$ has countably compact closure in $X$. Since $(\E,\Bas)$ is a $\ccs^{**}$-netbase, there exists an entourage $E\subset U$ such that for some point $x\in X$ the $E$-ball $E[x]$ contains infinitely many points of the sequence $(x_k)_{k\in\w}$. Find a number $k\in\w$ with $E\subset E_k$ and then choose two numbers $m>n>k$ such that $x_m,x_n\in E[x]\subset E_k[x]$. Then $x_m\in E_k[x]\subset E_k[E_k^{-1}[x_n]]=E_k^{\pm2}[x_n]$, which contradicts the choice of the point $x_m$.
This contradiction completes the proof of Claim~\ref{cl3.3}.
\end{proof}

Write the $\sigma$-discrete family $\F$ as the countable union $\bigcup_{i\in\w}\F_i$ of discrete families $\F_i$ in $X$.
 For every $i\in\w$ and $E\in\E$ consider the subfamily $\F_{i,E}=\{F\in\F_i:|D_{E,F}|<\w\}\subset \F_i$.
 For every $i\in\w$, $E\in\E$, $F\in\F_{i,E}$, and $x\in D_{E,F}$ consider the set $N_{i,E,F,x}=E^{\pm2}[x]\cap F$.
Observe that the family $\mathcal N_{i,E,F}=\{N_{i,E,F,x}:x\in D_{E,F}\}$ is finite, the family $N_{i,E}=\bigcup_{F\in\F_{i,E}}\mathcal N_{i,E,F}$ is $\sigma$-discrete, and so is the family $\mathcal N=\bigcup_{i\in\w}\bigcup_{E\in\E}\mathcal N_{i,E}$.

It remains to prove that the family $\mathcal N$ is a network for $X$.
Given a point $x\in X$ and a neighborhood $O_x\subset X$ of $x$, apply Proposition~\ref{p:p-lqu+qu} to find an entourage $B\in\Bas$ such that $B^{\pm4}[x]\subset O_x$.
Find a set $C\in\C$ such that $x\in C$. By Claim~\ref{cl3.3}, there is a set $F\in\F(C)$ and an entourage $E\in\E$ such that $E\subset B$ and the set $D_{E,F}$ is finite. Then $x\in F\subset E^{\pm2}[D_{E,F}]$ and hence $x\in E^{\pm2}[z]\cap F\in \mathcal N$ for some $z\in D_{E,F}$. It follows that $z\in E^{\pm2}[x]\subset B^{\pm2}[x]$ and hence $x\in E^{\pm2}[z]\cap F\subset B^{\pm2}[z]\subset B^{\pm4}[x]\subset O_x$,
witnessing that $\mathcal N$ is a $\sigma$-discrete network for $X$ and $X$ is a $\sigma$-space.
\end{proof}

\section{Countable locally quasi-uniform netbases in strong $\sigma$-spaces}\label{s:sigma-lqu}

A regular $T_0$-space $X$ is called a \index{strong $\sigma$-space}\index{topological space!strong $\sigma$-space}{\em strong $\sigma$-space} if $X$ has a network $\mathcal N$ that can be written as the countable union $\mathcal N=\bigcup_{i\in\w}\mathcal N_i$ of strongly discrete families $\mathcal N_i$, $i\in\w$, in $X$. Since each discrete family of subsets in a collectionwise normal space is strongly discrete and a $\sigma$-space is paracompact if and only if its is collectionwise normal \cite[p.446]{Grue}, we get the implications:
$$\mbox{paracompact $\sigma$-space $\Leftrightarrow$ collectionwise normal $\sigma$-space $\Ra$ strong $\sigma$-space $\Ra$ $\sigma$-space}.$$

A family $\C$ of subsets of a topological space $X$ is called \index{hereditary family}{\em hereditary} if for any set $C\in\C$ accumulating at a point $x\in X$ and for any neighborhood $O_x\subset X$ of $x$ there exists a set $C'\in \C$ such that $C'\subset C\cap O_x$ and $C'$ accumulates at $x$.

\begin{theorem}\label{t:s-sigma} Let $\C$ be a hereditary family of subsets in a strong $\sigma$-space $X$.
If $X$ has a countable locally quasi-uniform $\C^{*}$-netbase, then $X$ has a $\sigma$-discrete $\C^*$-network.
\end{theorem}

\begin{proof} Let $(\E,\Bas)$ be a countable locally quasi-uniform $\C^*$-netbase for $X$.
Let $\F$ be a strongly $\sigma$-discrete network for $X$. Write $\F$ as the countable union $\F=\bigcup_{i\in\w}\F_i$ of strongly discrete families $\F_i$ in $X$. For every $i\in\w$ and $F\in\F_i$ choose an open neighborhood $O_F\subset X$ such that the family $(O_F)_{F\in\F_i}$ is discrete in $X$.

It follows that for every $E\in\E$ the family $\mathcal F_{i,E}=\{O_F\cap E[F]:F\in\F_i\}$ is discrete in $X$ and the family $\N=\bigcup_{i\in\w}\bigcup_{E\in\E}\mathcal F_{i,E}$ is $\sigma$-discrete. We claim that this family is a $\C^*$-network for $X$.

Given a point $x\in X$, a neighborhood $O_x\subset X$ of $x$, and a set $C\in\C$ accumulating at $x$, we need to find a set $N\in\N$ such that $x\in N\subset O_x$ and $N\cap C$ is infinite.
Being locally quasi-uniform, the base $\Bas$ contains an entourage $B$ such that $B^2[x]\subset O_x$. The network $\F$ contains a set $F\in\F$ such that $x\in F\subset B[x]$. Find $i\in\w$ such that $F\in\F_i$. The family $\C$, being hereditary, contains a set $C'\subset C\cap O_F$ accumulating at $x$. Since $(\E,\Bas)$ is a $\C^*$-netbase, the family $\E$ contains an entourage $E\subset B$ such that $E[x]\cap C'$ is infinite. Then the set $O_F\cap E[F]\in \F_{i,E}$ has infinite intersection $(O_F\cap E[F])\cap C\supset E[x]\cap (O_F\cap C)\supset E[x]\cap C'$ with the set $C$. Finally, we observe that $x\in F\subset E[F]\subset B[B[x]]\subset O_x$.
\end{proof}

\begin{corollary}\label{c:Sigma=>aleph} Each collectionwise normal $\Sigma$-space $X$ with a locally uniform $\cccs^{**}$-netbase is a paracompact $\aleph$-space.
\end{corollary}

\begin{proof} By Theorem~\ref{t:Sigma}, the $\Sigma$-space $X$ is a $\sigma$-space. By \cite[p.446]{Grue}, the collectionwise normal $\sigma$-space $X$ is paracompact and hence is a strong $\sigma$-space. By Proposition~\ref{p:c*<->c**}, $X$ has a countable locally uniform $\cs^*$-netbase.
 By Theorem~\ref{t:s-sigma}, the strong $\sigma$-space $X$ has a $\sigma$-discrete $\cs^*$-network and hence is an $\aleph$-space.
 \end{proof}

\begin{corollary}\label{c:cs-eq} Let $\C$ be a hereditary family of subsets of a regular $T_0$-space $X$. The following conditions are equivalent:
\begin{enumerate}
\item $X$ has a countable $\C^*$-network;
\item $X$ is cosmic and has a countable locally uniform $\C^*$-netbase;
\item $X$ is cosmic and has a countable locally quasi-uniform $\C^*$-netbase.
\end{enumerate}
\end{corollary}

\begin{proof} The implication $(1)\Ra(2)$ follows from Propositions~\ref{p:enet<->network}(2), \ref{p:exist-lubase}, and $(2)\Ra(3)$ is trivial. To prove that $(3)\Ra(1)$, assume that the space $X$ is cosmic and has a countable locally quasi-uniform $\C^*$-netbase. Being cosmic, the space $X$ is collectionwise normal and hence is a strong $\sigma$-space. By Theorems~\ref{t:s-sigma}, the space $X$ has a $\sigma$-discrete $\C^*$-network $\mathcal N$. Since the space $X$ is Lindel\"of, the $\sigma$-discrete $\C^*$-network $\mathcal N$ is countable.
\end{proof}

The following theorem characterizes $\aleph_0$-spaces , i.e., regular $T_0$-spaces with a countable $\cs^*$-network.

\begin{theorem}\label{t:aleph0} For a topological space $X$ the following conditions are equivalent:
\begin{enumerate}
\item[\textup(1)] $X$ is an $\aleph_0$-space;
\item[\textup(2)] $X$ is cosmic and has a countable locally uniform $\cs^*$-netbase;
\item[\textup(3)] $X$ is cosmic and has a countable locally quasi-uniform $\cs^{*}$-netbase;
\item[\textup(4)] $X$ is cosmic and has a countable locally uniform $\cs^{**}$-netbase;
\item[\textup(5)] $X$ is cosmic and has a countable locally uniform $\ccs^{*}$-netbase;
\item[\textup(6)] $X$ is a $\Sigma$-space with countable extent and a countable locally uniform $\cccs^{**}$-netbase.
\end{enumerate}
The equivalent conditions \textup{(1)--(6)} imply the conditions
\begin{enumerate}
\item[\textup(7)] $X$ is cosmic and has a countable $\cs^*$-netbase.
\end{enumerate}
\end{theorem}

\begin{proof} The equivalences $(1)\Leftrightarrow(2)\Leftrightarrow(3)$ follow from Corollary~\ref{c:cs-eq} applied to the family $\cs$ of convergent sequences in $X$. The equivalence $(2)\Leftrightarrow(4)$ was proved in Proposition~\ref{p:c*<->c**}(1). The implication $(5)\Ra(2)$ is trivial and $(2)\Ra(5)$ follows from the (sequential) compactness of closures of countably compact sets in cosmic spaces. The implication $(5)\Ra(6)$ is trivial and
$(6)\Ra(4)$ follows from  Theorem~\ref{t:Sigma}. The final implication $(2)\Ra(7)$ is trivial.
\end{proof}

\begin{remark} The condition $(7)$ of Theorem~\ref{t:aleph0} is not equivalent to the conditions (1)--(6): according to \cite[4.6]{BL}, there exists a first-countable cosmic space $X$ which fails to be an $\aleph_0$-space. By Proposition~\ref{p:1-ebase}, the first-countable space $X$ has a countable entourage base $\mathcal B$. Then the pair $(\Bas,\Bas)$ is a countable $\cs^*$-netbase for $X$.
\end{remark}

\section{Characterizing metrizability via netbases}\label{s:lu-metr}

In this section we characterize metrizable spaces using countable locally uniform netbases.
The characterization involves the notion of a closed-$\bar G_\delta$ space.

A subset $F$ of a topological space $X$ is called a {\em $\bar G_\delta$-set} if $F=\bigcap_{n\in\w}W_n=\bigcap_{n\in\w}\overline{W}_n$ for some sequence $(W_n)_{n\in\w}$ of open sets in $X$. It is easy to see that a subset $F$ of a (perfectly normal) space is a $\bar G_\delta$-set (if and) only if $F$ a closed $G_\delta$-set.

A topological space $X$ is defined to be a \index{topological space!closed-$\bar G_\delta$}{\em closed-$\bar G_\delta$ space} if each closed subset of $X$ is a $\bar G_\delta$-set in $X$. It is clear that a normal space $X$ is closed-$\bar G_\delta$ if and only if each closed subset of $X$ is of type $G_\delta$.

The main result of this section is the following characterization.

\begin{theorem}\label{t:metr} For a $T_0$-space $X$ the following conditions are equivalent:
\begin{enumerate}
\item $X$ is metrizable;
\item $X$ is a first-countable closed-$\bar G_\delta$ space with a countable locally uniform $\cs^{*}$-netbase;
\item $X$ is a first-countable strong $\sigma$-space with a countable locally quasi-uniform $\cs^*$-netbase;
\item $X$ is a first-countable collectionwise normal $\Sigma$-space with a countable locally uniform $\cs^*$-netbase;
\item $X$ is an $M$-space with a countable locally uniform $\cccs^{**}$-netbase;
\item $X$ is a $q$-space, a collectionwise normal $\Sigma$-space, and $X$ has a countable locally uniform $\cccs^{**}$-netbase.
\end{enumerate}
\end{theorem}

\begin{proof} The implications $(1)\Ra(i)$ for $i\in\{2,3,4,5,6\}$ are trivial. So, it remains to prove the reverse implications.
\smallskip

To prove that $(2)\Ra(1)$, assume that $X$ is a first-countable closed-$\bar G_\delta$ space with a countable locally uniform $\cs^*$-netbase $(\E,\Bas)$.  We lose no generality assuming that $\E$ is closed under finite unions. The metrizability of $X$ will be proved using the Moore Metrization Theorem.

For every $x\in X$ and $E\in\E$ denote by $E[x]^\circ$ the interior of the $E$-ball $E[x]$ in $X$.
Then $\U_{E}:=\{E[x]^\circ:x\in X\}$ is a family of open subsets of $X$ and its union $\bigcup\U_{E}$ is an open set in $X$. Since the space $X$ is closed-$\bar G_\delta$,  $X\setminus\bigcup\U_E=\bigcap_{m\in\w}W_{E,m}=\bigcap_{m\in\w}\overline{W}_{E,m}$ for some  sequence $(W_{E,m})_{m\in\w}$ of open sets in $X$. For every $m\in\w$ consider the open cover $\U_{E,m}=\U_E\cup\{W_{E,m}\}$ of $X$.

It follows that $\{\U_{E,m}:E\in\E,\;m\in\w\}$ is a countable family of open covers of $X$. The metrizability of $X$ will follow from the Moore Metrization Theorem \cite[5.4.2]{Eng} as soon as for every point $x\in X$ and  neighborhood $O_x\subset X$ of $x$ we find a neighborhood $V_x\subset X$ of $x$ and a pair $(E,m)\in\E\times\w$ such that $\St(V_x;\U_{E,m})\subset O_x$.%, where $\St(V_x;\U_{E,m})=\bigcup\{U\in\U_{E,m}:V_x\cap U\ne\emptyset\}$.

The locally uniform base $\Bas$ contains an entourage $B$ such that $B^{\pm3}[x]\subset O_x$. Consider the subfamily $\E_B=\{E\in\E:E\subset B\}$.

\begin{claim}\label{cl:Ecirc} There exists an entourage $E\in\E_B$ such that $x\in E[x]^\circ$.
\end{claim}

\begin{proof} Since the countable family $\E$ is closed under finite unions, there exists an increasing sequence $\{E_n\}_{n\in\w}\subset\E_B$ such that each $E\in\E_B$ is contained in some $E_n$. The space $X$, being first-countable, admits a decreasing neighborhood base $(B_n)_{n\in\w}$ at $x$. Assuming that for every $E\in \E_B$ the point $x$ does not belong to the interior of the set $E[x]$, we can construct a sequence of points $x_n\in B_n\setminus E_n[x]$, $n\in\w$, which converges to $x$. Since the family $\{E[x]:E\in\E_B\}$ is a $\cs^*$-network at $x$, for some entourage $E\in\E_B$ the ball $E[x]$ contains infinitely many points $x_n$, $n\in\w$. Choose $k\in\w$ with $E\subset E_k$ and find $m>k$ such that $x_m\in E[x]$. Then $x_m\in E[x]\subset E_k[x]\subset E_m[x]$, which contradicts the choice of $x_m$.
\end{proof}

Claim~\ref{cl:Ecirc} yields an entourage $E\in\E_B$ such that $x\in E[x]^\circ\in\U_{E}$ and hence $x\in \bigcup\U_{E}\setminus \overline{W}_{E,m}$ for some $m\in\w$. We claim that the neighborhood $V_x= E[x]\setminus \overline{W}_{E,m}$ of $x$ and the cover $\U_{E,m}$ have the required property: $\St(V_x;\U_{E,m})\subset O_x$. Given any point $y\in \St(V_x;\U_{E,m})$, we can find a point $v\in V_x$ and a set $U\in\U_{E,m}$ such that $\{y,v\}\subset U$. Since $v\in V_x=E[x]^\circ\setminus\overline{W}_{E,m}$, the set $U$ is equal to the set $E[z]^\circ$ for some point $z\in X$.
Observe that $v\in V_x\subset E[x]\subset B[x]$.
On the other hand, the inclusions
$\{y,v\}\subset U=E[z]^\circ\subset B[z]$ imply
$y\in B[z]\subset BB^{-1} [v]\subset BB^{-1} B[x]=B^{\pm3}[x]\subset O_x$. This completes the proof of the implication $(2)\Ra(1)$.
\smallskip

To prove that $(3)\Ra(1)$, assume that $X$ is a first-countable strong $\sigma$-space with a countable locally quasi-uniform $\cs^*$-netbase. By Theorem~\ref{t:s-sigma}, $X$ is an $\aleph$-space and by \cite[11.4]{Grue}, the first-countable $\aleph$-space $X$ is metrizable.
\smallskip

To prove that $(4)\Ra(3)$, assume that $X$ is a first-countable collectionwise normal $\Sigma$-space with a countable locally uniform $\cs^*$-netbase $(\E,\Bas)$. Since $X$ is first-countable, each closed countably compact subspace of $X$ is sequentially compact. Consequently, the $\cs^*$-netbase $(\E,\Bas)$ is a $\cccs^{*}$-netbase. By Theorem~\ref{t:Sigma}, the $\Sigma$-space $X$ is a $\sigma$-space. Being collectionwise normal, the $\sigma$-space $X$ is a strong $\sigma$-space.

To prove that $(5)\Ra(1)$, assume that $X$ is an $M$-space with a countable locally uniform $\ccs^{**}$-netbase.  By Theorem~\ref{t:wD=>Gd}, the $M$-space $X$ (being a $w\Delta$-space) has a $G_\delta$-diagonal and by Corollary 3.8 in \cite{Grue}, $X$ is metrizable.
\smallskip

To prove the implications $(6)\Ra(4)$, assume that $X$ is a $q$-space, is a collectionwise normal $\Sigma$-space and has a countable locally uniform $\cccs^{**}$-netbase. By Theorem~\ref{t:1-enet}, the $q$-space $X$ is first-countable. By Proposition~\ref{p:c*<->c**}(1), the space $X$ has a countable locally uniform $\cs^*$-netbase.
\end{proof}

Theorem~\ref{t:metr} will help us to characterize metrizable separable spaces.

\begin{theorem}\label{t:metr-separ} For a $T_0$-space $X$ the following conditions are equivalent:
\begin{enumerate}
\item $X$ is metrizable and separable;
\item $X$ is a first-countable hereditarily Lindel\"of space with a countable locally uniform $\cs^*$-netbase;
\item $X$ is a hereditarily Lindel\"of $M$-space with a countable locally uniform $\cs^{**}$-netbase;
\item $X$ is a first-countable cosmic space with a countable locally quasi-uniform $\cs^*$-netbase;
\item $X$ is a first-countable $\Sigma$-space with countable extent and a countable locally uniform $\cs^{**}$-netbase;
\item $X$ is a $q$-space, a $\Sigma$-space with countable extent and $X$ has a countable locally uniform $\ccs^{**}$-netbase.
\end{enumerate}
\end{theorem}

\begin{proof} The implications $(1)\Ra(i)$ for $i\in\{2,3,4,5,6\}$ are trivial.
\smallskip

To prove that $(2)\Ra(1)$, assume that $X$ is a first-countable hereditarily Lindel\"of space with a countable locally uniform $\cs^*$-netbase. By Proposition~\ref{p:lu=>regular}, the space $X$ is regular. Being hereditarily Lindel\"of and regular, the space $X$ is closed-$\bar G_\delta$. By Theorem~\ref{t:metr}(2), the space $X$ is metrizable and being (hereditarily) Lindel\"of, is separable.
\smallskip

To prove that $(3)\Ra(1)$, assume that $X$ is a hereditarily Lindel\"of $M$-space with a countable locally uniform $\cs^{**}$-netbase $(\E,\Bas)$. By \cite{AW}, each countably compact subspace of $X$ is sequentially compact, which implies that the $\cs^{**}$-netbase $(\E,\Bas)$ is a $\cccs^{**}$-netbase. By Theorem~\ref{t:metr}(5), the $M$-space $X$ is metrizable and being hereditarily Lindel\"of, is separable.
\smallskip

To prove that $(4)\Ra(1)$, assume that $X$ is a first-countable cosmic space with a countable locally quasi-uniform $\cs^*$-netbase. Being cosmic, the space $X$ is a strong $\sigma$-space. By Theorem~\ref{t:metr}(3), the space $X$ is metrizable and being cosmic is separable.
\smallskip

To prove that $(5)\Ra(4)$, assume that $X$ is a first-countable $\Sigma$-space with countable extent and a countable locally uniform $\cs^{**}$-netbase $(\E,\Bas)$. The first-countability of $X$ implies that the $\cs^*$-netbase $(\E,\Bas)$ is a locally uniform $\cccs^*$-netbase. By Theorem~\ref{t:Sigma}, the $\Sigma$-space $X$ is a $\sigma$-space. So, $X$ has a $\sigma$-discrete network $\mathcal N$. Since $X$ has countable extent, the $\sigma$-discrete network $\mathcal N$ is countable, which means that $X$ is cosmic.
\smallskip

To prove that $(6)\Ra(5)$, assume that $X$ is a $q$-space, a $\Sigma$-space with countable extent and $X$ has a locally uniform $\cccs^{**}$-netbase. By Theorem~\ref{t:1-enet}, the space $X$ is first-countable.
\end{proof}

Next, we apply netbases to characterizing compact metrizable spaces.

\begin{theorem}\label{t:comp-metr} For a topological space $X$ the following conditions are equivalent:
\begin{enumerate}
\item $X$ is compact and metrizable;
\item $X$ is a countably compact regular $T_0$-space with a countable locally uniform $\cccs^{**}$-netbase;
\item $X$ is a sequentially compact regular $T_0$-space with a countable locally uniform $\cs^{**}$-netbase;
\item $X$ is a compact Hausdorff space with a countable locally quasi-uniform $\cccs^{**}$-netbase;
\item $X$ is a compact sequentially compact Hausdorff space with a countable locally quasi-uniform $\cs^{**}$-netbase.
\end{enumerate}
\end{theorem}

\begin{proof} The implication $(1)\Ra(5)$ is trivial. The implications $(5)\Ra(4)$ and $(3)\Ra(2)$ follow from the observation that each $\cs^{**}$-netbase in a sequentially compact space is a $\cccs^{**}$-netbase. The implications $(5)\Ra(3)$ and $(4)\Ra(2)$ follow from Theorem~\ref{t:c-lqu=lu}.
The final implication $(2)\Ra(1)$ follows from Theorem~\ref{t:metr}(5).
\end{proof}

Using Theorem~\ref{t:comp-metr}, we can complete Theorem~\ref{t:1-enet} with another condition.

\begin{proposition} A regular $T_0$-space $X$ is first-countable at a point $x\in X$ if $X$ is a $q$-space at $x$, $X$ has a countable locally quasi-uniform $\cccs^{*}$-netbase, and each closed countably compact subset of $X$ is compact.
\end{proposition}

\begin{proof} Assume that $X$ has a countable locally quasi-uniform $\cccs^{**}$-netbase, and each closed countably compact subset of $X$ is compact. If the regular space $X$ is a $q$-space at a point $x\in X$, then we can find a decreasing sequence $(V_n)_{n\in\w}$ of closed neighborhoods of $x$ in $X$ such that any sequence $(x_n)_{n\in\w}\in\prod_{n\in\w}V_n$ has an accumulation point $x_\infty$ in $X$. It follows that the intersection $K=\bigcap_{n\in\w}V_n$ is a closed countably compact subspace of $X$. By our assumption, the subset $K$ is compact and Hausdorff. Since the space $X$ has a countable locally quasi-uniform $\cccs^{*}$-netbase, the compact Hausdorff space $K\subset X$ also has a countable locally quasi-uniform $\cccs^{*}$-netbase. By Theorem~\ref{t:comp-metr}(4), the space $K$ is metrizable and hence first-countable at $x$. Then we can choose a decreasing sequence $(U_n)_{n\in\w}$ of closed neighborhoods of $x$ in $X$ such that $K\cap\bigcap_{n\in\w} U_n=\{x\}$.
We claim that the countable family $\{V_n\cap U_n\}_{n\in\w}$ is a neighborhood base at $x$. Assuming the opposite, we could find an open neighborhood $O_x\subset X$ of $x$ such that for every $n\in\w$ the set $U_n\cap V_n\setminus O_x$ contains some point $x_n$. The choice of the sequence $(V_n)_{n\in\w}$ guarantees that the sequence $(x_n)_{n\in\w}$ has an accumulation point $x_\infty\in \bigcap_{n\in\w}V_n\setminus O_x=K\setminus O_x$. On the other hand, $x_\infty\in K\cap \bigcap_{n\in\w}U_n=\{x\}$, which contradicts the inclusion $x_\infty\in K\setminus O_x$.
\end{proof}

We finish this section by analyzing netbase properties of the Mr\'owka-Isbell spaces $\Psi_\A$ (see, \cite{Mrowka1}, \cite{Mrowka2}, \cite{Hru}), defined with the help of an almost disjoint family $\A$. A family of sets $\A$ is called \index{almost disjoint family}{\em almost disjoint} if for any distinct sets $A,B\in\A$ the intersection $A\cap B$ is finite.
For an almost disjoint family $\A$ of infinite subsets of $\w$ let $\Psi_\A$ be the set $\w\cup\A$ endowed with the topology $$\tau=\{U\subset \w\cup\A:\forall A\in \A\cap U\;\;\;|A\setminus U|<\w\}.$$

\begin{example} For any uncountable almost disjoint family $\A$ of infinite subsets of $\w$, the Mr\'owka-Isbell space $\Psi_\A$ has the following properties:
\begin{enumerate}
\item $\Psi_\A$ is a Moore space;
\item $\Psi_\A$ is a first-countable $w\Delta$-space;
\item Every subset of $\Psi_\A\times \Psi_A$ is a $G_\delta$-set, so $\Psi_\A$ has a $G_\delta$-diagonal;
\item $\Psi_\A$ is a separable space with uncountable extent;
\item $\Psi_\A$ locally compact but not countably compact;
\item If $\A$ is maximal almost disjoint, then the set $\w$ is sequentially compact in $\Psi_\A$ and the space $\Psi_\A$ is pseudocompact;
\item $\Psi_\A$ is $\sigma$-discrete and hence is a $\sigma$-space;
\item $\Psi_\A$ is not an $\aleph$-space;
\item $\Psi_\A$ is not a strong $\sigma$-space;
\item $\Psi_\A$ has a countable quasi-uniform base;
\item $\Psi_\A$ does not have a countable locally uniform $\cs^{**}$-netbase.
\end{enumerate}
\end{example}

\begin{proof} Let $X:=\Psi_\A$. For every $n\in\w$ and $x\in X$ consider the closed-and-open neighborhood
$$U_n(x)=\begin{cases}
\{x\},&\mbox{if $x\in\w$},\\
\{A\}\cup(A\setminus n), &\mbox{if $x=A\in\A$}
\end{cases}
$$
of $x$ in $X$. Then $U_n=\bigcup_{x\in X}\{x\}\times U_n(x)$ is a neighborhood assignment on $X$. %Observe that each $A\in\A$ appears in two roles: as an element and as a subset of $X$. To avoid a possible misunderstanding, we shall the notation $U_n(A)$ for denoting the neighborhood $\{A\}\cup(A\setminus n)$ of the element $A\in X$ and $U_n[A]$ for denoting the set $\bigcup_{a\in A}U_n(a)=A\setminus n\subset X$.
\smallskip

1. For every $n\in\w$ consider the open cover $\U_n=\{U_n(x):x\in X\}$ of $X$ and observe that for every $x\in X$ the family $\{\St(x,\U_n)\}_{n\in\w}$ is a neighborhood base at $x$, which implies that $X$ is a Moore space.
\smallskip

2. Being a Moore space, $X$ is a first-countable $w\Delta$-space.
\smallskip

3. Given any subset $F\subset X\times X$, for every $n\in\w$ consider the open neighborhood $W_n[F]=\bigcup_{(x,y)\in F}U_n(x)\times U_n(y)$ of $F$ and observe that  $F=\bigcap_{n\in\w}W_n[F]$, which means that $F$ is a $G_\delta$-set in $X\times X$.
\smallskip

4. The Mr\'owka-Isbell space $X$ is separable as $\w$ is dense in $X$, and has uncountable extent as $\A$ is an uncountable closed discrete subset of $X$.
\smallskip

5. The space $X$ is locally compact since for every $x\in X$ the neighborhood $U_0(x)$ of $x$ is compact. The space $X$ is not countably compact since it contains an infinite closed discrete subspace $\A$.
\smallskip

6. If $\A$ is maximal almost disjoint, then each infinite subset $I\subset\w$ has infinite intersection $I\cap A$ with some set $A\in\A$ and then $I\cap A$ is subsequence of $I$, convergent to the point $A\in\A\subset X$. This means that the dense subset $\w$ is sequentially compact in $X$, which implies that $X$ is pseudocompact.
\smallskip

7. It is clear that the Mr\'owka-Isbell space $X=\A\cup\bigcup_{n\in\w}\{n\}$ is $\sigma$-discrete and hence is a $\sigma$-space.
\smallskip

8. By \cite[11.4]{Grue}, first-countable $\aleph$-spaces are metrizable. Being first-countable and non-metrizable, the Mr\'owka-Isbell space fails to be an $\aleph$-space.
\smallskip

9,10. It is easy to see that $\U=\{U_n\}_{n\in\w}$ is a countable quasi-uniform base for $X$ and hence $(\U,\U)$ is a countable quasi-uniform $\cs^*$-netbase for $X$. Assuming that $X$ is a strong $\sigma$-space and applying Theorem~\ref{t:s-sigma}, we would conclude that $X$ is an $\aleph$-space, which is not the case. This contradiction shows that $X$ is not a strong $\sigma$-space.
\smallskip

11. To derive a contradiction, assume that the Mr\'owka-Isbell space $X$ admits a countable locally uniform $\cs^{**}$-netbase. By Proposition~\ref{p:c*<->c**}, $X$ has a countable locally uniform $\cs^*$-netbase $(\E,\Bas)$. We can additionally assume that the family $\E$ is closed under finite unions. For every $E\in\E$ and $n\in\w$ let $\A_{E,n}=\{A\in\A:\forall m\in A\setminus n\;\;(A,m)\in E\}$. Endow $\A_{E,n}$ with the weak topology generated by the countable base consisting of the sets $[F,G]=\{A\in\A_{E,n}:A\cap G=F\}$ where $F\subset G$ run over finite subsets of $\w$. Let $\dot\A_{E,n}$ be the set of isolated points in $\A_{E_n}$. The second countability of $\A_{E,n}$ implies that the set $\dot\A_{E,n}$ is at most countable. Then the set $\dot\A=\bigcup_{E\in\E}\bigcup_{n\in\w}\dot\A_{E,n}$ is countable too, so we can choose an element $A\in\A\setminus\dot\A$. The base $\Bas$, being locally uniform, contains an entourage $B\in\Bas$ such that $B^{\pm3}[\{A\}]\subset \{A\}\cup A$. Since $X$ is first-countable at $A$ and the family $\E_B=\{E\in\E:E\subset B\}$ is a (closed under finite unions) $\cs^*$-network at $A$, there exists an entourage $E\in\E_B$ such that $E[\{A\}]$ is a neighborhood of $A$ in $X$. Consequently, we can find $n\in\w$ such that $A\setminus n\subset E[\{A\}]$, which implies that $A\in\A_{E,n}$. Fix any point $a\in A\setminus n$ and consider its neighborhood $[\{a\},\{a\}]\ni A$ in $\A_{E,n}$. Since $A\notin \dot\A_{E,n}$, the singleton $\{A\}$ is not open in $\A_{E,n}$. Consequently, $\{A\}\ne[\{a\},\{a\}]$ and we can find a set $A'\in [\{a\},\{a\}]\setminus\{A\}$. Then $a\in A'\setminus n\subset E[\{A'\}]$ and hence $a\in E[\{A'\}]\cap E[\{A\}]$. Consequently, $A'\in E^{-1}E[\{A\}]\subset B^{-1}B[\{A\}]\subset \{A\}\cup A$, which is a desired contradiction showing that $X$ has no locally uniform $\cs^{**}$-netbases.
\end{proof}

\begin{problem} Can the Mr\'owka-Isbell space $\Psi_\A$ be a closed-$\bar G_\delta$ space for a suitable uncountable almost disjoint family $\A$?
\end{problem}

\begin{remark} The properties (7),(8),(10) of Mr\'owka-Isbell space show that strong $\sigma$-spaces in Theorem~\ref{t:metr}(3) cannot be replaced by $\sigma$-spaces.
\end{remark}

Nonetheless we do not know the answer to the following question.

 \begin{problem} Is each first-countable $\sigma$-space with a countable  locally uniform $\cs^*$-netbase metrizable?
\end{problem}

\section{Topological spaces with a countable $\IR$-universal netbase}

We recall that a {\em based space} is a topological space $X$ endowed with an entourage base $\Bas_X$. The entourage base $\Bas_X$ generates the preuniformity $\U_X=\{U\subset X\times X:\exists B\in\Bas\;\;(B\subset U)\}$ on $X$. For a based space $(X,\Bas)$ by $C_u(X)$ and $C_\w(X)$ we denote the sets of uniformly continuous and $\w$-continuous real-valued functions on the preuniform space $(X,\U_X)$, respectively. We recall that a function $f:X\to\IR$ is {\em uniformly continuous} (resp. {\em $\w$-continuous}) if for every $\e>0$ there exists an entourage $U\in\Bas$ (resp. a countable family of entourages $\V\subset \Bas$) such that for every $x\in X$ (there exists $U\in\V$) such that $|f(x)-f(y)|<\e$ for all $(x,y)\in U$.
For a based space $X$ by $\delta\colon X\to\IR^{C_u(X)}$ we denote the canonical map assigning to each point $x\in X$ the Dirac measure $\delta_x\colon C_u(X)\to\IR$, $\delta_x\colon f\mapsto f(x)$.

A based space $(X,\Bas)$ is called
\begin{itemize}
\item {\em $\IR$-universal} if $C_\w(X)=C_u(X)$;
\item {\em $\IR$-separated} if the canonical map $\delta\colon X\to\IR^{C_u(X)}$ is injective.
\item {\em $\IR$-regular} if the canonical map $\delta\colon X\to\IR^{C_u(X)}$ is a topological embedding.
\end{itemize}
By Corollary~\ref{c:Rreg}, a based space $(X,\Bas)$ is $\IR$-regular if and only if $X$ is a $T_0$-space and the base $\Bas$ is locally $\infty$-uniform.

A netbase $(\E,\Bas)$ for a topological space $X$ is called {\em $\IR$-universal} (resp. {\em $\IR$-separated}, {\em $\IR$-regular}, {\em locally $\infty$-uniform}) if so is the based space $(X,\Bas)$.

For a subset $C$ of a topological space $X$ by $C^{\prime\as}$ we denote the set of accumulation points of $C$ in $X$. For a family $\C$ of subsets of a topological space $X$ let $X^{\prime\C}:=\bigcup_{C\in\C}C^{\prime\as}$ be the set of accumulation points of the sets $C\in\C$.
Observe that for a $T_1$-space $X$ and the family $\as$ of all subsets of $X$, the set $X^{\prime\as}$ coincides with the set $X'$ of non-isolated points of $X$.

A subset $B\subset X$ of a based space $(X,\Bas)$ is defined to be \index{subset!functionally bounded}{\em functionally bounded} if for each uniformly continuous function $f\colon X\to\IR$ the set $f(B)$ is bounded in $\IR$. A subset $B\subset X$ is \index{subset!$\sigma$-bounded}{\em $\sigma$-bounded} if $B$ can be written as the countable union of functionally bounded sets in $X$.

\begin{theorem}\label{t:universal=>s-bound} Let $\C$ be a family of subsets in a topological space $X$ and $(\E,\Bas)$ be a countable $\C^*$-netbase for $X$. If the netbase $(\E,\Bas)$ is $\IR$-universal and $\IR$-separated, then the set $X^{\prime\C}$ is $\sigma$-bounded in the based space $(X,\Bas)$.
\end{theorem}

\begin{proof} For every $E\in\E$ consider the set $$B_E:=\{x\in X:\forall O_x\in\Tau_x(X)\;\;\big|O_x\cap E[x]\big|\ge\w\}.$$ Let $\E_b=\{E\in\E:\mbox{$B_E$ is functionally bounded in $(X,\Bas)$}\}$.
To finish the proof of the theorem, it suffices to check that $X^{\prime\C}\subset\bigcup_{E\in\E_b}B_E$.

For this we shall ``kill'' all entourages $E$ that belong to the family $\E':=\E\setminus\E_b$. By Lemma~\ref{l:funbound}, there exists an $\w$-continuous function $f:X\to \IR$ such that for every  $E\in\E'$ the set $f[B_E]$ is unbounded in $\IR$. Then for each $E'$ we can inductively choose a point $b_E\in B_E$ such that $|f(b_E)-f(b_{F})|>5$ for any distinct sets $E,F\in\E'$. Observe that for every $E\in\E'$ the function $$\lambda_E:X\to[0,1],\;\;\lambda_E:x\mapsto \max\big\{0,\min\{1,2-|f(x)-f(b_E)|\}\big\},$$
is $\w$-continuous, takes value $1$ on the set $\{x\in X:|f(x)-f(b_E)|\le1\}$ and value $0$ on the set $\{x\in X:|f(x)-f(b_E)|\ge 2\}$.

For every $E\in \E'$ consider the neighborhood $O_E=\{x\in X:|f(x)-f(b_E)|<1\}$ of $b_E$ and choose a point $a_E\in O_E\cap E[b_E]\setminus\{b_E\}$ (such a point $a_E$ exists by the definition of the set $B_E$). Since the based space $(X,\Bas)$ is $\IR$-separated, there exists a uniformly continuous function $\varphi_E:X\to[0,2]$ such that $\varphi_E(a_E)=2$ and $\varphi_E(b_E)=0$. It follows that the function $\psi_E:X\to[0,1]$, $\psi_E:x\mapsto \varphi_E(x)\cdot\lambda_E(x)$, has the following properties: $\psi_E(a_E)=2$, $\psi_E(b_E)=0$ and $\psi_E(x)=0$ for any point $x\in X$ with $|f(x)-f(b_E)|\ge 2$. It can be shown that $g=f+\sum_{E\in\E'}\psi_E$ is a well-defined $\w$-continuous function on $X$ such that for every $E\in\E'$
$$
|g(a_E)-g(b_E)|=|f(a_E)+\psi_E(a_E)-f(b_E)-\psi_E(b_E)|\ge|\psi_E(a_E)-\psi_E(b_E)|-|f(a_E)-f(b_E)|\ge 2-1=1.
$$
Since $g\in C_\w(X)=C_u(X)$, there exists an entourage $B\in\Bas$ such that $|g(x)-g(y)|<1$ for any $(x,y)\in B$.

Now we can prove that $X^{\prime\C}\subset \bigcup_{E\in\E_b}B_E$. Given any point $x\in X^{\prime\C}$, find a set $C\in\C$ accumulating at $x$. Since $(\E,\Bas)$ is a $\C^*$-netbase, there exists an entourage $E\in\E$ such that $E\subset B$ and $E[x]\cap C$ is infinite. Consequently, $x\in B_E$. It remains to prove that $E\in\E_b$. Assuming that $E\in\E'$, we can consider the points $b_E\in B_E$ and $a_E\in E[b_E]$ and conclude that $a_E\in E[b_E]\subset B[b_E]$ and hence $|g(a_E)-g(b_E)|<1\le|g(a_E)-g(b_E)|$, which is a desired contradiction.
\end{proof}

We recall that a subset $B$ of a topological space $X$ is \index{subset of a topological space!$\w$-Urysohn}{\em $\w$-Urysohn} if each infinite closed discrete subset $B\subset A$ of $X$ contains an infinite strongly discrete subset $C\subset B$ of $X$.

\begin{lemma}\label{l:fb=>cm} If a topological space $X$ has a countable $\IR$-universal $\IR$-regular $\ccs^{**}$-netbase $(\E,\Bas)$, then
\begin{enumerate}
\item each functionally bounded $\w$-Urysohn closed subset of $X$ is compact and metrizable;
\item each functionally bounded hereditarily Lindel\"of $\bar G_\delta$-subset of $X$ is compact and metrizable.
\end{enumerate}
\end{lemma}

\begin{proof} Let $\U=\{U\subset X\times X:\exists B\in\Bas\;(B\subset U)\}$ be the preuniformity generated by the base $\Bas$ and $\U^{\pm\w}$ be the canonical uniformity of the preuniform space $(X,\U)$. By Proposition~\ref{p:Rr-u}, the uniformity $\U^{\pm\w}$ generates the topology of $X$. It follows that $(\E,\U^{\pm\w})$ is a countable (locally) uniform $\ccs^{**}$-netbase for $X$.

By Proposition~\ref{p:fb+wU=>cc}, each functionally bounded $\w$-Urysohn closed subset $B$ of $X$ is countably compact. By Theorem~\ref{t:comp-metr}, the countably compact set $B$ is compact and metrizable. By Lemma~\ref{l:hL+bG=>wU}, each functionally bounded hereditarily Lindel\"of $\bar G_\delta$-subset of $X$ is $\w$-Urysohn and hence is compact and metrizable.
\end{proof}

\begin{theorem}\label{t:cRunb} Let $\C$ be a family of subsets in an $\w$-Urysohn topological space $X$ and $(\E,\Bas)$ be a countable $\IR$-regular $\IR$-universal $\C^{*}$-netbase for $X$. Then
\begin{enumerate}
\item $X^{\prime\C}$ is contained in a countable union of closed countably compact subsets of $X$.
\item If $\cs\subset\C$ and each closed countably compact subset of $X$ is sequentially compact, then the set $X^{\prime\C}$ is contained in a countable union of compact metrizable subsets of $X$.
\item If $\ccs\subset\C$, then the set $X^{\prime\C}$ is contained in a countable union of compact metrizable subsets of $X$.
\end{enumerate}
\end{theorem}

\begin{proof} 1. By Theorem~\ref{t:universal=>s-bound}, the set $X^{\prime\C}$ is contained in the countable union $\bigcup_{n\in\w}B_n$ of closed functionally bounded subsets of $(X,\Bas)$. By Proposition~\ref{p:fb+wU=>cc}, each closed functionally bounded subset $B_n$ of the $\w$-Urysohn $\IR$-universal $\IR$-regular based space $(X,\Bas)$ is countably compact.
\smallskip

2. If each closed countably compact subset of $X$ is sequentially compact, then the closed countably compact subsets $B_n$, $n\in\w$, are sequentially compact. Let $\U$ be the preuniformity generated by the base $\Bas$ and $\U^{\pm\w}$ be the canonical uniformity of the preuniform space $(X,\U)$. By Proposition~\ref{p:Rr-u}, the uniformity $\U^{\pm\w}$ generates the topology of $X$. If $\cs\subset\C$, then $(\E,\U^{\pm\w})$ is a countable uniform $\cs^*$-netbase for $X$. By Theorem~\ref{t:comp-metr}(3), the sequentially compact spaces $B_n$, $n\in\w$, are compact and metrizable.
\smallskip

3. If $\ccs\subset\C$, then by Lemma~\ref{l:fb=>cm} the functionally bounded sets $B_n$, $n\in\w$, are compact and metrizable.
\end{proof}

A topological space $X$ is defined to be \index{topological space!$\sigma'$-compact}{\em $\sigma'$-compact} if the set $X'$ of non-isolated points of $X$ is $\sigma$-compact.

\begin{theorem}\label{t:netbase=>sigma'} For a topological space $X$ possessing a countable $\IR$-universal $\IR$-regular $\as^*$-netbase $(\E,\Bas)$, the following conditions are equivalent:
\begin{enumerate}
\item the set $X'$ of non-isolated points of $X$ is a hereditarily Lindel\"of $\sigma$-compact $G_\delta$-subset of $X$;
\item $X$ is a perfectly paracompact $\sigma'$-compact space;
\item $X$ is an $\w$-Urysohn $\Sigma$-space;
\item the set $X'$ of non-isolated points of $X$ is a cosmic $\sigma$-compact $G_\delta$-subset of $X$.
\end{enumerate}
\end{theorem}

\begin{proof} By the $\IR$-regularity of the base $\Bas$, the topological space $X$ is Tychonoff.
\smallskip

$(1)\Ra(2)$ Assume that the set $X'$ of non-isolated points of $X$ is a hereditarily Lindel\"of $\sigma$-compact $G_\delta$-subset of $X$. By Lemma~\ref{l:para'}, the space $X$ is paracompact. To see that $X$ is perfectly paracompact, we need to show that each open subset $U\subset X$ is an $F_\sigma$-set in $X$. Since the closed set $X'$ is a  $G_\delta$-set in $X$, the discrete space $U\setminus X'$ is an $F_\sigma$-set in $X$. Since $X'$ is a closed hereditarily Lindel\"of subspace of $X$, the open subset $U\cap X'$ is an $F_\sigma$-set in $X'$ and in $X$. Consequently, the open set $U=(U\cap X')\cup(U\setminus X')$ is an $F_\sigma$-set in $X$.
\smallskip

$(2)\Ra(3)$ Assume that $X$ is a perfectly paracompact $\sigma'$-compact space. Then the closed set $X'$ is a $\sigma$-compact $G_\delta$-set in $X$. Write $X'$ as the countable union $\bigcup_{n\in\w}K_n$ of compact subsets $K_n$, $n\in\w$, and $X\setminus X'$ as the countable union $\bigcup_{n\in\w}F_n$ of closed (discrete) subsets of $X$. Then the family $$\N=\{K_n\}_{n\in\w}\cup\bigcup_{n\in\w}\big\{\{x\}:x\in F_n\big\}$$ consists of compact sets in $X$ and is a $\sigma$-discrete $\N$-network for $\N$, witnesses that $X$ is a $\Sigma$-space. Being paracompact, the space $X$ is $\w$-Urysohn.
\smallskip

$(3)\Ra(4)$ Assume that $X$ is an $\w$-Urysohn $\Sigma$-space. By Theorem~\ref{t:Sigma}, $X$ is a $\sigma$-space. By Theorem~\ref{t:cRunb}(3), the set $X'=X^{\prime\as}$ is cosmic and $\sigma$-compact.
By \cite[p.446]{Grue}, the closed subset $X'$ of the $\sigma$-space $X$ is a $G_\delta$-set in $X$.
\smallskip

The implication $(4)\Ra(1)$ is trivial.
\end{proof}

\begin{theorem}\label{t:netbase=>sigma} For a topological space $X$ possessing a countable locally uniform $\css^*$-netbase $(\E,\Bas)$, the following conditions are equivalent:
\begin{enumerate}
\item $X$ is cosmic;
\item $X$ is an $\aleph_0$-space;
\item $X$ is an $\mathfrak P_0$-space;
\item $X$ is a $\Sigma$-space with countable extent.
\end{enumerate}
If the netbase $(\E,\Bas)$ is $\IR$-universal, then the conditions \textup{(1)--(4)} are equivalent to
\begin{itemize}
\item[(5)] $X$ is $\sigma$-compact.
\end{itemize}
\end{theorem}

\begin{proof} The equivalence of the conditions (1)--(4) follows from Theorems~\ref{t:aleph0} and \ref{t:s-sigma}.
The implication $(5)\Ra(4)$ is trivial. Now assuming that the netbase $(\E,\Bas)$ is $\IR$-universal, we shall prove that $(1)\Ra(5)$. By Lemma~\ref{l:C=Cw}, $C(X)=C_\w(X)=C_u(X)$. The complete regularity of the cosmic space $X$ and the equality $C(X)=C_u(X)$ guarantee that the based space $(X,\Bas)$ is $\IR$-regular. Being cosmic, the space $X$ has countable tightness and hence satisfies the equality $X'=X^{\prime\as}=X^{\prime\css}$. Moreover, the cosmicity of $X$ implies that the discrete subspace $X\setminus X'$ of $X$ is cosmic and hence at most countable. Being cosmic, the space $X$ is paracompact and  $\w$-Urysohn. By Theorem~\ref{t:netbase=>sigma'}, the set $X'=X^{\prime\css}$ is $\sigma$-compact and so is the space $X=X'\cup(X\setminus X')$.
\end{proof}

\section{Cardinality of spaces with a countable $\as^*$-netbase}\label{s:card-as}

In this section we shall prove that many known upper bounds for the cardinality of first-countable topological spaces remain true for topological spaces possessing a countable $\as^*$-base.
More precisely, we shall show that in many upper bounds on the cardinality $|X|$ of a topological space $X$ the character $\chi(X)$ of $X$ can be replaced by the $\as^*$-character $\chi_{\as^*}(X)$ of $X$.

\begin{definition} Let $\C$ be a family of subsets of $X$. The \index{$\C^*$-character}\index{topological space!$\C^*$-character of}{\em $\C^*$-character} $\chi_{\C^*}(X)$ of a topological space $X$ is the smallest cardinal $\kappa$ for which there exists a $\C^*$-netbase $(\E,\Bas)$ for $X$ with $|\E|\le\kappa$.
\end{definition}

We shall be interested in the $\as^*$-character $\chi_{\as^*}(X)$ for the family $\as$ of all subsets of $X$. This character nicely bounds the character $\chi(X)$ of a topological space $X$. We recall that the \index{character of a point}\index{topological space!character of a point}{\em character} $\chi(x;X)$ of a point $x$ in a topological space $X$ is defined as the smallest cardinality of a neighborhood base at $x$. The cardinal $\chi(X):=\sup_{x\in X}\chi(x;X)$ is called the \index{topological space!character of}{\em character} of the topological space $X$.

\begin{theorem} For each topological space $X$ we get
$\chi_{\as^*}(X)\le\chi(X)\le2^{\chi_{\as^*}(X)}.$
\end{theorem}

\begin{proof} Let $\kappa=\chi(X)$. If $\kappa$ is countable, then by Proposition~\ref{p:1-ebase}, the space $X$ has a countable entourage base $\Bas$. Then the countable $\as^*$-netbase $(\Bas,\Bas)$ witnesses that $\chi_{\as^*}(X)\le\w=\kappa$. It remains to consider the case of uncountable $\kappa$.

For every point $x\in X$, fix a neighborhood base $(U_\alpha[x])_{\alpha\in\kappa}$ at $x$.
For every $\alpha\in\kappa$ consider the neighborhood assignment $U_\alpha=\bigcup_{x\in X}\{x\}\times U_\alpha[x]$. Let $[\kappa]^{<\w}$ be the family of finite subsets of $\kappa$ and for every $F\in[\kappa]^{<\w}$ let $U_F=\bigcap_{\alpha\in F}U_\alpha$ (for $F=\emptyset$ we put $U_\emptyset=X\times X$). The family $\E=\{U_F:F\in[\kappa]^{<\w}\}$ is an entourage base for the space $X$ and $(\E,\E)$ is a $\as^*$-netbase for $X$, witnessing that $$\chi_{\as^*}(X)\le |\E|\le|[\kappa]^{<\kappa}|=\kappa=\chi(X).$$

To see that $\chi(X)\le 2^{\chi_{\as^*}(X)}$, fix a $\as^*$-netbase $(\E,\Bas)$ with $|\E|=\chi_{\as^*}(X)$. For every $x\in X$ the family $\E[x]=\{E[x]:E\in\E\}$ is a $\as^*$-network at $x$. Let $\Tau_x$ be the family of all (not necessarily open) neighborhoods of $x$ in $X$. Let $\ddot x=\bigcap\Tau_x(X)$. Consider the family $$\textstyle{\mathcal B_x=\{\ddot x\cup\bigcup\A:\A\subset \E[x],\;\ddot x\cup\bigcup\A\in\Tau_x(X)\}}$$ and observe that $|\mathcal B_x|\le 2^{|\E[x]|}\le 2^{|\E|}=2^{\chi_{\as^*}(X)}$. We claim that $\mathcal B_x$ is a neighborhood base at $x$. Given any neighborhood $O_x\in\Tau_x(X)$, consider the subfamily $\A=\{N\in \E[x]:N\subset O_x\}$. We claim that $\ddot x\cup \bigcup\A$ is a neighborhood of $x$. In the opposite case  $x$ is an accumulating point of the set  $X\setminus\bigcup\A$ and we can find a set $N\in \E[x]$ in the $\as^*$-network $\E[x]$ at $x$ such that $x\in N\subset O_x$ and $N\cap(X\setminus \bigcup\A)\ne\emptyset$. It follows that $N\in\A$ and hence $N\subset\bigcup\A$, which contradicts the choice of $N$. This contradiction shows that $\ddot x\cup\bigcup\A$ is a neighborhood of $x$ contained in $O_x$, which means that $\mathcal B_x$ is a neighborhood base at $x$. So, $\chi(x;X)\le|\mathcal B_x|\le  2^{\chi_{\as^*}(X)}$ and $\chi(X)=\sup_{x\in X}\chi(x;X)\le 2^{\chi_{\as^*}(X)}$.
\end{proof}

According to a famous theorem of Arhangel'ski\u\i\ \cite[3.12.10]{Eng}, each  Hausdorff space $X$ has cardinality $|X|\le 2^{\chi(X)L(X)}$ where $L(X)$ is the Lindel\"of number of $X$. The inequality $|X|\le 2^{\chi(X) L(X)}$ can be improved to $|X|\le 2^{\chi(X)a L_c(X)}$ (and $|X|\le 2^{\chi(X)wL_c(X)}$ for regular spaces), see \cite{Hodel}.

Let us recall that for a topological space $X$ its
\begin{itemize}
\item \index{Lindel\"of number}\index{$L(X)$}{\em Lindel\"of number} $L(X)$ is the smallest cardinal $\kappa$ such that each open cover $\V$ of $X$ has a subcover $\V'\subset\V$ of cardinality $|\V'|\le\kappa$;
\item \index{almost Lindel\"of number}\index{$aL(X)$}{\em almost Lindel\"of number} $aL(X)$
is the smallest cardinal $\kappa$ such that for any open cover $\V$ of $X$ there exists a subcollection $\V'\subset \V$ of cardinality $|\V'|\le\kappa$ such that $X=\bigcup_{V\in\V'}\bar V$;
\item \index{almost Lindel\"of number for closed subspaces}\index{$aL_c(X)$}{\em almost Lindel\"of number for closed subspaces} $aL_c(X)$
is the smallest cardinal $\kappa$ such that for any cover $\V$ of a closed subspace $C\subset X$ by open subsets of $X$ there exists a subcollection $\V'\subset \V$ of cardinality $|\V'|\le\kappa$ such that $C\subset\bigcup_{V\in\V'}\bar V$;
\item \index{weak Lindel\"of number}\index{$wL(X)$}{\em weak Lindel\"of number} $wL(X)$  is the smallest cardinal $\kappa$ such that for any open cover $\V$ of $X$  there exists a subcollection $\V'\subset \V$ of cardinality $|\V'|\le\kappa$ whose union $\bigcup \V'$ is dense in $X$;
\item \index{weak Lindel\"of number for closed subspaces}\index{$wL_c(X)$}{\em weak Lindel\"of number for closed subspaces} $wL_c(X)$
is the smallest cardinal $\kappa$ such that for any cover $\V$ of a closed subspace $C\subset X$ by open subsets of $X$ there exists a subcollection $\V'\subset \V$ of cardinality $|\V'|\le\kappa$ such that $C\subset\overline{\bigcup\V'}$;
\item \index{cellularity}\index{topological space!cellularity of}\index{$c(X)$}{\em cellularity} $c(X)$ is the smallest cardinal $\kappa$ such that every disjoint family $\U$ of non-empty open subsets of $X$ has cardinality $\le \kappa$;
\item \index{density}\index{topological space!density of}\index{$d(X)$}{\em density} $d(X)$ is the smallest cardinality of a dense subset of $X$;
\item \index{tightness}\index{topological space!tightness of}\index{$t(X)$}{\em tightness} $t(X)$ is the smallest cardinal $\kappa$ such that for any set $A\subset X$ and point $x\in\bar A$ there exists a subset $B\subset A$ of cardinality $|B|\le|A|$ such that $x\in\bar B$.
    \end{itemize}
The following diagram describes the relations between these cardinal characteristics for any topological space. In the diagram an arrow $f\to g$ between two cardinal invariants $f,g$ indicates that $f(X)\le g(X)$ for any topological space $X$.
$$
\xymatrix{
aL\ar[r]&aL_c\ar[r]&{L}\\
{wL}\ar[r]\ar[u]&{wL_c}\ar[r]\ar[u]&c\ar[r]&d
}
$$
It is known (and easy to see) that $aL(X)=aL_c(X)=L(X)$ for any regular space $X$ and $wL(X)=wL_c(X)$ for any normal space $X$.

By \cite{BC} (see also \cite{Hodel}), each Hausdorff space $X$ has cardinality $$|X|\le 2^{\chi(X) aL_c(X)}\le 2^{\chi(X) L(X)}.$$ Moreover, if $X$ is Urysohn, then $$|X|\le 2^{\chi(X) wL_c(X)}\le 2^{\chi(X) c(X)}\mbox{ \  and \ }|X|\le 2^{\chi(X) aL(X)},$$ see \cite{Alas}, \cite{BC}, \cite{Hodel}.
%First we prove an upper bound on the cardinality of a separable space  possessing an $\w^\w$-base.
%We recall that a topological space $X$ is {\em Urysohn} if any distinct points $x,y\in X$ have disjoint closed neighborhoods.
%For a subset $A$ of a topological space $X$ its {\em $\theta$-closure} if the set of all points $x\in X$ such that each closed neighborhood $\bar O_x\subset X$ of $x$ intersects $x$. Bella and Cammaroto \cite{BC} proved that for each subset $A$ of a Urysohn space $X$ its $\theta$-closure $\overline{A}^\theta$ has cardinality $|\overline{A}^\theta|\le |A|^{\chi(X)}$.

\begin{lemma}\label{l:card-clos} For any subset $A\subset X$ of a Hausdorff space $X$ its closure $\bar A$ has cardinality
\begin{enumerate}
\item $|\bar A|\le 2^{|A|\cdot\chi_{\as^*}(X)}$;
\item $|\bar A|\le |A|^{t(X)}\cdot 2^{\chi_{\as^*}(X)}$.
\end{enumerate}
\end{lemma}

 \begin{proof} Both inequalities trivially hold for finite $A$. So, we assume that the set $A$ is infinite. Let  $(\E,\Bas)$ be a $\as^*$-netbase for $X$ with $|\E|=\chi_{\as^*}(X)$.

 If $\chi_{\as^*}(X)$ is finite, then the character $\chi(X)\le 2^{\chi_{\as^*}(X)}$ is finite too. In this case the Hausdorff space $X$ is discrete, $\bar A=A$ and the inequalities (1), (2) trivially hold.

So, we assume that $\chi_{\as^*}(X)=|\E|$ is infinite. In this case we lose no generality assuming that $\E$ is closed under taking finite unions.

Let $A$ be any set in $X$ and $\mathcal P(A)$ be the family of all subsets of $A$. Consider the function $\varphi:\bar A\to \mathcal P(A)^{\E}$ assigning to each point $x\in \bar A$ the function $\varphi_x:\E\to \mathcal P(A)$, $\varphi_x:E\mapsto E[x]\cap A$. We claim that the function $\varphi$ is injective.

Given two distinct points $x,y\in \bar A$, use the Hausdorff property of $X$ and find an entourage $B\in\Bas$ such that $B[x]\cap B[y]=\emptyset$. Since the family $\{E[x]:E\in\E,\;E\subset B\}$ is a $\as^*$-network at $x\in\bar A$, there exists an entourage $E_x\in\E$ such that $E_x\subset B$ and $E_x[x]\cap A\ne\emptyset$. By the same reason, there exists an entourage $E_y\in\E$ such that $E_y\subset B$ and $E_y[y]\cap A\ne\emptyset$. Since $\E$ is closed under finite unions, the entourage $E=E_x\cup E_y\subset B$ belongs to $\E$ and has the property $E[x]\cap A\ne\emptyset \ne A\cap E[y]$. Since $E[x]\subset B[x]$ and $E[y]\subset B[y]$, the non-empty sets $E[x]\cap A$ and $E[y]\cap A$ are disjoint and hence distinct. Consequently, $\varphi_x(E)=E[x]\cap A\ne E[y]\cap A=\varphi_y(E)$, which means that the function $\varphi:\bar A\to\mathcal P(A)^\E$ is injective and hence $|\bar A|\le |\mathcal P(A)|^{\E}=2^{|A|\cdot\chi_{\as^*}(X)}$.
\smallskip

By the definition of the tightness, $\bar A=\bigcup\{\bar S:S\in[A]^{\le t(X)}\}$  where $[A]^{\le t(X)}=\{S\in\mathcal P(A):|S|\le t(X)\}$. Then
$$|\bar A|\le\sum_{S\in[A]^{\le t(X)}}|\bar S|\le \sum_{S\in[A]^{\le t(X)}}2^{|S|\cdot\chi_{\as^*}(X)}\le\sum_{S\in[A]^{\le t(X)}}2^{t(X)\cdot \chi_{\as^*}(X)}\le |A|^{t(X)}\cdot 2^{t(X)\cdot \chi_{\as^*}(X)}=|A|^{t(X)}\cdot 2^{\chi_{\as^*}(X)}.$$
\end{proof}

\begin{corollary}\label{c:cardinal}Every Hausdorff space $X$ has cardinality $|X|\le d(X)^{t(X)}\cdot 2^{\chi_{\as^*}(X)}$.
\end{corollary}

The following theorem can be deduced from Theorem 3.1 of \cite{Hodel} and Corollary~\ref{c:cardinal}.

\begin{theorem}\label{t:chip-cardinal} Any Hausdorff space $X$ has cardinality $|X|\le 2^{\chi_{\as^*}(X)\cdot aL_c(X)}\le 2^{\chi_{\as^*}(X)\cdot L(X)}\le 2^{\chi(X)\cdot L(X)}$.\newline If the space $X$ is regular, then $|X|\le 2^{\chi_{\as^*}(X)\cdot wL_c(X)}\le 2^{\chi_{\as^*}(X)\cdot c(X)}\le 2^{\chi(X)\cdot c(X)}.$
\end{theorem}

\begin{remark} The $\as^*$-character $\chi_{\as^*}(X)$ in the upper bound $|X|\le 2^{\chi_{\as^*}(X)\cdot L(X)}$ cannot be replaced by the $\css^*$-character: according to Example~\ref{ex:szeptycki} under CH there exists a linearly ordered compact Hausdorff space $X$ with $\chi_{\css^*}(X)=\w$ and $|X|=2^{\w_1}>\mathfrak c=2^{\chi_{\css^*}(X)\cdot L(X)}$.
\end{remark}

\section{Netbases in netportators}

In this section we construct natural netbases in netportators and then apply these netbases to studying topological properties of netportators. We recall that a \index{portator}{\em portator} is a pointed topological space $X$ with a distinguished point $e$ called the {\em unit} of $X$ and $X$ is endowed with a set-valued binary operation $\mathbf{xy}:X\times X\multimap X$, $\mathbf{xy}:(x,y)\mapsto xy\subset X$, satisfying for every $x\in X$ the following three conditions:
\begin{itemize}
\item $xe=\{x\}$;
\item for every $V\in\Tau_e(X)$ the set $xV:=\bigcup_{y\in V}xy$ is a neighborhood of $x$;
\item for every neighborhood $O_x\subset X$ of $x$ there is a neighborhood $V$ of $e$ such  that $xV\subset O_x$.
\end{itemize}
The binary operation $\mathbf{xy}$ is called the {\em multiplication} of a portator $X$. It induces another two set-valued binary operations
$$\mathbf{x}^{-1}\mathbf y:X\times X\multimap X,\;\;\mathbf{x}^{-1}\mathbf y:(x,y)\mapsto x^{-1}y:=\{z\in X:y\in xz\}$$ and $$\mathbf{xy}^{-1}:X\times X\multimap X,\;\;\mathbf{xy}^{-1}:(x,y)\mapsto xy^{-1}:=\{z\in X:x\in zy\},$$called the {\em left division} and the {\em right division} of the portator $X$.
It is convenient to extend the operations $\mathbf{xy}$, $\mathbf{x}^{1}\mathbf{y}$, $\mathbf{xy}^{-1}$ to subsets $A,B\subset X$ letting $AB:=\bigcup_{a\in A}aB=\bigcup_{b\in B}Ab$ where $aB:=\bigcup_{b\in B}ab$ and $Ab:=\bigcup_{a\in A}ab$. In a similar manner we define the sets $A^{-1}B$ and $AB^{-1}$.

A portator $X$ is called a\index{netportator}\index{portator!netportator} {\em netportator} if for any points $x,y\in X$ the sets $xy$ and $x^{-1}y$ are finite.

\subsection{Admissible families of sets in a portator}
We shall say that a family $\C$ of subsets of a portator $X$ is \index{admissible family}{\em  admissible} if for any set $C\in\C$ accumulating at some point $x\in X$ the set $x^{-1}C:=\{y\in X:xy\in C\}$ contains a subset $C'\in\C$ accumulating at $e$. In this case we shall also say that the portator $X$ is \index{portator!$\C$-admissible}{\em $\C$-admissible}.

We recall that for any topological space $X$ by $\as$ (resp. $\css$, $\cs^*$, $\ccs$) we denote the family of all subsets of $X$ (resp. all countable subsets, all convergent sequences, and all countable subsets with countably compact closure in $X$).

\begin{proposition}\label{p:C-admis} For any netportator $X$
\begin{enumerate}
\item the families $\css$ and $\as$ are admissible;
\item the family $\cs$ is admissible if for every $x\in X$ the  set-valued map $x^{-1}\mathbf{y}:X\multimap X$, $x^{-1}\mathbf{y}:y\mapsto x^{-1}y$, is semicontinuous at $x$ and $x^{-1}x=\{e\}$.
\item the family $\ccs$ is admissible if $X$ if for every $x\in X$
the  set-valued map $x^{-1}\mathbf{y}:X\multimap X$, $x^{-1}\mathbf{y}:y\mapsto x^{-1}y$, is semicontinuous.
\end{enumerate}
\end{proposition}

\begin{proof} 1. To show that the family $\css$ is admissible, fix any set $C\in\css$ accumulating at some point $x\in X$. Since $X$ is a netportator, the set-valued map $x^{-1}\mathbf{y}:X\multimap X$ has finite value $x^{-1}y$ for any point $y\in X$. This implies that the set $C'=x^{-1}C$ is countable and hence belongs to the family $\css$. Assuming that is $C'$ does not accumulate at $e$, we can find a neighborhood $V\in\Tau_e(X)$ that has finite intersection $V\cap C'$ with the set $C'$. Then $xV$ is a neighborhood of $x$ that has finite intersection $xV\cap C\subset x(V\cap C')$ with the set $C$, which is not possible as the set $C$ accumulates at $x$.

By analogy we can prove that the family $\as$ of all subsets of $X$ is admissible.
\smallskip

2. Assume that for every $x\in X$ the  set-valued map $x^{-1}\mathbf{y}:X\multimap X$, $x^{-1}\mathbf{y}:y\mapsto x^{-1}y$, is semicontinuous at $x$ and $x^{-1}x=\{e\}$. To show that the family $\cs$ is admissible, fix any set $C\in \cs$. The set $C$ is infinite, countable, and converges to some point $x\in X$ in the sense that any neighborhood $O_x$ of $x$ contains all but finitely many points of the set $C$. Consider the set $C'=x^{-1}C$. Since $X$ is a netportator, for every $c\in C$, the set $x^{-1}c$ is finite, which implies that $x^{-1}C$ is countable. We claim that the set $C'$ is infinite. In the opposite case, the set $xC'$ is finite and so is the set $C\cap xC'\supset C\cap xX$, which is not possible as $xX$ is a neighborhood of $x$, containing all but finitely many points of the infinite set $C$. To show that the set $C'$ converges to $e$, choose any neighborhood $V\in\Tau_e(X)$. Since $x^{-1}x=\{e\}\subset V$, the semicontinuity of the operation $x^{-1}\mathbf{y}$ at $x$ yields a neighborhood $O_x\subset X$ of $x$ such that $x^{-1}O_x\subset V$. Replacing $O_x$ by $O_x\cap xV$, we can assume that $O_x\subset xV$. Since $C$ converges to $x$, the set $F=C\setminus O_x$ is finite and so is the set $x^{-1}F$.
We claim that $C'\setminus V\subset x^{-1}F$. Indeed, take any point $c'\in C'\setminus V=(x^{-1}C)\setminus V$ and observe that the intersection $xc'\cap C$ is not empty and hence contains some point $c\in C$. We claim that $c\notin O_x$. Otherwise $c'\in x^{-1}c\subset x^{-1}O_x\subset V$, which contradicts the choice of $c'$. Since $X$ is a netportator, the set $x^{-1}F$ is finite and so is its subset $C'\setminus V$. This means that the infinite set $C'$ converges to $e$ and hence  accumulates at $e$ and belongs to the family $\cs$.
\smallskip

3. Assuming that for every $x\in X$ the  set-valued map $x^{-1}\mathbf{y}:X\multimap X$, $x^{-1}\mathbf{y}:y\mapsto x^{-1}y$, is semicontinuous, we shall prove that the netportator $X$ is $\ccs$-admissible.  Let $C\in\ccs$ be a countable set with countably compact closure $\bar C$ in $X$, which accumulates at some point $x\in X$. By the first statement, the countable set $C'=x^{-1}C$ accumulates at $e$. It remains to show that $C'$ has countably compact closure in $X$. This will follow as soon as we check that the set $x^{-1}\bar C$ is countably compact and closed in $X$.

First we show that the set $x^{-1}\bar C$ is closed in $X$. Given any point $y\in X\setminus x^{-1}\bar C$, observe that $xy\subset X\setminus \bar C$. The semicontinuity of the operation $x\mathbf{y}$ at $y$ yields a neighborhood $O_y\in\Tau_y(X)$ such that $xO_y\subset X\setminus \bar C$. Then $O_y\cap x^{-1}\bar C=\emptyset$, witnessing that the set $x^{-1}\bar C$ is closed in $X$.

Next, we show that  this set is countably compact. In the opposite case we could find an infinite subset $Z\subset x^{-1}\bar C$ without accumulating points in $X$. For every $z\in Z\subset x^{-1}\bar C$ choose a point $y_z\in xz\cap\bar C$ and consider the set $Y=\{y_z\}_{z\in Z}\subset\bar C$. Since for every $y\in X$ the set $x^{-1}y$ is finite, the set $Y$ is infinite. We claim that $Y$ has no accumulation point in $X$. Given any point $y\in X$, we should find a neighborhood $V_y\in\Tau_y(X)$ with finite intersection $V_y\cap Y$. Since the set $Z$ has no accumulation points in $X$, each point $z$ of the set $x^{-1}y$ has a neighborhood $U_z\subset X$ that has finite intersection with the set $Z$. Then  $U=\bigcup_{z\in x^{-1}y}U_z$ is a neighborhood of the finite set $x^{-1}y$ that has finite intersection with the set $Z$. The semicontinuity of the function $x^{-1}\mathbf{y}$ at the point $y$, yields a neighborhood $V_y\in\Tau_y(X)$ of $y$ such that $x^{-1}V_y\subset U$. Observe that the set $F=\{z\in Z:y_z\in V_y\}\subset Z\cap x^{-1}V_y\subset Z\cap U$ is finite and so is the set $Y\cap V_y\subset xF$. Therefore the infinite subset $Y$ of $\bar C$ has no accumulation points, which contradicts the countable compactness of $\bar C$. This contradiction shows that the closed set $x^{-1}\bar C$ is countably compact.
\end{proof}

We recall that a portator $X$ is \index{portator!left-topological}{\em left-topological} if for every $x\in X$ the set-valued map $x^{-1}\mathbf{y}:X\multimap X$, $x^{-1}\mathbf{y}:y\mapsto x^{-1}y$, is semicontinuous and $x^{-1}x=\{e\}$.

Proposition~\ref{p:C-admis} implies

\begin{corollary}\label{c:port-admis} For any left-topological netportator $X$ the families $\cs$, $\ccs$, $\css$, and $\as$ are admissible.
\end{corollary}

\subsection{Local networks generating netbases in netportators}

The translation structure of a netportator allows us to transform local $\C^*$-networks at the unit into $\C^*$-netbases for the whole netportator. To each subset $N\subset X$ of a portator $X$ assign the entourage $$\vec N=\{(x,y)\in X\times X:y\in\{x\}\cup xN\}$$and to each family $\N$ of subsets of $X$ assign the family of entourages
$\vec\N=\{\vec N:N\in\N\}$.

\begin{theorem}\label{t:trans-netbase} Let $X$ be a netportator and $\Bas$ be a neighborhood base at the unit $e$ of $X$. If $\N$ is a $\C^*$-network at $e$ for some admissible family $\C$ of subsets of $X$, then the pair $(\vec \N,\vec\Bas)$ is a $\C^*$-netbase for $X$. If the portator $(X,t_X)$ is (locally) [quasi-]uniform, then so is the $\C^*$-netbase $(\vec N,\vec\Bas)$.
\end{theorem}

\begin{proof} Let $\N$ be a $\C^*$-network at $e$. By Proposition~\ref{p:trans-base}, the family $\vec\Bas$ is an entourage base for $X$. To show that $(\vec\N,\vec\Bas)$ is a $\C^*$-netbase, it suffices to check that for any $B\in\Bas$ and point $x\in X$, the family $\{\vec N[x]:\vec N\in\vec\N,\;\vec N\subset \vec B\}$ is a $\C^*$-network at $x$. Let $O_x\subset X$ be any neighborhood of $x$ and $C\in\C$ be a set, accumulating at $x$. The family $\C$, being admissible, contains a set $C'\subset x^{-1} C$ accumulating at $e$.  The family $\N$, being a $\C^*$-network $\N$ at $e$, contains a set $N\subset B$ that has infinite intersection with the set $C'$. For every point $z\in C'\subset x^{-1}C$ choose a point $c_z\in C\cap xz$. Since each set $x^{-1}y$, $y\in X$, is finite, the set $\{c_z: z\in C'\cap N\}$ is infinite. Then $\vec N\subset\vec B$ and the ball $\vec N[x]=\{x\}\cup xN$ has infinite intersection $\vec N[x]\cap C\supset \{c_z:z\in C'\cap N\}$ with the set $C$.

If the portator $X$ is (locally) [quasi-]uniform, then so is the base $\Bas$ and so is the $\C^*$-netbase $(\vec \N,\vec \Bas)$.
\end{proof}

\subsection{Locally quasi-uniform netportators}\label{s:lqu-netport}

In this subsection we apply locally quasi-uniform netbases to establish some (non-obvious) network properties of locally quasi-uniform netportators. By Proposition~\ref{p:lqu-portator}, a portator $X$ is locally quasi-uniform if and only if its multiplication map $\mathbf{xy}$ is semicontinuous at each point $(x,e)\in X\times \{e\}$. In particular, each para-topological portator is locally quasi-uniform.

\begin{theorem}\label{t:trans-s-sigma} Let $X$ be a locally quasi-uniform netportator and $\C$ be an admissible family of subsets in $X$. If $X$ has a countable $\C^*$-network at the unit $e$, then
each strong $\sigma$-subspace of $X$ has a $\sigma$-discrete $\C^*$-network.
\end{theorem}

\begin{proof} Let $\Bas$ be a neighborhood base at the unit $e$ of $X$ and $\N$ be a countable $\C^*$-network at $e$. By Theorem~\ref{t:trans-netbase}, the pair $(\vec\N,\vec\Bas)$ is a countable locally quasi-uniform $\C^*$-netbase for $X$. By Proposition~\ref{p:netbase-her} and \ref{p:lqu-subspace}, for any subspace $Z\subset X$ the pair $(\vec\N|_Z,\vec\Bas|_Z)$ is a countable locally quasi-uniform $\C^*$-netbase for $Z$. If the space $Z$ is a strong $\sigma$-space, then by Theorem~\ref{t:s-sigma}, $X$ has a $\sigma$-discrete $\C^*$-network.
\end{proof}

\begin{corollary} If a $\cs$-admissible locally quasi-uniform netportator $X$ has a countable $\cs^*$-network at the unit $e$, then
\begin{enumerate}
\item each strong $\sigma$-subspace of $X$ is an $\aleph$-space;
\item each first-countable strong $\sigma$-subspace of $X$ is metrizable;
\item each cosmic subspace of $X$ is an $\aleph_0$-space;
\item each compact sequentially compact Hausdorff subspace of $X$ is metrizable.
\end{enumerate}
\end{corollary}

\begin{proof} 1. The first statement follows from Theorem~\ref{t:trans-s-sigma} applied to the family $\cs$ of all convergent sequences in $X$.

2. The second statement follows from the first one and the metrizability of first-countable $\aleph$-spaces, proved in \cite{O'M} (see also \cite[11.4]{Grue}).

3. The third statement follows from the first statement and the observations that each cosmic space is a strong $\sigma$-space and each cosmic $\aleph$-space is an $\aleph_0$-space.

4. Assume that $K$ is a compact sequentially compact Hausdorff subspace of $X$. By Theorem~\ref{t:trans-netbase}, the space $X$ has a locally quasi-uniform $\cs^*$-netbase. By Proposition~\ref{p:lqu-subspace}, the subspace $K$ of $X$ also has a locally quasi-uniform $\cs^*$-netbase. By Theorem~\ref{t:comp-metr}(5), the  space $K$ is metrizable.
\end{proof}

\begin{corollary} If a $\ccs^*$-admissible locally quasi-uniform netportator $X$ has a countable $\ccs^*$-network at the unit $e$, then each closed compact Hausdorff subspace of $X$ is metrizable.
\end{corollary}

\begin{proof} By Theorem~\ref{t:trans-netbase}, the space $X$ has a countable locally quasi-uniform $\ccs^*$-netbase. By Proposition~\ref{p:lqu-subspace}, every compact Hausdorff subspace $K\subset X$ has a countable locally quasi-uniform $\ccs^*$-netbase and by  Theorem~\ref{t:comp-metr}(4), $K$ is metrizable.
\end{proof}

\begin{corollary}\label{c:comp-in-s*} If a locally quasi-uniform netportator $X$ has a countable $\css^*$-network at the unit $e$, then each compact Hausdorff subspace of $X$ is metrizable.
\end{corollary}

\begin{proof} By Proposition~\ref{p:C-admis} and Theorem~\ref{t:trans-netbase}, the space $X$ has a countable locally quasi-uniform $\css^*$-netbase.
Proposition~\ref{p:lqu-subspace} implies that every compact Hausdorff subspace $K\subset X$ has a countable locally quasi-uniform $\css^*$-netbase, which is also a locally quasi-uniform $\ccs^{**}$-netbase. By  Theorem~\ref{t:comp-metr}(4), $K$ is metrizable.
\end{proof}

A topological space $X$ is called a {\em $k$-space} if a subset $F\subset X$ is closed in $X$ if and only if for every compact subset $K\subset X$ the intersection $F\cap K$ is relatively closed in $K$. It is well-known that each sequential space is a $k$-space and a $k$-space $X$ is sequential if each compact subset of $X$ is sequentially compact.

\begin{corollary}\label{c:k<=>sP} If a Hausdorff locally quasi-uniform netportator has the strong Pytkeev$^*$ property, then $X$ is a $k$-space if and only if $X$ is sequential.
\end{corollary}

\begin{proof} If $X$ has the strong Pytkeev$^*$-property, then $X$ has a countable $\css^*$-network at $e$ and by Corollary~\ref{c:comp-in-s*}, each compact subspace of $X$ is metrizable and hence sequentially compact. Then the $k$-space property of $X$ is equivalent to the sequentiality of $X$.
\end{proof}

\begin{corollary}\label{c:k<=>s*} If a Hausdorff leftpara-topological netportator $X$ has a countable $\ccs^*$-network at the unit, then
$X$ is a $k$-space if and only if $X$ is sequential.
\end{corollary}

\begin{proof} By Corollaries~\ref{c:TA=>lqu}(3) and \ref{c:port-admis}, the leftpara-topological portator $X$ is locally quasi-uniform and $\ccs$-admissible.
 By Theorem~\ref{t:trans-netbase}, the space $X$ has a countable locally quasi-uniform $\ccs^*$-netbase. By Proposition~\ref{p:lqu-subspace}, every compact subspace $K\subset X$ has a countable locally quasi-uniform $\ccs^*$-netbase and by  Theorem~\ref{t:comp-metr}(4), $K$ is metrizable.
 Since compact subsets of $X$ are metrizable, the $k$-space property for $X$ is equivalent to the sequentiality of
 $X$.
\end{proof}

\begin{remark}
The class of locally quasi-uniform netportators includes all paratopological groups (which are also leftpara-topological netportators). So, all results of Subsection~\ref{s:lqu-netport} are true for paratopological groups.
\end{remark}

\begin{remark} Corollaries~\ref{c:k<=>sP} and \ref{c:k<=>s*} yield an affirmative answer to Question 4.3  of Lin, Ravsky and Zhang \cite{LRZ} who asked whether the sequentiality and the $k$-space property are equivalent for topological groups with the strong Pytkeev property.
\end{remark}

\subsection{Locally uniform netportators}

In this subsection we apply netbases to studying locally uniform netportators. By Proposition~\ref{p:lu-portator}, a portator $X$ is locally uniform if and only if the multiplication $\mathbf{xy}$ and the right division $\mathbf{xy}^{-1}$ are semicontinuous at points of the set $X\times\{e\}$. By Corollary~\ref{c:TA=>lqu}(5), each paradiv-topological portator is locally uniform. In particular, each topological loop is a locally uniform portator.

\begin{theorem} If a $\cs$-admissible locally uniform netportator $X$ has a countable $\cs^*$-network at the unit, then  for any subspace $Z$ of $X$ the following conditions are equivalent:
\begin{enumerate}
\item $Z$ is metrizable;
\item $Z$ is a first-countable closed-$\bar G_\delta$ $T_0$-space;
\item $Z$ is an $M$-space such that each closed countably compact subset of $Z$ is sequentially compact;
\item $Z$ is a first-countable collectionwise normal $\Sigma$-space.
\end{enumerate}
In particular, all sequentially compact $T_0$-subspaces of $X$ are metrizable.
\end{theorem}

\begin{proof} Assume that the $\cs$-admissible locally uniform netportator $X$ has a countable $\cs^*$-network at the unit. By Theorem~\ref{t:trans-netbase}, $X$ admits a countable locally uniform $\cs^*$-netbase.
By Proposition~\ref{p:lqu-subspace}, each subspace $Z$ of $X$ has a countable locally uniform $\cs^*$-netbase $(\E,\Bas)$. If each closed countably compact subset of $Z$ is sequentially compact, then the $\cs^*$-netbase $(\E,\Bas)$ is a $\ccs^{**}$-base.

Now the equivalence of the conditions (1)--(4) follows from the Metrization Theorem~\ref{t:metr}. The metrizability of sequentially compact $T_0$-subspaces of $X$ follows from the statement (3) and the regularity of spaces admitting a locally uniform base (see Proposition~\ref{p:lu=>regular}).
\end{proof}

\begin{theorem}\label{t:trans-lu-ccs} Let $X$ be a $\ccs$-admissible locally uniform netportator $X$. If $X$ has a countable $\ccs^*$-network at the unit $e$, then any subspace $Z$ of $X$ has the following properties:
\begin{enumerate}
\item $Z$ is a $\sigma$-space if and only if $Z$ is a $\Sigma$-space;
\item $Z$ is first-countable if and only if $Z$ is a $q$-space;
\item If $Z$ is a $w\Delta$-space, then $Z$ has a $G_\delta$-diagonal;
\item $Z$ is metrizable if and only if $Z$ is an $M$-space;
\end{enumerate}
In particular, all countably compact $T_0$-subspaces of $X$ are metrizable.
\end{theorem}

\begin{proof} Assume that the locally uniform portator $X$ is $\ccs$-admissible and has a countable $\ccs^*$-network at $e$. By Theorem~\ref{t:trans-netbase}, $X$ admits a countable locally uniform $\ccs^*$-netbase. By Proposition~\ref{p:lqu-subspace}, each subspace $Z$ of $X$ has a countable locally uniform $\ccs^*$-netbase $(\E,\Bas)$.
Now the statements (1), (2), (3), (4) follow from Theorems~\ref{t:Sigma}, \ref{t:1-enet}, \ref{t:wD=>Gd}, \ref{t:metr}, respectively. The metrizability of countably compact $T_0$-subspaces of $X$ follows from the statement (4)  and the regularity of spaces admitting a locally uniform base (see Proposition~\ref{p:lu=>regular}).
\end{proof}

\begin{theorem}\label{t:trans-lu-s} If a locally uniform netportator $X$ has a countable $\css^*$-network at the unit, then any subspace $Z$ of $X$ has the following properties:
\begin{enumerate}
\item $Z$ is a $\sigma$-space if and only if $Z$ is a $\Sigma$-space;
\item if $Z$ is a strong $\sigma$-space, then $Z$ is a $\mathfrak P^*$-space;
\item $Z$ is cosmic if and only if $Z$ is a $\mathfrak P_0$-space;
\item $Z$ is first-countable if and only if $Z$ is a $q$-space;
\item If $Z$ is a $w\Delta$-space, then $Z$ has a $G_\delta$-diagonal;
\item $Z$ is metrizable if and only if $Z$ is an $M$-space.
\end{enumerate}
In particular, all countably compact $T_0$-subspaces of $X$ are metrizable.
\end{theorem}

\begin{proof} Assume that the locally uniform netportator $X$ is has a countable $\css^*$-network at the unit. By Proposition~\ref{p:C-admis} and Theorem~\ref{t:trans-netbase}, $X$ admits a countable locally uniform $\css^*$-netbase.

By Proposition~\ref{p:lqu-subspace}, each subspace $Z$ of $X$ has a countable locally uniform $\css^*$-netbase, that is also a $\ccs^*$-netbase and a $\cs^*$-netbase.
Now the statements (1), (2), (3), (4), (5), (6) follow from Theorems~\ref{t:Sigma}, \ref{t:s-sigma}, \ref{t:aleph0}, \ref{t:1-enet},  \ref{t:wD=>Gd}, \ref{t:metr}, respectively. The metrizability of countably compact $T_0$-subspaces of $X$ follows from the statement (6) and the regularity of spaces admitting a locally uniform base (see Proposition~\ref{p:lu=>regular}).
\end{proof}

By Proposition~\ref{p:C-admis} and Corollary~\ref{c:TA=>lqu}(5), each topological loop is a $\ccs$-admissible locally uniform netportator. Applying Theorems~\ref{t:trans-lu-ccs} and \ref{t:trans-lu-s}, we obtain the following two corollaries.

\begin{corollary}\label{c:loop-cs*} If a topological loop $X$ has a countable $\cs^*$-network at the unit $e$, then  for any subspace $Z$ of $X$ the following conditions are equivalent:
\begin{enumerate}
\item $Z$ is metrizable;
\item $Z$ is a first-countable closed-$\bar G_\delta$ $T_0$-space;
\item $Z$ is an $M$-space such that each closed countably compact subset of $Z$ is sequentially compact;
\item $Z$ is a first-countable collectionwise normal $\Sigma$-space.
\end{enumerate}
In particular, all sequentially compact $T_0$-subspaces of $X$ are metrizable.
\end{corollary}

\begin{corollary}\label{c:loop-ccs*} If a topological loop has a countable $\ccs^*$-network at the unit, then any subspace $Z$ of $X$ has the following properties:
\begin{enumerate}
\item $Z$ is first-countable if and only if $Z$ is a $q$-space;
\item $Z$ is metrizable if and only if $Z$ is an $M$-space;
\item If $Z$ is a $w\Delta$-space, then $Z$ has a $G_\delta$-diagonal;
\item $Z$ is a $\sigma$-space if and only if $Z$ is a $\Sigma$-space;
\end{enumerate}
In particular, all countably compact $T_0$-subspaces of $X$ are metrizable.
\end{corollary}

\begin{problem} Can Corollaries~\ref{c:loop-cs*} and \ref{c:loop-ccs*} be generalized to topological lops?
\end{problem}

\chapter{$\w^\w$-bases in topological spaces}\label{Ch:ww-base}

In this chapter we study topological spaces with an $\w^\w$-base.
We recall that for a directed poset $P$ a topological space $X$ has a \index{topological space!$P$-base of}{\em $P$-base} if at each point $x\in X$ the space $X$ has a \index{topological space!neighborhood $P$-base of}\index{$P$-base}{\em neighborhood $P$-base} at $x$, i.e. a neighborhood base $\{U_\alpha[x]\}_{\alpha\in P}$ at $x$ such that $U_\beta[x]\subset U_\alpha[x]$ for all $\alpha\le\beta$ in $P$. For every $\alpha\in P$ the neighborhoods $U_\alpha[x]$, $x\in X$, compose the entourage $U_\alpha=\{(x,y)\in X\times X:y\in U_\alpha[x]\}$. The family $\{U_\alpha\}_{\alpha\in\w^\w}$ is called an \index{entourage $P$-base}\index{$P$-base}{\em entourage $P$-base} or briefly a {\em $P$-base} for $X$. By Lemma~\ref{l:p-Tukey}, a topological space $X$ has an $P$-base iff $\Tau_x(X)\le_T P$ for each point $x\in X$.

An $P$-base $\{U_\alpha\}_{\alpha\in P}$ for a topological space $X$
is called
\begin{itemize}
\item \index{$P$-base!uniform}\index{uniform $P$-base}{\em uniform} if for any $\alpha\in P$ there exists $\beta\in P$ such that $U_\beta^{\pm3}\subset U_\alpha$;
\item \index{$P$-base!quasi-uniform}\index{quasi-uniform $P$-base}{\em quasi-uniform} if for any $\alpha\in P$ there exists $\beta\in P$ such that $U_\beta^{2}\subset U_\alpha$;
\item\index{$P$-base!locally uniform}\index{locally uniform $P$-base} {\em locally uniform} if for any $x\in X$ and neighborhood $O_x\subset X$ of $x$ there exists $\beta\in P$ such that $U_\beta^{\pm3}[x]\subset O_x$;
\item \index{$P$-base!locally quasi-uniform}\index{locally quasi-uniform $P$-base}{\em locally quasi-uniform} if for any $x\in X$ and neighborhood $O_x\subset X$ of $x$ there exists $\beta\in P$ such that $U_\beta^{2}[x]\subset O_x$.
\end{itemize}

\section{Hereditary properties of spaces with a $P$-base}
\label{s:stable}

In this section we establish some stability properties of the class of topological spaces with an $\w^\w$-base or more generally, a $P$-base for a directed poset $P$.  %The following stability properties of the class of topological spaces with an $\w^\w$-base was proved by Gabriyelyan, K\c akol and Leiderman \cite[Propositions 2.7, 2.9]{GabKakLei_2} in the framework of topological groups. But the proofs easily generalize to the general case.

\begin{proposition}\label{p:subspace} For any poset $P$ the class of topological spaces with a $P$-base is closed under taking subspaces and images under open continuous maps.
\end{proposition}

\begin{proof} If a topological space $X$ has a $P$-base, then for any subspace $Y$ of $X$ and any point $y\in Y$ the monotone cofinal map $\Tau_y(X)\to\Tau_y(Y)$, $U\mapsto U\cap A$, defines a reduction $\Tau_y(X)\succcurlyeq \Tau_y(Y)$. Taking into account that $P\succcurlyeq\Tau_y(X)\succcurlyeq\Tau_y(Y)$, we conclude that $P\succcurlyeq \Tau_y(Y)$, which means that the space $Y$ has a $P$-base.

If $f:X\to Y$ is an open continuous surjective map between topological spaces, then for every point $x\in X$ the monotone cofinal map $\Tau_x(X)\to\Tau_{f(x)}(Y)$, $U\mapsto f(U)$, defines a reduction $\Tau_x(X)\succcurlyeq \Tau_{f(x)}(Y)$. If the space $X$ has a $P$-base, then $P\succcurlyeq\Tau_x(X)\succcurlyeq \Tau_{f(x)}(Y)$ implies that $P\succcurlyeq \Tau_y(Y)$ for all $y\in Y$, which means that the space $Y$ has a $P$-base.
\end{proof}

We recall that the \index{box-product}{\em box-product} $\square_{\alpha\in\kappa}X_\alpha$ of a family $(X_\alpha)_{\alpha\in \kappa}$ is the Cartesian product $\prod_{\alpha\in\kappa}X_\alpha$ endowed with the topology generated by the base consisting of products $\prod_{\alpha\in\kappa}U_\alpha$ of open subsets $U_\alpha\subset X_\alpha$, $\alpha\in\kappa$.
It is clear that for any point $x=(x_\alpha)_{\alpha\in\kappa}$ in the box-product $X=\square_{\alpha\in\kappa}X_\alpha$ the cofinal monotone map $\prod_{\alpha\in\kappa}\Tau_{x_\alpha}(X_\alpha)\to\Tau_x(X)$, $(U_\alpha)_{\alpha\in\kappa}\mapsto\prod_{\alpha\in\kappa}U_\alpha$, defines the reduction $\prod_{\alpha\in\kappa}\Tau_{x_\alpha}(X_\alpha)\succcurlyeq\Tau_x(X)$, which allows us to observe the following fact.

\begin{proposition}\label{p:boxproduct} Let $P$ be a directed poset. If topological spaces $X_\alpha$, $\alpha\in\kappa$, have a $P$-base, then the box-product $\square_{\alpha\in \kappa}X_\alpha$ has a $P^\kappa$-base. Consequently, the class of topological spaces with an $\w^\w$-base is closed under taking countable box-products.
\end{proposition}

Now we prove the preservation of the class of topological spaces with an $\w^\w$-base by countable Tychonoff products. Given a cardinal $\kappa$ and a poset $P$ with a smallest element $0$, consider the poset
$$P^{\odot\kappa}=\{(x_\alpha)_{\alpha\in\kappa}\in P^\kappa:|\{\alpha\in\kappa:x_\alpha\ne0\}|<\w\}$$endowed with the partial order, inherited from the power $P^\kappa$.

\begin{proposition}\label{p:Tprod} Let $P$ be a poset with a smallest element $0$ and $(X_\alpha)_{\alpha\in\kappa}$ be a family of topological spaces with a $P$-base. Then the Tychonoff product $X=\prod_{\alpha\in\kappa}X_\alpha$ has a $P^{\odot\kappa}$-base.
\end{proposition}

\begin{proof} Given any element $x=(x_\alpha)_{\alpha\in\kappa}$ of the Tychonoff product $X=\prod_{\alpha\in\kappa}X_\alpha$, for every point $x_\alpha$ fix a cofinal monotone map $f_\alpha:P\to \Tau_{x_\alpha}(X_\alpha)$ such that  $f_\alpha(0)=X_\alpha$. Then the function $$f:P^{\odot\kappa}\to\Tau_x(X),\;\;
f:(p_\alpha)_{\alpha\in\kappa}\mapsto
\prod_{\alpha\in\kappa}f_\alpha(p_\alpha),$$
defines a reduction $P^{\odot\kappa}\succcurlyeq \Tau_x(X)$, witnessing that the Tychonoff product $X=\prod_{\alpha\in\kappa}X_\alpha$ has a $P^{\odot\kappa}$-base.
\end{proof}

\begin{corollary}\label{c:Tprod} The class of topological spaces with an $\w^\w$-base is closed under countable Tychonoff products.
\end{corollary}

\begin{proof} This corollary will follow from Proposition~\ref{p:Tprod} as soon as we show that $\w\times (\w^\w)^\w\succcurlyeq (\w^\w)^{\odot\w}$. This reduction is established by the monotone cofinal map $f:\w\times(\w^\w)^\w\to(\w^\w)^{\odot\w}$, $f:(n,(x_i)_{i\in\w})\mapsto (y_i)_{i\in\w}$ where $y_i=x_i$ for $i\le n$ and $y_i=0$ for all $i>n$.
\end{proof}

Observing that the constructions of $\w^\w$-bases in Propositions~\ref{p:subspace}, \ref{p:boxproduct},  \ref{p:Tprod} preserve (locally) [quasi] uniform bases, we can prove the following fact.

\begin{proposition}\label{p:lqu-her} The class of topological spaces with a (locally) [quasi-] uniform $\w^\w$-base is closed under taking subspaces, countable Tychonoff products and countable box-products.
\end{proposition}

A map $f:X\to Y$ between topological spaces is defined to be \index{map!hereditarily quotient}{\em hereditarily quotient} if  for every subspace $Z\subset Y$ the restriction $f|f^{-1}(Z):f^{-1}(Z)\to Z$ is a quotient map. By  \cite[2.4.F(a)]{Eng}, a map $f:X\to Y$ is hereditarily quotient if and only if for every   point $y\in Y$ and a neighborhood $U\subset X$ of the preimage $f^{-1}(y)$ the set $f(U)$ is a neighborhood of the point $y$ in $Y$. This characterization implies that open or closed maps are hereditarily quotient.

\begin{proposition} Let $f:X\to Y$ be a hereditarily quotient map. The space $Y$ has a local $\w^\w$-base at a point $y\in Y$ if one of the following conditions is satisfied:
\begin{enumerate}
\item[\textup{(1)}] $f^{-1}(y)$ is countable and $X$ has a local $\w^\w$-base at each point $x\in f^{-1}(y)$;
\item[\textup{(2)}] the space $X$ is metrizable and $f^{-1}(y)$ is $\sigma$-compact.
\end{enumerate}
\end{proposition}

 \begin{proof} First assume that the preimage $f^{-1}(y)$ is countable and at each point $x\in f^{-1}(y)$ the space $X$ has a neighborhood $\w^\w$-base, which means that $\w^\w\succcurlyeq \Tau_x(X)$. To show that $Y$ has a local $\w^\w$-base at $y$, we shall establish the reduction $\prod_{x\in f^{-1}(y)}\Tau_x(X)\succcurlyeq\Tau_y(Y)$. The characterization of hereditarily quotient maps given in \cite[2.4.F(a)]{Eng} ensures that the map $u: \prod_{x\in f^{-1}(y)}\Tau_x(X)\to \Tau_y(Y)$, $u:(U_x)_{x\in f^{-1}(y)}\mapsto f\big(\bigcup_{x\in f^{-1}(y)}U_x\big)$, is well-defined, monotone, and cofinal. Taking into account that each poset $\Tau_x(X)$ admits a reduction $\w^\w\succcurlyeq \Tau_x(X)$, we get the required reduction
$$\w^\w\cong (\w^\w)^{f^{-1}(y)}\succcurlyeq\prod_{x\in f^{-1}(y)}\Tau_x(X)\succcurlyeq\Tau_y(Y),$$ witnessing that the space $Y$ has a neighborhood $\w^\w$-base at $y$.

Next, assume that the space $X$ is metrizable and the preimage $f^{-1}(y)$ is $\sigma$-compact. Then $f^{-1}(y)=\bigcup_{n\in\w}K_n$ for some increasing sequence $(K_n)_{n\in\w}$ of compact sets. Fix a metric $d$ generating the topology of the metrizable space $X$ and for every function $\alpha\in\w^\w$ consider the open neighborhood $$U_\alpha=\bigcup_{n\in\w}\bigcup_{x\in K_n}B_d(x,2^{-\alpha(n)})$$ of $f^{-1}(y)$. Here by $B_d(x,\e)=\{y\in X:d(x,y)<\e\}$ we denote the open $\e$-ball centered at $x$. The characterization of hereditarily quotient maps given in \cite[2.4.F(a)]{Eng} guarantees that the set $V_\alpha=f(U_\alpha)$ is a neighborhood of $y$ in $Y$. It is easy to check that the map $\w^\w\to\Tau_y(Y)$, $\alpha\mapsto V_\alpha$, is monotone and cofinal. Therefore, $\w^\w\succcurlyeq \Tau_y(Y)$ and the space $Y$ has a neighborhood $\w^\w$-base at $y$.
\end{proof}

We shall say that a topological space $X$ \index{topology!inductive}{\em has the inductive topology at a point $x\in X$, with respect to a family} $\C$ of subspaces of $X$ if a subset $U\subset X$ containing $x$ is a neighborhood of $x$ in $X$ if and only if for every $C\in\C$ with $x\in C$ the intersection $C\cap U$ is a neighborhood of $x$ in $C$.

\begin{proposition} A topological space $X$ has a neighborhood $\w^\w$-base at a point $x\in X$ if $X$ has inductive topology at $x$ with respect to a countable family $\C$ of subspaces such that each space $C\in\C$ with $x\in C$ has a neighborhood  $\w^\w$-base at $x$.
\end{proposition}

\begin{proof} Let $\C_x=\{C\in\C:x\in C\}$. The definition of the inductive topology at $x$ ensures that the map $u: \prod_{C\in\C_x}\Tau_x(C)\to \Tau_x(X)$, $u:(U_C)_{C\in\C_x}\to\bigcup_{C\in\C_x}U_C$, is well-defined, monotone, and cofinal. Taking into account that each space $C\in\C_x$ has a neighborhood $\w^\w$-base
at $x$, we get the reduction
$$\w^\w\cong (\w^\w)^{\C_x}\succcurlyeq\prod_{C\in\C_x}\Tau_x(C)\succcurlyeq\Tau_x(X).$$
\end{proof}

A topological space $X$ is defined to carry the {\em inductive topology with respect to a cover} $\C$ if $X$ {has the inductive topology at each point $x\in X$, with respect to the  family $\C$}. This topology consists of set $U\subset X$ such that for every $C\in\C$ the intersection $U\cap C$ is relatively open in $C$. Topological spaces carrying the inductive topology with respect to some countable cover by compact sets are called {\em $k_\w$-spaces}.

\begin{proposition}\label{p:cosmic-kw} Each cosmic $k_\w$-space $X$ has a uniform $\w^\w$-base.
\end{proposition}

\begin{proof} Let $\square^\w\IR$ be the countable box-product of countably many copies of the real line.  By Proposition~\ref{p:lqu-her}, the space $\square^\w\IR$ has a uniform $\w^\w$-base. Then its subspace $$\IR^\infty=\{(x_n)_{n\in\w}\in\square^\w\IR:\exists n\in\w\;\forall m\ge n\;\;x_m=0\}$$ has a uniform $\w^\w$-base, too. It is well-known (and easy to see) that $\IR^\w$ is a cosmic $k_\w$-space. Moreover, each cosmic $k_\w$-space $X$ embeds into the product $[0,1]^\w\times\IR^\infty$ of the Hilbert cube $[0,1]^\w$ and $\IR^\infty$. By Proposition~\ref{p:lqu-her}, the space $[0,1]^\w\times\IR^\infty$ has a uniform $\w^\w$-base and so does its subspace $X$.
\end{proof}

\section{Locally uniform $P$-bases in baseportators}

Many examples of topological spaces with a $P$-base naturally appear in Topological Algebra. Namely, the transport structure of a baseportator $X$ allows us to transform a neighborhood $P$-base at the unit $e\in X$ into an entourage  $P$-base for the whole space $X$.

We recall that a \index{baseportator}{\em baseportator} is a pointed topological space $X$ with a distinguished point $e\in X$ (called the {\em unit} of $X$) and a family $(t_x)_{x\in X}$ of monotone cofinal functions $t_x:\Tau_e(V)\to\Tau_x(X)$, $x\in X$, called the {\em transport structure} of the baseportator. The transport structure can be recovered from the set-valued binary operation $$\mathbf{xV}:X\times\Tau_e(X)\multimap X,\;\;\mathbf{xV}:(x,V)\mapsto xV:=t_x(V),$$called the {\em multiplication} operation of the baseportator. The multiplication induces another set-valued binary operation
$$\mathbf{xV}^{-1}:X\times\mathcal P(X)\multimap X,\;\;\mathbf{xV}^{-1}:(x,V)\mapsto xV^{-1}:=\{y\in X:x\in xV\},$$
called the {\em division operation} of the baseportator.

A special kind of baseportators are portators. Their transport maps $t_x$ are determined by a set-valued binary operation $\mathbf{xy}:X\times X\multimap X$, $\mathbf{xy}:(x,y)\mapsto xy\subset X$, in the sense that $t_x(V)=xV:=\bigcup_{y\in V}xy$ for all $x\in X$ and $V\in\Tau_e(X)$. Portators are studied in details in Chapter~\ref{ch:portator}

The transport structure of a baseportator $X$ transforms each neighborhood $V\in\Tau_e(X)$ of the unit $e$ into the entourage $$\vec V=\{(x,y)\in X\times X:y\in xV\}.$$ These entourages form the base $\{\vec V:V\in\Tau_e(X)\}$ of the {\em canonical preuniformity} $\vec\Tau_X$ on the baseportator.
A baseportator $X$ is called
{uniform} (resp. {quasi-uniform}, {locally uniform}, {locally quasi-uniform}, {symmetrizable}) if so it its canonical preuniformity $\vec\Tau_X$.

By Propositions~\ref{p:lqu-bport} and \ref{p:lu-bport}, a baseportator $X$ is locally quasi-uniform (resp. locally uniform) if and only if its multiplication $\mathbf{xV}$ (and its division $\mathbf{xV}^{-1}$) is continuous at each point $(x,e)\in X\times \{e\}$.

 Proposition~\ref{p:trans-base} and Corollary~\ref{c:TA=>lqu} imply

\begin{theorem}\label{t:algebra-Pbase} Let $P$ be a directed poset, $X$ be a baseportator, and $\Bas=(B_\alpha)_{\alpha\in P}$ be a neighborhood $P$-base at the unit of $X$. Then the family $\vec \Bas=(\vec B_\alpha)_{\alpha\in P}$ is a $P$-base for $X$. The $P$-base $\vec\Bas$ is uniform
(resp. quasi-uniform, locally uniform, locally quasi-uniform, symmetrizable) if and only if so is the baseportator $X$. In particular, the $P$-base $\vec\Bas$ is
\begin{enumerate}
\item locally quasi-uniform if $X$ is a para-topological portator, for example, a paratopological lop;
\item quasi-uniform if $X$ is a paratopological group;
\item locally uniform if $X$ is a paradiv-topological portator, for example, a topological loop;
\item uniform if $X$ is a locally associative invpara-topological portator, for example, a topological group.
\end{enumerate}
\end{theorem}

A bit different canonical $P$-bases can be defined for symmetrizable baseportators. We recall that a baseportator $X$ is {\em symmetrizable} if for any neighborhood $V\in\Tau_e(X)$ of the unit $e$ and any $x\in X$ the set $\vec V^{-1}[x]:=\{y\in X:x\in yV\}$ is a neighborhood of $x$.
 By Proposition~\ref{p:sym-portator} a portator $X$ is {\em symmetrizable} if for every $y\in X$ the set-valued map $\mathbf{x}^{-1}y:X\multimap X$, $\mathbf{x}^{-1}y:x\mapsto x^{-1}y$, is continuous at $y$ and $y^{-1}y=\{e\}$.

\begin{theorem}\label{t:algebra-sPbase} Let $P$ be a directed poset, $X$ be a symmetrizable baseportator, and $\Bas=\{B_\alpha\}_{\alpha\in P}$ be a neighborhood $P$-base at the unit of $X$. Then the family $\overleftrightarrow\Bas=\{\vec B_\alpha\cap\vec B_\alpha^{-1}\}_{\alpha\in P}$ is a symmetric $P$-base for $X$. The $P$-base $\vec\Bas$ is (locally) uniform if the baseportator $X$ is (locally) quasi-uniform. In particular, the $P$-base $\overleftrightarrow\Bas$ is
\begin{enumerate}
\item locally uniform if $X$ is a quasipara-topological portator, for example, a topological lop;
\item uniform if $X$ is a locally associative quasipara-topological portator, for example, a topological group.
\end{enumerate}
\end{theorem}

Since each rectifiable space is homeomorphic to a topological lop, Theorem~\ref{t:algebra-sPbase} implies the following important corollary.

\begin{corollary}\label{c:rectif} Let $P$ be a directed poset. A rectifiable space $X$ has a locally uniform $P$-base if and only if $X$ has a neighborhood $P$-base at some point.
\end{corollary}

\section{$P$-bases in function spaces}

In this section we construct (locally uniform) $P$-bases in certain function spaces. For two topological spaces $X,Y$ by $C(X,Y)$ we denote the set of all continuous functions from $X$ to $Y$.
For a family $\K$ of compact subsets of $X$ the \index{topology!$\K$-open}{\em $\K$-open topology} on $C(X,Y)$ is generated by the subbase consisting of the sets $[K,U]=\{f\in C(X,Y):f(K)\subset U\}$ where $K\in\K$ and $U$ is an open set in $Y$. The space $C(X,Y)$ endowed with the $\K$-open topology will be denoted by $C_\K(X,Y)$. For the family $\K$ of (finite) compact subsets of $X$, the space $C_\K(X,Y)$ is denoted  by $C_k(X,Y)$ (resp. $C_p(X,Y)$).\index{$C_\K(X,Y)$}\index{$C_p(X,Y)$}\index{$C_k(X,Y)$}

 Any family $\K$ of compact subsets of $X$ will be considered as a poset endowed with the natural inclusion order ($A\le B$ iff $A\subset B$). If $P\succcurlyeq \K$ for some poset $P$, then we shall say that the family $\K$ is {\em $P$-dominated}.

 \begin{proposition}\label{p:C(X,Y)-const} Let $P,Q$ be directed posets and $\K$ be a $P$-dominated family of compact subsets of a topological space $X$. If a topological space $Y$ has a neighborhood $Q$-base at a point $y\in Y$, then the function space $C_\K(X,Y)$ has a neighborhood $P\times Q$-base at the constant function $\bar y:X\to\{y\}\subset Y$.
 \end{proposition}

 \begin{proof} We claim that the monotone map
 $$\mu:\K\times\Tau_{y}(Y)\to \Tau_{\bar y}(C_\K(X,Y)),\;\;\mu:(K,U)\mapsto [K,U],$$ is cofinal. Indeed, given any neighborhood $O_{\bar y}\subset C_\K(X,Y)$ of the constant function $\bar y$, we can find sets $K_1,\dots,K_n\in\K$ and neighborhoods $U_1,\dots,U_n\subset Y$ of $y$ such that $\bigcap_{i=1}^n[K_i,U_i]\subset O_{\bar y}$. Using the $P$-dominacy of the family $\K$, we can show that union $K_1\cup\dots\cup K_k$ is contained in some set $K\in\K$. Then for the neighborhood $U=\bigcap_{i=1}^nU_i$ we get the inclusion $[K,U]\subset \bigcap_{i=1}^n[K_i;U_i]\subset O_{\bar y}$, witnessing that the monotone map $\mu$ is cofinal and hence $P\times Q\succcurlyeq\K\times\Tau_y(Y)\succcurlyeq \Tau_{\bar y}(C_\K(X,Y))$, which means that the space $C_\K(X,Y)$ has a neighborhood $P\times Q$-base at the constant function $\bar y$.
 \end{proof}

 We recall that a topological space $X$ is called \index{topological space!topologically homogeneous}{\em topologically homogeneous} if for any points $x,y\in X$ there exists a homeomorphism $h:X\to Y$ such that $h(x)=y$. It is clear that a topological space $X$ has an $P$-base if and only if $X$ has a neighborhood $P$-base at some point $x\in X$. This observation and  Proposition~\ref{p:C(X,Y)-const} imply:

 \begin{corollary}\label{c:C(X,Y)-homogen} Let $P,Q$ be two directed posets and $\K$ be a  $P$-dominated family of compact subsets of a topological space $X$. If a topological space $Y$ has a $Q$-base and  the function space $C_\K(X,Y)$ is topologically homogeneous, then $C_\K(X,Y)$ has a $P\times Q$-base.
 \end{corollary}

 \begin{theorem}\label{t:C(X,Y)} Let $P,Q$ be two directed posets, $\K$ be a $P$-dominated family of compact subsets of a topological space $X$, and $Y$ be a topological space with a $Q$-base.
 The function space $C_\K(X,Y)$ has
 \begin{enumerate}
 \item a locally uniform $P\times Q$-base if the space $Y$ is rectifiable;
  \item a uniform $P\times Q$-base if $Y$ is a topological group.
 \end{enumerate}
 \end{theorem}

 \begin{proof}  By Proposition~\ref{p:C(X,Y)-const}, the function space $C_\K(X,Y)$ has a neighborhood $P\times Q$-base at some point.

 If $Y$ is a topological group (a rectifiable space), then so is the function space $C_\K(X,Y)$ (see  Proposition~4.3 \cite{BL}). By Theorem~\ref{t:algebra-Pbase}(4) (or Corollary~\ref{c:rectif}), the function space $C_\K(X,Y)$ has a (locally) uniform $P\times Q$-base.
 \end{proof}

 Applying Theorem~\ref{t:C(X,Y)} to the poset $P=Q=\w^\w$ we get the following result whose last statement generalizes a  result of Ferrando and K\c akol \cite{feka}. In the proof we should also use a Christensen's Theorem~\ref{t:Chris} saying that for a Polish space $X$ the family $\K(X)$ of all compact subsets of $X$ is $\w^\w$-dominated.

 \begin{corollary}\label{c:C(X,Y)ww} For a Polish space $X$ and a topological space $Y$ with a $\w^\w$-base, the function space $C_k(X,Y)$ has
 \begin{enumerate}
 \item a locally uniform $\w^\w$-base if the space $Y$ is rectifiable;
 \item a uniform $\w^\w$-base if $Y$ is a topological group.
 \end{enumerate}
 \end{corollary}

In the following corollary, a cardinal $\lambda$ is identified with the set $[0,\lambda)$ of ordinals $\alpha<\lambda$ and is endowed with the topology generated by the subbase consisting of the sets $[0,\alpha)$ and $(\alpha,\lambda)$ for $\alpha\in\lambda$.

\begin{corollary} For an $\w^\w$-dominated cardinal $\lambda$ and a topological space $Y$ with a $\w^\w$-base, the function space $C_k(\lambda,Y)$ has
 \begin{enumerate}
 \item a locally uniform $\w^\w$-base if the space $Y$ is rectifiable;
 \item a uniform $\w^\w$-base if $Y$ is a topological group.
 \end{enumerate}
 \end{corollary}

In the following theorem we establish some additional properties of the function spaces $C_k(\lambda,Y)$. A topological space $X$ is called \index{topological space!strongly zero-dimensional}{\em strongly zero-dimensional} if each open cover of $X$ can be refined by a disjoint open cover. It follows that each strongly zero-dimensional space is paracompact.

\begin{theorem}\label{t:C_k-L} Let $\lambda$ be a regular uncountable cardinal and $Y$ be a metrizable space of density $d(Y)<\lambda$. The function space $X:=C_k(\lambda,Y)$ has the following properties.
\begin{enumerate}
\item Each open cover $\U$ of $X$ has a subcover $\V\subset\U$ of cardinality $|\V|<\lambda$.
\item If $|Y|\ge 2$, then $X$ contains $\lambda$ many pairwise disjoint non-empty open sets and $X$ is not a $\Sigma$-space.
\item If the space $Y$ is discrete, then the function space $X$ is strongly zero-dimensional and its universal uniformity  is generated by the base $\{U_\alpha\}_{\alpha\in\lambda}$ consisting of the entourages $U_\alpha=\{(f,g)\in X\times X:f|[0,\alpha]=g|[0,\alpha]\}$, $\alpha\in\lambda$.
\item If the space $Y$ is discrete and the cardinal $\lambda$ is $\w^\w$-dominated, then the universal uniformity $\U_X$ of $X$ has an $\w^\w$-base.
\end{enumerate}
\end{theorem}

\begin{proof} 1. Let $\U$ be an open cover of the function space $X:=C_k(\lambda,Y)$. Fix a metric $d$ generating the topology of the space $Y$. For any functions $f,g\in C_k(\lambda,Y)$ and an ordinal $\alpha<\lambda$ let $d_\alpha(f,g)=\sup_{x\in[0,\alpha]}d(f(x),g(x))$. Let also $d_\lambda(f,g)=\sup_{\alpha\in\lambda}d_\alpha(f,g)$.
For every $f\in X$, $\alpha\in\lambda$ and $\e>0$ consider the open neighborhood $B_\alpha[f;\e):=\{g\in C_k(\lambda,Y):d_\alpha(g,f)<\e\}$ of $f$ in the function space $C_k(\lambda,Y)$.

%let $2^{-\w}=\{2^{-n}:n\in\w\}\subset(0,1]$.
For every $f\in X$ let $$
\begin{aligned}
\e_f&:=\sup\{\e\in (0,1]:\exists \alpha\in\lambda\;\exists U\in\U\;\;B_\alpha[f;5\e)\subset U\},\mbox{ and }\\
\alpha_f&:=\min\{\alpha\in\lambda:\exists U\in\U\;\;B_\alpha[f,4\e_f)\subset U\}.
\end{aligned}
$$
Choose also a set $U_f\in\U$ such that $B_{\alpha_f}[f;4\e_f)\subset U_f$.

\begin{claim}\label{cl:C_k} For any functions $f,g\in C_k(\lambda,Y)$ we get $5\e_f\ge 5\e_g-d_\lambda(f,g)$.
\end{claim}

\begin{proof} Assuming that $5\e_f<5\e_g-d_\lambda(f,g)$, we can choose $\delta>0$ such that  $d_\lambda(f,g)+5\e_f+\delta\le 5\e_g-\delta$. By the definition of $\e_g$, there exists an ordinal $\alpha\in\lambda$ such that $B_\alpha[g;5\e_g-\delta)\subset U$ for some $U\in\U$. Then $$B_\alpha[f;5\e_f+\delta)\subset B_\alpha[g;d_\lambda(f,g)+5\e_f+\delta)\subset B_\alpha[g;5\e_g-\delta)\subset U\in\U,$$ which contradicts the definition of the number $\e_f$.
\end{proof}

 For every ordinal $\alpha\in\lambda$, identify the function space $C_k([0,\alpha],Y)$ with the subspace $\{f\in C_k(\lambda,Y):f|[\alpha,\lambda)\equiv\mbox{const}\}$ of $C_k(\lambda,Y)$, consisting of functions, which are constant on the interval $[\alpha,\lambda)$. By \cite[3.4.16]{Eng}, the function space $C_k([0,\alpha],Y)$ has weight (and Lindel\"of number) $\le d(Y)+|\alpha|<\lambda$.
For two ordinals $\alpha,\beta\in\lambda$ by $\alpha\vee\beta$ we denote their maximum $\max\{\alpha,\beta\}$.

We shall construct inductively a non-decreasing sequence of ordinals $(\alpha_n)_{n\in\w}$ and a sequence $(F_n)_{n\in\w}$ of subsets $F_n\subset C_k(\lambda,Y)$ such that for every $n\in\w$ the following conditions are satisfied:
\begin{enumerate}
\item[$(1_n)$] $C_k([0,\alpha_n],Y)\subset \bigcup_{f\in F_n}B_{\alpha_n\!{\vee}\alpha_f}[f;\e_f)$;
\item[$(2_n)$] $F_n\subset C_k([0,\alpha_n],Y)$ and $|F_n|<\lambda$;
\item[$(3_n)$] $\alpha_{n+1}=\sup\limits_{f\in F_n}(\alpha_n\vee\alpha_f)$.
\end{enumerate}
We start the inductive construction letting $\alpha_0=0$. Assume that for some $n\in\w$ an ordinal $\alpha_n$ has been constructed. Since the metrizable space $X_n:=C_k([0,\alpha_n],Y)$ has Lindel\"of number $L(X_n)\le w(X_n)<\lambda$, for the open cover $\{B_{\alpha_n\!{\vee}\alpha_f}[f;\e_f):f\in X_n\}$ of $X_n$ there exists a subset $F_n\subset X_n$ of cardinality $|F_n|\le L(X_{n})<\lambda$ such that  $X_{n}\subset \bigcup_{f\in F_n}B_{\alpha_n\!{\vee}\alpha_f}[f;\e_f)$. The regularity of the cardinal $\lambda$ guarantees that the ordinal $\alpha_{n+1}:=\sup_{f\in F_n}(\alpha_n\!{\vee}\alpha_f)$ is strictly smaller than $\lambda$. Now we see that the conditions $(1_n)$--$(3_n)$ are satisfied.
\smallskip

After completing the inductive construction, consider the ordinal $\alpha_\w:=\sup_{n\in\w}\alpha_n$ and the set $F:=\bigcup_{n\in\w}F_n\subset C_k([0,\alpha_\w],Y)\subset C_k(\lambda,Y)$ of cardinality $|F|\le\sum_{n\in\w}|F_n|<\lambda$. We claim that the family $\V=\{U_f:f\in F\}\subset\U$ is a required subcover of $C_k(\lambda,Y)$ of cardinality $|\V|\le|F|<\lambda$.

Given any function $g\in C_k(\lambda,Y)$, for every ordinal $n\le\w$  consider the (unique) function $g_n\in C_k([0,\alpha_n],Y)$ such that $g_n|[0,\alpha_n]=g|[0,\alpha_n]$.
By the continuity of the function $g$ at $\alpha_\w$, there exists a number $n\in\w$ such that $d(g(x),g(\alpha_\w))<\e_{g_\w}$ for all $x\in[\alpha_n,\alpha_\w]$. This implies that  $d_\lambda(g_n,g_\w)<\e_{g_\w}$ and $5\e_{g_n}\ge 5\e_{g_\w}-d_\lambda(g_n,g_\w)>4\e_{g_\w}$ according to Claim~\ref{cl:C_k}.
%By Claim~\ref{??}, $$3\e_{g_n}\ge 3\e_{g_\w}-d_\lambda(g_n,g_\w)>2\e_{g_\w}.$$

By the inductive condition $(1_n)$, for the function $g_n\in C_k([0,\alpha_n],Y)$ there exists $f\in F_n$ such that $g_n\in B_{\alpha_n\!{\vee}\alpha_f}[f;\e_f)$ and hence $d_{\alpha_n\!{\vee}\alpha_f}(g_n,f)<\e_f$. By the inductive condition $(3_n)$, $\alpha_f\le\alpha_{n+1}\le\alpha_\w$. Taking into account that the functions $f$ and $g_n$ are constant on the interval $[\alpha_n,\lambda)$, we conclude that $d_\lambda(f,g_n)=d_{\alpha_n\vee\alpha_f}(f,g_n)<\e_f$ and hence
 $5\e_{g_n}>5\e_{f}-d_\lambda(f,g_n)>4\e_{f}$ according to Claim~\ref{cl:C_k}. Then $d_\lambda(f,g_n)<\e_f<\frac54\e_{g_n}$ and by Claim~\ref{cl:C_k}, $5\e_f>5\e_{g_n}-d_\lambda(f,g_n)>5\e_{g_n}-\tfrac54\e_{g_n}=\tfrac{15}4\e_{g_n}$ and hence $\e_{g_n}<\tfrac{20}{15}\e_f=\tfrac43\e_f$. Then  $\e_{g_\w}<\frac54\e_{g_n}<\frac54\frac{4}{3}\e_f=\tfrac53\e_f$.

 We claim that $g\in B_{\alpha_f}[f;4\e_f)\subset U_f$.  Indeed, for every $x\in [0,\alpha_f]$ we get $$d(g(x),f(x))\le d(g(x),g_n(x))+d(g_n(x),f(x))<\e_{g_\w}+\e_f<\tfrac{5}3\e_{f}+\e_f<4\e_f,$$
 and hence $g\in B_{\alpha_f}[f;4\e_f)\subset U_f$.
 \smallskip

2. Assume that $|Y|>1$ and choose two non-empty disjoint open sets $V,W\subset W$. For every ordinal $\alpha\in\lambda$ consider the open set $$U_\alpha=\{f\in C_k(\lambda,Y):f\big[[0,\alpha]\big]\subset V,\;f(\alpha+1)\in W\}$$in $C_k(\lambda,Y)$ and observe that the family $(U_\alpha)_{\alpha\in\lambda}$ is disjoint.

To show that $X$ is not a $\Sigma$-space, choose two distinct points $y_0,y_1\in Y$ and consider the closed subspace $Z=\{f\in C_k(\lambda,Y):f[\lambda]\subset\{y_0,y_1\}\}$. Assuming that $X$ is a $\Sigma$-space, we conclude that $Z$ is a $\Sigma$-space, too. Since $Z$ is a $P$-space, Lemma~\ref{l:Sigma+P=d} implies that $Z$ is discrete, which is not true. So, $X$ is not a $\Sigma$-space.
\smallskip

3,4. Assume that the space $Y$ is discrete. Let $\U$ be the uniformity on the space $X:=C_k(\lambda,Y)$ generated by the base $\Bas=\{U_\alpha\}_{\alpha\in\lambda}$ consisting of the entourages $U_\alpha=\{(f,g)\in X\times X:f|[0,\alpha]=g|[0,\alpha]\}$, $\alpha\in\lambda$. We claim that each open cover $\W$ of $X$ can be refined by the disjoint cover $\{U_\alpha[f]:f\in X\}$ for a suitable ordinal $\alpha\in\lambda$. For every $f\in X$ find a set $W_f\in\W$ containing $f$ and an ordinal $\alpha_f\in\lambda$ such that $U_{\alpha_f}[f]\subset W_f$. By the first statement, for the open cover $\{U_{\alpha_f}[f]:f\in X\}$ of $X$, there exists a subset $F\subset X$ of cardinality $|F|<\lambda$ such that $X=\bigcup_{f\in F}U_{\alpha_f}[f]$. By the regularity of the cardinal $\lambda$, the ordinal $\alpha:=\sup_{f\in F}\alpha_f$ is strictly smaller than $\lambda$. Then $\{U_\alpha[f]:f\in X\}$ is a required disjoint refinement of the cover $\W$. This implies that the space $X$ is strongly zero-dimensional and the universal uniformity $\U_X$ of $X$ coincides with the uniformity $\U$.
\smallskip

If the cardinal $\lambda$ is $\w^\w$-dominated, then we can fix a monotone cofinal map $\varphi:\w^\w\to\lambda$ and conclude that $(U_{\varphi(\alpha)})_{\alpha\in\w^\w}$ is an $\w^\w$-base for the universal uniformity $\U_X=\U$ of the space $X=C_k(\lambda,Y)$.
\end{proof}

Theorem~\ref{t:C_k-L} and Corollary~\ref{c:bdcf(d)} imply the following two corollaries.

\begin{corollary}\label{c:CkZ} The function space $Z:=C_k(\w_1,\IZ)$ has the following properties:
\begin{enumerate}
\item $Z$ is a topological group;
\item $Z$ is Lindel\"of $P$-space of uncountable cellularity;
\item under $\w_1=\mathfrak b$ the universal uniformity $\U_Z$ of has an $\w^\w$-base.
\end{enumerate}
\end{corollary}

\begin{corollary}\label{c:Ck2} The function space $X:=C_k(\w_1,\IR)$ has the following properties:
\begin{enumerate}
\item $X$ is a locally convex space;
\item $X$ is Lindel\"of, has uncountable cellularity, and is not a $\Sigma$-space;
\item under $\w_1=\mathfrak b$ the space $X$ has a uniform $\w^\w$-base.
\end{enumerate}
\end{corollary}

\begin{remark} Theorem~\ref{t:fb=>ms} implies that the universal uniformity $\U_X$ of the function space $X:=C_k(\w_1;\IR)$ does not have $\w^\w$-bases.
\end{remark}

\begin{remark}  The Lindel\"of property of the function spaces $C_k(\w_1,\IZ)$ and $C_k(\w_1,\IR)$ was first proved by Gul'ko \cite{Gul77}, \cite{Gul78} (see also \cite[5.35]{Gul98}). By \cite[16.12]{kak}, the space $C_k(\w_1,\IR)$ is not countably tight. By \cite[Theorem 3]{FKLS} (see also \cite[16.15]{kak}), the space $C_k(\w_1,\IR)$ has a (uniform) $\w^\w$-base if and only if $\w_1=\mathfrak b$.
\end{remark}

%According to a famous theorem of Arhangel'ski\u\i\ \cite[3.12.10]{Eng}, each first-countable compact Hausdorff space has cardinality $\le\mathfrak c$. We do not know if the same is true for spaces with a local $\w^\w$-base.

 %\begin{problem} Let $X$ be a compact Hausdorff space with a local $\w^\w$-base. Is $|X|\le\mathfrak c$?
%\end{problem}

\section{Countable $\css^*$-networks in spaces with an $\w^\w$-base}
\label{s:local}

In this section we prove an important Theorem~\ref{t:lP*} establishing the strong Pytkeev$^*$ property of topological spaces with an $\w^\w$-base. We recall that a topological space $X$ has the \index{strong Pytkeev$^*$ property}{\em strong Pytkeev$^*$ property} if $X$ has a countable $\css^*$-network at each point $x\in X$.

Given a family of sets $(U_\alpha)_{\alpha\in\w^\w}$ and a subset $A\subset\w^\w$ we put $U_A:=\bigcap_{\alpha\in A}U_\alpha$. In particular, for a finite sequence $\beta\in\w^{<\w}$ by $U_{{\uparrow}\beta}$ we denote the intersection $\bigcap_{\alpha\in{\uparrow}\beta}U_\alpha$. Here ${\uparrow}\beta:=\{\alpha\in\w^\w:\alpha|n=\beta\}$ for any finite sequence $\beta\in\w^n\subset\w^{<\w}$. Sometimes it will be convenient to denote the set $U_{{\uparrow}\beta}$ by $U_\beta$.

\begin{theorem}\label{t:lP*}  If $\{U_\alpha\}_{\alpha\in\w^\w}$ is a local $\w^\w$-base at a point $x$ of a topological space $X$, then for every $\alpha\in\w^\w$ and every sequence $(x_n)_{n\in\w}\in X^\w$ accumulating at $x$ there exists $k\in\w$ such that the set  $U_{\alpha|k}$ contains infinitely many points of the sequence $(x_n)_{n\in\w}$. Consequently, the family $\{U_\beta\}_{\beta\in\w^{<\w}}$ is a countable $\css^*$-network at $x$.
\end{theorem}

\begin{proof} Given any $\alpha\in\w^\w$ and any sequence $(x_n)_{n\in\w}\in X^\w$ accumulating at $x$, we need to find $k\in\w$ such that the set  $U_{\alpha|k}$ contains infinitely many points of the sequence $(x_n)_{n\in\w}$.

Let $\ddot x$ be the intersection of all neighborhoods of $x$. If the set $\Omega=\{n\in\w:x_n\notin \ddot x\}$ has infinite complement in $\w$, then the number $k=0$ has the required property as the set $U_{\alpha|0}=\ddot x$ contains each point $x_k$, $k\in\w\setminus \Omega$.

So, we assume that the set $\Omega$ has finite complement in $\w$. Let $\Tau_x(X)$ be the family of all neighborhoods of $x$ in $X$. For every $U\in\Tau_x(X)$ consider the infinite set $F(U)=\{n\in\Omega:x_n\in U\}$ and observe that $\F=\{F(U):U\in\Tau_x(X)\}$ is a free filter on $\Omega$. Moreover,  the family $\{F(U_\alpha)\}_{\alpha\in\w^\w}$ is a base of the filter $\F$ such that $F(U_\beta)\subset F(U_\alpha)$ for any $\alpha\le\beta$ in $\w^\w$.

Consider the filter $\F$ as a subset of the power-set $\mathcal P(\Omega)$, endowed with the natural compact metrizable topology. Observe that for every $\alpha\in\w^\w$ the set $K_\alpha=\{F\in\mathcal P(\Omega):F(U_\alpha)\subset F\}$ is a compact subset of $\mathcal F$. For any functions $\alpha,\beta\in\w^\w$ the inequality $\alpha\le\beta$ implies that $K_\alpha\subset K_\beta$. Taking into account that $\F=\bigcup_{\alpha\in\w^\w}K_\alpha$, we conclude that $(K_\alpha)_{\alpha\in\w^\w}$ is a compact resolution of the metrizable space $\F$.  By Lemma~\ref{l:analytic}, the space $\F$ is analytic and by Lemma~\ref{l:am}, the free filter $\F$ is meager in $\mathcal P(\Omega)$. By Lemma~\ref{l:Tal}, there exists a finite-to-one map $\varphi:\Omega\to\w$ such that for every set $F\in\F$ the image $\varphi(F)$ is has finite complement in $\w$.

  We claim that for some $k\in\w$ the set $U_{\alpha|k}$ contains infinitely many points $x_k$, $k\in\Omega$. To derive a contradiction, assume that for every $k\in\w$ the set $J_k=\{n\in\Omega:x_n\in U_{\alpha|k}\}$ is finite.   Then we can choose an increasing number sequence $(y_k)_{k\in\w}\in\w^\w$ such that  $J_k\cap\varphi^{-1}(y_k)=\emptyset$ for all $k\in\w$.

For every $k\in\w$ and every $i\in\varphi^{-1}(y_k)$ the point $x_i$ does not belong to the intersection $U_{\alpha|k}=\bigcap\{U_\beta:\beta\in\w^\w,\;\beta|k=\alpha|k\}$ and hence $x_i\notin U_{\beta_{k,i}}$ for some $\beta_{k,i}\in\w^\w$ with $\beta_{k,i}|k=\alpha|k$. %Let $E_k=\{x_i:i\in\varphi^{-1}(y_k)\}$.
Consider the function $\beta_k=\max\{\beta_{k,i}:i\in\varphi^{-1}(y_k)\}$ and observe that the inclusion $U_{\beta_k}\subset U_{\beta_{k,i}}$ for $i\in \varphi^{-1}(y_k)$ implies that $x_i\notin U_{\beta_k}$ for all $i\in\varphi^{-1}(y_k)$. It follows that $\beta_k|k=\alpha|k$ and hence $\beta_k(k-1)=\alpha(k-1)$ if $k\ge 1$. Let $\beta\in\w^\w$ be the function defined by $\beta(k)=\max\{\beta_i(k):i\le k+1\}$. We claim that $\beta\ge\beta_k$ for every $k\in\w$. Fix any number $n\in\w$. If $k\le n+1$, then $\beta(n)=\max\{\beta_i(n):i\le n+1\}\ge \beta_k(n)$. If $k>n+1$, then $\beta(n)\ge \beta_{n+1}(n)=\alpha(n)=\beta_k(n)$.
For every $k\in\w$ the inequality $\beta\ge\beta_k$ implies the inclusion $U_\beta\subset U_{\beta_k}$ and hence $x_i\notin U_\beta$ for all $i\in\varphi^{-1}(y_k)$. Then the set $F(U_\beta)$ is disjoint with $\bigcup_{k\in\w}\varphi^{-1}(y_k)$ and hence the image $\varphi(F(U_\beta))$ has infinite complement in $\w$. But this contradicts the choice of the finite-to-one function $\varphi$.
\end{proof}

Theorem~\ref{t:lP*} implies the following important result, which will allow us to apply powerful results on $\css^*$-netbases to studying topological spaces with an $\w^\w$-base.

\begin{theorem}\label{t:ww=>netbase} If $\{U_\alpha\}_{\alpha\in\w^\w}$ is an $\w^\w$-base for a  topological space $X$, then the pair $$(\E,\Bas):=\big(\{U_{{\uparrow}\beta}\}_{\beta\in\w^{<\w}},\{U_\alpha\}_{\alpha\in\w^\w}\big)$$ is a countable $\css^*$-netbase  for $X$. If the space $X$ is countably tight, then the pair $(\E,\Bas)$ is a countable $\as^*$-netbase for $X$.
\end{theorem}

Theorem~\ref{t:lP*} combined with Theorem~\ref{t:1=C*+fan} implies the following characterization of first-countable spaces.

\begin{corollary}\label{c:1-ww} For a topological space $X$ (and a point $x\in X$) the following conditions are equivalent:
\begin{enumerate}
\item $X$ is first-countable (at $x$);
\item $X$ is countably fan-tight and has a (neighborhood) $\w^\w$-base (at $x$).
\end{enumerate}
If the space $X$ is semi-regular (at $x$), then the conditions \textup{(1),(2)} are equivalent to
\begin{enumerate}
\item[(3)] $X$ is countably ofan-tight (at $x$) and has a (neighborhood) $\w^\w$-base (at $x$).
\end{enumerate}
\end{corollary}

%\section{Local $\w^\w$-bases satisfying the condition (D)}\label{s:D=T}

Following \cite{GabKakLei_1}, we say that a neighborhood $\w^\w$-base $(U_\alpha[x])_{\alpha\in\w^\w}$ at a point $x$ of a topological space $X$ satisfies \index{$\w^\w$-base!condition (D)}{\em the condition} (D), if for every $\alpha\in\w^\w$ the neighborhood $U_\alpha[x]$ is equal to the union $\bigcup_{k\in\w}U_{\alpha|k}[x]$ of the sets $U_{\alpha|k}[x]=\bigcap\{U_\beta[x]:\beta\in\w^\w,\;\beta|k=\alpha|k\}$.
The condition (D) appears in many results involving $\w^\w$-bases, see \cite[\S3]{GabKakLei_1}, \cite{GabKak_2}.

\begin{theorem}\label{t:wwD} For a topological space $X$ and a point $x\in X$ the following conditions are equivalent:
\begin{enumerate}
\item[\textup{(1)}] $X$ has a neighborhood $\w^\w$-base at $x$ satisfying the condition  (D);
\item[\textup{(2)}] $X$ is countably tight at $x$ and $X$ has a neighborhood $\w^\w$-base at $x$.
\end{enumerate}
\end{theorem}

\begin{proof} $(1)\Ra(2)$ Assuming that a local $\w^\w$-base $(U_\alpha)_{\alpha\in\w^\w}$ at $x$ satisfies the condition (D), we shall prove that the space $X$ is countably tight at $x$.

Given a subset $A\subset X$ with $x\in\bar A$, consider the subset $\Omega=\{\beta\in\w^{<\w}:A\cap U_{\beta}\ne \emptyset\}$. For every $\beta\in \Omega$ choose a point $x_\beta\in A\cap U_\beta$. We claim that the countable subset $B=\{x_\beta\}_{\beta\in\Omega}\subset A$ contains the point $x$ in its closure. In the opposite case the set $B$ is disjoint with some neighborhood $U_\alpha$ of $x$. The condition $(D)$ guarantees that $U_\alpha=\bigcup_{k\in\w}U_{\alpha|k}$. Since $x\in \bar A$, the set $A$ intersects the neighborhood $U_\alpha$ and hence intersects some set $U_{\alpha|k}$, $k\in\w$. Then $\alpha|k\in\Omega$ and hence $x_{\alpha|k}\in B\cap U_{\alpha|k}\subset B\cap U_\alpha$, which contradicts the choice of the neighborhood $U_\alpha$.
\smallskip

$(2)\Ra(1)$ Assume that $X$ is countably tight at $x$ and $(U_\alpha)_{\alpha\in\w^\w}$ is a neighborhood $\w^\w$-base at $x$. For every $\beta\in\w^{<\w}$ consider the intersection $U_\beta=\bigcap_{\alpha\in{\uparrow}\beta}U_\alpha$. The countable tightness of $X$ and Theorem~\ref{t:lP*} implies that for every $\alpha\in\w^\w$ the union $V_\alpha=\bigcup_{k\in\w}U_{\alpha|k}$ is a neighborhood of $x$. Using the monotonicity of the enumeration $(U_\alpha)_{\alpha\in\w^\w}$ we can show that $V_\beta\subset V_\alpha$ for all $\alpha\le\beta$ in $\w^\w$

It remains to check that the neighborhood $\w^\w$-base $(V_\alpha)_{\alpha\in\w^\w}$ satisfies the condition (D). Given any $\alpha\in\w^\w$ and $k\in\w$ observe that
$$V_{\alpha|k}= \bigcap_{\beta\in{\uparrow}(\alpha|k)}V_\beta=
\bigcap_{\beta\in{\uparrow}(\alpha|k)}\bigcup_{n\ge k}U_{\beta|n}\supset U_{\alpha|k}$$ and hence $$V_\alpha\supset\bigcup_{k\in\w}V_{\alpha|k}\supset \bigcup_{k\in\w}U_{\alpha|k}=V_\alpha,$$
which means that $V_\alpha=\bigcup_{k\in\w}V_{\alpha|k}$ and the neighborhood $\w^\w$-base $(V_\alpha)_{\alpha\in\w^\w}$ satisfies the property (D).
\end{proof}

\begin{remark} Theorem~\ref{t:wwD} answers affirmatively Question 5 in \cite{GabKakLei_1}.
\end{remark}

\section{Topological spaces with a locally quasi-uniform $\w^\w$-base}\label{s:lqu-ww}

In this section we study topological spaces possessing a locally quasi-uniform $\w^\w$-base. We recall that an $\w^\w$-base $\{U_\alpha\}_{\alpha\in \w^\w}$ for a topological space $X$ is \index{$\w^\w$-base!locally quasi-uniform}{\em locally quasi-uniform} if for any point $x\in X$ and neighborhood $O_x\subset X$ there exists $\alpha\in \w^\w$ such that $U_\alpha U_\alpha[x]\subset O_x$.

\begin{theorem}\label{t:lqu-ww} Each topological space $X$ with a locally quasi-uniform $\w^\w$-base has the following properties:
\begin{enumerate}
\item $X$ has a countable locally quasi-uniform $\css^*$-netbase;
\item each compact Hausdorff subspace of $X$ is metrizable;
\item $X$ is a $\mathfrak P^*$-space if $X$ is a strong $\sigma$-space;
\item $X$ is a $\mathfrak P_0$-space if and only if $X$ is cosmic;
\item $X$ is metrizable if $X$ is a first-countable strong $\sigma$-space;
\item $X$ is metrizable and separable if and only if $X$ is a first-countable cosmic space.
\end{enumerate}
\end{theorem}

\begin{proof} By Theorem~\ref{t:ww=>netbase}, the space $X$ has a countable locally quasi-uniform $\css^*$-netbase. Now the statements (2), (3), (4), (5), (6) follow from Theorem~\ref{t:comp-metr}(4), Theorem~\ref{t:s-sigma}, Corollary~\ref{c:cs-eq}, Theorem~\ref{t:metr}(3), and Theorem~\ref{t:metr-separ}(4), respectively.
\end{proof}

\begin{remark} By Proposition~\ref{p:trans-base} and Definition~\ref{d:trans-lqu}, each locally quasi-uniform portator $X$ with a neighborhood $\w^\w$-base at the unit has a locally quasi-uniform $\w^\w$-base and hence has the properties \textup{(1)--(6)} of Theorem~\ref{t:lqu-ww}. By Proposition~\ref{p:lqu-portator}, a portator $X$ is locally quasi-uniform if and only if its multiplication map $\mathbf{xy}$ is semicontinuous at each point of the set $X\times\{e\}$. By  Corollary~\ref{c:TA=>lqu}, the class of locally quasi-uniform portators includes all para-topological portators, in particular, all paratopological groups (which are quasi-uniform portators).
\end{remark}

\begin{remark} The metrizability of compact Hausdorff spaces with a uniform $\w^\w$-base was first proved by Cascales and Orihuela \cite{CO}. In \cite{DH} Dow and Hart generalized their result proving that a compact Hausdorff space $X$ is metrizable if its diagonal $\Delta_X$ can be written as the intersection $\bigcap_{\alpha\in\w^\w}U_\alpha$ of a family $\{U_\alpha\}_{\alpha\in\w^\w}$ of open sets in $X\times X$ such that $U_\beta\subset U_\alpha$ for all $\alpha\le\beta$ in $\w^\w$.
\end{remark}

\section{Topological spaces with a locally uniform $\w^\w$-base}

In this section we study topological spaces possessing a locally uniform $\w^\w$-base. We recall that an $\w^\w$-base $\{U_\alpha\}_{\alpha\in \w^\w}$ for a topological space $X$ is \index{$\w^\w$-base!locally uniform}{\em locally uniform} if for any point $x\in X$ and neighborhood $O_x\subset X$ there exists $\alpha\in \w^\w$ such that $U_\alpha^{\pm3}[x]\subset O_x$.

\begin{theorem}\label{t:1-luww} For a topological space $X$ with a locally uniform $\w^\w$-base (and a point $x\in X$) the following conditions are equivalent:
\begin{enumerate}
\item $X$ is first-countable (at $x$);
\item $X$ is ofan $\css$-tight (at $x$);
\item $X$ is a $q$-space (at $x$).
\end{enumerate}
\end{theorem}

\begin{proof} By Theorem~\ref{t:ww=>netbase}, the space $X$ has a countable locally uniform $\css^*$-netbase $(\E,\Bas)$, which is also a $\ccs^*$-netbase for $X$.

The implications $(2)\Leftarrow(1)\Ra(3)$ are trivial and $(3)\Ra(1)$ follow from Theorem~\ref{t:1-enet}.
To see that $(2)\Ra(1)$, assume that $X$ is ofan $\css$-tight at some point $x\in X$. By Proposition~\ref{p:lu=>regular}, the space $X$ is regular. Since $(\E,\Bas)$ is a $\css^*$-netbase  for $X$, the family $\{E[x]:E\in\E\}$ is a countable $\css^*$-network at $x$. By Theorem~\ref{t:1=C*+fan}(2), the space $X$ is first-countable at $x$.
\end{proof}

The following characterization of metrizable spaces can be derived from  Theorems~\ref{t:ww=>netbase} and \ref{t:metr}.

\begin{theorem}\label{t:metr-ww} For a topological space $X$ with a locally uniform $\w^\w$-base the following conditions are equivalent:
\begin{enumerate}
\item $X$ is metrizable;
\item $X$ is first-countable closed-$\bar G_\delta$ $T_0$-space;
\item $X$ is an $M$-space;
\item $X$ is first-countable collectionwise normal $\Sigma$-space;
\item $X$ is first-countable strong $\sigma$-space.
\end{enumerate}
\end{theorem}

\begin{remark} In Example~\ref{ex:michael} we shall prove that the Michael's line $\IR_\IQ$  has a (locally) uniform $\w^\w$-base so the closed-$\bar G_\delta$ requirement cannot be removed from Theorem~\ref{t:metr}(2) even for hereditarily paracompact spaces with countable set of non-isolated points.
\end{remark}

\begin{problem} Is a $T_0$-space $X$ metrizable if $X$ has a locally uniform $\w^\w$-base and each closed subset of $X$ is of type $G_\delta$?
\end{problem}

The following characterization of metrizable separable spaces can be derived from  Theorems~\ref{t:ww=>netbase} and \ref{t:metr-separ}.

\begin{theorem} For a topological space $X$ with a locally uniform $\w^\w$-base the following conditions are equivalent:
\begin{enumerate}
\item $X$ is metrizable and separable;
\item $X$ is a first-countable hereditarily Lindel\"of $T_0$-space;
\item $X$ is first-countable $\Sigma$-space with countable extent;
\item $X$ is first-countable cosmic space.
\end{enumerate}
\end{theorem}

The following characterization of ``small'' spaces can be derived from  Theorems~\ref{t:ww=>netbase}, \ref{t:aleph0}, and Corollary~\ref{c:cs-eq}.

\begin{theorem}\label{tc:small} For a topological space $X$ with a locally uniform $\w^\w$-base the following conditions are equivalent:
\begin{enumerate}
\item[\textup(1)] $X$ is a $\Sigma$-space with countable extent;
\item[\textup(2)] $X$ is cosmic;
\item[\textup(3)] $X$ is an $\aleph_0$-space;
\item[\textup(4)] $X$ is a $\mathfrak P_0$-space.
\end{enumerate}
\end{theorem}
Finally we establish some properties of topological spaces with a locally uniform $\w^\w$-base.

\begin{theorem}\label{t:prop-luww} Any topological space $X$ with a locally uniform $\w^\w$-base has the following properties:
\begin{enumerate}
\item $X$ is a $\Sigma$-space if and only if $X$ is a $\sigma$-space;
\item If $X$ is a $w\Delta$-space, then $X$ has a $G_\delta$-diagonal;
\item If $X$ is a collectionwise normal $\Sigma$-space, then $X$ is a paracompact $\mathfrak P^*$-space;
\item all countably compact subsets of $X$ are metrizable.
\end{enumerate}
\end{theorem}

\begin{proof} By Theorem~\ref{t:ww=>netbase}, $X$ has a countable locally uniform $\css^*$-netbase, which is also a $\ccs^*$-netbase. Now the statements (1), (2), (4) follow from Theorems~\ref{t:Sigma}, \ref{t:wD=>Gd}, \ref{t:metr}(5), respectively.
The statement (3) follows from Corollary~\ref{c:Sigma=>aleph} and Theorem~\ref{t:s-sigma}.
\end{proof}

 \begin{problem} Is each first-countable $\sigma$-space with a locally uniform $\w^\w$-base
metrizable?
\end{problem}

\section{Spaces with locally uniform $\w^\w$-bases under $\w_1<\mathfrak b$ and PFA}\label{s:PFA}

In this section we shall establish some consistent properties of topological spaces with a locally uniform $\w^\w$-base, which hold under additional set-theoretic assumptions like $\w_1<\mathfrak b$ or PFA  (the Proper Forcing Axiom). It is known \cite[\S7]{Moore} that PFA implies $\mathfrak b=\mathfrak c=\w_2$.

We recall that for a family $\Bas$ of entourages on a set $X$ by $\cov(X;\Bas)$ we denote the smallest cardinal $\kappa$ such that for every entourage $U\in\Bas$ there exists a set $D\subset X$ of cardinality $|D|\le\kappa$ such that $X=U[D]$. For $n\in\IN$ we put $\Bas^{\mp n}=\{U^{\mp n}:U\in\Bas\}$ and $\Bas^{-1}=\{U^{-1}:U\in\Bas\}$.

\begin{theorem}\label{t:w1<b} Under $\w_1<\mathfrak b$,  for any topological space $X$ with a locally uniform $\w^\w$-base $\Bas=\{U_\alpha\}_{\alpha\in\w^\w}$ the following conditions are equivalent:
\begin{enumerate}
\item[\textup{(1)}] $X^\w$ is weakly cosmic;
\item[\textup{(2)}] all finite powers of $X$ are weakly cosmic;
\item[\textup{(3)}] $X$ is weakly cosmic;
\item[\textup{(4)}] $X$ is separable;
\item[\textup{(5)}] $X$ is Lindel\"of;
\item[\textup{(6)}] $X$ has countable discrete cellularity;
\item[\textup{(7)}] $\cov(X;\Bas^{\mp n})\le \w$ for some $n\in\IN$;
\item[\textup{(8)}] $\cov(X;\Bas)\le \w$;
\item[\textup{(9)}] $\cov(X;\Bas^{-1})\le \w$.
\end{enumerate}
Under \index{PFA} PFA the conditions \textup{(1)--(9)} are equivalent to each of the following:
\begin{itemize}
\item[\textup{(10)}] $X$ is cosmic.
\item[\textup{(11)}] $X$ is an $\aleph_0$-space;
\item[\textup{(12)}] $X$ is a $\mathfrak P_0$-space;
\item[\textup{(13)}] $X$ contains a dense $\Sigma$-subspace with countable extent.
\end{itemize}
\end{theorem}

\begin{proof} First we prove the implications $(3)\Ra(4,5)\Ra(6)\Ra(7)\Ra(3)$ of which the implications $(3)\Ra(4,5)\Ra(6)$ are trivial or well-known.

$(6)\Ra(7)$ Assuming that $X$ has countable discrete cellularity, we shall prove that $\cov(X;\Bas^{\mp 4})\le \w$. Assuming that $\cov(X;\Bas^{\mp 4})>\w$, we can find a function $\delta\in\w^\w$ such that $X\ne U_\delta^{\mp4}[C]$ for any countable set $C\subset X$. This allows us to construct a transfinite sequence of points $(x_\alpha)_{\alpha\in\w_1}$ such that $x_\alpha\notin \bigcup_{\beta<\alpha}U_\delta^{\mp 4}[x_\beta]$ for every $\alpha<\w_1$. We claim that the family  $(U_\delta[x_\alpha])_{\alpha\in\w_1}$ is discrete in $X$. Assuming the opposite, we could find a point $z\in X$ and countable ordinals $\alpha<\beta$ such that $U_\delta[z]\cap U_\delta[x_\alpha]\ne\emptyset \ne U_\delta[z]\cap U_\delta[x_\beta]$. Then $x_\beta\in U_\delta^{\mp2}[z]\subset U_\delta^{\mp2}[U_\delta^{\mp2}[x_\alpha]]=U_\delta^{\mp4}[x_\alpha]$, which contradicts the choice of the point $x_\beta$. This contradiction shows that $\cov(X;\Bas^{\mp 4})\le \w$.
\smallskip

$(7)\Ra(3)$ Assume that $\cov(X;\Bas^{\mp n})\le \w$ for some $n\in\w$. Replacing $n$ by a larger number, if necessary, we can assume that $n$ is even. To derive a contradiction, assume that the space $X$ is not weakly cosmic and hence contains an uncountable weakly separated subset $\{x_i\}_{i\in\w_1}$. Then each point $x_i$ has a neighborhood $O(x_i)\subset X$ such that for any countable ordinals $i<j$ either $x_i\notin O(x_j)$ or $x_j\notin O(x_i)$. By Proposition~\ref{p:p-lqu+qu}, for each $i\in\w_1$ there is an index $\alpha_i\in\w^\w$ such that $U_{\alpha_i}^{\mp 2n}[x]\subset O(x_i)$.

Since $\w_1<\mathfrak b$, the set $\{\alpha_i:i\in\w_1\}\subset\w^\w$ is dominated by a countable subset $B\subset \w^\w$. By the Pigeonhole Principle, for some $\beta\in B$ the set $\Omega=\{i\in\w_1:\alpha_i\le\beta\}$ is uncountable.

Since $\cov(X;\Bas^{\mp n})\le\w$, there exists a countable subset $Z\subset X$ such that $U_\beta^{\mp n}[Z]=X$. By the Pigeonhole Principle, there are two ordinals $i<j$ in $\Omega$ such that $x_i,x_j\in U_\beta^{\mp n}[z]$ for some point $z\in Z$. Then $z\in (U_\beta^{\mp n})^{-1}[x_i]=U_\beta^{\mp n}[x_i]$ and hence $x_j\in U_\beta^{\mp n}[z]\subset U_\beta^{\mp n}U_{\beta}^{\mp n}[x_i]=U_\beta^{\mp 2n}[x_i]\subset O(x_i)$. By analogy we can prove that $x_j\in O(x_i)$. But this contradicts the choice of the neighborhoods $O(x_i)$, $O(x_j)$. This contradiction shows that the space $X$ is weakly cosmic and completes the proof of the equivalence of the conditions $(3)$--$(7)$.
\smallskip

The implications $(5)\Ra(8)\Ra(7)$ and $(4)\Ra(9)\Ra(7)$ are trivial.

The implication $(2)\Ra(3)$ is trivial and $(3)\Ra(2)$ follows from the equivalence $(3)\Leftrightarrow(7)$ and the observation that for every $n\in\w$ the countability of the covering number $\cov(X;\Bas^{\mp n})$ is preserved by finite products of based topological spaces.

The implication $(1)\Ra(2)$ is trivial and $(2)\Ra(1)$ easily follows from the Pigeonhole Principle.

Now assuming that PFA holds, we prove that the conditions $(1)$--$(9)$ are equivalent to the conditions $(10)$--$(13)$. The equivalence of the conditions (10)--(12) was proved in Theorem~\ref{tc:small} and the implications $(10)\Ra(4)\Ra(13)$ are trivial. It is known \cite[\S7]{Moore} that PFA implies $\w_1<\mathfrak b=\w_2$. By Theorem~\ref{t:TodoPFA}, under PFA each regular space with weakly cosmic finite powers is cosmic. This yields the implication $(2)\Ra(10)$ and hence the equivalence of the conditions $(1)$--$(12)$. It remains to prove that $(13)\Ra(4)$. Assuming that $X$ contains a dense $\Sigma$-subspace $Z$ with countable extent, we can apply Theorem~\ref{tc:small} to conclude that the space $Z$ is cosmic and hence separable. Then $X$ is separable, too.
 \end{proof}

\begin{remark} Theorem~\ref{t:w1<b} cannot be proved in ZFC. By Corollary~\ref{c:Ck2}, the function space $C_k(\w_1,\IR)$ is Lindel\"of, has uncountable cellularity, is not a $\Sigma$-space, and has a uniform $\w^\w$-base under $\w_1=\mathfrak b$. This example shows that in ZFC the condition (5)--(8) are not equivalent to the other conditions of Theorem~\ref{t:w1<b}. Nonetheless, in  Theorem~\ref{t:small} we shall prove that  the conditions $(1)$--$(4)$ and $(10)$--$(13)$ are equivalent for topological spaces with a uniform $\w^\w$-base.
\end{remark}

\begin{remark}  By Propositions~\ref{p:trans-base} and \ref{p:sym+lqu=lu},
a portator $X$ has a locally uniform $\w^\w$-base if $X$ has a neighborhood $\w^\w$-base at the unit $e$ and $X$ is either locally uniform or symmetrizable and locally quasi-uniform. Consequently (see Propositions~\ref{p:lu-portator}, \ref{p:lqu-portator}, \ref{p:sym-portator} and Definition~\ref{d:trans-algebra}), a portator $X$ has a locally uniform $\w^\w$-base if $X$ has a neighborhood $\w^\w$-base at the unit $e$ and $X$ is paradiv-topological or invpara-topological. By Definitions~\ref{d:trans-algebra} and \ref{d:lops}, the class of invpara-topological portators includes all topological lops. Since each rectifiable space is homeomorphic to a topological lop, all results of Sections~\ref{s:lqu-ww}--\ref{s:PFA} are true for rectifiable spaces with an $\w^\w$-base.
\end{remark}

\section{Topological spaces with a uniform $\w^\w$-base}\label{s:ub}

In this section we start a more detail study of topological spaces with a uniform $\w^\w$-base. We recall that an entourage base $\{U_\alpha\}_{\alpha\in P}$ for a topological space $X$ is  \index{$P$-base!uniform}\index{$\w^\w$-base!uniform}{\em uniform} if for any $\alpha\in P$ there exists $\beta\in P$ such that $U_\beta^{\pm3}\subset U_\alpha$. In this case the canonical preuniformity $\U_X$ generated by the base $\{U_\alpha\}_{\alpha\in P}$ coincides with the canonical uniformity $\U_X^{\pm\w}$ of the preuniform space $(X,\U_X)$. %In this chapter we shall study topological spaces with a uniform $\w^\w$-base using the space $C_u(X)$ of all uniformly continuous real-valued functions on the uniform space $(X,\U_X)=(X,\U_X^{\pm\w})$.
Since each uniform $\w^\w$-base is locally (quasi-)uniform, all results proved in Sections~\ref{s:lqu-ww}~--~\ref{s:PFA} remain true also topological spaces with a uniform $\w^\w$-base.

Propositions~\ref{p:p-lqu+qu}(4) and \ref{p:cu-Pb} imply the following characterization.

\begin{proposition}\label{p:uP<=>liuP} For a directed poset $P\cong P^\w$, a topological space $X$ has a uniform $P$-base if and only if $X$ has a locally $\infty$-uniform $P$-base.
\end{proposition}

An entourage base $\Bas$ for a topological space $X$ is called \index{base!$\IR$-regular}{\em $\IR$-regular} is so is the based space $(X,\Bas)$. This happens if and only if the canonical map $\delta:X\to\IR^{C_u(X)}$, $\delta\colon x\mapsto\delta_x\colon f\mapsto f(x)$, is a topological embedding. Here by $C_u(X)$ we denote the space of functions $f:X\to \IR$ such that for every $\e>0$ there exist an entourage $B\in\Bas$ such that $|f(x)-f(y)|<\e$ for all $(x,y)\in B$.

Proposition~\ref{p:uP<=>liuP} and Corollary~\ref{c:Rreg} imply the following characterization.

\begin{corollary} For a poset $P\cong P^\w$ and a topological space $X$ the following conditions are equivalent:
\begin{enumerate}
\item[\textup(1)] $X$ has an $\IR$-regular $P$-base;
\item[\textup(2)] $X$ is a $T_0$-space with a locally $\infty$-uniform $P$-base;
\item[\textup(3)] $X$ is a $T_0$-space with a uniform $P$-base.
\end{enumerate}
\end{corollary}

A map $f:X\to Y$ between uniform spaces is called \index{map!uniformly quotient}{\em uniformly quotient} if a pseudometric $d$ on $Y$ is uniform if and only if the pseudometric $d(f\times f)$ on $X$ is uniform. This definition implies that each uniformly quotient map $f:X\to Y$ is uniformly continuous. The following proposition is proved in \cite{BL-LG} with the help of free topological Abelian groups.

\begin{proposition}\label{p:uniquot} Let $f:X\to Y$ be a uniformly quotient map of uniform spaces. If the uniformity $\U_X$ of $X$ has an $\w^\w$-base, then the uniformity $\U_Y$ has an $\w^\w$-base, too.
\end{proposition}

Let $P$ be a poset. A uniform space $X$ is defined to be \index{uniform space!$P$-based}{\em $P$-based} if the uniformity $\U_X$ of $X$ has a base $\Bas_X\subset \U_X$ which can be written as $\Bas_X=\{U_\alpha\}_{\alpha\in P}$ so that $U_\beta\subset U_\alpha$ for all $\alpha\le\beta$ in $P$. In this case the family $\{U_\alpha\}_{\alpha\in P}$ is called a \index{uniform space!$P$-base of}{\em $P$-base} of the uniformity $\U_X$.

\begin{proposition}\label{p:completion} For any directed poset $P$, the completion $\bar X$ of a $P$-based uniform space $X$ is a $P$-based uniform space.
\end{proposition}

\begin{proof} Fix any $P$-base $\{U_\alpha\}_{\alpha\in P}$ of the uniformity of $X$ and for every $\alpha\in P$ consider the closure $\bar U_\alpha$ of the entourage $U_\alpha$ in $\bar X\times\bar X$. It is easy to see that $(\bar U_\alpha)_{\alpha\in P}$ is an $P$-base of the uniformity of $\bar X$, witnessing that the uniform space $\bar X$ is $P$-based.
\end{proof}

\begin{proposition}\label{p:ww=>pc-ms} If uniform space $(X,\U_X)$ has an $\w^\w$-base, then each precompact subset of $X$ is metrizable and separable.
\end{proposition}

\begin{proof} By Proposition~\ref{p:completion}, the completion $\bar X$ of the uniform space $(X,\U_X)$ is $\w^\w$-based. By \cite[8.3.17]{Eng}, each precompact subset $B\subset X$ has compact closure $\bar B$ in $\bar X$. By Theorem~\ref{t:prop-luww}(4), the compact subset $\bar B$ of the $\w^\w$-based uniform space $\bar X$ is metrizable and separable. Consequently, the subset $B\subset \bar B$ is metrizable and separable, too.
\end{proof}

\section{Function spaces on topological spaces with a uniform $\w^\w$-base}

For a based topological space $(X,\Bas_X)$ by \index{$C_u(X)$}$C_u(X)$ we denote the subspace of $\IR^X$ consisting of uniformly continuous functions. A function $f:X\to \IR$ is called \index{function!uniformly continuous}{\em uniformly continuous} if for every $\e>0$ there exists an entourage $B\in\Bas_X$ such that $|f(x)-f(y)|<\e$ for all $(x,y)\in B$.

A subset $B$ of a based topological space $(X,\Bas_X)$ is called
\begin{itemize}
\item \index{subset!functionally bounded}
 {\em functionally bounded} if for each uniformly continuous function $f:X\to \IR$ the set $f[B]$ is bounded in the real line;
 \item \index{subset!$\sigma$-bounded} {\em $\sigma$-bounded} if $B$ is the union of a countable family of functionally bounded sets in $X$.
 \end{itemize}

In this section we are interested in properties of the function spaces $C_u(X)$ related to the $K$-analycity and compact resolutions (see \cite{kak} or Subsection~\ref{ss:Ka} for more information).

\begin{definition}\label{d:spec-res} Let $(X,\Bas_X)$ be a based topological space. A compact resolution $(K_\alpha)_{\alpha\in\w^\w}$ of the function space $C_u(X)$ is called a \index{special compact resolution}\index{compact resolution!special}{\em special compact resolution} if the following conditions are satisfied:
\begin{enumerate}
\item[(1)] $(K_\alpha)_{\alpha\in\w^\w}$ is \index{resolution!upper semicontinuous}{\em upper semicontinuous}, which means that for every open subset $U\subset X$ the set $\{\alpha\in\w^\w:K_\alpha\subset U\}$ is open in $\w^\w$;
\item[(2)] for every $\alpha\in \w^\w$ the function $\hat K_\alpha:X\to\IR$, $\hat K_\alpha:x\mapsto\sup\{|f(x)|:f\in K_\alpha\}$, is uniformly continuous;
\item[(3)] for every $\alpha\in\w^\w$ there exists a sequence of entourages $(U_n)_{n\in\w}\in \Bas_X^\w$ such that
$$K_\alpha=\big\{f\in \IR^X:|f|\le\hat K_\alpha\big\}\cap\bigcap_{n\in\w}\bigcap_{(x,y)\in U_n}\big\{f\in\IR^X:|f(x)-f(y)|\le\tfrac1{2^n}\big\};$$
\item[(4)] for every pair $((U_n)_{n\in\w},\varphi)\in \Bas_X^\w\times C_u(X)$ there exists $\alpha\in\w^\w$ such that $|\varphi|\le\hat K_\alpha$ and
$$\big\{f\in \IR^X:|f|\le |\varphi|\big\}\cap\bigcap_{n\in\w}\bigcap_{(x,y)\in U_n}\big\{f\in\IR^X:|f(x)-f(y)|\le\tfrac1{2^n}\big\}\subset K_\alpha.$$
\end{enumerate}
\end{definition}

In the following theorem the function space $C_u(X)$ is considered as a poset endowed with the partial order inherited from the product $\IR^X$.

\begin{theorem}\label{t:func-eq}  For an $\w^\w$-based topological space $(X,\Bas_X)$ the  following conditions are equivalent:
\begin{enumerate}
\item[\textup{(1)}] the poset $C_u(X)$ is $\w^\w$-dominated;
\item[\textup{(2)}] for some dense subspace $Z\subset X$ the set $\{f|Z:f\in C_u(X)\}$ is $\w^\w$-dominated in $\IR^Z$;
\item[\textup{(3)}] the set $C_u(X)$ is $\w^\w$-dominated in the poset $\IR^X$;
\item[\textup{(4)}] the space $C_u(X)$ has a special compact resolution;
\item[\textup{(5)}] the space $C_u(X)$ is $K$-analytic;
\item[\textup{(6)}] the space $C_u(X)$ has a compact resolution.
\end{enumerate}
The equivalent conditions \textup{(1)--(6)} follow from the condition
\begin{enumerate}
\item[\textup{(7)}] $X$ contains a dense $\sigma$-bounded subset.
\end{enumerate}
%If the preuniform space $X$ is $\IR$-regular and $X$ is either $\IR$-complete or separable, then the conditions \textup{(1)--(6)} are equivalent to:
%\begin{enumerate}
%\item[(7)] $X$ is cosmic.
%\end{enumerate}
\end{theorem}

\begin{proof} We shall prove the implications $(4)\Ra(1)\Ra(2)\Ra(3)\Ra(4)\Ra(5)\Ra(6)\Ra(2)\Leftarrow(7)$. In fact, the implications $(4)\Ra(1)\Ra(2)$, $(4)\Ra(5)$ are trivial and $(5)\Ra(6)$ follows from Proposition 3.10(i) in \cite{kak}. Let $\Bas_X=\{U_\alpha\}_{\alpha\in\w^\w}$ be the $\w^\w$-base of the $\w^\w$-based topological space $(X,\Bas_X)$.
\smallskip

$(2)\Ra(3)$ Assume that for some dense subspace $Z\subset X$ the set $C_u(X|Z)=\{f|Z:f\in C_u(X)\}$  is $\w^\w$-dominated in $\IR^Z$. Then there exists a monotone map $f:\w^\w\to \IR^Z$ whose image $D=f[\w^\w]$ dominates the set $C_u(X|Z)$ in $\IR^Z$.  First we establish the reduction $\Bas_X^\w\times D\succcurlyeq C_u(X)$. For every pair $P=\big((U_n)_{n\in\w},\varphi\big)\in \Bas_X^\w\times D$ consider the compact subset
$$K_P=\bigcap_{x\in Z}\big\{f\in \IR^X:|f(x)|\le\max\{0,\varphi(x)\}\big\}\cap\bigcap_{n\in\w}\bigcap_{(x,y)\in U_n}\big\{f\in \IR^X:|f(x)-f(y)|\le \tfrac1{2^n}\big\}$$of $\IR^X$ and observe that $K_P\subset C_u(X)$.
%We claim that the set $K_P$ is compact. It suffices to check that for every $x\in X$ the set $K_P(x)=\{f(x):f\in K_P\}$ is bounded in the real line. Choose any point $z\in Z\cap U_1[x]$ and observe that for every $f\in K_P$ we get $|f(x)|\le |f(z)|+|f(x)-f(z)|\le |\varphi(z)|+1$ and hence $K_P(x)\subset [-1-|\varphi(z)|,|\varphi(z)|+1]$.
Moreover, the function $\hat K_P:X\to \IR$, $\hat K_P:x\mapsto \sup\{f(x):x\in K_P\}$ is uniformly continuous and hence belongs to $C_u(X)$.

Taking into account that $\w^\w\succcurlyeq \Bas_X$ and $\w^\w\succcurlyeq D$, we can find a monotone cofinal map $p:\w^\w\to \Bas_X^\w\times D$. Then $(K_{p(\alpha)})_{\alpha\in\w^\w}$ is a compact resolution of $C_u(X)$ and $\{\hat K_{p(\alpha)}\}_{\alpha\in\w^\w}$ is a monotone cofinal family in $C_u(X)$ witnessing that the poset $C_u(X)$ is $\w^\w$-dominated.
\smallskip

$(3)\Ra(4)$ Assume that $C_u(X)$ is $\w^\w$-dominated in $\IR^X$ and find a monotone map $\varphi_*:\w^\w\to \IR^X$ whose image $D=\varphi_*[\w^\w]$ dominates the set $C_u(X)$ in $\IR^X$. For every $\alpha\in\w^\w$ it will be convenient to denote the function $\varphi_*(\alpha)\in\IR^X$ by $\varphi_\alpha$.
For a finite sequence $\beta\in\w^{<\w}$ let $\hat \varphi_\beta:X\to(-\infty,+\infty]$ be the function defined by $$\hat\varphi_\beta(x)=\sup(\{0\}\cup\{\varphi_{\alpha}(x):\alpha\in{\uparrow}\beta\})\mbox{ \  for $x\in X$},$$ where ${\uparrow}\beta=\{\alpha\in\w^\w:\exists n\in\w\;\;(\alpha|n=\beta)\}\subset\w^\w$ is a basic open set in $\w^\w$.

Observe that for every $\alpha\in\w^\w$ and $k\le n$ in $\w$ we get ${\uparrow}(\alpha|n)\subset {\uparrow}(\alpha|k)$ and hence $\varphi_\alpha\le \hat \varphi_{\alpha|n}\le \hat\varphi_{\alpha|k}$.
For an infinite sequence $\alpha\in\w^\w$ let $\hat\varphi:X\to(-\infty,+\infty]$ be the function defined by $\hat\varphi_\alpha(x)=\inf_{k\in\w}\varphi_{\alpha|k}(x)$.

Lemma~\ref{l:loc-bound} implies that for every $x\in X$ and $\alpha\in\w^\w$ the value $\hat\varphi_\alpha(x)$ is finite. Therefore, $\hat\varphi_\alpha\in\IR^X$ is a well-defined function such that $\hat\varphi\ge\varphi$. Moreover, for every $\alpha\le\beta$ in $\w^\w$ we get the inequality $\hat\varphi_\alpha\le\hat\varphi_\beta$.

Now we ready to prove that the space $C_u(X)$ has a special compact resolution. For every pair $P=\big((U_n)_{n\in\w},\alpha\big)\in \Bas_X^\w\times \w^\w$ consider the compact subset
$$K_P=\bigcap_{x\in X}\{f\in \IR^X:|f(x)|\le\hat\varphi_\alpha(x)\}\cap\bigcap_{n\in\w}\bigcap_{(x,y)\in U_n}\{f\in \IR^X:|f(x)-f(y)|\le \tfrac1{2^n}\big\}\subset \IR^X$$and the function
$$\hat K_P:X\to\IR,\;\;\hat K_P:x\mapsto\sup\{f(x):f\in K_P\}.$$
Observe that $\hat K_P\in K_P\subset C_u(X)$ and
$$K_P=\bigcap_{x\in X}\{f\in \IR^X:|f(x)|\le\hat K_P(x)\}\cap\bigcap_{n\in\w}\bigcap_{(x,y)\in U_n}\{f\in \IR^X:|f(x)-f(y)|\le \tfrac1{2^n}\big\}\subset \IR^X.$$

Fix any monotone cofinal map $p:\w^\w\to \Bas_X^\w\times\w^\w$ and observe that the family $(K_{p(\alpha)})_{\alpha\in\w^\w}$ is a compact resolution of $C_u(X)$.
To show that this compact resolution is special, it suffices to show that it is upper semicontinuous. Given $\alpha\in\w$ and an open neighborhood $W\subset C_u(X)$ of $K_\alpha$ we need to find $k\in\w$ such that $K_{p(\beta)}\subset W$ for all $\beta\in{\uparrow}(\alpha|k)$. We lose no generality assuming that $W$ is of subbasic form $W=\{f\in C_u(X):|f(x)|<a\}$ for some $x\in X$ and some $a>0$. Taking into account that the set $K_{p(\alpha)}$ contains the constant function $\IR\to \{\hat\varphi_\alpha(x)\}$, we conclude that $\hat\varphi_\alpha(x)<a$ and hence $\hat\varphi_{\alpha|k}(x)<a$ for some $k\in\w$. Then for any $\beta\in{\uparrow}(\alpha|k)$ we get $\hat\varphi_\beta(x)\le \hat\varphi_{\beta|k}(x)=\hat\varphi_{\alpha|k}(x)<a$ and hence $K_{p(\beta)}\subset W$.
Therefore, $(K_{p(\alpha)})_{\alpha\in\w^\w}$ is an upper semicontinuous (and hence special) compact resolution of $C_u(X)$.
\smallskip

$(6)\Ra(2)$ Assume that the function space $C_u(X)$ has compact resolution $(K_\alpha)_{\alpha\in\w^\w}$. For every $\alpha\in\w^\w$ consider the function $\varphi_\alpha:X\to\IR$, $\varphi_\alpha:x\mapsto \sup\{f(x):f\in K_\alpha\}$, and observe that $\varphi_\alpha\le \varphi_\beta$ for any $\alpha\le\beta$ in $\w^\w$. Since $C_u(X)=\bigcup_{\alpha\in\w^\w}K_\alpha$, the set $\{\varphi_\alpha\}_{\alpha\in\w^\w}$ dominates $C_u(X)$ in $\IR^X$.
\smallskip

$(7)\Ra(2)$ Assume that the based topological space $(X,\Bas_X)$ contains a dense $\sigma$-bounded subset $Z$. Then $Z=\bigcup_{n\in\w}Z_n$ for some functionally bounded sets $Z_n$ in $X$. For every $\alpha\in\w^\w$ consider the function $\varphi_\alpha\in\IR^Z$ assigning to each $z\in Z$ the number $\alpha(n)$ where $n\in\w$ is the unique number such that $z\in\bigcup_{k\le n}Z_k\setminus\bigcup_{k<n}Z_k$. It is easy to see that the correspondence $\varphi_*:\w^\w\to\IR^Z$, $\alpha\mapsto\varphi_\alpha$, is monotone. The functional boundedness of the sets $Z_n$ in $X$ guarantees that for every $f\in C_u(X)$ there is a function $\alpha\in\w^\w$ such that  $\alpha(n)\ge\sup_{x\in Z_n}|f(x)|$ for all $n\in\w$. For this function $\alpha$ we get $f|Z\le\varphi_\alpha$, which means that the set $\{f|Z:f\in C_u(X)\}$ is $\w^\w$-dominated in $\IR^Z$.
\end{proof}

\begin{remark} For the universal uniformity $\U_X$ on a Tychonoff space $X$ the function space $C_u(X)$ coincides with $C_p(X)$. In this case the equivalences (1)$\Leftrightarrow$(5)$\Leftrightarrow$(6) in Theorem~\ref{t:func-eq} were proved by Tkachuk \cite{Tk}.
\end{remark}

\section{Characterizing ``small'' topological spaces with a uniform $\w^\w$-base}\label{s:ww-small}

We recall that  a base $\Bas_X$ for a topological space $X$ is  \index{base!$\IR$-regular}{\em $\IR$-regular} (resp. {\em $\IR$-complete}) if the canonical map $\delta:X\to \IR^{C_u(X)}$ is a (closed) topological embedding. For a based topological space $(X,\Bas_X)$ by $\U_X$ we denote the preuniformity generated by the base $\Bas_X$ and by $\U_X^{\pm\w}$ the canonical uniformity of the preuniform space $(X,\U_X)$, see Section~\ref{s:cu}. A based preuniform space $(X,\Bas_X)$ will be called \index{based topological space!$\w$-narrow}{\em $\w$-narrow} if for every uniformly continuous map $f:X\to Y$ to a metric space $Y$ the image $f(X)$ is separable. This happens if and only if the canonical uniformity $\U_X^{\pm\w}$ of $X$ is $\w$-narrow in the sense that $\cov(X;\U_X^{\pm\w})\le\w$.

The following theorem collecting many equivalent smallness properties of topological spaces with an $\IR$-regular $\w^\w$-base is the main result of this section and of this paper-book, too.

\begin{theorem}\label{t:small} For a $T_0$-space $X$ with a uniform $\w^\w$-base $\Bas_X=\{U_\alpha\}_{\alpha\in\w^\w}$ the following conditions are equivalent:
\begin{enumerate}
%\item[\textup{(1)}] $X$ is a $\Sigma$-space with countable extent;
\item[\textup{(1)}] $X$ is cosmic;
\item[\textup{(2)}] $X$ is an $\aleph_0$-space;
\item[\textup{(3)}] $X$ is a $\mathfrak P_0$-space.
\item[\textup{(4)}] $X$  contains a dense $\Sigma$-subspace with countable extent;
\item[\textup{(5)}] $X$ is separable;
\item[\textup{(6)}] $C_u(X)$ is analytic;
\item[\textup{(7)}] $C_u(X)$ is cosmic.
\end{enumerate}
The conditions $(1)$--$(7)$ imply the condition:
\begin{enumerate}
\item[\textup{(8)}] $X$ contains a dense $\sigma$-bounded subset;
\end{enumerate}
The condition $(8)$ implies the equivalent conditions:
\begin{enumerate}
\item[\textup{(9)}] the poset $C_u(X)$ is $\w^\w$-dominated;
\item[\textup{(10)}] $C_u(X)$ is $\w^\w$-dominated in $\IR^X$;
\item[\textup{(11)}] for some dense subspace $Z\subset X$ the set $\{f|Z:f\in C_u(X)\}$ is $\w^\w$-dominated in $\IR^Z$;
\item[\textup{(12)}] $C_u(X)$ has a special compact resolution;
\item[\textup{(13)}] $C_u(X)$ is $K$-analytic;
\item[\textup{(14)}] $C_u(X)$ has a compact resolution.
\end{enumerate}
If the completion of $X$ by the canonical uniformity $\U_X^{\pm\w}$ is  $\IR$-complete, then the conditions  $(1)$--$(14)$ are equivalent.
\smallskip

Under $\w_1<\mathfrak b$ the conditions $(1)$--$(7)$ are equivalent to:
\begin{enumerate}
\item[\textup{(15)}] $X$ is Lindel\"of;
\item[\textup{(16)}] $X$ has countable discrete cellularity;
\item[\textup{(17)}] the based space $(X,\Bas_X)$ is $\w$-narrow.
\end{enumerate}
%Under $\w_1=\mathfrak b$ the conditions \textup{(1)--(8)} are not equivalent to $(16)$.
\end{theorem}

\begin{proof} By Theorem~\ref{tc:small}, the conditions $(1)$--$(3)$ are equivalent.
Next, we prove that $(1)\Ra(4)\Leftrightarrow(5)\Ra(6)\Ra(7)\Ra(1)$. The implications $(1)\Ra(4)\Leftarrow(5)$ and $(6)\Ra(7)$ are trivial and $(4)\Ra(5)$ follows from Theorem~\ref{tc:small}.

To prove that $(5)\Ra(6)$, assume that the space $X$ is separable and hence contains a countable dense subset $Z$. Then the subset $C_u(X|Z)=\{f|Z:f\in C_u(X)\}$ is $\w^\w$-dominated in $\IR^Z$ and by Theorem~\ref{t:func-eq}, the function space $C_u(X)$ is $K$-analytic. The injective restriction operator $C_u(X)\to \IR^Z$, $f\mapsto f|Z$, witnesses that the $K$-analytic space $C_u(X)$ is submetrizable and hence is analytic, see Lemma~\ref{l:analytic}.

 Finally, we prove the implication $(7)\Ra(1)$.  By the $\IR$-regularity of the based space $(X,\Bas_X)$, the canonical map $\delta:X\to \IR^{C_u(X)}$ is a topological embedding. If the function space $C_u(X)$ is cosmic, then by \cite[I.1.3]{Arch}, the function space $C_p(C_u(X))$ is cosmic too and so is its subspace $\delta(X)$, which is homeomorphic to $X$. This completes the proof of the equivalence of the conditions $(1)$--$(7)$.
\smallskip

The implication $(5)\Ra(8)$ is trivial and $(8)\Ra(9)\Leftrightarrow(10)\Leftrightarrow(11)\Leftrightarrow(12)\Leftrightarrow(13)
\Leftrightarrow(14)$ were proved in Theorem~\ref{t:func-eq}.
\smallskip

Now assume that the completion $\bar X$ of the space $X$ by the canonical uniformity $\U_X^{\pm\w}=\U_X$ is $\IR$-complete.
By Proposition~\ref{p:completion} the complete uniform space $\bar X$ has an $\w^\w$-base $\bar \Bas_X$. To prove that the statements $(1)$--$(14)$ are equivalent, it suffices to prove that $(11)\Ra(1)$.
Assuming that $(11)$ holds and applying the implication $(11)\Ra(12)$ to the $\IR$-complete $\w^\w$-based space $(\bar X,\bar \Bas_X)$, we conclude that the function space $C_{u}(\bar X)$ has a special compact resolution $(K_\alpha)_{\alpha\in\w^\w}$. This means that for every $\alpha\in\w^\w$ there exists a function $\varphi_\alpha\in\ K_\alpha$ and a sequence of entourages $(U_n)_{n\in\w}\in \bar \Bas_X^\w$ such that $$K_\alpha=\bigcap_{x\in \bar X}\{f\in \IR^{\bar X}:|f(x)|\le\varphi_\alpha(x)\}\cap\bigcap_{n\in\w}\bigcap_{x,y\in U_n}\{f\in\IR^{\bar X}:|f(x)-f(y)|\le \tfrac1{2^n}\}.$$
It follows that for every $\alpha\in\w^\w$ and $x\in \bar X$ we get the equality $\varphi_\alpha(x)=\sup\{f(x):x\in K_\alpha\}$.
For every $\beta\in\w^{<\w}$ and $x\in \bar X$ consider the (finite or infinite) number
$$\varphi_\beta(x)=\sup\{\varphi_\alpha(x):\alpha\in{\uparrow}\beta\}\in(-\infty,+\infty].$$ Lemma~\ref{l:loc-bound} implies that for every $\alpha\in\w^\w$ there is a number $k\in\w$ such that $\varphi_{\alpha|k}(x)$ is finite. Moreover, the upper semicontinuity of the compact resolution $(K_\alpha)_{\alpha\in\w^\w}$ implies that
$$\varphi_\alpha(x)=\inf_{n\in\w}\varphi_{\alpha|n}(x)=\lim_{n\to\infty}\varphi_{\alpha|n}(x).$$

Let $[-\infty,+\infty]$ be the standard two-point compactification of the real line and let $\delta:\bar X\to[-\infty,+\infty]^{C_{u}(\bar X)}$ be the map assigning to each point $x\in \bar X$ the Dirac measure $\delta_x:C_{u}(\bar X)\to\IR$, $\delta_x:\varphi\mapsto\varphi(x)$. The $\IR$-completeness of $\bar X$ guarantees that $\delta(\bar X)$ is a closed subset of $\IR^{C_{u}(\bar X)}\subset [-\infty,+\infty]^{C_{u}(\bar X)}$, so $\overline{\delta(\bar X)}\cap \IR^{C_{u}(\bar X)}=\delta(\bar X)$.

Let every $n\in\w$ and $\beta\in\w^{<\w}$ consider the compact subset
$F_{n,\beta}=\bigcap_{\alpha\in{\uparrow}\beta}\{\mu\in\overline{\delta(\bar X)}:\mu(\varphi_\alpha)\in[-n,n]\}$ in the compact Hausdorff space $\overline{\delta(\bar X)}\subset  [-\infty,+\infty]^{C_{u}(\bar X)}$. We claim that the countable family $\{F_{n,\beta}:\beta
\in\w^{<\w},\;n\in\w\}$ separates the points of the set $\delta(\bar X)$ from the points of the remainder $\overline{\delta(\bar X)}\setminus \delta(\bar X)$. Indeed, take any two points $x\in\delta(\bar X)$ and $y\in\overline{\delta(\bar X)}\setminus \delta(X)$. The $\IR$-completeness of $X$ guarantees that $y(\varphi)\notin\IR$ for some $\varphi\in C_{u}(\bar X)$. Replacing the function $\varphi$ by $-\varphi$, if necessary, we can assume that $y(\varphi)=+\infty$. Find $\alpha\in\w^\w$ with $\varphi\le\varphi_\alpha$ and observe that $y(\varphi_\alpha)=+\infty$. Choose any $n\in\w$ such that $\varphi_\alpha(x)<n$. Then $\varphi_{\alpha|k}(x)<n$ for some $k\in\w$ and hence $x\in F_{n,\alpha|k}$.  On the other hand, $y\notin F_{n,\alpha}$. By Proposition IV.9.2 in~\cite{Arch}, $\delta(\bar X)$ is a Lindel\"of $\Sigma$-space and so is its topological copy $\bar X$. Being Lindel\"of, the space $\bar X$ has countable extent. By Theorem~\ref{tc:small}, the space $\bar X$ is cosmic. Then the subspace $X$ of $\bar X$ is cosmic, too.
\smallskip

Under $\w_1<\mathfrak b$ the equivalence of the conditions $(1)$--$(5)$ to any of $(15)$--$(17)$ follows from Theorem~\ref{t:w1<b} and Proposition~\ref{p:nar=cov}.
\end{proof}

\begin{remark} The conditions $(1)$--$(14)$ of Theorem~\ref{t:small} are not equivalent without the $\IR$-completeness of the completion of $X$: just take any non-separable closed bounded convex subset $X$ of a Banach space and observe that $X$ considered as a uniform (metric) space is functionally bounded, so satisfies the conditions (8)--(14) but not (1)--(7).
\end{remark}

\begin{remark}\label{r:w1=b} The equivalence of the conditions (1)--(7) to (15)--(17) in Theorem~\ref{t:small} cannot be proved in ZFC: by Corollary~\ref{c:Ck2}, the function space $C_k(\w_1)$ is Lindel\"of, has uncountable cellularity, is not a $\Sigma$-space, and has a uniform $\w^\w$-base under $\w_1=\mathfrak b$.
\end{remark}

\begin{remark}  By Proposition~\ref{p:trans-base} and Definition~\ref{d:trans-lqu}, each  uniform baseportator $X$ with a neighborhood $\w^\w$-base at the unit has a  uniform $\w^\w$-base and hence has the properties described in Theorem~\ref{t:small}. By Proposition~\ref{p:u-port}, a portator $X$ is uniform if its multiplication $\mathbf{xy}$  is locally associative and semicontinuous at $(e,e)$, the inversion $\mathbf{x}^{-1}e$ is continuous at $e$, and  $e^{-1}e=\{e\}$. By  Corollary~\ref{c:TA=>lqu}, the class of uniform portators includes all topological groups. So, any topological group $X$ with an $\w^\w$-base has the properties described in Theorem~\ref{t:small}.
\end{remark}

\chapter{Topological spaces with a universal $\w^\w$-base}\label{Ch:univer}

In this chapter we study topological spaces with a universal $\w^\w$-base. A base $\Bas_X$ for a topological space $X$ is defined to be \index{base!universal}{\em universal} if each continuous map $f:X\to M$ to a metric space $(M,d_M)$ is uniformly continuous in the sense that for every $\e>0$ there exists an entourage $B\in\Bas_X$ such that $d_M(f(x),f(y))<\e$ for all $(x,y)\in B$.

The notion of a universal base can be parametrized by a cardinal parameter $\kappa$.

Namely, we define a function $f:X\to Y$ between based spaces $(X,\Bas_X)$ and $(Y,\Bas_Y)$ to be \index{function!$\kappa$-continuous}{\em $\kappa$-continuous} if for every entourage $U\in \Bas_Y$ there exists a subfamily $\V\subset \Bas_X$ of cardinality $|\V|\le\kappa$ such that for every $x\in X$ there exists an entourage $V\in \V$ such that $f\big[V[x]\big]\subset U\big[f[x]\big]$. Observe that uniformly continuous maps coincide with 1-continuous maps.

A base $\Bas_X$ for a  topological space $X$ will be called \index{base!$\kappa$-universal}{\em $\kappa$-universal} if each $\kappa$-continuous map $f:X\to M$ to a metric space $M$ of density $d(M)\le \kappa$ is uniformly continuous. Observe that an entourage base is universal if and only if it is $\kappa$-universal for every cardinal $\kappa$.

A base $\Bas$ for a topological space $X$ is called {\em $\IR$-universal} if each $\w$-continuous map $f:X\to\IR$ is uniformly continuous.

For any base $\Bas$ for a topological space $X$ we get the implications
$$\mbox{universal $\Ra$ $\w_1$-universal $\Ra$ $\w$-universal $\Ra$ $\IR$-universal.}$$

Tychonoff spaces admitting $\IR$-regular $\w^\w$-bases which are $\IR$-universal, $\w$-universal, and $\w_1$-universal, will be studied in Sections~\ref{s:Ru}, \ref{s:Fat}, and \ref{s:w1}, respectively. We shall prove that such spaces are close to being $\sigma'$-compact.

%Observe that a base $\Bas_X$ of a based topological space $X$ is universal (resp. $k$-universal, $\IR$-universal) if and only if so is the canonical (pre- and quasi-) uniformity of $X$.

In Section~\ref{s:AC} we generalize a result \cite{AC} of Arhangelskii and Calbrix who proved that every Tychonoff space $X$ with $\w^\w$-dominated function space $C_p(X)$ is \index{topological space!projectively $\sigma$-compact}{\em projectively $\sigma$-compact} in the sense that each metrizable separable image of $X$ is $\sigma$-compact. In Theorem~\ref{t:dominat} we shall prove a similar fact for Tychonoff spaces $X$ with $\w^\w$-dominated function space $C_\w(X)$. In Section~\ref{s:Ru} we shall prove that many ``smallness'' properties are equivalent for a Tychonoff space whose topology is generated by a universal $\w^\w$-base.
In Sections~\ref{s:pu-ww}, \ref{s:qu-ww}, \ref{s:u-ww} we detect topological spaces $X$ whose universal preuniformity $p\U_X$, the universal quasi-uniformity $q\U_X$ or the universal uniformity $\U_X$ have a $\w^\w$-base. Proposition~\ref{p:cu-Pb} implies that for any topological space $X$ we have the implications:
$$\mbox{$p\U_X$ has an $\w^\w$-base $\Ra$ $q\U_X$ has an $\w^\w$-base $\Ra$ $\U_X$ has an $\w^\w$-base}.$$

%Based spaces $X$ satisfying the equality $C_\w(X)=C_u(X)$ will be considered in Section~\ref{??}, where we prove that each complete based space $(X,\Bas_X)$ with $C_w(X)=C_u(X)$ is $\IR$-complete. In Section~\ref{??} we prove that a based space $(X,\Bas_X)$ with $\w^\w$-dominated function space $C_\w(X)$ is projectively $\sigma$-compact. In Section~\ref{??} we shall study $\w$-universally based topological spaces and their images under uniformly continuous maps into metric spaces.

%A based space $X$ is called {\em $\w$-narrow} if for each uniformly continuous map $f:X\to M$ to a metric space $M$ the image $f(X)$ is separable. This is equivalent to saying that for any entourage $U\in\U_X^{\pm\w}$ there is a countable set $C\subset X$ such that $X=U[C]$.

\section{Based topological spaces with $\w^\w$-dominated function space $C_\w(X)$}\label{s:AC}

For a based space $X$ by $C_\w(X)$ we denote the space of all $\w$-continuous real-valued functions on $X$. Since each uniformly continuous function is $\w$-continuous, we get the inclusions $C_u(X)\subset C_\w(X)\subset C_p(X)\subset\IR^X$. If the base $\Bas_X$ of $X$ is universal, then $C_u(X)=C_\w(X)=C_p(X)$. If the space $X$ is Lindel\"of, then $C_\w(X)=C_p(X)$.

The following result generalizes Arhangel'skii-Calbrix Theorem \cite{AC} (see also \cite[Theorem 9.9]{kak}).

\begin{theorem}\label{t:dominat} If for a based topological space $X$ the set $C_\w(X)$ is $\w^\w$-dominated in $\IR^X$, then for every $\w$-continuous map $f:X\to M$ to a metric space $M$ the image $f[X]$ is $\sigma$-compact.
\end{theorem}

\begin{proof} We lose no generality assuming that $M = f[X]$. First we consider the case of a totally bounded metric space $M$. In this case the completion $\bar M$ of $M$ is compact. Denote by $d$ the metric of the compact metric space $\bar M$. Let $\{\varphi_\alpha\}_{\alpha\in\w^\w}\subset \IR^X$ be a subset witnessing that the set $C_\w(X)$ is $\w^\w$-dominated in $\IR^X$.   Replacing each function $\varphi_\alpha$ by $\max\{1,\varphi_\alpha\}$, we can assume that $\varphi_\alpha(X)\subset[1,\infty)$.

For every $\alpha\in\w^\w$ consider the open set $U_\alpha=\bigcup_{x\in X}B_d(f(x),1/\varphi_\alpha(x))$ and the compact set
 $K_\alpha=\bar M\setminus U_\alpha$ in $\bar M$. Here by $B_d(x,\e)=\{y\in\bar M:d(x,y)<\e\}$ we denote the open $\e$-ball centered at a point $x$ of the metric space $\bar M$. Observe that for any $\alpha\le\beta$ in $\w^\w$ the inequality $\varphi_\alpha\le \varphi_\beta$ implies the inclusion $K_\alpha\subset K_\beta$. We claim that the family $(K_\alpha)_{\alpha\in\w^\w}$ is cofinal in $\K(\bar M\setminus M)$. Given any compact set $K\subset\bar M\setminus M$, consider the $\w$-continuous function $\varphi:X\to\IR$ defined by
 $\varphi(x)=1/d(f(x),K)$ where $d(f(x),K)=\min_{y\in K}d(f(x),y)$. Find $\alpha\in\w^\w$ such that $\varphi_\alpha\ge\varphi$ and observe that $K\subset K_\alpha$, witnessing that $\w^\w\succcurlyeq \K(\bar M\setminus M)$. By Theorem~\ref{t:Chris},
 the space $\bar M\setminus M$ is Polish and hence of type $G_\delta$ in $\bar M$. Then $M$ is $\sigma$-compact, being of type $F_\sigma$ in the compact space $\bar M$.

Next, we consider the case of separable metric space $M = f(X)$. In this case we can take any homeomorphism $h:M\to T$ onto a totally bounded metric space $T$ and observe that the $\w$-continuity of the map $f:X\to M$ implies the $\w$-continuity of the composition $h\circ f:X\to T$. By the preceding case, the image $h\circ f(X)$ is $\sigma$-compact. Since $h$ is a homeomorphism, the space $f(X) = M$ is $\sigma$-compact too.

Finally, we can prove the general case of arbitrary metric space $M$.
By the preceding case, it is sufficient to show that the image $f(X) = M$ is separable. Assuming that $M = f(X)$ is non-separable, we can find a closed discrete subspace $D\subset M$ of cardinality $|D|=\w_1$. By \cite[Corollary 1]{BP}, the separable Hilbert space $\ell_2$ contains an uncountable linearly independent compact set $K\subset\ell_2$. Fix a subset $D'\subset K$  of cardinality $|D'|=\w_1$ which is not $\sigma$-compact. Next, consider the linear hull $L$ of the set $D'$ in $\ell^2$ and observe that $D' = L\cap K$, so $D'$ is a closed subset of $L$. Take any surjective map $g:D\to D'$. By Dugundji Theorem \cite{Dug}, the map $g:D\to D'\subset L$ has a continuous extension $\bar g:M\to L$. Then the map $\bar g\circ f:X\to L$ is $\w$-continuous. Since the space $L$ is separable, the preceding case guarantees that the image $\bar g\circ f(X)$ is $\sigma$-compact and so is its closed subspace $D'=\bar g\circ f(X)\cap K$. But this contradicts the choice of $D'$.
\end{proof}

\section{Topological spaces with an $\IR$-universal $\w^\w$-base}\label{s:Ru}

In this section we study topological spaces possessing an $\IR$-universal $\IR$-regular $\w^\w$-base.

We recall that a base $\Bas_X$ for a topological space $X$ is called
\begin{itemize}
\item{\em universal} if each continuous map $f:X\to Y$ to a metric space is uniformly continuous;
\item{\em $\IR$-universal} if each $\w$-continuous map $f:X\to \IR$ is uniformly continuous;
\item{\em $\IR$-regular} if the canonical map $\delta:X\to\IR^{C_u(X)}$ is a topological embedding.
\end{itemize}
By Lemma~\ref{l:u=>Ru+Rr}, each universal base for a Tychonoff space is $\IR$-universal and $\IR$-regular.

We recall that a subset $B\subset X$ of a based space $(X,\Bas)$ is called
\begin{itemize}
\item {\em precompact} if for each uniformly continuous map $f:X\to Y$ to a complete metric space $Y$ the set $f(B)$ has compact closure in $Y$;
\item {\em functionally bounded} if for each uniformly continuous function $f:X\to \IR$  the set $f(B)$ is bounded in $\IR$;
\item {\em $\sigma$-bounded} if $B$ is the countable union of functionally bounded subsets of $X$.
\end{itemize}

\begin{theorem}\label{t:fb=>ms} If a topological space $X$ has an $\IR$-universal $\IR$-regular $\w^\w$-base $\Bas$, then
\begin{enumerate}
\item a subset $B\subset X$ is functionally bounded in $(X,\Bas)$ if and only if $B$ is precompact in $(X,\Bas)$;
\item each precompact subset of $(X,\Bas)$ is metrizable and separable;
\item each precompact $\w$-Urysohn closed subset of $(X,\Bas)$ is compact and metrizable;
\item each precompact $\bar G_\delta$-subset of $(X,\Bas)$ is compact and metrizable;
\item each $\sigma$-bounded subset of $(X,\Bas)$ is a $\mathfrak P_0$-space;
\item each $\sigma$-bounded $\bar G_\delta$-subset of $X$ is a $\sigma$-compact $\mathfrak P_0$-space;
\smallskip

\item the set $X^{\prime\css}$ is closed and $\sigma$-bounded in $(X,\Bas)$;
\item $X^{\prime\css}$ is a $\sigma$-metrizable $\mathfrak P_0$-space;
\smallskip

\item a closed subset $F\subset X$ is $\sigma$-compact if $|F\setminus X^{\prime\css}|\le\w$ and $F$ is either $\w$-Urysohn or a $\bar G_\delta$-set in $X$;
\item the space $X$ is $\sigma'$-compact if and only if $|X'\setminus X^{\prime\css}|\le\w$.
\end{enumerate}
\end{theorem}

\begin{proof} Let $\U=\{U\subset X\times X:\exists B\in\Bas\;(B\subset U)\}$ be the preuniformity generated by the base $\Bas$ and $\U^{\pm\w}$ be the canonical uniformity of the preuniform space $(X,\U)$. By Propositions~\ref{p:cu-Pb} and \ref{p:Rr-u}, the uniformity $\U^{\pm\w}$ has an $\w^\w$-base and generates the topology of $X$. By Theorem~\ref{t:ww=>netbase}, the space $X$ has a countable uniform $\css^*$-netbase.
\smallskip

1. The first statement is proved in Proposition~\ref{p:fb=pc}.
\smallskip

2. The second statement follows from Proposition~\ref{p:ww=>pc-ms}.
\smallskip

3. By Proposition~\ref{p:fb+wU=>cc}, each precompact $\w$-Urysohn closed subset of $(X,\Bas)$ is countably compact and being metrizable is compact.
\smallskip

4. Let $B$ be a precompact $\bar G_\delta$-subset of $(X,\Bas)$. By the first statement, the space $B$ is metrizable and separable and by Lemma~\ref{l:hL+bG=>wU}, $B$ is $\w$-Urysohn. By the preceding statement, $B$ is compact and metrizable.
\smallskip

5. By the first and second statements, each $\sigma$-bounded subset $B$ of $(X,\Bas)$ is cosmic and by Theorem~\ref{t:lqu-ww}(4), $B$ is a $\mathfrak P_0$-space.
\smallskip

6. The sixth statement follows from the statements (1), (4) and (5).
\smallskip

7. Since the space $X$ has a countable uniform $\css^*$-netbase, the statement (7) follows from Theorem~\ref{t:universal=>s-bound}.
\smallskip

8. The statement (8) follows from the statements (1), (3), (5), and (7).
\smallskip

%9. The statement (9) follows from the statements (1), (4) and (7).
%\smallskip

9. Let $F\subset X$ be a closed set in $X$ such that $|F\setminus X^{\prime\css}|\le\w$. By the statement (7), the set $X^{\prime\css}$ is $\sigma$-bounded and so is the set $F=(F\cap X^{\prime\css})\cup (F\setminus X^{\prime\css})$. If $F$ is $\w$-Urysohn in $X$, then by the statements (1) and (3) the set $F$ is $\sigma$-compact. If $F$ is $\bar G_\delta$ in $X$, then
the statements (1) and (4) the set $F$ is $\sigma$-compact.
\smallskip

10. If the set $X'$ is $\sigma$-compact, then by the statement (2), the space $Z=X'$ is separable and hence $Z'=Z^{\prime\css}\subset X^{\prime \css}$ and the complement $X'\setminus X^{\prime \css}\subset Z\setminus Z'$ is at most countable.  Now assume that $|X'\setminus X^{\prime \css}|\le\w$. By the statements (1), (2), and (7), the space $X^{\prime\css}$ is cosmic and so is the space $X'$. By Lemma~\ref{l:para'}, the space $X$ is paracompact and hence $\w$-Urysohn. By the statement (9), the space $X'$ is $\sigma$-compact.
\end{proof}

\section{Topological spaces with an $\w$-universal $\w^\w$-base}\label{s:Fat}

In this section we study properties of Tychonoff spaces admitting an $\w$-universal $\w^\w$-base.

\begin{definition} A base $\Bas_X$ for a topological space $X$ is \index{base!$\w$-universal}{\em $\w$-universal\/} if each $\w$-conti\-nuous map $f:X\to Y$ to a metric separable space $Y$ is uniformly continuous.
\end{definition}

 We recall that a map $f:X\to Y$ from a based topological space $X$ to a metric space $Y$ \index{function!$\w$-continuous}{\em $\w$-continuous} if for every $\e>0$ there is a countable subfamily $\V\subset\Bas_X$ such that for every point $x\in X$ there exists an entourage $V\in\V$ such that $\diam\big( f[V[x]]\big)<\e$.

It is clear that each $\w$-universal base for a topological space is $\IR$-universal, so the results of Section~\ref{s:Ru} remain true for Tychonoff spaces with an $\w$-universal $\IR$-regular $\w^\w$-base.

\begin{definition} Let $f:X\to Y$ be a map from a uniform space $X$ to a set $Y$. A point $y\in Y$ is called a \index{function!fat value of}{\em fat value} of $f$ if for any entourage $U\in\U_X$ there exists a point $x\in f^{-1}(y)$ such that $f[U[x]]\ne\{y\}$. This is equivalent to saying that $y\in f[U[f^{-1}(y)]]\ne\{y\}$ where $U[f^{-1}(y)]=\bigcup_{x\in f^{-1}(y)}U[x]$.

By \index{$\Fat(f)$}$\Fat(f)$  denote the set of fat values of the map $f$.
\end{definition}

\begin{theorem}\label{t:Fat} For every $\w$-continuous map $f:X\to M$ from an $\w$-universally $\w^\w$-based topological space $X$ to a metric space $M$, any closed subset $F\subset \Fat(f)$ of $f[X]$ is $\sigma$-compact.
\end{theorem}

\begin{proof} First we consider a partial case when the metric space $M$ is totally bounded. In this case the completion $\bar M$ of $M$ is compact. Let $d$ be the metric of the compact metric space $\bar M$. Replacing $d$ by $\min\{1,d\}$, we can assume that $d\le 1$. Let $\bar F$ be the closure of $F$ in $\bar M$.

Let $\Bas_X=\{U_\alpha\}_{\alpha\in\w^\w}$ be the $\w$-universal $\w^\w$-base for the based topological space $X$. For any $\alpha\in\w^\w$ and $y\in F$ consider the real number
$\e_\alpha(y)=\inf\{\e>0:f[U_\alpha[f^{-1}(y)]]\subset B_d(x,\e)\}$, where $B_d(x,\e)=\{y\in \bar M:d(x,y)<\e\}$ is the $\e$-ball around $x$ in $\bar M$. Since $y\in F\subset\Fat(f)$, the number $\e_\alpha(y)$ is strictly positive. Consider the open set $V_\alpha=\bigcup_{y\in F}B_d(y,\e_\alpha(y))\subset \bar M$ and the compact set $K_\alpha=\bar F\setminus V_\alpha\subset \bar F\setminus F$. For every $\alpha\le\beta$ in $\w^\w$ the inclusion $U_\beta\subset U_\alpha$ implies $\e_\beta\le\e_\alpha$, $V_\beta\subset V_\alpha$ and finally $K_\alpha\subset K_\beta$. We claim that the family $(K_\alpha)_{\alpha\in\w^\w}$ is cofinal in $\K(\bar F\setminus F)$. Fix any compact subset $K\subset \bar F\setminus F$. Taking into account that $F$ is closed in $f[X]$, we conclude that $F=\bar F\cap f[X]$ and hence $K\cap f[X]=\emptyset$. Using the paracompactness of the metric space $f[X]$ and Theorem 8.1.10 \cite{Eng}, we can construct a continuous metric $\rho$ on $f[X]$ such that for every $y\in f[X]$ the unit ball $B_\rho(y,1)=\{z\in f[X]:\rho(y,z)<1\}$ is contained  in the ball $B_d(y,d(y,K))$. Since the based space $X$ is $\w$-universal, the $\w$-continuous map $f:X\to (M,\rho)$ is uniformly continuous. Consequently the entourage $U=\{(x,y)\in X\times X:\rho(x,y)<1\}$ contains the basic entourage $U_\alpha$ for some $\alpha\in\w^\w$.

We claim that $K\subset K_\alpha$. This will follow as soon as we check that $K\cap V_\alpha=\emptyset$. Take any point $y\in F$ and $x\in f^{-1}(y)$ and observe that $f[U_\alpha[x]]\subset f[U[x]]\subset B_\rho(f(x),1)\subset B_d(y,d(y,K))$, which implies that $\e_\alpha(y)\le d(y,K)$ and $B_d(y,\e_\alpha(y))\subset B_d(y,d(y,K))\subset M\setminus K$. Then $V_\alpha=\bigcup_{y\in F}B_d(y,\e_\alpha(y))\subset M\setminus K$ and $K_\alpha=\bar F\setminus V_\alpha\supset \bar F\setminus (M\setminus K)=K$. Therefore, the family $\{K_\alpha\}_{\alpha\in\w^\w}$ in cofinal in $\K(\bar F\setminus F)$, which yields the reduction $\w^\w\succcurlyeq \K(\bar F\setminus F)$. By Christensen's Theorem~\ref{t:Chris}, the space $\bar F\setminus F$ is Polish and hence its complement $F$ in the compact space $\bar F$ is $\sigma$-compact.
\smallskip

Next, we consider the case of separable space $M=f[X]$. Fix any homeomorphism $h:M\to N$ to a totally bounded metric space $N$ and observe that the map $h\circ f:X\to N$ is $\w$-continuous. By the preceding case, the closed subset $h[F]\subset\Fat(h\circ f)$ of $h(M)=N$ is $\sigma$-compact and so is its homeomorphic copy $F$.
\smallskip

Finally, we consider the case of arbitrary metric space $(M,d)$. First we show that the set $\Fat(f)$ is separable. In the opposite case we can find a subspace $D\subset \Fat(f)$ of cardinality $|D|=\w_1$, which is separated in the sense that $\delta=\inf\{d(x,y):x,y\in D,\;x\ne y\}>0$. In the separable Hilbert space $\ell_2$ choose any uncountable linearly independent compact subset $K$ and fix a subset $D'\subset K$  of cardinality $|D'|=\w_1$ which is not $\sigma$-compact. Consider the linear hull $L$ of the set $D'$ in $\ell^2$ and observe that $D'=L\cap K$, so $D'$ is a closed subset of $L$. Take any surjective map $g:D\to D'$. Fix any non-zero vector $v\in L$ and for every $x\in D$ define a map $\tilde g_x:B_d(x,\delta/3)\to L$ by $g_x(z)=g(x)+d(z,x)\cdot v$. The maps $\tilde g_x$, $x\in D$, compose a continuous map $\tilde g:\bigcup_{x\in D}B_d(x,\delta/3)\to L$ defined by $\tilde g|B_d(x,\delta/3)=\tilde g_x$.
By Dugundji Theorem \cite{Dug}, the map $\tilde g$ has a continuous extension $\bar g:M\to L$. The choice of the maps $\tilde g_x$ guarantees that $D'=\bar g[D]\subset \Fat(\bar g\circ f)$. Since the space $L$ is separable, the preceding case guarantees that the closed subset $D'$ of $L$ is $\sigma$-compact, which is not the case.

So, the set $\Fat(f)$ is separable and we can choose a closed topological embedding $h:\Fat(f)\to L$ to a separable normed space $L$. By Dugundji Theorem \cite{Dug}, the map $h$ can a continuous extension $\bar h:M\to L$. Finally, consider the map $\tilde h:M\to L\times\IR$ defined by $\tilde h(y)=(\bar h(y),d(y,F))$, $y\in M$. It follows that $\tilde h[F]$ is a closed subset of $L\times\{0\}$ and $\tilde h[F]\subset\Fat(\tilde h\circ f)$. By the preceding case, the space $\tilde h[F]$ is $\sigma$-compact and so is its topological copy $F$.
\end{proof}

\section{Topological spaces with an $\w_1$-universal $\w^\w$-base}\label{s:w1}

In this section we study topological spaces possessing an $\w_1$-universal $\IR$-regular $\w^\w$-base.
We recall that a base $\Bas$ for a topological space $X$ is {\em $\w_1$-universal} if each $\w_1$-continuous map $f:X\to Y$ of a metric space $Y$ of density $\le\w_1$ is uniformly continuous.

For  a topological space $X$ by $X'$ we denote the set of non-isolated points of $X$ and by $X^{\prime P}$ the set of points $x\in X$ which are not $P$-points in $X$. It is clear that $X^{\prime\css}\subset X^{\prime P}\subset X'$, where $X^{\prime\css}$ is the set of accumulation points of countable sets in $X$.

\begin{theorem}\label{t:U+P} If a Tychonoff space $X$ has an $\w_1$-universal $\IR$-regular  $\w^\w$-base $\Bas$, then the set $X^{\prime P}$ is $\w$-narrow in the based space $(X,\Bas)$.
\end{theorem}

\begin{proof} Assuming that the set $X^{\prime P}$ is not $\w$-narrow in $X$, we could find a uniformly continuous map $f:X\to M$ to a metric space $(M,d)$ of density $\w_1$ such that the image $f[X^{\prime P}]$ is not separable. Then there exists a subset $D\subset f[X^{\prime P}]$ of cardinality $|D|=\w_1$ such that $\delta=\inf\{d(x,y):x,y\in D,\;x\ne y\}>0$.
For every $y\in D$ choose a point $s(y)\in f^{-1}(y)\cap X^{\prime P}$ and consider its  neighborhood $U_{s(y)}=f^{-1}[B_d(y,\delta/3)]$. Since $s(y)$ is not a $P$-point in $X$ and the based space $(X,\Bas)$ is $\IR$-regular, there exists a $\w$-continuous function $\lambda_y:X\to [0,1]$ such that $\lambda_y(s(y))=0$, $\lambda_y[X\setminus U_{s(y)}]\subset\{1\}$ and  $\lambda_y(O_x)\ne\{0\}$ for every neighborhood $O_x$ of $x$. Let $\lambda:X\to[0,1]$ be the $\w_1$-continuous function defined by $\lambda(x)=\lambda_y(x)$ if $x\in U_{s(y)}$ for some $y\in D$ and $\lambda(x)=1$ otherwise.  Then for the $\w_1$-continuous function $\tilde f:X\to M\times [0,1]$, $\tilde f:x\mapsto (f(x),\lambda(x))$, we get $D\times\{0\}\subset\Fat(\tilde f)$. By the $\w_1$-universality of the base $\Bas$, the $\w_1$-continuous map $\tilde f$ is uniformly continuous.
By Theorem~\ref{t:Fat}, the set $D\times \{0\}$ is $\sigma$-compact, which is not possible as $D$ is uncountable and discrete.
\end{proof}

\begin{theorem}\label{t:w1ub} Assume that $\w_1<\mathfrak b$. If a Tychonoff space $X$ has an $\w_1$-universal $\IR$-regular $\w^\w$-base $\Bas$, then
\begin{enumerate}
\item $X^{\prime P}$ is closed in $X$ and the complement $X^{\prime P}\setminus X^{\prime\css}$ is countable and discrete;
\item the set $X^{\prime P}$ is $\sigma$-bounded in $X$ and $X^{\prime P}$ is a $\sigma$-metrizable $\mathfrak P_0$-space;
\item a closed subset $F\subset X$ is $\sigma$-compact if $|F\setminus X^{\prime P}|\le\w$ and $F$ is either $\w$-Urysohn or a $G_\delta$-set in $X$;
\item the set $X'$ is $\sigma$-compact if and only if $|X'\setminus X^{\prime P}|\le\w$.
\end{enumerate}
\end{theorem}

\begin{proof} Let $\U:=\{U\subset X\times X:\exists B\in\Bas\;(B\subset U)\}$ be the preuniformity generated by the base $\Bas$ and $\U^{\pm\w}$ be the canonical uniformity of the preuniform space $(X,\Bas)$. By Propositions~\ref{p:cu-Pb} and \ref{p:Rr-u}, the uniformity $\U^{\pm\w}$ has an $\w^\w$-base and generates the topology of $X$. By Theorem~\ref{t:ww=>netbase}, the space $X$ has a countable uniform $\css^*$-netbase. By Theorem~\ref{t:U+P}, the set $X^{\prime P}$ is $\w$-narrow in $(X,\Bas)$. Then the closure $\bar X^{\prime P}$ of $X^{\prime P}$ is $\w$-narrow too. By the implication $(17)\Ra(3)$ in Theorem~\ref{t:small}, the space $Z=\bar X^{\prime P}$ is a $\mathfrak P_0$-space.
\smallskip

1. It follows that each point $x\in \bar X^{\prime P}\setminus X^{\prime P}$ is not a $P$-point and hence $Z\setminus X^{\prime P}\subset X^{\prime P}$, which means that the set $X^{\prime P}=Z$ is closed in $X$. The separability of the $\mathfrak P_0$-space $Z$ implies that set $Z\setminus Z'$ is countable and $Z'=Z^{\prime\css}\subset X^{\prime\css}$. Consequently, $X^{\prime P}\setminus X^{\prime\css}\subset Z\setminus Z'$ is countable and discrete.
\smallskip

2. By Theorem~\ref{t:fb=>ms}(7), the set $Z'\subset X^{\prime\css}$ is closed and $\sigma$-bounded in $X$ and so is the set $X^{\prime P}=(Z\setminus Z')\cup Z'$. By Theorem~\ref{t:fb=>ms}(1,2,5,7), the space $X^{\prime P}$ is cosmic $\sigma$-metrizable, and by Theorem~\ref{t:small}, $X^{\prime P}$ is a $\mathfrak P_0$-space.
\vskip3pt

3. Let $F\subset X$ be a closed subset of $X$ such that $|F\setminus X^{\prime P}|\le\w$.
By the first statement, $|X^{\prime P}\setminus X^{\prime\css}|\le\w$ and hence $|F\setminus X^{\prime\css}|\le\w$. If the set $F$ is $\w$-Urysohn in $X$, then $F$ is $\sigma$-compact by Theorem~\ref{t:fb=>ms}(9). If $F$ is a $G_\delta$-set in $X$, then we can apply  Lemma~\ref{l:P=>Gd=bGd} and conclude that $F$ is a $\bar G_\delta$-set in $X$. By Theorem~\ref{t:fb=>ms}(9), the set $F$ is $\sigma$-compact.
\smallskip

4. If the set $Z=X'$ is $\sigma$-compact, then by Theorem~\ref{t:small}, $Z$ is cosmic and hence $Z\setminus Z'$ is countable. Taking into account that $Z'\subset Z^{\prime \css}\subset X^{\prime P}$, we conclude that the set $X'\setminus X^{\prime P}\subset Z\setminus Z'$ is at most countable.

Now assume that $|X'\setminus X^{\prime P}|\le \w$. By the statement (2), the space $X^{\prime P}$ is cosmic and so is the space $X'$. By Lemma~\ref{l:para'}, the space $X$ is paracompact and hence $\w$-Urysohn. By the statement (3), the closed $\w$-Urysohn subset $X'$ of $X$ is $\sigma$-compact.
\end{proof}

\begin{theorem} Under $\w_1<\mathfrak b$, for a Tychonoff space $X$ with an $\w_1$-universal $\IR$-regular $\w^\w$-base the following conditions are equivalent:
\begin{enumerate}
\item $X'$ is a $\sigma$-compact $G_\delta$-set in $X$;
\item $|X'\setminus X^{\prime P}|\le\w$ and $X'$ is a $G_\delta$-set in $X$;
\item $X$ is a $\Sigma$-space;
\item $X$ is a $\sigma$-space.
\end{enumerate}
\end{theorem}

\begin{proof} The equivalence $(1)\Leftrightarrow(2)$ follows from Theorem~\ref{t:w1ub}(4).
\smallskip

$(1)\Ra(3)$ Assume that $X'$ is a $\sigma$-compact $G_\delta$-set in $X$. Write the $\sigma$-compact space $X'$ as the countable union $\bigcup_{n\in\w}K_n$ of compact subsets of $X$. Also write the $F_\sigma$-set $X\setminus X'$ as the countable union $\bigcup_{n\in\w}F_n$ of closed subsets of $X$. Then the family of compact sets $\mathcal N=\{K_n\}_{n\in\w}\cup\bigcup_{n\in\w}\big\{\{x\}:x\in F_n\big\}$ is $\sigma$-discrete and is a $\N$-network, witnessing that $X$ is a $\Sigma$-space.
\smallskip

$(3)\Ra(4)$ Assume that $X$ is a $\Sigma$-space. Let $\U=\{U\subset X\times X:\exists B\in\Bas\;(B\subset U)\}$ be the preuniformity generated by the base $\Bas$ and $\U^{\pm\w}$ be the canonical uniformity of the preuniform space $(X,\U)$. By Propositions~\ref{p:cu-Pb} and \ref{p:Rr-u}, the uniformity $\U^{\pm\w}$ has an $\w^\w$-base and generates the topology of the space $X$. By Theorem~\ref{t:prop-luww}, the $\Sigma$-space $X$ is a $\sigma$-space.
\smallskip

$(4)\Ra(1)$ Assume that $X$ is a $\sigma$-space. By \cite[p.446]{Grue}, the closed subset $X'=X^{\prime P}$ of $X$ is a $G_\delta$-set in $X$. By Theorem~\ref{t:w1ub}, the set $X'$ is $\sigma$-compact.
\end{proof}

%\begin{problem} Is Corollary~\ref{c:sigma'3} true is ZFC?
%\end{problem}
%We recall that a topological space $X$ is called a {\em $\sigma$-space} if $X$ is regular and possesses a $\sigma$-discrete network.
%It is well-known that the class of $\sigma$-spaces contains all cosmic and all metrizable spaces. More information on $\sigma$-spaces can be found in the classical survey of Gruenhage \cite[\S4]{Grue}.

\section{Characterizing ``small'' spaces with a universal $\w^\w$-base}\label{s:u-small}

For a Tychonoff space $X$ by \index{$C_k(X)$}$C_k(X)$ we denote the function space $C(X)$ endowed with the compact-open topology.

\begin{theorem}\label{t:tsmall} For a Tychonoff space $X$ with an $\IR$-universal $\IR$-regular $\w^\w$-base the following conditions are equivalent:
\begin{enumerate}
\item[\textup{(1)}] $|X\setminus X^{\prime\css}|\le\w$;
\item[\textup{(2)}] $X$ is $\sigma$-compact;
\item[\textup{(3)}] $X$ is separable;
\item[\textup{(4)}] $X$  contains a dense $\Sigma$-subspace with countable extent;
%\item[\textup{(4)}] $X$ is a $\Sigma$-space with countable extent;
\item[\textup{(5)}] $X$ is cosmic;
\item[\textup{(6)}] $X$ is an $\aleph_0$-space;
\item[\textup{(7)}] $X$ is a $\mathfrak P_0$-space;
\item[\textup{(8)}] the poset $C(X)$ is $\w^\w$-dominated;
\item[\textup{(9)}] the set $C(X)$ is $\w^\w$-dominated in $\IR^X$;
\item[\textup{(10)}] $C_u(X)$ has a special compact resolution;
\item[\textup{(11)}] $C_u(X)$ is $K$-analytic;
\item[\textup{(12)}] $C_u(X)$ has a compact resolution.
\item[\textup{(13)}] $C_u(X)$ is cosmic;
\item[\textup{(14)}] $C_u(X)$ is analytic.
\end{enumerate}
If $C_u(X)=C(X)$, then the conditions \textup{(1)--(14)} are equivalent to each of the following conditions:
\begin{enumerate}
\item[\textup{(15)}] $C_k(X)$ is $K$-analytic;
\item[\textup{(16)}] $C_k(X)$ has a compact resolution;
\item[\textup{(17)}] $C_k(X)$ is cosmic;
\item[\textup{(18)}] $C_k(X)$ is analytic.
\end{enumerate}
Under $\w_1<\mathfrak b$ the conditions $(1)$--$(18)$ are equivalent to:
\begin{enumerate}
\item[\textup{(19)}] $X$ is Lindel\"of;
\item[\textup{(20)}] $X$ has countable discrete cellularity.
\end{enumerate}
If $\w_1<\mathfrak b$ and $X$ has an $\w_1$-universal $\IR$-regular $\w^\w$-base, then the conditions \textup{(1)--(20)} are equivalent to
\begin{enumerate}
\item[(21)] $|X\setminus X^{\prime P}|\le\w$.
\end{enumerate}
\end{theorem}

\begin{proof} Let $\Bas_X$ be an $\IR$-universal $\IR$-regular $\w^\w$-base for the space $X$. Let $\U_X=\{U\subset X\times X:\exists B\in\Bas_X\;\;(B\subset U)\}$ be the preuniformity on $X$, generated by the base $\Bas_X$ and $\U_X^{\pm\w}$ be the canonical uniformity of the preuniform space $(X,\U_X)$. By Corollary~\ref{c:Rreg} and Proposition~\ref{p:p-lqu+qu}(4), the uniformity $\U_X^{\pm\w}$ generates the topology of $X$. By Proposition~\ref{p:cu-Pb}, the uniformity $\U_X^{\pm\w}$ has an $\w^\w$-base $\Bas_X^{\pm\w}$, which is a uniform $\w^\w$-base for the space $X$.

 The equivalence of the conditions $(3)$--$(14)$ follow from Theorem~\ref{t:small} applied to the  uniform $\w^\w$-base $\Bas_X^{\pm\w}$ for the space $X$. Also this theorem guarantees that under $\w_1<\mathfrak b$ the conditions $(3)$--$(14)$ are equivalent to $(19)$ and $(20)$.
 The implication $(2)\Ra(4)$ is trivial and $(5)\Ra(2)$ follows from Theorems~\ref{t:netbase=>sigma} and \ref{t:ww=>netbase}. The implication $(3)\Ra(1)$ is trivial and $(1)\Ra(2)$ follows from Theorem~\ref{t:fb=>ms}(10). The implication $(1)\Ra(21)$ is trivial and the implication $(21)\Ra(1)$ follows from Theorem~\ref{t:w1ub}(1).
\smallskip

Now assume that $C_u(X)=C(X)$. We shall prove that $(3,10)\Ra (18)\Ra(15)\Ra(16)\Ra(12)$, $(18)\Ra(17)\Ra(13)$.

To prove that $(3,10)\Ra(18)$, assume that the space $X$ is separable and the function space $C_u(X)$ admits a special compact resolution $(K_\alpha)_{\alpha\in\w^\w}$. The condition (3) of Definition~\ref{d:spec-res} and (one direction of) the Ascoli Theorem \cite[8.2.10]{Eng} (holding without $k$-space requirements) guarantees that each set $K_\alpha$, $\alpha\in\w^\w$, is compact in the function space $C_k(X)$. Now we see that the special compact resolution $(K_\alpha)_{\alpha\in\w^\w}$ of $C_u(X)=C(X)$ is a compact resolution of the function space $C_k(X)$. Fix a countable dense subset $D$ in the separable space $X$ and observe that the restriction operator $C_k(X)\to\IR^D$, $\varphi\mapsto\varphi|D$, is injective, which implies that the space $C_k(X)$ is submetrizable. By Lemma~\ref{l:analytic}, the space $C_k(X)$ is analytic.

The implication $(18)\Ra(15)$ is trivial, $(15)\Ra(16)$ follows from Proposition~3.10(i) \cite{kak} and the implications $(16)\Ra(12)$, $(17)\Ra(13)$ trivially follow from the continuity of the identity map $C_k(X)\to C_u(X)\subset\IR^X$. The implication $(18)\Ra(17)$ is trivial.
\end{proof}

\begin{remark} Under $\w_1=\mathfrak b$ the conditions \textup{(1)--(18)} of Theorem~\ref{t:tsmall} are not equivalent to $(19)$:  By Corollary~\ref{c:CkZ}, the function space $Z=C_k(\w_1,\IZ)$ is a Lindel\"of $P$-space of uncountable cellularity, whose universal uniformity $\U_Z$ has an $\w^\w$-base under $\w_1=\mathfrak b$. The space $Z$ has $Z'=Z$ and $Z^{\prime P}=\emptyset$.
\end{remark}

\begin{problem} Is each universally $\w^\w$-based Tychonoff space $X=X^{\prime P}$ cosmic?
\end{problem}

\section{Topological spaces with $\w^\w$-based universal preuniformity}\label{s:pu-ww}

\index{universal preuniformity}\index{$p\U_X$}
In this section we detect topological spaces whose universal preuniformity $p\U_X$ has an $\w^\w$-base. The universal preuniformity $p\U_X$ consists of all neighborhood assignments on $X$, i.e., entourages $U\subset X\times X$ such that for every $x\in X$ the $U$-ball $U[x]$ is a neighborhood of $x$ in $X$.

Proposition~\ref{p:u-sub}(1) implies:

\begin{corollary}
  If the universal preuniformity $p\U_X$ of a topological space $X$ has an $\w^\w$-base, then for each subspace $Z\subset X$ the universal preuniformity $\U_Z$ of $Z$ has an $\w^\w$-base, too.
\end{corollary}

A point $x$ of a topological space is defined to be \index{singular point}{\em singular} if the intersection $\ddot x:=\bigcap\Tau_x(X)$ of all neighborhoods of $x$ is a neighborhood of $x$. It is clear that each singular point is a $P$-point. A point of a $T_1$-space is singular if and only if it is isolated.

For a topological space $X$ by $X'$ we denote the set of non-singular points in $X$ and by $X^{\prime P}$ the set of points which are not $P$-points in $X$. It is clear that $X^{\prime P}\subset X'$. For a $T_1$-space $X$ the set $X'$ coincides with the set of non-isolated points of $X$.

\begin{theorem}\label{t:pu-ww} For any topological space $X$ we have the reductions $$\prod_{x\in X'}\Tau_x(X)\cong p\U_X\succcurlyeq\w^{X^{\prime P}}.$$
\end{theorem}

\begin{proof} The monotone cofinal map $f:\prod_{x\in X'}\Tau_x(X)\to p\U_X$ establishing the reduction  $\prod_{x\in X'}\Tau_x(X)\succcurlyeq p\U_X$ can be defined by the formula $f((U_x)_{x\in X'})=\big(\bigcup_{x\in X\setminus X'}\{x\}\times\ddot x\big)\cup\big(\bigcup_{x\in X'}\{x\}\times U_x\big)$. On the other hand, the reduction $p\U_X\succcurlyeq \prod_{x\in X'}\Tau_x(X)$ is established by the monotone cofinal map $p\U_X\to \prod_{x\in X'}\Tau_x(X)$, $U\mapsto(U[x])_{x\in X'}$.

Next, we prove that $p\U_X\succcurlyeq \w^{X^{\prime P}}$. For every point $x\in X^{\prime P}$ choose a decreasing sequence $\{U_n[x]\}_{n\in\w}\subset \Tau_{x}(X)$ whose intersection $\bigcap_{n\in\w}U_n[x]$ is not a neighborhood of $x$. Consider the monotone cofinal map $\mu_*:p\U_X\to \w^{X^{\prime P}}$ assigning to each entourage $U\in p\U_X$ the function $\mu_U:X^{\prime P}\to\w$, $\mu_U:x\mapsto\min\{n\in\w:U[x]\not\subset U_n[x]\}$. The map $\mu_*$ establishes the reduction $p\U_X\succcurlyeq \w^{X^{\prime P}}$.
\end{proof}

\begin{corollary}\label{c:up-ww} For a topological space $X$ the universal preuniformity $p\U_X$ has an $\w^\w$-base if $X$ has an $\w^\w$-base and $|X'|\le\w$.
\end{corollary}

\begin{example}[The Michael's line]\label{ex:michael} The \index{Michael's line}{\em Michael's line} $\IR_\IQ$ is the real line $\IR$ endowed with the topology generated by the base $\{(a,b):a<b\}\cup\big\{\{x\}:x\in \IR\setminus\IQ\big\}$, see \cite[5.1.22, 5.1.32, 5.5.2]{Eng}. It is well-known that $\IR_\IQ$ is a first-countable hereditarily paracompact space and $\IQ$ is a closed subset of $\IR_\IQ$ which is not $G_\delta$ in $\IR_\IQ$. So, $\IR_\IQ$ is not metrizable. Being first-countable, the Michael's line $\IR_\IQ$ has an $\w^\w$-base. Taking into account that the Michael's line $\IR_\IQ$ has countable set $\IQ$ of non-isolated points, we can apply Corollary~\ref{c:up-ww} and conclude that the universal preuniformity of $\IR_\IQ$ has an $\w^\w$-base. By Proposition~\ref{p:cu-Pb}, the universal uniformity of $\IR_\IQ$ has an $\w^\w$-base, too.
\end{example}

\begin{proposition}\label{p:dw} If  the universal preuniformity $p\U_X$ of a topological space $X$ has an $\w^\w$-base, then  $|X^{\prime P}|\le\w$ and $|X'|<\mathfrak d$.
\end{proposition}

\begin{proof} Let $\{U_\alpha\}_{\alpha\in\w^\w}$ be an $\w^\w$-base of the universal preuniformity $p\U_X$.

The inequality $|X^{\prime P}|\le\w$ follows from the reduction $\w^\w\succcurlyeq p\U_X\succcurlyeq \w^{X^{\prime P}}$ established in Theorem~\ref{t:pu-ww} and the non-reduction $\w^\w\not\succcurlyeq \w^{\w_1}$ proved in Proposition~\ref{p:e(X)}(2).

Next, we prove that $|X'|<\mathfrak d$. To derive a contradiction, assume that $|X'|\ge\mathfrak d$. Fix a cofinal subset $D\subset \w^\w$ of cardinality $|D|=\mathfrak d=\cof(\w^\w)$. Let $x:D\to X'$ be an injective map. Choose a neighborhood assignment $U\subset X\times X$ such that $U[x(\alpha)]\not\subset U_{\alpha}[x(\alpha)]$ for all $\alpha\in D$ (the choice of $U$ is possible as the points $x(\alpha)$, $\alpha\in D$, are non-singular). Since $(U_\alpha)_{\alpha\in D}$ is a base of the universal preuniformity $p\U_X\ni U$, there exists $\alpha\in D$ such that $U_\alpha\subset U$ and hence $U_\alpha[x(\alpha)]\subset U[x(\alpha)]$, which contradicts the choice of $U$. This contradiction shows that $|X'|<\mathfrak d$.
\end{proof}

Under $\mathfrak b=\mathfrak d$ we have the following characterization of topological spaces whose universal preuniformity has an $\w^\w$-base.

\begin{theorem}\label{t:pu-b=d} Under $\mathfrak b=\mathfrak d$ the universal preuniformity $p\U_X$ of a topological space $X$ has an $\w^\w$-base if and only if $X$ has an $\w^\w$-base, $|X'|<\mathfrak b=\mathfrak d$ and $|X^{\prime P}|\le\w$.
\end{theorem}

\begin{proof} The ``only if'' part was proved in Proposition~\ref{p:dw}. To prove the ``if'' part, assume that the space $X$ has an $\w^\w$-base, $|X'|<\mathfrak b=\mathfrak d$ and $|X^{\prime P}|\le\w$. By Corollary~\ref{c:P=b=d}, for every $x\in X'\setminus X^{\prime P}$, we get $\mathfrak b\cong \Tau_x(X)$. By Proposition~\ref{p:e(X)}(3), $\mathfrak e(\mathfrak b)\ge \cf(\mathfrak b)=\mathfrak b>|X'\setminus X^{\prime P}|$. Consequently, $\w^\w\succcurlyeq \mathfrak b^{X'\setminus X^{\prime P}}\cong\prod_{x\in X'\setminus X^{\prime P}}\Tau_x(X)$. Applying Theorem~\ref{t:pu-ww}, we conclude that
$$\w^\w\cong\w^\w\times(\w^\w)^\w\succcurlyeq\prod_{x\in X'\setminus X^{\prime P}}\Tau_x(X)\times \prod_{x\in X^{\prime P}} \Tau_x(X)\succcurlyeq p\U_X.$$
\end{proof}

Under $\mathfrak d=\w_1$ the characterization given in Theorem~\ref{t:pu-b=d}  simplifies to the following form (whose ZFC part is proved in Corollary~\ref{c:up-ww}).

\begin{theorem}  Under $\mathfrak d=\w_1$ the universal preuniformity $p\U_X$ of a topological space $X$ has an $\w^\w$-base if and only if $X$ has an $\w^\w$-base and the set $X'$ of non-singular points of $X$ is at most countable.
\end{theorem}

\section{Topological spaces with $\w^\w$-based universal quasi-uniformity}\label{s:qu-ww}
\index{universal quasi-uniformity}\index{$q\U_X$}
In this section we study topological spaces whose universal quasi-uniformity  has an $\w^\w$-base. We recall that the universal quasi-uniformity $q\U_X$ of a topological space $X$ is generated by the base consisting of the entourages $\bigotimes_{n\in\w}U_n$ where $(U_n)_{n\in\w}\in p\U_X^\w$ is a sequence of neighborhood assignments on $X$. This description of $q\U_X$ implies the following proposition.

\begin{proposition}\label{p:pu=>qu} The universal quasi-uniformity $q\U_X$ of a topological space $X$ has an $\w^\w$-base if the universal preuniformity $p\U_X$ of $X$ has an $\w^\w$-base.
\end{proposition}

\begin{corollary}\label{c:qu-ww'} The universal quasi-uniformity $q\U_X$ of a topological space $X$ has an $\w^\w$-base if the space $X$ has an $\w^\w$-base and the set $X'$ of non-singular points of $X$ is countable.
\end{corollary}

\begin{corollary} The universal quasi-uniformity $q\U_X$ of a topological space $X$ has an $\w^\w$-base if the space $X$ has an $\w^\w$-base, $|X'|<\mathfrak b=\mathfrak d$ and $|X^{\prime P}|\le\w$.
\end{corollary}

Proposition~\ref{p:u-sub}(2) implies

\begin{corollary}\label{c:qu-sub} If the universal quasi-uniformity of a topological space $X$ has an $\w^\w$-base, then for every closed subspace $Z\subset X$ the universal quasi-uniformity $q\U_Z$ has an $\w^\w$-base.
\end{corollary}

Now we our aim is to prove  that for any topological space $X$ with $\w^\w$-based universal quasi-uniformity $q\U_X$ all compact Hausdorff subspaces of $X$ are countable. To prove this result we need to establish some reductions involving the poset $q\U_X$.
\smallskip

  A function $f:X\to \w$ defined on a topological space $X$ is called \index{function!lower semicontinuous}{\em lower semicontinuous} if for any $n\in \w$ the set $\{x\in X:f(x)>n\}$ is open in $X$. By $LSC(X)$ we denote the set of all lower semicontinuous functions $f:X\to \w$. The set $LSC(X)$ is endowed with the partial order inherited from $\IR^X$.

For a topological space $X$ by $\K(X)$ we denote the hyperspace of all compact subsets endowed with the Vietoris topology. If $X$ is compact and metrizable, then so is the hyperspace $\K(X)$, see \cite[3.12.27, 4.5.23]{Eng}. The set $\K(X)$ is endowed with the inclusion partial order ($A\le B$ iff $A\subset B$). In the countable power $\K(X)^\w$ consider the subset
$$cov_\w(X)=\{(K_n)_{n\in\w}\in\K(X)^\w:X=\bigcup_{n\in\w}K_n\}$$
endowed with the partial order, inherited from $\K(X)^\w$.

\begin{lemma}\label{l:LSC>cov} For any compact Hausdorff space $X$ we get the reduction $LSC(X)\succcurlyeq cov_\w(X)$.
\end{lemma}

\begin{proof} The monotone cofinal map $LSC(X)\to cov_\w(X)$, $\varphi\mapsto \big(\varphi^{-1}\big[(-\infty,n]\big]\big)_{n\in\w}$, witnesses that $LSC(X)\succcurlyeq cov_\w(X)$.
\end{proof}

The proof of the following lemma was suggested by Zoltan Vidnyanszky\footnote{http://mathoverflow.net/questions/248023/is-the-space-of-countable-closed-covers-of-the-cantor-set-analytic}.

\begin{lemma}\label{l:nona} For the Cantor cube $X=2^\w$ the subspace $cov_\w(X)$ of the compact metrizable space $\K(X)^\w$ is not analytic.
\end{lemma}

\begin{proof} By \cite[14.2, 14.3]{Ke}, there exists a $G_\delta$-subset $G\subset X\times X$ whose projection $A=\{x\in X:\exists y\in X\;(x,y)\in G\}$ is not Borel. Write the complement $(X\times X)\setminus G$ as the countable union $\bigcup_{n\in\w}F_n$ of compact subsets $F_n$ of $X\times X$. It is easy to see that  for every $n\in\w$ the map $s_n:X\to\K(X)$, $s_n:x\mapsto \{y\in X:(x,y)\in F_n\}$, is Borel. Then maps $s_n$, $n\in\w$, determine a Borel map $s:X\to\K(X)^\w$, $s:x\mapsto (s_n(x))_{n\in\w}$. Observe that $$s^{-1}(cov_\w(X))=\{x\in X:X=\bigcup_{n\in\w}s_n(x)\}=X\setminus A.$$ Assuming that the set $cov_\w(X)$ is analytic in $\K(X)^\w$, we can apply Proposition 14.4 of \cite{Ke} and conclude that its preimage $X\setminus A=s^{-1}(cov_\w(X))$ is an analytic subset of $X$. By Souslin Theorem \cite[14.11]{Ke}, the set $A$ is Borel (being analytic and coanalytic). But this contradicts the choice of the set $G$.
\end{proof}

 \begin{lemma}\label{l:qu>LSC} For the Cantor cube $X=2^\w$ we get the reduction $q\U_X\succcurlyeq LSC(X)$.
 \end{lemma}

 \begin{proof} Identify the Cantor cube $2^\w$ with the standard Cantor set on $[0,1]$ and endow $X$ with the metric $d$ inherited from the real line $\IR$. This metric has the following property: for any $\e,\delta>0$ and $x\in X$ the inclusion $B_d(x,\delta)\subset B_d(x,\e)$ implies $\delta\le2\e$.  For every entourage $U\in q\U_X$ consider the function $n_U:X\to\w$ assigning to each point $x\in X$ the number $n_U(x)=\min\{n\in\w:B_d(x,2^{-n})\subset U[x]\}$. The function $n_U$ determines the lower-semicontinuous function $$\check n_U:X\to\w,\;\;\check n_U:x\mapsto \max_{O_x\in\Tau_x(X)}\min n_U(O_x).$$ It is clear that $\check n_U\le n_U$ and the map $\check n_*:q\U_X\to LSC(X)$, $\check n_*:U\mapsto\check n_U$, is monotone. It remains to prove that the map $\check n_*$ is cofinal.

Given any  lower semicontinuous function $\varphi\in LSC(X)$, consider the entourage $$U=\{(x,y)\in X\times X:y\in B_d(x,2^{-\varphi(x)-1})\}.$$

\begin{claim}\label{cl:UinqU} $U\in q\U_X$.
\end{claim}

\begin{proof}
For every $n\in \w$ consider the neighborhood assignment $$V_n=\{(x,y)\in X\times X:d(x,y)<\tfrac{2^{-\varphi(x)}}{2^{n+4}},\;\;\varphi(y)>\varphi(x)-\tfrac1{2^n}\}\in p\U_X.$$
The lower semicontinuity of the function $\varphi$ ensures that the entourage $V_n$ is a well-defined neighborhood assignment on $X$.
The definition of the universal quasi-uniformity $q\U_X$ guarantees that $\bigotimes_{n\in\w}V_n\in q\U_X$.

We claim that $\bigotimes_{n\in\w}V_n\subset U$. Given any pair $(x,y)\in \bigotimes_{n\in\w}V_n$, find $n\in\w$ and a permutation $\sigma\in S_n$ such that $(x,y)\in V_{\sigma(0)}\cdots V_{\sigma(n-1)}$. It follows that there exists a sequence of points $x=x_0,x_1,\dots,x_n=y$ such that $x_{i+1}\in V_{\sigma(i)}[x_i]$ for all $i\in n$. Then for every $i\in n$ the definition of the entourage $V_{\sigma(i)}$ yields the inequalities $$d(x_{i+1},x_i)<\frac{2^{-\varphi(x_i)}}{2^{\sigma(i)+4}}\mbox{ \ and \ }\varphi(x_{i+1})>\varphi(x_i)-\frac1{2^{\sigma(i)}},$$ which imply $$\varphi(x_{i+1})>\varphi(x_0)-\sum_{j\le i}\frac1{2^{\sigma(j)}}>\varphi(x)-2,\;\; d(x_{i+1},x_i)<\frac{2^{-\varphi(x_i)}}{2^{\sigma(i)+4}}<\frac{2^{2-\varphi(x)}}{2^{\sigma(i)+4}}=\frac{2^{-\varphi(x)}}{2^{\sigma(i)+2}}$$ and finally $$d(x,y)\le\sum_{i\in n}d(x_i,x_{i+1})<\sum_{i\in n}\frac{2^{-\varphi(x)}}{2^{\sigma(i)+2}}<2^{-\varphi(x)-1}.$$ Now we see that $(x,y)\in U$ and hence $\bigotimes_{n\in\w}V_n\subset U$ and $U\in q\U_X$.
\end{proof}

By Claim~\ref{cl:UinqU}, $U\in q\U_X$.  For every $x\in X$ the definition of the number $n=n_U(x)$ ensures that $B_d(x,2^{-n})\subset U[x]=B_d(x,2^{-\varphi(x)-1}$ and hence $2^{-n}\le 2\cdot 2^{-\varphi(x)-1}=2^{-\varphi(x)}$ according to the choice of the metric $d$.
Then $n_U(x)=n\ge \varphi(x)$. By the lower semicontinuity of the function $\varphi$, the point $x$ has a neighborhood $O_x\in\Tau_x(X)$ such that $\varphi[O_x]\subset[\varphi(x),+\infty)$. Then $n_U[O_x]\subset [\varphi(x),+\infty)$ and hence $\check n_U(x)\ge\inf n_U[O_x]\ge \varphi(x)$.
Therefore $\varphi\le \check n_U$, which completes the proof of the cofinality of the monotone map $\check n_*:q\U_X\to LSC(X)$.
\end{proof}

Now we can prove the promised theorem.

\begin{theorem}\label{t:qu-cc} If the universal quasi-uniformity $q\U_X$ of a compact Hausdorff space $X$ has an $\w^\w$-base, then $X$ is countable.
\end{theorem}

\begin{proof}  By Theorem~\ref{t:lqu-ww}(2), the compact Hausdorff space $X$ is metrizable. Assuming that $X$ is uncountable, we can apply \cite[6.2]{Ke} and find a subspace $Z\subset X$, homeomorphic to the Cantor cube $2^\w$. By Lemmas~{c:qu-sub}, \ref{l:LSC>cov} and \ref{l:qu>LSC}, $$\w^\w\succcurlyeq q\U_Z\succcurlyeq LSC(Z)\succcurlyeq cov_\w(Z).$$ So, we can fix a monotone cofinal map $f:\w^\w\to cov_\w(Z)$. For every $\alpha\in\w^\w$ consider the upper set $K_\alpha=\{y\in \K(Z)^\w:y\ge f(\alpha)\}\subset \K(Z)^\w$ and observe that $(K_\alpha)_{\alpha\in\w^\w}$ is a compact resolution of the subspace $ cov_\w(Z)$ of the compact metrizable space $\K(Z)^\w$. By Lemma~\ref{l:analytic}, the space $cov_\w(Z)$ is analytic, which contradicts Lemma~\ref{l:nona}. This contradiction shows that compact space $X$ is countable.
\end{proof}

\begin{theorem}\label{t:qu-ww} Assume that for a Tychonoff space $X$ the universal quasi-uniformity $q\U_X$ has an $\w^\w$-base. Then
\begin{enumerate}
\item[\textup{(1)}] the universal uniformity $\U_X$ of $X$ has an $\w^\w$-base.
\item[\textup{(2)}] the set $X^{\prime \css}$ is countable and closed in $X$.
\item[\textup{(3)}] If $\w_1<\mathfrak b$, then the set $X^{\prime P}$ is countable and closed in $X$.
\end{enumerate}
\end{theorem}

\begin{proof} 1. Proposition~\ref{p:cu-Pb} implies that  $\w^\w\cong(\w^\w)^\w\succcurlyeq (q\U_X)^\w\succcurlyeq \U_X$, which means that $\U_X$ has an $\w^\w$-base.
\smallskip

2. By Theorem~\ref{t:fb=>ms}, the space $Z=X^{\prime \css}$ is cosmic and closed in $X$. By Corollary~\ref{c:qu-sub}, the universal quasi-uniformity $q\U_Z$ of the space $Z$ has an $\w^\w$-base. By the first statement, the universal uniformity $\U_Z$ of $Z$ has an $\w^\w$-base. By Theorem~\ref{t:fb=>ms}(9), the separable space $Z$ is $\sigma$-compact. By Corollary~\ref{c:qu-sub} and Theorem~\ref{t:qu-cc}, the $\sigma$-compact space $Z=X^{\prime\css}$ is countable.
\smallskip

3. If $\w_1<\mathfrak b$, then by Theorem~\ref{t:w1ub}(1), the set $X^{\prime P}$ is closed in $X$ and the complement $X^{\prime P}\setminus X^{\prime \css}$ is at most countable. By the second statement, the set $X^{\prime \css}$ is countable and so is the set $X^{\prime P}$.
\end{proof}

\begin{corollary} For a Tychonoff\/ $\Sigma$-space $X$ with $|X'\setminus X^{\prime\css}|\le\w$ the following conditions are equivalent:
\begin{enumerate}
\item[\textup{(1)}] the universal pre-uniformity $p\U_X$ has an $\w^\w$-base;
\item[\textup{(2)}] the universal quasi-uniformity $q\U_X$ has an $\w^\w$-base;
\item[\textup{(3)}] the space $X$ has an $\w^\w$-base and the set $X'$ of non-isolated points of $X$ is at most countable.
\end{enumerate}
\end{corollary}

\begin{proof} The implications $(3)\Ra(1)\Ra(2)$ were proved in Corollary~\ref{c:up-ww} and Proposition~\ref{p:pu=>qu}, respectively. The implication $(2)\Ra(3)$ follows from  Theorem~\ref{t:qu-ww}(2).
\end{proof}

\section{Topological spaces with $\w^\w$-based universal uniformity}\label{s:u-ww}

\index{universal uniformity}\index{topological space!universal uniformity of}\index{$\U_X$}
In this section we shall detect topological spaces $X$ whose universal uniformity $\U_X$ has an $\w^\w$-base. Such spaces will be called \index{topological space!universally $\w^\w$-based}{\em universally $\w^\w$-based}. We recall that the universal uniformity $\U_X$ on $X$ is generated by the base consisting of the entourages $[d]_{<1}=\{(x,y)\in X\times X:d(x,y)<1\}$ where $d$ runs over the family of all continuous pseudometrics on $X$. For a paracompact space $X$ the universal uniformity $\U_X$ is generated by the family of entourages $\bigcup_{V\in\V}V\times V$ where $\V$ runs over open covers of $X$. If the universal uniformity $\U_X$ of a Tychonoff space $X$ has an $\w^\w$-base $\Bas$, then the base $\Bas$ is universal and $\IR$-regular. So, all the results proved in Sections~\ref{s:Ru}, \ref{s:Fat}, \ref{s:w1}, \ref{s:u-small} hold for universally $\w^\w$-based Tychonoff spaces.

Proposition~\ref{p:cu-Pb} and Corollary~\ref{c:up-ww} imply the following proposition.

\begin{proposition}\label{p:uu-ww'} If the universal preuniformity $p\U_X$ of a topological space has an $\w^\w$-base, then the universal uniformity $\U_X$ of $X$ has an $\w^\w$-base, too. Consequently, the universal uniformity $\U_X$ has an $\w^\w$-base if the topological space $X$ has an $\w^\w$-base and the set $X'$ of non-singular points of $X$ is countable.
\end{proposition}

Next we define a property  guaranteeing that an $\w^\w$-based topological space is universally $\w^\w$-based.

\begin{definition}\label{d:ks-port} A topological space $X$ is defined to be \index{topological space!$k_\sigma$-baseportating}\index{$k_\sigma$-baseportating} {\em $k_\sigma$-baseportating} if $X=\bigcup_{n\in\w}K_n$ for some sequence $(K_n)_{n\in\w}$ of compact subsets of $X$  such that for every $n\in\w$ there exists a point $e_n\in X$ and an indexed family $\big(t_{n,x}:\Tau_{e_n}(X)\to \Tau_x(X)\big)_{x\in K_n}$ of monotone cofinal maps $t_{n,x}:\Tau_{e_n}(X)\to \Tau_x(X)$ such that for every $x\in K_n$ and neighborhood $O_x\in\Tau_x(X)$ there exist neighborhoods $V_x\in\Tau_x(K_n)$ and $V\in\Tau_e(X)$ such that $\bigcup_{y\in V_x}t_{n,y}(V)\subset O_x$.
\end{definition}

This definition implies that each $k_\sigma$-baseportating space is $\sigma$-compact. Also each countable space is $k_\sigma$-baseportating. Each $\sigma$-compact locally quasi-uniform baseportator $(X,(t_x)_{x\in X})$ is $k_\sigma$-baseportating: write $X$ as the countable union $X=\bigcup_{n\in\w}K_n$ of compact sets and for every $n\in\w$ put $e_n$ be the unit of $X$ and $t_{n,x}:=t_x$ for all $x\in K_n$. Proposition~\ref{p:lqu-bport} guarantees that the  transport maps $t_{n,x}$ is semicontinuous at $(x,e)$ and has the continuity property required in Definition~\ref{d:ks-port}. In particular, each $\sigma$-compact para-topological portator is $k_\sigma$-baseportating.

 \begin{theorem}\label{t:local} A $k_\sigma$-baseportating Tychonoff space $X$ has an $\w^\w$-base if and only if its universal uniformity $\U_X$ has an $\w^\w$-base.
 \end{theorem}

  \begin{proof} The ``if'' part is trivial. To prove the ``only if'' part, assume that $X$ is $k_\sigma$-baseportating and write $X$ as the countable union $X=\bigcup_{n\in\w}K_n$ of compact sets such that for every $n\in\w$ there exists a point $e_n\in X$ and a family $t_n=(t_{n,x}:\Tau_{e_n}(X)\to\Tau_x(X))_{x\in K_n}$ of monotone cofinal maps having the continuity property required in Definition~\ref{d:ks-port}.

We are going to define a reduction $\prod_{n\in\w}\Tau_{e_n}(X)\succcurlyeq\U_X$. For every sequence $S=(V_n)_{n\in\w}\in\prod_{n\in\w}\Tau_{e_n}(X)$, consider the neighborhood $$
W_S=\bigcup_{n\in\w}\bigcup_{x\in K_n}t_{n,x}(V_n)^2\subset X\times X$$of the diagonal. Being $\sigma$-compact, the space $X$ is paracompact. Consequently, the set $W_S$ belongs to the universal uniformity $\U_X$ of $X$.

It is clear that the map $W_*:\prod_{n\in\w}\Tau_{e_n}(X)\to\U_X$, $W_*:S\mapsto W_S$, is monotone. It remains to show that it is cofinal.

Given any entourage $V\in\U_X$, find an entourage $E\in\U_X$ such that $EE^{-1}\subset V$. For every $n\in\w$ and every $x\in K_n$ the continuity property of the family $t_n$ yields neighborhoods $O_x\in\Tau_x(K_n)$ and $V_{n,x}\in\Tau_{e_n}(X)$ such that $\bigcup_{y\in O_x}t_{n,y}(V_{n,x})\subset E[x]$. By the compactness of $K_n$, the open cover $\{O_x:x\in K_n\}$ has a finite subcover $\{O_x:x\in F_n\}$. Put $V_n=\bigcap_{x\in F_n}V_{n,x}\in\Tau_{e_n}(X)$ and observe that for every $x\in K_n$ there exists $z\in K_n$ with $x\in O_z$, which implies that $t_{n,x}(V_n)\subset t_{n,x}(V_{n,z})\subset E[z]$ (by the choice of the neighborhoods $O_z$ and $V_{n,z}$).
Then for any points $u,v\in t_{n,x}(V_n)\subset E[z]$ we get $z\in E^{-1}[v]$ and $u\in EE^{-1}[v]$, which implies $(v,u)\in EE^{-1}$ and finally $t_{n,x}(V_n)\times t_{n,x}(V_n)\subset EE^{-1}$. Then for the sequence $S=(V_n)_{n\in\w}\in\prod_{n\in\w}\Tau_{e_n}(X)$ the set $W_S$ is contained in $EE^{-1}\subset V$, witnessing that the map $W_*$ is cofinal.

If the space $X$ has a neighborhood $\w^\w$-base at each point $e_n$, $n\in\w$, then we get the reduction $$\w^\w\cong(\w^\w)^\w\succcurlyeq\prod_{n\in\w}\Tau_{e_n}(X)\succcurlyeq \U_X,$$
witnessing that the universal uniformity $\U_X$ of $X$ has an $\w^\w$-base.
\end{proof}

Taking into account that $\sigma$-compact paratopological groups are $k_\sigma$-baseportating, we get the following corollary.

\begin{corollary} A $\sigma$-compact Tychonoff paratopological group (more generally, para-topological lop) $X$ has an $\w^\w$-base if and only if the universal uniformity $\U_X$ has an $\w^\w$-base.
\end{corollary}

This corollary is not true for quasi-topological groups. We recall that a \index{quasi-topological group}{\em quasi-topolo\-gical group} is a group $G$ endowed with a topology making the map $G\times G\to G$, $(x,y)\mapsto xy^{-1}$, separately continuous. It is clear that each quasi-topological group is a topologically homogeneous  space.

\begin{example}\label{e:cosmic} By \cite[4.11]{Ban}, there exists  a first-countable cosmic $\sigma$-compact quasi-topological group $X$, which is not an $\aleph_0$-space. Being first countable, the space $X$ has an $\w^\w$-base.  By Theorem~\ref{tc:small}, the space $X$ has no (locally) uniform $\w^\w$-base, which implies that the universal uniformity $\U_X$ of $X$ does not have $\w^\w$-base.
\end{example}

For metrizable $\sigma'$-compact spaces an $\w^\w$-base of the universal uniformity can be constructed without any homogeneity properties of the space.
 We recall that a topological space $X$ is \index{topological space!$\sigma'$-compact}{\em $\sigma'$-compact} if the set $X'$ of non-isolated points of $X$ is $\sigma$-compact. The following fact was proved in \cite{LPT}.

\begin{theorem}[Leiderman-Pestov-Tomita]\label{t:msigma'} For any metrizable $\sigma'$-compact space $X$ the universal uniformity $\U_X$ of $X$ has an $\w^\w$-base.
\end{theorem}

\begin{proof} For a convenience of the reader we present a short proof of this theorem. Fix a metric $d$ generating the topology of $X$ and for $\e>0$ and $x\in X$ let $B_d(x,\e)=\{y\in X:d(x,y)<\e\}$ be the $\e$-ball centered at $x$. Since $X$ is $\sigma'$-compact, the set $X'$ of non-isolated points of $X$ can be written as the union $X'=\bigcup_{n\in\w}K_n$ of an increasing sequences of compact sets in $X$. For every function $\alpha\in\w^\w$ consider the entourage $$U_\alpha=\Delta_X\cup\bigcup_{n\in\w}\bigcup_{x\in K_n}B_d\big(x,\tfrac1{2^{\alpha(n)}}\big)\times B_d\big(x,\tfrac1{2^{\alpha(n)}}\big).$$ Using the paracompactness of the metrizable space $X$, it can be shown that $(U_\alpha)_{\alpha\in\w^\w}$ is an $\w^\w$-base of the universal uniformity $\U_X$ of $X$.
\end{proof}

Combining Theorem~\ref{t:msigma'} with Theorem~\ref{t:fb=>ms}(9) we get the following characterization.

\begin{theorem}\label{t:ms'<=>uww} A metrizable space $X$ is $\sigma'$-compact if and only if its universal uniformity $\U_X$ has an $\w^\w$-base.
\end{theorem}

\begin{remark} Theorem~\ref{t:ms'<=>uww} can be compared with a results of Ginsburg \cite{Ginsburg} saying that  the universal uniformity $\U_X$ of a Tychonoff space $X$  has an $\w$-base if and only if $X$ is a metrizable space with compact set  $X'$ of non-isolated points.
\end{remark}

A map $f:X\to Y$ between topological spaces is called \index{map!$\IR$-quotient}{\em $\IR$-quotient} if for any function $\varphi:Y\to \IR$ the continuity of $\varphi$ is equivalent to the continuity of the composition $\varphi\circ f:X\to\IR$. It is clear that each quotient maps is $\IR$-quotient.

\begin{proposition}\label{p:Rquot} Let $f:X\to Y$ be an $\IR$-quotient map between Tychonoff spaces. Then $f$ is uniformly quotient with respect to the universal uniformities on the spaces $X$ and $Y$. Consequently, the space $Y$ is universally $\w^\w$-based if the space $X$ is universally $\w^\w$-based.
\end{proposition}

%{\color{red} This is Corolary 9.3 of \cite{BL}}.

\begin{proof}
 This proposition will follow from Proposition~\ref{p:uniquot} as soon as we check that the $\IR$-quotient map $f:X\to Y$ is uniformly quotient (with respect to the universal uniformities of the spaces $X$ and $Y$).

Given a continuous pseudometric $d_Y$ on $Y$ we need to show that the pseudometric $d_X=d_Y(f\times f)$ is continuous. By the triangle inequality, it suffices to show that for every point $x_0\in X$ the map $d_X(x_0,\cdot):X\to \IR$, $d_X(x_0,\cdot):x\mapsto d_X(x_0,x)=d_Y(f(x_0),f(x))$, is continuous. Since the map $f:X\to Y$ is $\IR$-quotient, the continuity of the map $d_Y(f(x_0),\cdot):Y\to\IR$ implies the continuity of the map $d_X(x_0,\cdot)=d_Y(f(x_0),\cdot)\circ f$.
\end{proof}

Combining Theorem~\ref{t:msigma'} with Proposition~\ref{p:Rquot}, we get the following theorem.

\begin{theorem}\label{t:s'} The universal uniformity $\U_X$ of a Tychonoff space $X$ has an $\w^\w$-base if $X$ is the image of a $\sigma'$-compact metrizable space $M$ under an $\IR$-quotient map $f:M\to X$.
\end{theorem}

This theorem motivates the following problem (which will be resolved affirmatively for La\v snev spaces in Theorem~\ref{t:Lasnev}).

\begin{problem}\label{prob:quots-s'} Assume that the universal uniformity $\U_X$ of a $\sigma$-compact Tychonoff space $X$ has an $\w^\w$-base. Is $X$ an $\IR$-quotient image of a $\sigma$-compact metrizable space?
\end{problem}

We shall say that a topological space $X$ \index{topology!inductive}{\em carries the inductive topology generated by}  a family $\C$ of subspaces of $X$ if a subset $U\subset X$ is open in $X$ if and only if for every $C\in\C$ the intersection $C\cap U$ is relatively open in $C$.

\begin{corollary}\label{c:limit} A Tychonoff space $X$ is universally $\w^\w$-based if $X$ carries the inductive topology generated by a countable cover $\C$ consisting of universally $\w^\w$-based subspaces of $X$.
\end{corollary}

\begin{proof}  Endow the family $\C$ with the discrete topology and consider the topological sum $\oplus\C=\bigcup_{C\in\C}\{C\}\times C\subset \C\times X$ of the family $\C$. Taking into account that $\U_{\oplus\C}\cong\prod_{C\in\C}\U_C$, we conclude that the space $\oplus\C$ is $\w^\w$-based. Since the projection $\pr:\oplus\C\to X$, $\pr:(C,x)\mapsto x$, is quotient and hence $\IR$-quotient, we can apply Theorem~\ref{t:s'} and conclude that the space $X$ is universally $\w^\w$-based.
\end{proof}

We recall that a topological space $X$ is called a \index{$k_\w$-space}{\em $k_\w$-space} if $X$ carries the inductive topology generated by a countable cover $\C$ consisting of compact subsets of $X$.
  Corollary~\ref{c:limit} and Theorem~\ref{t:msigma'} imply the following known fact (see \cite[1.14]{GabKak_2}), which improves Proposition~\ref{p:cosmic-kw}.

 \begin{corollary} Each cosmic $k_\w$-space is universally $\w^\w$-based.
 \end{corollary}

Now we evaluate  some cardinal characteristics of universally $\w^\w$-based spaces.
%We recall that a uniform space $X$ is called {\em discrete} if its uniformity $\U_X$ contains the diagonal $\Delta_X=\{(x,y)\in X\times X:x=y\}$ of the square $X\times X$. Observe that a uniform space $X$ is not discrete if and only if its uniformity $\U_X$ is an unbounded poset. In this case the cardinals $\add(\U_X)$ and $\cof(\U_X)$ are well-defined.
%Lemmas~\ref{l:b-bound} and \ref{l:b-reg} imply the following proposition.
%\begin{proposition}\label{p:b-bound} Each $\w^\w$-based non-discrete uniform space $(X,\U_X)$ has $$\add(\U_X),\cof(\U_X)\in\{\w\}\cup[\mathfrak b,\mathfrak d]\mbox{ \; and \; }\add(\U_X)\le\cf(\cof(\U_X))\le\cof(\U_X).$$
%\end{proposition}
First we remark that for a Tychonoff space $X$ the additivity $\add(\U_X)$ of the universal uniformity $\U_X$ is local in the following sense.

\begin{proposition}\label{p:b-local} For any non-discrete Tychonoff space $X$ we get
$$\add(\U_X)=\min\{\add(\Tau_x(X)):x\in X'\}.$$
\end{proposition}

\begin{proof} Consider the cardinal $\kappa=\min\{\add(\Tau_x(X):x\in X'\}\ge\w$ and find a (non-isolated) point $x\in X'$ such that $\add(\Tau_x(X))=\kappa$.

By the definition of the cardinal $\add(\Tau_x(x))$, there exists a family $(U_\alpha)_{\alpha\in\kappa}$ of neighborhoods of $x$ such that $\bigcap_{\alpha\in\kappa}U_\alpha$ is not a neighborhood of $x$. For every $\alpha\in\kappa$ choose an entourage $V_\alpha\in\U_X$ such that $V_\alpha[x]\subset U_\alpha$ and observe that $\bigcap_{\alpha\in\kappa}V_\alpha\notin\U_X$ witnessing that $\add(\U_X)\le\kappa$.

Now we prove that $\add(\U_X)\ge\kappa$. This inequality is clear if $\kappa=\w$.
So, we assume that $\kappa>\w$. Given any subfamily $\V\subset\U_X$ of cardinality $|\V|<\kappa$, we need to check that $\bigcap\V\in\U_X$. Replacing $\V$ by a largest subfamily of cardinality $<\kappa$, we can additionally assume that for every $V,V'\in\V$ there exists $U\in\V$ such that $U=U^{-1}$ and $U U\subset V\cap V'$. Then the intersection $E=\bigcap\V$ is an equivalence relation on $X$. For every non-isolated point $x\in X$ the inequality $|\V|<\kappa\le\add(\Tau_x(X))$ implies that $E[x]=\bigcap_{V\in\V}V[x]$ is a neighborhood of $x$. This means that the equivalence classes of  the relation $E$ are open and hence the quotient map $q:X\to X/E$ to the quotient space endowed with the discrete topology is continuous. It follows that the $\{0,1\}$-valued pseudometric $d:X\times X \to\{0,1\}$ defined by
$$d(x,y)=\begin{cases}1,&\mbox{if $x\in E(y)$},\\
0,&\mbox{otherwise},
\end{cases}$$is continuous and hence the entourage $\bigcap\V=E=\{(x,y)\in X:d(x,y)<1\}$ belongs to the universal uniformity $\U_X$ of $X$.
\end{proof}

By Theorems~\ref{t:fb=>ms}(7) and \ref{t:U+P} for every continuous map $f:X\to Y$ from a universally $\w^\w$-based Tychonoff space $X$ to a metric space $Y$ the image $f[X^{\prime\css}]$ is $\sigma$-compact and the image $f[X^{\prime P}]$ is separable. In the following Theorem~\ref{t:X'-d} we evaluate the density of the image $f[X']$. First we prove a lemma.

\begin{lemma}\label{l:X'-d} If a Tychonoff space $X$ is universally $\w^\w$-based, then for every entourage $U\in\U_X$ there exists a set $C\subset X'$ of cardinality $|C|<\mathfrak d$ such that $X'\subset U[C]$.
\end{lemma}

\begin{proof}  Given any entourage $U\in\U_X$, choose an entourage $V\in\U_X$ such that $V^{-1} V\subset U$. Using the Zorn Lemma, choose a maximal subset $C\subset X'$ such that $V[x]\cap V[y]=\emptyset$ for any distinct points $x,y\in X'$. By the maximality of $C$, for every $x\in X'$ there exists $y\in C$ such that $V[x]\cap V[y]\ne\emptyset$ and hence $x\in V^{-1}V[y]\subset U[y]$. So, $X'\subset U[C]$.

It remains to prove that $|C|<\mathfrak d$. To derive a contradiction, assume that $|C|\ge\mathfrak d$ and choose a transfinite sequence $(x_\alpha)_{\alpha\in\mathfrak d}$ of pairwise distinct points of the set $C$. Let $W\in\U_X$ be an entourage such that $W=W^{-1}$ and $WW\subset V$. Since the family $(V[x_\alpha])_{\alpha\in\mathfrak d}$ is disjoint,  the family $(W[x_\alpha])_{\alpha\in\mathfrak d}$ is discrete in $X$.

Let $D\subset\w^\w$ be a dominating set of cardinality $|D|=\mathfrak d$. By our assumption, the uniformity $\U_X$ has an $\w^\w$-base $\{U_\alpha\}_{\alpha\in\w^\w}$.  It follows that $\{U_\alpha\}_{\alpha\in D}$ is a base of the uniformity $\U_X$. Replacing each entourage $U_\alpha$ by $U_\alpha\cap W$, we can additionally assume that $U_\alpha\subset W$ for all $\alpha\in\w^\w$. In this case the family $(U_\alpha[x_\alpha])_{\alpha\in\mathfrak d}$ is discrete in $X$.

For every $\alpha\in D$ choose a point $y_\alpha\in U_\alpha[x_\alpha]\setminus\{x_\alpha\}$.
Such point $y_\alpha$ exists since the point $x_\alpha$ is not isolated in $X$.
Since the space $X$ is Tychonoff, there exists a continuous map $f_\alpha:X\to[0,1]$ such that $f_\alpha(x_\alpha)=1$ and $f_\alpha(z)=0$ for any point $z\in\{y_\alpha\}\cup (X\setminus U_\alpha[x_\alpha])$. Now consider the
continuous map $f:X\to \ell_2(\mathfrak d)$, $f:x\mapsto (f_\alpha(x))_{\alpha\in\mathfrak d}$, to the Hilbert space $\ell_2(\mathfrak d)$ of density $\mathfrak d$. This map is continuous since the family $(U_\alpha[x_\alpha])_{\alpha\in\mathfrak d}$ is discrete in $X$.   It follows that the entourage $E=\{(x,y)\in X\times X:\|f(x)-f(y)\|<1\}$ belongs to the universal uniformity $\U_X$ of $X$. Consequently, $U_\alpha\subset E$ for some $\alpha\in D$ and hence $\|f(y_\alpha)-f(x_\alpha)\|<1$. On the other hand, $1=|f_\alpha(x_\alpha)-f_\alpha(y_\alpha)|\le\|f(x_\alpha)-f(y_\alpha)\|<1$. This contradiction completes the proof of the strict inequality $|C|<\mathfrak d$.
\end{proof}

\begin{theorem}\label{t:X'-d} If the universal uniformity of a Tychonoff space $X$ has an $\w^\w$-base, then for any continuous map $f:X\to Y$ to a metric space $Y$ the image $f(X')$ has density $<\mathfrak d$.
\end{theorem}

\begin{proof} Let $d_Y$ be the metric of the metric space $Y$.  By the definition of the universal uniformity $\U_X$, for every $n\in\w$ the entourage $U_n=\{(x,x')\in X\times X:d(f(x),f(x'))<2^{-n}\}$ belongs to the uniformity $\U_X$. By Lemma~\ref{l:X'-d}, there exists a subset $C_n\subset X'$ of cardinality $|C'|<\mathfrak d$ such that $X'\subset U_n[C_n]$. Since $\cf(\mathfrak d)\ge\mathfrak b>\w$, the set $D=\bigcup_{n\in\w}f(C_n)$ has cardinality $|D|<\mathfrak d$. It is easy to see that the set $D$ is dense in $f(X')$, which implies that the set $f(X')$ has density $<\mathfrak d$.
\end{proof}

\begin{corollary} Under $\mathfrak d=\w_1$, for any universally $\w^\w$-based Tychonoff space $X$ the set $X'$ is $\w$-narrow in $X$.
\end{corollary}

\begin{corollary} If the universal uniformity of a Tychonoff space $X$ has an  $\w^\w$-base, then the space $X'$ of non-isolated points has topological weight $w(X')\le\mathfrak d$.
\end{corollary}

\begin{proof} Let $D\subset \w^\w$ be a dominating set of cardinality $|D|=\mathfrak d$.
To prove that $X'$ has topological weight $\le\mathfrak d$, for every $\alpha\in D$ choose a maximal subset $C_\alpha\subset X'$ such that $U_\alpha[x]\cap U_\alpha[y]=\emptyset$ for any distinct points $x,y\in C_\alpha$. Lemma~\ref{l:X'-d} implies that the set $C_\alpha$ has cardinality $|C_\alpha|<\mathfrak d$. Then the union $C'=\bigcup_{\alpha\in D}C_\alpha$ has cardinality $|C'|\le\mathfrak d$. It is easy to see that $\mathcal B=\{X'\cap U_\alpha[c]:\alpha\in D,\;c\in C'\}$ is a base of the topology of the space $X'$. Consequently, $w(X')\le|\mathcal B|\le\mathfrak d$.
\end{proof}

Lemma~\ref{l:X'-d} can be reformulated in the language of the sharp covering number $\cov^\sharp(X';\U_X)$ of the set $X'$ in $X$.

  For a subset $A$ of a uniform space $X$ the \index{$\cov^\sharp(X)$}{\em sharp covering number} $\cov^\sharp(A;\U_X)$ of $A$ in $X$ is defined as the smallest cardinal $\kappa$ such  for every entourage $U\in\U_X$ there exists a set $C\subset A$ of cardinality $|C|<\kappa$ such that $A\subset U[C]$.
 Lemma~\ref{l:X'-d} can be reformulated as follows.

  \begin{lemma}\label{l:nd} If the universal uniformity of a Tychonoff space $X$ has an $\w^\w$-base, then $\cov^\sharp(X';\U_X)\le\mathfrak d$.
  \end{lemma}

\chapter{$\w^\w$-Bases in $\sigma$-products of cardinals, La\v snev and GO-spaces}\label{Ch:ex}

In this chapter we construct some ``pathological'' examples of topological spaces with an $\w^\w$-base.
In Section~\ref{s:ex1} we detect $\w^\w$-bases in $\sigma$-products of cardinals and construct a consistent example of a non-separable Lindel\"of topological group whose universal uniformity has an $\w^\w$-base. Section~\ref{s:Lasnev} is devoted to detecting $\w^\w$-bases in La\v snev spaces and constructing an example of a countable sequential $\mathfrak P_0$-space without $\w^\w$-base. In Sections~\ref{s:uww-Lasnev}~--~\ref{s:uww-metr} we apply Theorem~\ref{t:Lasnev} to characterizing universally $\w^\w$-based La\v snev or metrizable spaces and detecting universally $\w^\w$-based $\sigma$-compact spaces. In Section~\ref{s:GO} we detect generalized ordered spaces with an $\w^\w$-base and in Section~\ref{s:ww-cardinal} prove some upper bounds on the cardinality of countably tight Hausdorff spaces with an $\w^\w$-base.

\section{$\w^\w$-Bases in $\sigma$-products of cardinals}\label{s:ex1}

Given a cardinal $\kappa$ and a transfinite sequence of cardinals  $(\lambda_i)_{i\in\kappa}$, consider the $\sigma$-product $$\mbox{$\bigodot_{i\in\kappa}\lambda_i=\big\{x\in \prod_{i\in\kappa}\lambda_i:|\{i\in\kappa:x(i)\ne 0|<\w\big\}$,}$$ endowed with the topology generated by the base consisting of the sets $$\mbox{$U_\alpha[x]=\{y\in \bigodot_{i\in\kappa}\lambda_i:y|\alpha=x|\alpha\}$, \;\;$x\in\bigodot_{i\in\kappa}\lambda_i,\;\;\alpha\in\kappa$}.$$
\index{$\bigodot_{i\in\w}\lambda_i$}
It is easy to see that the space $\bigodot_{i\in\kappa}\lambda_i$ is homeomorphic to a topological group.

If all cardinals $\lambda_i$, $i\in\kappa$, are equal to a fixed cardinal $\lambda$, then the $\sigma$-product $\bigodot_{i\in\kappa}\lambda_i$ will be denoted by \index{$\lambda^{\odot\kappa}$} $\lambda^{\odot\kappa}$. For $\kappa=\w_1$ the space $\lambda^{\odot\kappa}$ was introduced by Comfort and Ross in \cite{CR66} and then discussed in the book \cite[4.4.11]{AT}.

\begin{theorem}\label{t:contra} For any regular uncountable cardinal $\kappa$ and a transfinite sequence $(\lambda_i)_{i\in\kappa}$ of cardinals $1<\lambda_i<\kappa$ the space $X=\bigodot_{i\in\kappa}\lambda_i$ has the following properties:
\begin{enumerate}
\item[\textup{(1)}] For each open cover $\V$ there exists an ordinal $\alpha\in\kappa$ such that for each $x\in X$ the set $U_\alpha[x]=\{y\in X:y|\alpha=x|\alpha\}$ is contained in some set $V\in\V$;
\item[\textup{(2)}] The space $X$ is paracompact and the universal uniformity $\U_X$ of $X$ is generated by the base consisting of the sets $U_\alpha=\{(x,y)\in X\times X:x|\alpha=y|\alpha\}$ where $\alpha\in\kappa$.
\item[\textup{(3)}] $\U_X\cong \kappa\cong\Tau_x(X)$ for all $x\in X$.
\item[\textup{(4)}]  $p\U_X\cong \kappa^\kappa\cong q\U_X$.
%\item $X'=X$ and $\mathfrak b(\U(X))=\mathfrak d(\U(X))=\mathfrak n(X';X)=\kappa$.
\end{enumerate}
\end{theorem}

\begin{proof}  Let $\U$ be the uniformity on the space $X=\bigodot_{i\in\kappa}\lambda_i$, generated by the base $\{U_\alpha\}_{\alpha\in\kappa}$ where $U_\alpha=\{(x,y)\in X\times X:x|\alpha=y|\alpha\}$. It is clear that this uniformity generates the topology of $X$. %We shall prove that the space $X$ is paracompact and $\U$ coincides with the universal uniformity $\U(X)$ of $X$.
By the definition of the space $X=\bigodot_{i\in\kappa}\lambda_i$, every element $x\in X$ has finite support $\supp(x)=\{i\in\kappa:x(i)\ne0\}\subset \kappa$.
%\smallskip

1,2. Let $\V$ be any open cover of $X$. For any $x\in X$ find a set $V_x\in\V$ containing $x$ and an ordinal $i_x\in\kappa$ such that $U_{i_x}[x]\subset V_x$. Let $\alpha_0=0$ and $\alpha_{n+1}=\sup\{i_x+1:x\in X,\;\supp(x)\subset i_n\}$ for $n\in\w$. The regularity of the uncountable cardinal $\kappa$ ensures that the
set $\{x\in X:\supp(x)\subset \alpha_n\}$ has cardinality $|\bigodot_{i<\alpha_n}\lambda_i|<\kappa$, which implies that the ordinals $\alpha_n$, $n\in\w$, are well-defined. Since $\cf(\kappa)=\kappa>\w$, the ordinal $\alpha=\sup_{n\in\w}\alpha_n=\bigcup_{n\in\w}\alpha_n$ belongs to $\kappa$. It follows that for every $x\in X$ with $\supp(x)\subset \alpha$ we get $i_x<\alpha$ and hence $U_\alpha[x]\subset U_{i_x}[x]\subset V_x\in\V$. Then $\{U_\alpha[x]:x\in X,\;\supp(x)\subset\alpha\}$ is a disjoint open cover of $X$ refining the cover $\V$. This implies that the space $X$ is paracompact and the universal uniformity of $X$ coincides with the uniformity $\U$.
\smallskip

3. The description of the universal uniformity $\U_X$ of the space $X$ implies that $\U_X\cong\kappa\cong\Tau_x(X)$ for all $x\in X$.
\smallskip

4. Since $X'=X$, Theorem~\ref{t:pu-ww} guarantees that $p\U_X\cong\prod_{x\in X}\Tau_x(X)\cong\kappa^X\cong\kappa^\kappa$. Proposition~\ref{p:cu-Pb} implies that $\kappa^\kappa\cong(\kappa^\kappa)^\w\cong(p\U_X)^\w\succcurlyeq q\U_X$.
It remains to prove that $q\U_X\succcurlyeq\kappa^\kappa$.

 Identify each ordinal $\alpha\in\kappa$ with its characteristic function $\bar\alpha:\kappa\to\{0,1\}$ (so that $\bar\alpha^{-1}(1)=\{\alpha\}$). For each entourage $U\in q\U_X$ consider the function $\xi_U:\kappa\to\kappa$ assigning to each ordinal $\alpha\in\kappa$ the smallest ordinal $\xi_U(\alpha)\in\kappa$ such that $U_{\xi_U(\alpha)}[\bar \alpha]\subset U[\bar \alpha]$. It is clear that the map $\xi_*:q\U_X\to\kappa^\kappa$, $\xi_*:U\mapsto\xi_U$, is well-defined and monotone. To show that this map is cofinal, take any function $\varphi\in\kappa^\kappa$ and consider the function $\psi\in\kappa^\kappa$ defined by $\psi(\alpha)=\sup(\{\alpha\}\cup\{\varphi(\beta):\beta\le\alpha\})$ for $\alpha\in\kappa$. Let $\mu:X\to\kappa$ be the function assigning to each element $x\in X$ the largest element of the finite set $\{0\}\cup\supp(x)$.

Consider the neighborhood assignment $U=\{(x,y):y\in U_{\psi\circ\mu(x)}[x]\}$ and observe that $UU\subset U$ and hence $U\in q\U_X$. We claim that $\xi_U\ge\varphi$. Indeed, for every $\alpha\in\kappa$ we get $\psi\circ\mu(\bar\alpha)=\psi(\alpha)$ and hence $U[\bar\alpha]=U_{\psi(\alpha)}[\bar\alpha]$, which implies that $\xi_U(\alpha)=\psi(\alpha)\ge\varphi(\alpha)$. Therefore, the monotone map $\xi_*:q\U_X\to \kappa^\kappa$ is cofinal and hence $q\U_X\succcurlyeq\kappa^\kappa$.
\end{proof}

\begin{remark}
For any infinite cardinal $\kappa$ and a sequence $(\lambda_i)_{i\in\kappa}$ of non-zero at most countable cardinals $\lambda_i$ the space $\bigodot_{i\in\kappa}\lambda_i$ is homeomorphic to the closed subspace
$$\{f\in C_k(\kappa,\w):\forall i\in\kappa\;\;f(i+1)\in\lambda_i\}\cap\{f\in C_k(\kappa,\w):f(i)=0\mbox{ for any limit ordinal }i\in\kappa\}$$of the function space $C_k(\kappa,\w)$.
\end{remark}

Theorem~\ref{t:contra} allows us to construct the following example whose properties are similar to the properties of the function space $C_k(\w_1,\IZ)$ (see Corollary~\ref{c:CkZ}).

\begin{example}\label{ex:L-ww} For any transfinite sequence $(\lambda_i)_{i\in\w_1}$ of at most countable cardinals $\lambda_i>1$ the space $X=\bigodot_{i\in\w_1}\lambda_i$ has the following properties:
\begin{enumerate}
\item[\textup{(1)}] $X$ is a non-separable Lindel\"of $P$-space, homeomorphic to a topological group.
\item[\textup{(2)}] The universal uniformity $\U_X$ of $X$ has an $\w^\w$-base if and only if $\w_1=\mathfrak b$.
\item[\textup{(3)}] The universal preuniformity $p\U_X$ of $X$ has an $\w^\w$-base if and only if the universal quasi-uniformity $q\U_X$ of $X$ has an $\w^\w$-base if and only if $\w^\w\succcurlyeq (\w_1)^{\w_1}$.
\end{enumerate}
\end{example}

\begin{remark} By Theorem~\ref{t:e>w1}, the reduction $\w^\w\succcurlyeq(\w_1)^{\w_1}$ is consistent with ZFC. On the other hand, $\mathfrak e^\flat=\w_1$ (which follows from $\mathfrak b=\mathfrak d$) implies that $\w^\w\not\succcurlyeq(\w_1)^{\w_1}$.
\end{remark}

Theorem~\ref{t:contra} admits a self-generalization.

\begin{theorem}\label{t:Odot} For any uncountable regular cardinal $\kappa$ and a transfinite sequence of non-zero cardinals $(\lambda_i)_{i\in\kappa}$ such that $1<\lambda_i<\kappa$ for all non-zero $i\in\kappa$, the space $X=\bigodot_{i\in\kappa}\lambda_i$ and its universal uniformity $\U_X$ have the following properties:
\begin{enumerate}
\item[\textup{(1)}] $p\U_X\cong \kappa^{X}\cong q\U_X$,  $\U_X\cong \kappa^{\lambda_0}$ and $\Tau_x(X)\cong\kappa$ for all $x\in X$;
\item[\textup{(2)}] $\add(\U_X)=\kappa$, $\cof(\U_X)=\cof(\kappa^{\lambda_0})$, and $\cov^\sharp(X;\U_X)=\max\{\kappa,\lambda_0^+\}$;
\item[\textup{(3)}] $X'=X$, $X^{\prime P}=\emptyset$;
\item[\textup{(4)}] $X$ is a paracompact $P$-space, homeomorphic to a topological group.
%\item if $\w^\w\succcurlyeq \kappa^\delta$ and $\kappa\le\delta$, then $\cov(X';\U_X)\not\le \add(\U_X)$.
\end{enumerate}
\end{theorem}

\section{$\w^\w$-Bases in La\v snev spaces}\label{s:Lasnev}

In this section we detect La\v snev spaces with a (universal) $\w^\w$-base. We recall that a topological space $X$ is \index{topological space!La\v snev}\index{La\v snev space}{\em La\v snev} if $X$ is the image of a metrizable space $M$ under a closed continuous map $f:M\to X$. By \cite[5.4]{Grue}, the map $f$ can be assumed to be \index{map!irreducible}{\em irreducible}, which means that $f[M]=X$ but $f(Z)\ne X$ for any closed subset $Z\ne M$ of $M$. By \cite[\S5]{Grue}, La\v snev spaces are Fr\'echet-Urysohn, stratifiable and perfectly paracompact. Separable La\v snev spaces are Fr\'echet-Urysohn $\aleph_0$-spaces \cite[11.3]{Grue} and hence  $\mathfrak P_0$-spaces, see \cite[1.10]{Ban}.

\begin{theorem}\label{t:Lasnev} Let $f:X\to Y$ be a closed irreducible map of a metrizable space $X$ onto a La\v snev space $Y$.
\begin{enumerate}
\item[\textup{(1)}] $Y$ is first-countable at a point $y\in Y$ if and only if $f^{-1}(y)$ is compact.
\item[\textup{(2)}] $Y$ is first-countable if and only if $Y$ is metrizable if and only if  $f^{-1}(y)$ is compact for every $y\in Y$.
\item[\textup{(3)}] $Y$ carries the inductive topology at a point $y\in Y$ with respect to a countable cover of $Y$ by (closed) subspaces which are first countable at $y$ if and only if $f^{-1}(y)$ is locally compact and $\sigma$-compact.
\item[\textup{(4)}] $Y$ carries the inductive topology with respect to a countable cover by first-countable (metrizable) closed subspaces if and only if $X=\bigcup_{n\in\w}X_n^\circ$ for an increasing sequence $(X_n)_{n\in\w}$ of closed subsets of $X$ such that $X_n\cap f^{-1}(y)$ is compact for every $n\in\w$ and $y\in Y$.
\item[\textup{(5)}] $Y$ has a neighborhood $\w^\w$-base at a point $y\in Y$ if and only if $f^{-1}(y)$ is $\sigma$-compact.
\item[\textup{(6)}] $Y$ is universally $\w^\w$-based if and only if $X$ is $\sigma'$-compact only if $Y$ is $\sigma'$-compact.
\end{enumerate}
\end{theorem}

\begin{proof} %Let $d$ be a metric generating the topology of the space $X$.
The irreducibility of $f$ implies that for every isolated point $y\in Y$ the preimage $f^{-1}(y)$ is a singleton and for every non-isolated point $y\in Y$ the preimage $f^{-1}(y)$ is nowhere dense in $X$. Fix any function $s:Y\to X$ such that $f\circ s$ is the identity map of $Y$. The irreducibility of the map $f$ ensures that the set $s(Y)$ is dense in $X$.
\smallskip

1. Assume that for some point $y\in Y$ the preimage $K=f^{-1}(y)$ is compact.
For every $n\in\w$ consider the open neighborhood $U_n=\{x\in X:\exists z\in f^{-1}(y)\;\;d(z,x)<\tfrac1{2^n}\}$ of $f^{-1}(y)$ in $X$ and the open neighborhood $V_n=Y\setminus f[X\setminus U_n]$ of $y$ in $Y$. We claim that $\{V_n\}_{n\in\w}$ is a neighborhood base at $y$. Indeed, for any open neighborhood $V\subset Y$ of $y$, the preimage $f^{-1}[V]$ is an open neighborhood of $f^{-1}(y)$. By the compactness of $f^{-1}(y)$, there is $n\in\w$ such that $U_n\subset f^{-1}[V]$ and then $V_n\subset V$. Therefore, $\{V_n\}_{n\in\w}$ is a countable neighborhood base at $y$ and $Y$ is first-countable at $y$. This completes the proof the ``if'' part of the first statement.

To prove the ``only if'' part, assume that $Y$ is first-countable at some point $y\in Y$. If the point $y$ is isolated in $Y$, then $f^{-1}(y)$ is compact, being a singleton. So, we assume that $y$ is not isolated in $Y$. The irreducibility of $f$ ensures that the preimage $f^{-1}(y)$ is nowhere dense in $X$ and hence is contained in the closure of the set $f^{-1}[Y\setminus\{y\}]$. The following claim guarantees that the set $f^{-1}(y)$ is compact.

\begin{claim}\label{cl:comp} Assume that a subspace $Z\subset Y$ is first-countable at some point $y\in Z$. Then the set $P=f^{-1}(y)\cap\overline{f^{-1}[Z\setminus\{y\}]}$ is compact.
\end{claim}

\begin{proof} Assuming that the set $P$ is not compact, we can find  a discrete family $(U_n)_{n\in\w}$ of open subsets of $X$ intersecting $P$. For every $n\in\w$ choose a point $x_n\in U_n\cap P$. Since $P$ is contained in the closure of $f^{-1}[Y\setminus\{y\}]$, the point $x_n$ is the limit of some sequence $\{x_{n,m}\}_{m\in\w}\subset  U_n\cap f^{-1}[Z\setminus\{y\}]$. Being first-countable at the point $y$, the space $Z$ has a countable neighborhood base $\{V_n\}_{n\in\w}$ at $y$. For every $n\in\w$ the preimage $f^{-1}[V_n]$ is an open neighborhood of the set $f^{-1}(y)\supset\{x_n\}_{n\in\w}$ in $f^{-1}[Z]\supset\{x_{n,m}:n,m\in\w\}$. So, for every $n\in\w$ we can find a number $m_n\in\w$ such that $x_{n,m_n}\in f^{-1}[V_n]$. Since the family $(U_n)_{n\in\w}$ is discrete in $X$, the set $D=\{x_{n,m_n}:n\in\w\}$ is closed in $X$ and the set $V=Y\setminus f[D]$ is an open neighborhood of $y$ in $Y$ such that $V\cap Z\not\subset V_n$ for all $n\in\w$.   But this contradicts the choice of  $(V_n)_{n\in\w}$ as a neighborhood base of $Z$ at $y$.
\end{proof}
\smallskip

2. By Theorem~\ref{t:Lasnev}(1), the space $Y$ is first-countable if and only if for every $y\in Y$ the preimage $f^{-1}(y)$ is compact. In this case the map $f$ is perfect and the space $Y$ is metrizable by \cite[4.4.15]{Eng}.
\smallskip

3. If for some point $y\in Y$ the space $f^{-1}(y)$ is locally compact and $\sigma$-compact, then using the normality of the metrizable space $X$, we can find a continuous map $\hbar:X\to[0,\infty)$ such that for every $n\in \w$ the set $\hbar^{-1}\big[[0,n]\big]\cap f^{-1}(y)$ is compact and non-empty. For every $n\in\w$ consider the closed subset $X_n:=\hbar^{-1}\big[[0,n]\big]$ of $X$ and its image $Y_n=f[X_n]$, which is a closed subset of $Y$. By (the proof of) the first item, the compactness of the set $X_n\cap f^{-1}(y)$ implies that the space $Y_n$ is first countable at $y$. We claim that the space $Y$ has the inductive topology at $y$ with respect to the closed countable cover $\{Y_n\}_{n\in\w}$. We need to prove that a subset $V\subset Y$ is a neighborhood of $y$ if for every $n\in\w$ the intersection $V\cap Y_n$ is a neighborhood of $y$ in $Y_n$. This means that for every $n\in\w$ there exists an open neighborhood $V_n\subset Y$ of $y$ such that $V_n\cap Y_n\subset V$. It follows that $f^{-1}[V_n]\cap X_n\subset f^{-1}[V]$ and hence $X_n\cap f^{-1}[V]$ is a neighborhood of $X_n\cap f^{-1}(y)$ in $X_n$.
Since $X_n$ is contained in the interior of $X_{n+1}$, for every $x\in X_n$ we can find a neighborhood $O_x\subset X$ of $x$ such that $O_x\subset X_{n+1}\cap f^{-1}(V)$. Then the interior $U$ of the set $f^{-1}[V]$ contains the set $f^{-1}(y)$  and then $V\supset Y\setminus f[X\setminus U]$ is a neighborhood of $y$ in $Y$.

Now assume that $Y$ has the inductive topology at a point $y\in Y$ with respect to a countable cover $(Y_n)_{n\in\w}$ of $Y$ by closed subsets that contain $y$ and are first countable at $y$. If the point $y$ is isolated in $Y$, then the preimage $f^{-1}(y)$ is a singleton and hence is locally compact and $\sigma$-compact. So, we assume that $y$ is not isolated.
Replacing each set $Y_n$ by $\bigcup_{i\le n}Y_i$, we can assume that the sequence $(Y_n)_{n\in\w}$ is increasing. Since the point $y$ is not isolated in $Y$, it is not isolated in $Y_n$ for some $n\in\w$. Replacing the sequence $(Y_i)_{i\in\w}$ by the sequence $(Y_i)_{i\ge n}$ we can assume that $y$ is not isolated in each space $Y_n$.

 For every $n\in\w$ let $X_n$ be the closure of the set $f^{-1}[Y_n\setminus\{y\}]$. By Claim~\ref{cl:comp},  the set $X_n\cap f^{-1}(y)$ is compact. The local compactness and $\sigma$-compactness of $f^{-1}(y)$ will follow as soon as we check that each point $x\in f^{-1}(y)$ has a neighborhood, contained in some set $X_n$. In the opposite case, for every $n\in\w$ we can construct a sequence $\{x_{n,m}\}_{m\in\w}\in X\setminus X_n$ that converges to $x$. Since $f^{-1}(y)$ is nowhere dense in $X$ we can additionally assume that $\{x_{n,m}\}_{m\in\w}\cap f^{-1}(y)=\emptyset$. Using the first-countability of $X$ at $x$, we can choose a number sequence $(m_n)_{n\in\w}\in\w^\w$ such that the diagonal sequence $(x_{n,m_n})_{n\in\w}$ converges to $x$. The continuity of $f$ guarantees that the set $F=\{f(x_{n,m_n})\}_{n\in\w}$ contains the point $y=f(x)$ in its closure. On the other hand, for every $n\in\w$ the intersection $F\cap Y_n$ is finite, which implies that $Y_n\setminus F$ is a neighborhood of $y$ in $Y_n$. Since $Y$ has the inductive topology at $y$ with respect to the cover $\{Y_n\}_{n\in\w}$, the set $Y\setminus F$ is a neighborhood of $y$, which contradicts the inclusion $y\in\bar F$. This contradiction completes the proof of the local compactness and $\sigma$-compactness of the preimage $f^{-1}(y)$.
 \smallskip

4. Assume that $X=\bigcup_{n\in\w}X_n^\circ$ for an increasing sequence $(X_n)_{n\in\w}$ of closed subsets of $X$ such that $X_n\cap f^{-1}(y)$ is compact for every $n\in\w$ and $y\in Y$. This means that for every $n\in\w$ the restriction $f|X_n:X_n\to Y$ is a perfect map and by \cite[4.4.15]{Eng}, the image $Y_n=f[X_n]$ is a closed metrizable subset of $X$. Repeating the argument from the preceding case, we can prove that the space $Y$ carries the inductive topology with respect to the cover $(Y_n)_{n\in\w}$.

Now assume conversely that the space $Y$ carries the inductive topology with respect to a countable cover $(Y_n)_{n\in\w}$ of $Y$ by closed first-countable subspaces of $X$. Remind that $s:Y\to X$ is a function such that $f\circ s$ is the identity map of $Y$. For every $n\in\w$ let $X_n$ be the closure of the set $s[Y_n]$ in $X$. We claim that for every $y\in Y$ the set $f^{-1}(y)\cap X_n$ is compact. If the point $y$ does not belong to $Y_n$, then the set $f^{-1}(y)\cap X_n\subset f^{-1}(y)\cap f^{-1}[Y_n]=\emptyset$ is empty and hence compact. If $y$ is isolated in $Y_n$, then the set $f^{-1}(y)\cap X_n=\{s(y)\}$ is a singleton and hence compact.  If $y$ is not isolated in $Y_n$, then the set $f^{-1}(y)\cap X_n\subset \{s(y)\}\cup (f^{-1}(y)\cap\overline{s[Y_n\setminus\{y\}]})$ is compact by Claim~\ref{cl:comp}. It remains to prove that $X=\bigcup_{n\in\w}X_n^\circ$. Assuming the opposite, we could find a point $x\in X$ such that for each neighborhood $O_x\subset X$ the sets $O_x\setminus X_n$, $n\in\w$, are not empty. The density of $s[Y]$ in $X$ implies that the sets $O_x\cap s[Y]\setminus X_n$, $n\in\w$, also are not empty.  This allows us to construct a sequence $(x_n)_{n\in\w}\in\prod_{n\in\w} s[Y]\setminus X_n$, convergent to $x$.
Then the sequence $\big\{f(x_n)\big\}_{n\in\w}\subset Y$ converges to $y$.

We claim that for every $n\in\w$ the point $f(x_n)$ does not belong to $Y_n$. In the opposite case, $s\circ f(x_n)\subset X_n$. Since $x_n\in s[Y]$, there exists a point $y_n\in Y$ such that $x_n=s(y_n)$. Applying to this equality the map $f$, we get $f(x_n)=f\circ s(y_n)=y_n$ and hence $x_n=s(y_n)=s\circ f(x_n)\in X_n$, which contradicts the choice of the point $x_n$. This contradiction shows that $f(x_n)\notin X_n$ for all $n\in\w$. For every $n\in\w$ consider the set $D_n=\{f(x_i)\}_{i\ge n}$ and observe that $D_n\cap Y_m\subset \{f(x_i):n\le i<m\}$ is finite and hence closed in $Y_m$. Since $X$ has the inductive topology with respect to the cover $(Y_m)_{m\in\w}$, for every $n\in\w$ the set $D_n$ is closed in $Y$. Since $(Y_m)_{m\in\w}$ is a cover of $Y$, for some $n\in\w$ the set $Y_n$ contains the point $f(x)$. Then $f(x)\notin D_n$ and hence $Y\setminus D_n$ is an open neighborhood of $y$ in $Y$, which contradicts the convergence of the sequence $D_n=\{f(x_k)\}_{k\ge n}$ to $f(x)$. This contradiction completes the proof of the inclusion $X=\bigcup_{n\in\w}X_n^\circ$ and also the proof of the ``only if'' part of the fourth statement.
\smallskip

5. To prove the ``if'' part of the fifth statement, assume that the preimage $f^{-1}(y)$ of some point $y\in Y$ is $\sigma$-compact. Then $f^{-1}(y)$ can be written as the union $\bigcup_{n\in\w}K_n$ of an increasing sequence $(K_n)_{n\in\w}$ of compact subsets of $X$. Fix any metric $d$ generating the topology of the metrizable space $X$. For every function $\alpha\in\w^\w$ consider the neighborhood $$U_\alpha=\bigcup_{n\in\w}\bigcup_{x\in K_n}B_d\big(x,\tfrac1{2^{\alpha(n)}}\big)$$
of $f^{-1}(y)$ and the neighborhood
$V_\alpha=Y\setminus f[X\setminus U_\alpha]$ of $y$.
It follows that for every functions $\alpha\le\beta$ in $\w^\w$ we have the inequalities $U_\beta\subset U_\alpha$ and $V_\beta\subset V_\alpha$.

Using the compactness of the sets $K_n$, $n\in\w$, it can be shown that each neighborhood of $f^{-1}(y)$ in $X$ contains some neighborhood $U_\alpha$, $\alpha\in\w^\w$, which implies that each neighborhood $V\subset Y$ of $y$ contains some neighborhood $V_\alpha$, $\alpha\in\w^\w$. This means that $(V_\alpha)_{\alpha\in\w^\w}$ is a neighborhood $\w^\w$-base at $y$.

Now we prove the ``only if'' part of the fifth statement. Assuming that the space $Y$ has a neighborhood $\w^\w$-base at some point $y\in Y$, we shall prove that the preimage $f^{-1}(y)$ is $\sigma$-compact. If the point $y$ is isolated in $Y$, then by the irreducibility of the map $f$, the preimage $f^{-1}(y)$ is a singleton and hence is $\sigma$-compact. So, assume that $y$ is not isolated in $Y$. In this case the set $f^{-1}(y)$ is nowhere dense in $X$. First we show that this set is separable. Assuming that it is not separable, we can find a discrete family $(U_\gamma)_{\gamma\in\w_1}$ of open subsets of $X$, intersecting the non-separable set $f^{-1}(y)$. For every $\gamma\in\w_1$ choose a point $x_\gamma\in U_\gamma\cap f^{-1}(y)$ and a sequence $\{x_{\gamma,n}\}\subset s[Y]\cap U_\gamma\setminus f^{-1}(y)$, convergent to $x_\gamma$. For every neighborhood $V\in\Tau_y(Y)$ of $y$ consider the function $\varphi_V:\w_1\to\w$ assigning to each $\gamma\in\w_1$ the smallest number $n_\gamma\in\w$ such that $\{x_{\gamma,n}\}_{n\ge n_\gamma}\subset f^{-1}[V]$.
It is clear that the correspondence $\varphi_*:\Tau_y(Y)\mapsto\w^{\w_1}$, $\varphi_V:V\mapsto \varphi_V $, is monotone. We claim that is also cofinal.
Given any function $\psi\in\w^{\w_1}$, consider the closed subset $F=\{x_{\gamma,n}:\gamma\in\w_1,\;n\le\psi(\gamma)\}\subset s[Y]\setminus f^{-1}(y)$. Since the map $f$ is closed, the set $f[F]$ is closed in $Y$ and the set $V=Y\setminus f[F]$ is an open neighborhood of $y$. Taking into account that the sets $f^{-1}[V]$ and $F$ are disjoint, we conclude that $\varphi_V\ge \psi$, which completes the proof of the cofinality of the map $\varphi_*:\Tau_y(Y)\to\w^{\w_1}$. Therefore $\Tau_y(Y)\succcurlyeq \w^{\w_1}$ and by Proposition~\ref{p:e(X)}(2), $\w^\w\not\succcurlyeq \Tau_y(Y)$, which is a desired contradiction showing that the set $f^{-1}(y)$ is separable.

Taking into account that the closed separable subset $F_y:=f^{-1}(y)$ is nowhere dense in $X$, we can find a closed separable subspace $S\subset X$ containing $F_y$ as a nowhere dense subset. Let $\bar S$ be any metrizable compactification of the separable metrizable space $S$. Then the closure $\bar F_y$ of $F_y$ in $\bar S$ is a compact nowhere dense subset of $\bar S$.

By our assumption, the space $Y$ has a neighborhood $\w^\w$-base $\{U_\alpha\}_{\alpha\in\w^\w}$ at $y$. For every $\alpha\in\w^\w$ consider the neighborhood $V_\alpha=f^{-1}[U_\alpha]$ of $F_y$. It follows that  $K_\alpha:=\overline{F}_y\cap\overline{S\setminus V_\alpha}$ is a compact subset of $\bar F_y\setminus F_y$. Moreover, $K_\alpha\subset K_\beta$ for all $\alpha\le\beta$ in $\w^\w$. It remains to check that the family $(K_\alpha)_{\alpha\in\w^\w}$ is cofinal in the pospace $\K(\bar F_y\setminus F_y)$. Given any compact set $K\subset \bar F_y\setminus F_y=\bar F_y\setminus S$, we should find $\alpha\in\w^\w$ such that $K\subset K_\alpha$.

Using the nowhere density of the set $F_y$ in $S$, find a closed discrete subset $D\subset S$ such that $D\cap F_y=\emptyset$ and $K\subset \bar D=K\cup D$. Since the map $f$ is closed, the image $f[D]$ is closed in $Y$ and $U=Y\setminus f[D]$ is an open neighborhood of $y$. Choose $\alpha\in\w^\w$ such that $U_\alpha\subset U$. It follows from $U_\alpha\cap f[D]=\emptyset$ that $V_\alpha\cap D\subset f^{-1}[U_\alpha\cap D]=\emptyset$. Then $D\subset S\setminus V_\alpha$ and $K\subset \bar F_y\cap \bar D\subset \overline{F}_y\cap\overline{S\setminus V_\alpha}=K_\alpha$.

The monotone cofinal family $(K_\alpha)_{\alpha\in\w^\w}$ witnesses that $\w^\w\succcurlyeq \K(\bar F_y\setminus F_y)$. By Christensen's Theorem~\ref{t:Chris}, the space $\bar F_y\setminus F_y$ is Polish and hence is a $G_\delta$-set in the compact space $\bar F_y$. Then its complement $F_y$ is $\sigma$-compact.
\vskip5pt

6. If the metrizable space $X$ is $\sigma'$-compact, then by Theorem~\ref{t:s'}, the space $Y$ is universally $\w^\w$-based being a closed (and hence $\IR$-quotient) image of the metrizable $\sigma'$-compact space $X$. The irreducibility of the map $f$ ensures that $f[X']=Y'$ and hence $Y'$ is $\sigma$-compact (being a continuous image of the $\sigma$-compact space $X'$).

Now assume that the space $Y$ is universally $\w^\w$-based. Then $Y$ has a neighborhood $\w^\w$-base at each point $y\in Y$, and by the preceding statement the preimage $f^{-1}(y)$ is $\sigma$-compact. Since the La\v snev space $Y$ is Fr\'echet-Urysohn, the set $Y'$ of non-isolated points of $Y$ coincides with the set $Y^{\prime\mathsf s}$ of accumulation points of countable subsets of $Y$. By Theorem~\ref{t:fb=>ms}, the set $Y'=Y^{\prime\mathsf s}$ is cosmic and $\sigma$-compact. By Lemma~\ref{l:para'}, the space $Y$ is paracompact.

The irreducibility of $f$ implies that $f^{-1}[Y']=X'$.
We claim that the set $X'$ of non-isolated points in $X$ is separable.
This will follow as soon as we check that each closed discrete subset $D\subset X'$ is at most countable. Taking into account that the map $f$ is closed, we conclude that the image  $f[D]$ is a closed discrete subset of $Y'$. Since $Y'$ is cosmic, the set $f[D]$ is countable.
For every $y\in f[D]$ the set $D\cap f^{-1}(y)$ is at most countable (being a closed discrete subspace of the $\sigma$-compact set $f^{-1}(y)$). Then the set $D=\bigcup_{y\in f[D]}D\cap f^{-1}(y)$ is at most countable and the set $X'$ is separable.

Choose a closed separable subspace $S\subset X$ that contain the separable set $X'=f^{-1}[Y']$ so that for every $y\in Y'$ the nowhere dense subset $f^{-1}(y)$ of $X$ remains nowhere dense in the space $S$.
Fix any metrizable compactification $\bar S$ of the separable metrizable space $S$. Let $\{U_\alpha\}_{\alpha\in\w^\w}$ be an $\w^\w$-base of the universal uniformity $\U_Y$ of the space $Y$. For every $\alpha\in\w^\w$ and every $y\in Y'$ let $V_{\alpha,y}:=f^{-1}\big[U_\alpha[y]\big]$ and let $K_{\alpha,y}$ be the closure of the set $S\setminus V_{\alpha,y}$ in $\bar S$. Then $\tilde V_{\alpha,y}=\bar S\setminus K_{\alpha,y}$ is an open neighborhood of $f^{-1}(y)$ in $\bar S$. It follows that $K_\alpha=\bigcap_{y\in Y'}K_{\alpha,y}=\bar S\setminus\bigcup_{y\in Y'}\tilde V_{\alpha,y}$ is a closed subset of $\bar S$, disjoint with $X'$. So, $K_\alpha$ belongs to the poset $\K(\bar S\setminus X')$ of compact subsets of $\bar S\setminus X'$.

For every functions $\alpha\le\beta$ in $\w^\w$ and every $y\in Y'$ the inclusion $U_\beta\subset U_\alpha$ implies $U_\beta[y]\subset U_\alpha[y]$, $V_{\beta,y}\subset V_{\alpha,y}$, and $K_{\alpha,y}=\overline{S\setminus V_{\alpha,y}}\subset\overline{S\setminus V_{\beta,y}}=K_{\beta,y}$. Then $K_\alpha=\bigcap_{y\in Y'}K_{\alpha,y}\subset\bigcap_{y\in Y'}K_{\beta,y}= K_\beta$. This means that the family $(K_\alpha)_{\alpha\in\w^\w}$ is monotone. We claim that this family is cofinal in $\K(\bar S\setminus X')$.

  Given any compact set $K\subset \bar S\setminus X'$, we should find $\alpha\in\w^\w$ such that $K\subset K_\alpha$.  In the proof of the fifth statement it was shown that for every point $y\in Y'$ there is $\alpha_y\in\w^\w$ such that  $K\subset K_{\alpha_y,y}$.
%It follows that the union $W=\bigcup_{y\in Y}U_{\alpha_y}[y]\times U_{\alpha_y}[y]$  is a neighborhood of the diagonal $\Delta_Y$ in $Y\times Y$. By the paracompactness of $Y$ (which follows from the cosmicity of $Y$), there entourage $W$ belongs to the universal uniformity $\U(Y)$ of $Y$. Then we can find $\alpha\in\w^\w$ such that $U_\alpha^2\subset W$. We claim that $K\subset K_\alpha$, which is equivalent to $K\cap \bar V^\circ_{\alpha,y}=\emptyset$ for all $y\in Y$. Given any $y\in Y$

Using the paracompactness of the space $Y$,  find an open cover $\V$ of $Y$ that star-refines the cover $\U'=\{U_{\alpha_y}[y]:y\in Y'\}\cup\{\{y\}:y\in Y\setminus Y'\}$ of $Y$. This means that for every $V\in\V$ there is a set $U'\in\U'$ containing the $\V$-star $\St(V;\V)=\{V'\in\V:V\cap V'\ne\emptyset\}$ of $V$. By the paracompactness of $Y$ the entourage $\bigcup_{V\in\V}V\times V$ belongs to the universal uniformity $\U_Y$ of $Y$. Since $(U_\alpha)_{\alpha\in\w^\w}$ is a base of the uniformity $\U_Y$, there exists $\alpha\in\w^\w$ such that $U_\alpha\subset \bigcup_{V\in\V}V\times V$.

We claim that $K\subset K_\alpha=\bigcap_{y\in Y'}K_{\alpha,y}$. It suffices to check that $K\subset K_{\alpha,y}$ for every point $y\in Y'$. Given any $y\in Y'$, find $V_y\in\V$ containing $y$ and observe that $U_\alpha[y]\subset \St(V_y;\V)$. Indeed, given any point $z\in U_{\alpha}[y]$ and taking into account that $(y,z)\in U_\alpha\subset \bigcup_{V\in\V}V\times V$, we can find a set $V\in\V$ such that $(y,z)\in V\times V$. It follows from $z\in V\cap V_y$ that $z\in V\subset \St(V_y;\V)$. Since the cover $\{\St(V;\V):V\in\V\}$ refines the cover $\U'$ of $Y$, there is set $U'\in\U'$ such that $U_\alpha[y]\subset\St(V_y;\V)\subset U'$. The set $U'\in\U'$ intersects $Y'$ and hence is equal to $U_{\alpha_z}[z]$ for some $z\in Y'$. Observe that $U_\alpha[y]\subset U_{\alpha_z}[z]$ implies $V_{\alpha,z}\subset V_{\alpha_y,y}$, and $K\subset K_{\alpha_y,y}=\overline{C\setminus V_{\alpha_y,y}}\subset \overline{X\setminus V_{\alpha,z}}=K_{\alpha,z}$.

Therefore $(K_\alpha)_{\alpha\in\w^\w}$ is a monotone cofinal family in  $\K(\bar S\setminus X')$ witnessing that $\w^\w\succcurlyeq \K(\bar S\setminus X')$. By Christensen Theorem~\ref{t:Chris}, the space $\bar S\setminus X'$ is Polish and hence is a $G_\delta$-set in $\bar S$. Then its complement $X'$ is $\sigma$-compact.
\end{proof}

We shall apply Theorem~\ref{t:Lasnev} to quotient spaces $X/A$.

\begin{corollary}\label{c:factor} Let $X$ be a metrizable space and $A$ be a closed nowhere dense subspace of $X$. The quotient space $X/A$
\begin{enumerate}
\item[\textup{(1)}] is first countable if and only if the set $A$ is compact;
\item[\textup{(2)}]  has inductive topology with respect to a countable cover of $X/A$ by closed (metrizable) first-countable subsets if and only if $A$ is locally compact and $\sigma$-compact;
\item[\textup{(3)}]  has an $\w^\w$-base if and only if $A$ is $\sigma$-compact;
\item[\textup{(4)}]  is universally $\w^\w$-based if and only if the space $X$ is $\sigma'$-compact only if $X/A$ is $\sigma'$-compact.
\end{enumerate}
\end{corollary}

\begin{proof} This corollary follows from Theorem~\ref{t:Lasnev} and the observation that the quotient map $q:X\to X/A$ is closed and irreducible.
\end{proof}

\begin{example} Take any separable metrizable space $X$ containing a closed nowhere dense subset $A\subset X$ such that the complement $X\setminus A$ is discrete and $A$ is not $\sigma$-compact. By Corollary~\ref{c:factor}, the countable space $X/A$ does not have an $\w^\w$-base. Yet, $X/A$ is a La\v snev $\mathfrak P_0$-space. In particular, $X/A$ is Fr\'echet-Urysohn and stratifiable.
\end{example}

\begin{problem} Characterize La\v snev spaces that possess a (locally) uniform $\w^\w$-base.
\end{problem}

\section{Detecting universally $\w^\w$-based La\v snev spaces}\label{s:uww-Lasnev}

In this section we detect La\v snev spaces whose universal uniformity has an $\w^\w$-base. We remind that such spaces are called {\em universally $\w^\w$-based}.

\begin{theorem}\label{t:tLasnev} For a Tychonoff space $X$ the following conditions are equivalent:
\begin{enumerate}
\item[\textup{(1)}] $X$ is universally $\w^\w$-based and La\v snev.
\item[\textup{(2)}] $X$ is the image of a metrizable $\sigma'$-compact space under a closed continuous map.
\item[\textup{(3)}] $X$ is the image of a metrizable $\sigma'$-compact space under an $\IR$-quotient map, $X$ is Fr\'echet-Urysohn and $X'$ is a $G_\delta$-set in $X$.
\item[\textup{(4)}] $X$ is an universally $\w^\w$-based Fr\'echet-Urysohn $\sigma$-space.
\item[\textup{(5)}] $X$ is an universally $\w^\w$-based $\sigma'$-compact paracompact La\v snev $\mathfrak P$-space.
\end{enumerate}
\end{theorem}

\begin{proof} We shall prove the implications $(1)\Ra(2)\Ra(3)\Ra(4)\Ra(5)\Ra(1)$.
The implication $(1)\Ra(2)$ follows from Theorem~\ref{t:Lasnev}(6) and Lemma 5.4 of \cite{Grue}.
\smallskip

$(2)\Ra(3)$. Assume that $X$ is the image of a $\sigma'$-metrizable space $M$ under a closed continuous map $f:M\to X$. By Lemma 5.4 of \cite{Grue}, we can assume that the map $f$ is irreducible. In this case $M'=f^{-1}[X']$.
By \cite[2.4.G]{Eng}, the space $X$ is Fr\'echet-Urysohn. The closed map $f$ is quotient and hence $\IR$-quotient. To see that the set $X'$ is of type $G_\delta$ in $X$, fix a decreasing sequence $(U_n)_{n\in\w}$ of open sets in $M$ such that $f^{-1}[X']=M'=\bigcap_{n\in\w}U_n$. For every $n\in\w$ the complement $M\setminus U_n$ is closed in $M$ and its image $f[M\setminus U_n]$ is closed in $X$. Then the complement $V_n=X\setminus f[M\setminus U_n]$ is an open neighborhood of $X'$ such that $f^{-1}[V_n]\subset U_n$. It follows that $f^{-1}[X']=M'=\bigcap_{n\in\w}f^{-1}[V_n]\subset \bigcap_{n\in\w}U_n=f^{-1}[X']$, so $X'=\bigcap_{n\in\w}V_n$ is a $G_\delta$-set in $X$.
\smallskip

$(3)\Ra(4)$ Assume that the space $X$ is Fr\'echet-Urysohn, $X'$ is a $G_\delta$-set in $X$, and $X$ is the image of a metrizable $\sigma'$-compact space $M$ under an $\IR$-quotient map $f:M\to X$.
By Theorem~\ref{t:s'}, the space $X$ is universally $\w^\w$-based.

We claim that $X'\subset f[M']$. Given any point $x\in X\setminus f[M']$, observe that the preimage $f^{-1}(x)$ does not intersect the set $M'$ and hence is closed-and-open in $M$. Let $\chi:X\to\{0,1\}$ be the characteristic function of the singleton $\{x\}=\chi^{-1}(1)$. Observe that the composition $g=\chi\circ f:M\to\{0,1\}$ is continuous (as $g^{-1}(1)=f^{-1}(x)$ is closed-and-open in $M$). Since the map $f$ is $\IR$-quotient, the characteristic function $\chi$ of $\{x\}$ is continuous, which means that $x$ is an isolated point of $X$. This completes the proof of the inclusion $X'\subset f[M']$. Since the space $M'$ is cosmic and $\sigma$-compact, so is its image $f[M']$ and the closed subset $X'$ of $f[M']$.
Let $\C$ be a countable network for the cosmic space $X'$. The $G_\delta$-property of $X'$ in $X$ implies that the family of singletons  $\F=\big\{\{x\}:x\in X\setminus X'\big\}$ is $\sigma$-discrete in $X$. Then the union $\C\cup\F$ is a $\sigma$-discrete network for the space $X$, which means that $X$ is a $\sigma$-space.
\smallskip

$(4)\Ra(5)$ Assume that $X$ is a universally $\w^\w$-based Fr\'echet-Urysohn $\sigma$-space. Since $X$ is Fr\'echet-Urysohn, $X'=X^{\prime\css}$. By Theorem~\ref{t:fb=>ms}, the set $X'=X^{\prime\css}$ is cosmic. By Lemma~\ref{l:para'}, the space $X$ is paracompact and hence $\w$-Urysohn. By Theorem~\ref{t:fb=>ms}, the set $X'$ is $\sigma$-compact. By Theorem~\ref{t:prop-luww}(3), the paracompact $\Sigma$-space $X$ is an $\aleph$-space. By \cite[7.1, 7.5]{BBK}, the Fr\'echet-Urysohn $\aleph$-space $X$ is La\v snev.
\smallskip

The implication $(5)\Ra(1)$ is trivial.
\end{proof}

\section{On universally $\w^\w$-based $\sigma$-compact spaces}

The following theorem gives a partial answer to Problem~\ref{prob:quots-s'}.
% and for Fr\'echet-Urysohn feebly collectionwise Hausdorff $\Sigma$-spaces. We recall that a topological space $X$ is {\em feebly collectionwise Hausdorff} if each uncountable closed discrete subset of $X$ contains an uncountable strongly discrete subset.

\begin{theorem} For a Tychonoff space $X$ we have the implications $(1)\Ra(2)\Ra(3)\Ra (4)\Ra(5)\Ra(6)$ of the following conditions:
\begin{enumerate}
\item[\textup{(1)}] $X$ is metrizable and $\sigma$-compact;
\item[\textup{(2)}]  $X$ is a closed image of a $\sigma$-compact metrizable space;
\item[\textup{(3)}]  $X$ is a quotient image of a $\sigma$-compact metrizable space;
\item[\textup{(4)}]  $X$ is an $\IR$-quotient image of a $\sigma$-compact metrizable space;
\item[\textup{(5)}]  $X$ is universally $\w^\w$-based and $\sigma$-compact;
\item[\textup{(6)}]  $X$ has an $\w^\w$-base.
\end{enumerate}
If $X$ is a $q$-space or has countable ofan tightness, then $(5)\Ra(1)$.\newline
If the space $X$ is Fr\'echet-Urysohn, then $(5)\Ra(2)$. \newline
If the space $X$ is $k_\sigma$-baseportating, then $(5)\Leftrightarrow(6)$.
\end{theorem}

\begin{proof} The implications $(1)\Ra(2)\Ra(3)\Ra(4)$, $(5)\Ra(6)$ are trivial and $(4)\Ra(5)$ follows from Theorem~\ref{t:s'}.

Assuming that $X$ is a $q$-space or has countable ofan tightness, we shall prove that $(5)\Ra(1)$. If $X$ is universally $\w^\w$-based and $\sigma$-compact, then by Theorem~\ref{t:1-luww}, $X$ is first-countable and by Theorem~\ref{tc:small}, the space $X$ is cosmic and hence closed-$\bar G_\delta$. By Theorem~\ref{t:metr-ww}, the space $X$ is metrizable.

Now for Fr\'echet-Urysohn spaces, we shall prove the implication $(5)\Ra (2)$. If the space $X$ is universally $\w^\w$-based and $\sigma$-compact, then $X$ is a Lindel\"of $\Sigma$-space and by Theorem~\ref{tc:small}, $X$ is an $\aleph_0$-space.
By \cite[7.1, 7.5]{BBK}, the Fr\'echet-Urysohn $\aleph_0$-space $X$ is La\v snev. So,  $X$ is the image of a metrizable space $M$ under a closed continuous map $f:M\to X$. By Lemma~5.4 \cite{Grue}, we can assume that the map $f$ is irreducible.
The separability of $X$ and the irreducibility of the map $f$ imply that the metrizable space $M$ is separable. By Theorem~\ref{t:Lasnev}(6), the space $M$ is $\sigma'$-compact and being separable, is $\sigma$-compact. Thus $X$ is a closed image of the metrizable $\sigma$-compact space $M$.

 For $k_\sigma$-beseportating spaces the equivalence $(5)\Leftrightarrow(6)$ was proved in Theorem~\ref{t:local}.
\end{proof}

\section{Characterizing metrizable $\sigma'$-compact spaces}\label{s:uww-metr}

Metrizable $\sigma'$-compact space can be characterized as follows.

\begin{theorem}\label{t:sigma'} For a Tychonoff space $X$ the following conditions are equivalent:
\begin{enumerate}
\item[\textup{(1)}] $X$ is metrizable and $\sigma'$-compact.
\item[\textup{(2)}] $X$ is universally $\w^\w$-based, first-countable and perfectly normal.
\item[\textup{(3)}] $X$ is a universally $\w^\w$-based ofan $\css$-tight closed-$\bar G_\delta$ space;
\item[\textup{(4)}] $X$ is a universally $\w^\w$-based closed-$\bar G_\delta$ q-space.
\end{enumerate}
\end{theorem}

\begin{proof} The implication $(1)\Ra(2)$ follows from Theorem~\ref{t:msigma'} and $(2)\Ra(3,4)$ is trivial. To prove the  implications $(3)\Ra(1)$ and $(4)\Ra(1)$, assume that the space $X$ is universally $\w^\w$-based, closed-$\bar G_\delta$, and is a $q$-space or ofan $\css$-tight. By Theorem~\ref{t:1-luww}, the space $X$ is first-countable and by Theorem~\ref{t:metr-ww}, $X$ is metrizable. By Theorem~\ref{t:ms'<=>uww}, the metrizable space $X$ is $\sigma'$-compact.
\end{proof}

\section{$\w^\w$-Bases in scattered topological spaces}

In this section we study scattered spaces with an $\w^\w$-base. We recall that a topological space $X$ is \index{topological space!scattered}{\em scattered} if each non-empty subspace of $X$ has an isolated point. The complexity of a scattered space can be measured by the scattered height, defined as follows.

For a subspace $A$ of a topological space $X$ by $A'$ we denote the (closed) set of non-isolated points of $A$. Let $X^{(0)}=X$ and for every ordinal $\alpha$ define the {\em derived set} $X^{(\alpha)}$ by the recursive formula $$X^{(\alpha)}=\bigcap_{\beta<\alpha}(X^{(\beta)})'.$$ Observe that a topological space $X$ is scattered if and only if $X^{(\alpha)}=\emptyset$ for some ordinal $\alpha$. In this case for every point $x\in X$ there exists a unique ordinal $\hbar(x)$ such that $x\in X^{\hbar(x)}\setminus X^{\hbar(x)+1}$. The ordinal $$\hbar(X)=\sup_{x\in X}\hbar(x)$$is called the \index{topological space!scattered height of}\index{scattered height}{\em scattered height} of the scattered space $X$.

\begin{theorem}\label{t:scat1} Let $X$ be a scattered $\Sigma$-space. Assume that $X$ has countable extent and finite scattered height. If $X$ has countable $\ccs^*$-network at each point, then $X$ is countable.
\end{theorem}

\begin{proof} By induction we shall prove that each closed subspace $Z\subset X$ of finite scattered height $\hbar(Z)$ is countable Observe that $Z$ is a $\Sigma$-space that has a countable $\ccs^*$-network at each point.

If $\hbar(Z)=0$, then the space $Z$ is discrete. Since $X$ has countable extent, the closed discrete subspace $Z$ of $X$ is countable. Now assume that for some $n\in\IN$ we have proved that any closed subspace $Z\subset X$ of countable extent and scattered height $\hbar(Z)<n$ is countable. We claim that each closed subspace $Z\subset X$ of countable extent and scattered height $\hbar(Z)=n$ is countable.
Consider the closed subset $Z'\subset Z$ consisting of non-isolated points of $Z$ and observe that $\hbar(Z')=\hbar(Z)-1<n$. By the induction hypothesis, the space $Z'$ is countable.

By our assumption, at each point $z\in Z'$ the space $Z$ has a countable $\ccs^*$-network $\mathcal N_z$. Consider the countable family of entourages $\mathcal E=\big\{(N\times N)\cup\Delta_Z:N\in\bigcup_{z\in Z'}\mathcal N_z\big\}$. By Proposition~\ref{p:lu=>regular}, the $\Sigma$-space $Z$, being regular, admits a locally uniform base $\mathcal B$. By Proposition~\ref{p:enet<->network}, the pair $(\E,\Bas^{\pm2})$ is a locally uniform $\ccs^*$-netbase for the space $Z$. By Theorem~\ref{t:Sigma}, the $\Sigma$-space $Z$ is a $\sigma$-space. Since $Z$ has countable extent, the $\sigma$-space $Z$ is cosmic. Then the set $Z\setminus Z'$ of isolated points of $Z$ is countable and so is the space $Z=Z'\cup(Z\setminus Z')$.
\end{proof}

\begin{corollary}\label{c:scat} If a scattered compact Hausdorff space $X$ of finite scattered height has  countable $\ccs^*$-network at each point, then $X$ is countable and metrizable.
\end{corollary}

\begin{problem}\label{prob:scat-ccs} Assume that a scattered compact Hausdorff space $X$ of countable scattered height has a countable $\ccs^{*}$-network at each point. Is $X$ countable?
\end{problem}

The following example constructed by Jakovlev \cite{Jak} and improved by Abraham, Gorelic and Juh\'asz \cite{AGJ} shows that the $\ccs^*$-network in Problem~\label{prob:scat-ccs} cannot be weakened to a $\cs^*$-network.

\begin{example}[Jakovlev-Abraham-Gorelik-Juh\'asz] Under $\mathfrak b=\mathfrak c$ there exists a first-countable scattered locally compact locally countable regular $T_0$-space $K$ of cardinality $\mathfrak c$ whose one-point compactification $K^*$ is sequential and has a $\cs^*$-network at each point. Moreover, the space $K^*$ has scattered height $\hbar(K^*)=\w$ and has sequential order 2.
\end{example}

Theorems~\ref{t:lP*} and \ref{t:scat1} imply the following fact.

\begin{theorem} Let $X$ be a scattered $\Sigma$-space with countable extent and finite scattered height. If $X$ has an $\w^\w$-base, then $X$ is countable.
\end{theorem}

\begin{corollary} A scattered compact Hausdorff space $X$ with finite scattered height is metrizable if and only if $X$ has an $\w^\w$-base.
\end{corollary}

\begin{corollary} The one-point compactification of an uncountable discrete space fails to have an $\w^\w$-base.
\end{corollary}

A scattered compact space with an $\w^\w$-base needs not be metrizable: a suitable example is the segment of ordinals $[0,\w_1]$, which has an $\w^\w$-base under $\w_1=\mathfrak b$. Nonetheless, the following questions seem to be open.

\begin{problem} Let $X$ be a scattered compact Hausdorff space with an $\w^\w$-base and countable scattered height $\hbar(X)$. Is $X$ metrizable?
\end{problem}

\begin{problem} Assume that $\w_1<\mathfrak b$. Is each scattered compact Hausdorff space with an $\w^\w$-base metrizable?
\end{problem}

\section{$\w^\w$-Bases in generalized ordered spaces}\label{s:GO}

In this section we study $\w^\w$-bases in generalized ordered spaces. A topological space $X$ is called a \index{topological space!generalized ordered}\index{topological space!$GO$-space}{\em $GO$-space} (abbreviated from {\em generalized ordered space}) if $X$ has a base of the topology consisting of order-convex sets with respect to some closed linear order $\le $ on $X$, see \cite{BenLut}. A subset $A$ of linearly ordered space $(X,<)$ is called {\em order-convex} if for any points $x<y$ in $A$ the order interval $[x,y]=\{z\in X:x\le z\le y\}$ is contained in $A$.

\begin{theorem}\label{t:GOP} Let $P\cong P^2$ be a poset. A $GO$-space $X$ has a local $P$-base at a point $x\in X$ if and only if the cardinals $\add(\Tau_x(X))$ and $\cof(\Tau_x(X))$ are $P$-dominated.
\end{theorem}

\begin{proof} The ``only if'' part follows from Proposition~\ref{p:add=cof}. To prove the ``if'' part, assume that for some point $x\in X$ the cardinals $\add(\Tau_x(X))$ and $\cof(\Tau_x(X))$ are $P$-dominated. Let $\tau$ be the topology of $X$ and $<$ be a linear order on $X$ such that the family $\mathcal B=\{U\in\tau:U$ is order-convex in $(X,<)\}$ is a base of the topology $\tau$. First we define two cardinals $\cf_-(x)$ and $\cf_+(x)$ and two maps $\lambda:\cf_-(x)\to X$, $\rho:\cf_+(x)\to X$. Four cases are possible.

1. The point $x$ is isolated in $X$. In this case put $\cf_-(x)=\cf_+(x)=1$ and $\lambda:\cf_-(x)\to\{x\}\subset X$, $\rho:\cf_+(x)\to\{x\}\subset X$ be the constant maps.

2. The point $x$ is not isolated in $X$ but is right-isolated in the sense that the set $(\leftarrow,x]=\{y\in X:y\le x\}$ is open in $X$. In this case let $\cf_+(x)=1$ and $\cf_-(x)$ be the cofinality of the linearly ordered set $(\leftarrow,x):=({\leftarrow},x]\setminus\{x\}$. Let $\rho:\cf_+(x)\to\{x\}\subset X$ be the constant map and $\lambda:\cf_-(x)\to(\leftarrow,x)$ be any injective monotone cofinal map.

3. The point $x$ is not isolated in $X$ but is left-isolated in the sense that the set $[x,\to)=\{y\in X:x\le y\}$ is open in $X$. In this case let $\cf_-(x)=1$ and $\cf_+(x)$ be the coinitiality of the linearly ordered set $(x,{\to}):=[x,\to)\setminus\{x\}$, i.e., the smallest cardinality of a subset $B\subset (x,\to)$ without lower bound in $(x,\to)$. Let $\lambda:\cf_-(x)\to\{x\}\subset X$ be the constant map and $\rho:\cf_-(x)\to(x,\to)$ be an injective monotone cofinal map to the linearly ordered set $(x,\to)$ endowed with the revered linear order.

4. The point $x$ is neither left-isolated nor right-isolated. In this case let $\cf_-(x)$ be the cofinality of the linearly ordered set $(\leftarrow,x)$ and $\cf_+(x)$ be the coinitiality of the linearly ordered set $(x,\to)$. Let $\lambda:\cf_-(x)\to(\leftarrow,x)$ be an injective monotone cofinal map and $\rho:\cf_-(x)\to(x,\to)$ be an injective monotone cofinal map to the linearly ordered set $(x,\to)$ endowed with the revered linear order.

It is easy to see that $\add(\Tau_x(X))=\min\{\cf_-(x),\cf_+(x)\}$ and $\cof(\Tau_x(X))=\max\{\cf_-(x),\cf_+(x)\}$, which implies that the cardinals $\cf_-(x)$ and $\cf_+(x)$ are $P$-dominated and hence admit monotone cofinal maps $f_-:P\to\cf_-(x)$ and $f_+:P\to\cf_+(x)$. Then the map $$f:P^2\to \Tau_x(X),\;\;f:(\alpha,\beta)\mapsto[\lambda(f_-(\alpha)),\rho(f_+(\beta))],$$
is monotone and cofinal, witnessing that $P\cong P^2\succcurlyeq \Tau_x(X)$.
\end{proof}

\begin{theorem}\label{t:GOP2} For a poset $P\cong P^2$, a (compact) $GO$-space $X$ has an $P$-base if (and only if) for any point $x\in X$ every non-zero cardinal $\kappa\le \chi(x;X)$ is $P$-dominated.
\end{theorem}

\begin{proof} The ``if'' part follows from Theorem~\ref{t:GOP} and the inequality $\add(\Tau_x(X))\le\cof(\Tau_x(X))=\chi(x;X)$ holding for every point $x\in X$. To prove the ``only if'' part, assume that the $GO$-space $X$ is compact and has a  $P$-base. We should prove that for any $x\in X$, every non-zero cardinal $\kappa\le\chi(x;X)$ is $P$-dominated. %By Proposition~\ref{p:add=cof}, $P\succcurlyeq \Tau_x(X)\succcurlyeq \cof(\Tau_x(X))=\chi(x;X)$, which implies that the cardinal $\chi(x;X)$ is $P$-dominated. So, we can assume that $\kappa<\chi(x;X)$.

Let $\lambda:\cf_-(x)\to X$ and $\rho:\cf_+(x)\to X$ be the injective maps defined in the proof of Theorem~\ref{t:GOP}. Taking into account that $\chi(x;X)=\cof(\Tau_x(X))=\max\{\cf_-(x),\cf_+(x)\}$, we conclude that $\kappa\le \cf_-(x)$ or $\kappa\le\cf_+(x)$.

If $\kappa\le\cf_-(x)$, then the set $\lambda\big[[0,\kappa)\big]\subset X$ is well-defined and has the least upper bound $y\in X$ (by the compactness of the $GO$-space $X$). It follows that $\kappa=\cf_-(y)\in\{\add(\Tau_y(X)),\cof(\Tau_y(X))\}$. By Proposition~\ref{p:add=cof}, the cardinals $\add(\Tau_y(X)),\cof(\Tau_y(X))$ are $\Tau_y(X)$-dominated and so is the cardinal $\kappa$. Since the space $X$ has a $P$-base, the poset $\Tau_y(X)$ is $P$-dominated and so is the cardinal $\kappa$.

If $\kappa\le\cf_+(x)$, then the set $\rho\big[[0,\kappa)\big]\subset X$ is well-defined and has the greatest lower bound $z\in X$ (by the compactness of the $GO$-space $X$). It follows that $\kappa=\cf_+(z)\in\{\add(\Tau_x(X)),\cof(\Tau_x(X))\}$. By Proposition~\ref{p:add=cof}, the cardinals $\add(\Tau_x(X)),\cof(\Tau_x(X))$ are $P$-dominated and so is the cardinal $\kappa$.
\end{proof}

Applying Theorems~\ref{t:GOP}, \ref{t:GOP2} to the poset $P=\w^\w$, we obtain two corollaries.

\begin{corollary}  A $GO$-space $X$ has a local $\w^\w$-base at a point $x\in X$ if and only if the cardinals $\add(\Tau_x(X))$ and $\cof(\Tau_x(X))$ are $\w^\w$-dominated.
\end{corollary}

\begin{corollary}\label{c:GOww-char} A (compact) $GO$-space $X$ has an $\w^\w$-base if (and only if) for any point $x\in X$ any non-zero cardinal $\kappa\le \chi(x;X)$ is $\w^\w$-dominated.
\end{corollary}

In its turn, Corollary~\ref{c:GOww-char} implies two (opposite) consistent characterizations of GO-spaces with an $\w^\w$-base.

\begin{theorem}\label{t:w1<b:ww=1} Assume that $\w_1<\mathfrak b$. A compact GO-space $X$ has an $\w^\w$-base if and only if $X$ is first-countable.
\end{theorem}

\begin{proof} The ``if'' part is trivial. To prove the ``only if'' part, assume that $X$ has an $\w^\w$-base. If $X$ is not first-countable, then some point $x\in X$ has character $\chi(x;X)\ge\w_1$. By Corollary~\ref{c:GOww-char}, the cardinal $\w_1$ is $\w^\w$-dominated, which contradicts Lemma~\ref{l:b-bound}.
\end{proof}

\begin{theorem}\label{t:GOww-char} Assume that $\mathfrak b=\w_1$ and $\mathfrak d\le\w_2$. A GO-space $X$ has an $\w^\w$-base if and only if $\chi(X)\le\mathfrak d$.
\end{theorem}

\begin{proof} If $X$ has an $\w^\w$-base, then $\chi(X)\le\mathfrak d$ by Lemma~\ref{l:b-bound}.

If $\chi(X)\le\mathfrak d$, then each cardinal $\kappa\le\chi(X)$ belongs to the set $\w\cup\{\mathfrak b,\mathfrak d\}$ and hence is $\w^\w$-dominated by Corollary~\ref{c:bdcf(d)}. By Corollary~\ref{c:GOww-char}, the GO-space $X$ has an $\w^\w$-base.
\end{proof}

We do not know if Theorem~\ref{t:w1<b:ww=1} generalizes to all compact Hausdorff spaces.

\begin{problem} Assume that $\w_1<\mathfrak b$. Is each compact Hausdorff space with an $\w^\w$-base first-countable?
\end{problem}

\section{The cardinality of topological spaces with an $\w^\w$-base}
\label{s:ww-cardinal}

In this section we shall apply the results of Section~\ref{s:card-as} to establish some upper bounds on the cardinality of Hausdorff spaces with an $\w^\w$-base.

First observe that Lemma~\ref{l:b-bound} and Theorem~\ref{t:ww=>netbase} imply the following bounds on the characters of a topological space with an $\w^\w$-base.

\begin{proposition}\label{p:character} Every topological space $X$ with an $\w^\w$-base has character $$\chi(X)=\sup_{x\in X}\cof(\Tau_x(X))\in\{1,\w\}\cup[\mathfrak b,\mathfrak d]$$ and $\css^*$-character $\chi_{\css^*}(X)\le\w$.
\end{proposition}

By \cite{BC} (see also \cite{Hodel}), each Hausdorff space $X$ has cardinality $$|X|\le 2^{\chi(X) aL_c(X)}\le 2^{\chi(X) L(X)}.$$ Moreover, if $X$ is Urysohn, then $$|X|\le 2^{\chi(X) wL_c(X)}\le 2^{\chi(X) c(X)}\mbox{ \  and \ }|X|\le 2^{\chi(X) aL(X)},$$ see \cite{Alas}, \cite{BC}, \cite{Hodel}. Combining these inequalities with Proposition~\ref{p:character}, we get the following upper bounds on the cardinality of a Hausdorff (Urysohn) space.

\begin{corollary}\label{c:HUcard} Let $X$ be a topological space with an $\w^\w$-base.
\begin{enumerate}
\item If $X$ is Hausdorff, then $|X|\le 2^{\mathfrak d}\cdot 2^{aL_c(X)}\le 2^{\mathfrak d}\cdot 2^{L(X)}$.
\item If $X$ is Urysohn, then $|X|\le 2^{\mathfrak d}\cdot\min\{2^{wL_c(X)},2^{aL(X)}\}\le 2^{\mathfrak d}\cdot 2^{c(X)}$.
\end{enumerate}
\end{corollary}

\begin{corollary}\label{c:dcard} Assume that $2^{\mathfrak d}=2^\w$. Let $X$ be a topological space with an $\w^\w$-base.
\begin{enumerate}
\item If $X$ is Hausdorff, then $|X|\le 2^{aL_c(X)}\le 2^{L(X)}$.
\item If $X$ is Urysohn, then $|X|\le \min\{2^{wL_c(X)},2^{aL(X)}\}\le 2^{c(X)}$.
\end{enumerate}
\end{corollary}

Now we prove that for a countably tight (regular) Hausdorff space $X$ with an $\w^\w$-base the upper bound $|X|\le 2^{aL_c(X)}$ (and $|X|\le \min\{2^{wL_c(X)},2^{aL(X)}\}$) holds in ZFC.

\begin{theorem}\label{t:ww-cardinal} If a Hausdorff space $X$ is countably tight and has an $\w^\w$-base, then $|X|\le d(X)^\w$ and $|X|\le 2^{aL_c(X)}\le 2^{L(X)}$. If $X$ is regular, then $|X|\le 2^{wL_c(X)}\le 2^{c(X)}$.
\end{theorem}

\begin{proof} The inequalities trivially hold if $X$ is finite. So, we assume that $X$ is infinite. In this case $\min\{aL_c(X),wL_c(X)\}\ge\w$. By Proposition~\ref{p:character}, the space $X$ has countable $\css^*$-character $\chi_{\css^*}(X$ and being countably tight has countable $\as^*$-character $\chi_{\as^*}(X)=\chi_{\css^*}(X)\le\w$. By Corollary~\ref{c:cardinal}, $$|X|\le d(X)^{t(X)}\cdot 2^{\chi_{\as^*}(X)}\le d(X)^\w\cdot 2^{\w}=d(X)^\w.$$ By Theorem~\ref{t:chip-cardinal}, $$|X|\le 2^{\chi_{\as^*}(X)\cdot aL(X)}\le 2^{\w\cdot aL_c(X)}=2^{aL_c(X)}\le 2^{L(X)}.$$ If the space $X$ is regular, then Theorem~\ref{t:chip-cardinal} yields the upper bound
$$|X|\le  2^{\chi_{\as^*}(X)\cdot wL_c(X)}\le 2^{\w\cdot wL_c(X)}=2^{wL_c(X)}\le 2^{c(X)}.$$
\end{proof}

The countable tightness of essential in Theorem~\ref{t:ww-cardinal} as shown by the following example suggested by Paul Szeptycki.

\begin{example}[Szeptycki]\label{ex:szeptycki} The space $X=2^{\w_1}$ endowed with the topology generated by the lexicographic order is compact, has an $\w^\w$-base (and hence a countable $\css^*$-netbase) under $\w_1=\mathfrak b$ and  cardinality $|X|>\mathfrak c=2^{L(X)}$ under $2^{\w_1}>\mathfrak c$ (which follows from $\w_1=\mathfrak c$).
\end{example}

\begin{proof} It is clear that $|X|=2^{\w_1}$ and $L(X)=\w$. By Corollary~\ref{c:bdcf(d)}, the cardinal $\mathfrak b$ is $\w^\w$-dominated, which implies that under $\w_1=\mathfrak b$, each cardinal $\kappa\le \chi(X)=\mathfrak b=\w_1$ is $\w^\w$-dominated. By Theorem~\ref{t:GOP2}, the GO-space $X$ has an $\w^\w$-base.
\end{proof}

\begin{remark} By Corollary~\ref{c:dcard}, under $2^{\mathfrak d}=\mathfrak c$ each compact Hausdorff space $X$ with an $\w^\w$-base has cardinality $|X|\le\mathfrak c$.
\end{remark}

\printindex


\begin{thebibliography}{XXX}

\bibitem{AGJ} U.~Abraham, I.~Gorelic, I.~Juh\'asz, {\em On Jakovlev spaces}, Israel J. Math. {\bf 152} (2006), 205--219.

\bibitem{Alas} O.T.~Alas, {\em More topological cardinal inequalities}, Colloq. Math. {\bf 65}:2 (1993), 165--168.

\bibitem{AW} O.~Alas, R.~Wilson, {\em When is a compact space sequentially compact?} Topology Proc. {\bf 29}:2 (2005), 327--335.

%\bibitem{Ar86} A.V.~Arhangel'skii, {\em Hurewicz spaces, analytic sets and fan-tightness of spaces of functions}, Soviet Math. Dokl. {\bf 33}:2 (1986), 396--399.

 \bibitem{Arch}
A.V.~Arhangel'skii, {\em Topological function spaces}, Kluwer Academic, Dordrecht, 1992.

\bibitem{Ar02} A.V.~Arhangel'skii, {\em Topological invariants in algebraic environment}, in: Recent Progress in General Topology, II, 1--57, North-Holland, Amsterdam, 2002.

\bibitem{AC} A.V.~Arhangel'skii, J.~Calbrix, {\em A characterization of $\sigma$-compactness of a cosmic space $X$  by means of subspaces of ${\IR}^X$}, Proc. Amer. Math. Soc., {\bf 127} (1999), 2497--2504.

\bibitem{AT} A.V.~Arhangel'skii, M.G.~Tkachenko,
{\em Topological Groups and Related Structures},
 Atlantis Press and
World Scientific, Paris--Amsterdam, 2008.

\bibitem{Ban} T.~Banakh, {\em $\mathfrak P_0$-spaces}, Topology Appl. {\bf 195} (2015) 151--173.

 \bibitem{Ban2} T.~Banakh, {\em The strong Pytkeev$^*$ property of topological spaces}, preprint  (http://arxiv.org/abs/1607.03599).

\bibitem{BBK} T.~Banakh, V.~Bogachev, A.~Kolesnikov, {\em  $k^*$-metrizable spaces and their applications}, J. Math. Sci. (N. Y.) {\bf 155}:4 (2008), 475--522.

 %\bibitem{BG} T.~Banakh, S.~Gabriyelyan, {\em On the $C_k$-stable closure of the class of (separable) metrizable spaces},
 %Monatshefte Math. (DOI 10.1007/s00605-015-0840-6).

\bibitem{BL} T.~Banakh, A.~Leiderman, {\em The strong Pytkeev property in topological spaces}, Topology Appl.  (https://arxiv.org/abs/1412.4268).


\bibitem{BL-LG} T.~Banakh, A.~Leiderman, {\em $\w^\w$-Dominated function spaces and $\w^\w$-bases in free objects of Topological Algebra}, preprint (https://arxiv.org/abs/1611.06438).

%\bibitem{BL-G} T.~Banakh, A.~Leiderman, {\em Local $\w^\w$-bases in free topological groups}, preprint.


\bibitem{BP} T.~Banakh, A.~Plichko, {\em The algebraic dimension of linear metric spaces and Baire properties of their hyperspaces}, Rev. R. Acad. Cien. Serie A. Math. {\bf 100}:1-2 (2006), 31--37.

\bibitem{BR16} T.~Banakh, A.~Ravsky, {\em Verbal covering properties of topological spaces}, Topology Appl. 201 (2016) 181--205.

\bibitem{BR16q} T.~Banakh, A.~Ravsky, {\em Quasi-pseudometrics on quasi-uniform spaces and quasi-metrization of topological monoids}, Topology Appl. {\bf 200} (2016) 19--45.

\bibitem{BR17} T.~Banakh, A.~Ravsky, {\em Each regular paratopological group is completely regular}, Proc. Amer. Math. Soc. {\bf 145}:3 (2017) 1373--1382 (doi.org/10.1090/proc/13318).

\bibitem{BanRep} T.~Banakh, D.~Repov\v s, {\em Sequential rectifiable spaces with countable $cs^*$-character}, Bull. Malaysian Math. Sci. Soc. (to appear).

\bibitem{BZ} T.~Banakh, L.~Zdomskyy, {\em Coherence of Semifilters}, e-book available at\newline (http://prima.lnu.edu.ua/faculty/mechmat/Departments/Topology/booksite.html).

\bibitem{BH} T.~Bartoszy\'nski, L.~Halbeisen, {\em On a theorem of Banach and Kuratowski and $K$-Lusin sets}, Rocky Mountain J. Math. {\bf 33}:4 (2003) 1223--1231.


\bibitem{Baum} J.~Baumgartner, {\em Applications of the proper forcing axiom}, in: \textit{Handbook of Set Theoretic Topology} (eds. K. Kunen and J. Vaughan), North Holland, Amsterdam, 1984, 913--959.


\bibitem{BC} A.~Bella, F.~Cammaroto, {\em On the cardinality of Urysohn spaces}, Canad. Math. Bull. {\bf 31}:2 (1988), 153--158.

\bibitem{BenLut} H.R.~Bennett, D.J.~Lutzer, {\em Recent developments in the topology of ordered spaces}, Recent progress in general topology, II, 83--114, North-Holland, Amsterdam, 2002.

\bibitem{Blass} A.~Blass, {\em Combinatorial cardinal characteristics of the continuum}, in: {\em Handbook of set theory}, Springer, Dordrecht (2010), 395--489.

\bibitem{BM} A.~Blass, H.~Mildenberger, {\em On the cofinality of ultrapowers}, J. Symbolic Logic {\bf 64}:2 (1999), 727--736

 %\bibitem{Bourbaki}
 %N.~Bourbaki, {\em Topological vector spaces, Chapters 1--5}, Springer--Verlag, 1987.

%\bibitem{BKR} T.Banakh, L.Karchevska, A.Ravsky, {\em The closed Steinhaus properties of $\sigma$-ideals on topological groups}, preprint (http://arxiv.org/abs/1509.09073).

\bibitem{Canjar} M.~Canjar, {\em Cofinalities of countable ultraproducts: the existence theorem}, Notre Dame J. Formal Logic {\bf 30}:4 (1989), 539--542.

\bibitem{CKS} B.~Cascales, J.~K\c akol, S.~Saxon, {\em Metrizability vs. Fr\'echet-Urysohn property}, Proc. Amer. Math. Soc. {\bf 131}:11 (2003), 3623--3631.

 \bibitem{CO} B.~Cascales, J.~Orihuela, {\em On compactness in locally convex spaces}, Math. Z. {\bf 195} (1987), 365--381.

 \bibitem{COT} B.~Cascales, J.~Orihuela, V.~Tkachuk, {\em Domination by second countable spaces and Lindel\"of $\Sigma$-property}, Topology Appl. {\bf 158}:2 (2011), 204--214.

 % \bibitem{CFHT} C.~Chis, M.~V.~Ferrer, S.~Hernandez, B.~Tsaban,
%{\em  The character of topological groups, via bounded systems, Pontryagin--van Kampen duality and pcf theory,} J. Algebra, {\bf 420} (2014), 86--119.

 \bibitem{Chris} J.~Christensen, {\em Topology and Borel structure. Descriptive topology and set theory with applications to functional analysis and measure theory}, North-Holland Publ., NY, 1974.


\bibitem{CR66} W.W.~Comfort, K.A.~Ross, {\em Pseudocompactness and uniform continuity in topological groups}, Pacific J. Math. {\bf 16}:3 (1966), 483--496.



\bibitem{CS} J.~Cummings, S.~Shelah, {\em Cardinal invariants above the continuum}, Ann. Pure Appl. Logic {\bf 75}:3 (1995), 251--268.

\bibitem{Day} M.~Day, {\em Oriented systems}, Duke Math. J. {\bf 11} (1944) 201--229.

 \bibitem{DT}
N.~Dobrinen, S.~Todorcevic, {\em  Tukey types of ultrafilters}, Illinois J. Math. {\bf 55} (2011), 907--951.

 \bibitem{Douwen}
E.~van~Douwen, {\em The integers and topology}, in: \textit{Handbook of Set Theoretic Topology} (eds. K. Kunen and J. Vaugan), North Holland, Amsterdam: 1984, 111--167.

\bibitem{DH} A.~Dow, K.P.~Hart, {\em Compact spaces with a  $\mathbb P$-diagonal}, Indagationes. Mathematicae, {\bf 27} (2016), 721--726.  (http://arxiv.org/abs/1508.01541).

 \bibitem{Dug} J.~Dugundji, {\em An extension of Tietze's theorem}, Pacific J. Math. {\bf 1} (1951), 353--367.

 \bibitem{Eng} R.~Engelking, {\em General Topology}, Heldermann Verlag, Berlin, 1989.

 %\bibitem{FK}
%J.~C.~Ferrando and J.~K{\c{a}}kol, {\em A note on spaces $C_p(X)$ $K$-analytic-framed in ${\IR}^X$}, Bull. Aust. Math. Soc.
 %{\bf 78} (2008), 141--146.

 \bibitem{feka}
J.C.~Ferrando, J.~K\c{a}kol, {\em On precompact sets in spaces $C_{c}\left( X\right) $}, Georgian. Math. J., {\bf 20} (2013), 247--254.

\bibitem{FKLS}
J.~C.~Ferrando, J.~K{\c{a}}kol, M.~L\'{o}pez Pellicer, S.~A.~Saxon, {\em Tightness and distinguished Fr\'{e}chet spaces}, J. Math. Anal. Appl. {\bf 324} (2006), 862--881.

\bibitem{Foged} L.~Foged, {\em Characterizations of $\aleph $-spaces}, Pacific J. Math. {\bf 110}:1 (1984),  59--63.

\bibitem{Frem} D.H.~Fremlin, {\em The partially ordered sets of measure theory and Tukey's ordering}, Note Mat. {\bf 11} (1991), 177--214.

\bibitem{F-MT} D.H.~Fremlin, {\em Measure Theory}, Vol.5, Part.I, 2008.

 \bibitem{GabKak_2}
S.~Gabriyelyan, J.~K{\c{a}}kol,
{\em  On topological spaces and topological groups with certain local countable networks},
Topol. Appl. {\bf 90} (2015), 59--73.

\bibitem{GK-P}
S.~Gabriyelyan, J.~K{\c{a}}kol,  {\em On $\mathfrak{P}$-spaces and related concepts}, Topology Appl. {\bf 191} (2015), 178--198.


 \bibitem{GK_L(X)}
S.~Gabriyelyan, J.~K{\c{a}}kol, {\em Free locally convex spaces with a small base}, RACSAM (2016); doi:10.1007/s13398-016-0315-1


  \bibitem{GabKakLei_1}
S.~Gabriyelyan, J.~K{\c{a}}kol, A.~Leiderman,  {\em The strong Pytkeev property for topological groups and topological vector spaces},
Monatshefte Math. {\bf 175} (2014), 519--542.

 \bibitem{GabKakLei_2}
S.~Gabriyelyan, J.~K{\c{a}}kol, A.~Leiderman, {\em On topological groups with a small base and metrizability}, Fundam. Math.
 {\bf 229} (2015), 129--158.

\bibitem{GM16} P.~Gartside, A.~Mamatelashvili, {\em The Tukey order on compact subsets of
separable metric spaces}, J. Symb. Log. {\bf 81}:1 (2016), 181--200.

%\bibitem{GartMor} P.~Gartside, J.~Morgan, {\em The strong Pytkeev property in function spaces}, in preparation.


 %\bibitem{GabKak_1}
%S.~S.~Gabriyelyan, J.~K{\c{a}}kol,
%{\em Metrization conditions for topological vector spaces with Baire type properties,}Topol. Appl. {\bf 173} (2014), 135--141.


 %\bibitem{GabMor}
%S.~S.~Gabriyelyan, S.~A.~Morris, {\em Free topological vector spaces}, preprint
%(http://arxiv.org/abs/1604.04005).

  \bibitem{Ginsburg}
J.~ Ginsburg, {\em The metrizability of spaces whose diagonals have a countable base}, Canad. Math. Bull. {\bf 20} (1997), 513--514.


 \bibitem{Grue} G.~Gruenhage, {\em Generalized metric spaces}, in: \textit{Handbook of Set Theoretic Topology} (eds. K. Kunen and J. Vaugan), North Holland, Amsterdam: 1984, 423--501.

\bibitem{Gul} A.~Gul'ko, {\em Rectifiable spaces}, Topology Appl. {\bf 68}:2 (1996), 107--112.

\bibitem{Gul77} S.P.~Gul'ko, {\em  Properties of sets that lie in $\Sigma$-products}, Dokl. Akad. Nauk SSSR. {\bf 237}:3 (1977) 505--508 (in Russian).

\bibitem{Gul78} S.P.~Gul'ko, {\em Properties of some function spaces}, Dokl. Akad. Nauk SSSR. {\bf 243}:4 (1978) 839--842 (in Russian).

\bibitem{Gul98} S.P.~Gul'ko, {\em Semilattice of retractions and the properties of continuous function spaces of partial maps}, Quad. Mat. {\bf 3} (1998) 93--155.

\bibitem{Halb} L.~Halbeisen, {\em Combinatorial set theory. With a gentle introduction to forcing}, Springer Monographs in Mathematics. Springer, London, 2012.

\bibitem{Hodel} R.E.~Hodel, {\em Arhangelʹski\u\i's solution to Alexandroff's problem: a survey}, Topology Appl. {\bf 153}:13 (2006), 2199--2217.

%\bibitem{EG} R.E.~Hodel, {\em Modern Metrization Theorems}, in: Encyclopedia of General Topology\newline (K.P.Hart, J.Nagata, J.E.Vaughan eds.), Elsevier, (2004), 242--246.

\bibitem{HM} K.~Hofmann, J.~Martin, {\em Topological left-loops}, Topology Proc. {\bf 39} (2012), 185--194.

\bibitem{HS} K.~Hofmann, K.~Strambach, {\em Topological and analytic loops}, in: Quasigroups and Loops: Theory and Applications, 205--262,
Sigma Ser. Pure Math., 8, Heldermann, Berlin, 1990.

\bibitem{H11} M.~Hru\v s\'ak, {\em Combinatorics of filters and ideals}, Set theory and its applications, 29--69, Contemp. Math., 533, Amer. Math. Soc., Providence, RI, 2011.

\bibitem{Hru} M.~Hru\v s\'ak, {\em Almost disjoint families and topology}, Recent progress in general topology. III, 601--638, Atlantis Press, Paris, 2014.

\bibitem{Jak} N.N.~Jakovlev, {\em On the theory of $o$-metrizable spaces}, Dokl. Akad. Nauk SSSR, {\bf 229}:6 (1976) 1330--1331; English transl. in: Soviet Math. Dokl. {\bf 17}:4 (1976) 1217--1219.

 \bibitem{kak}
J.~K\c akol, W.~Kubi\'s, M.~Lopez-Pellicer, {\em Descriptive Topology in Selected  Topics of Functional Analysis}, Developments in Mathematics, Springer, 2011.

 %\bibitem{KocSch} Lj.~D.~R.~Ko\u{c}inac and M.~Scheepers, {\em Function spaces and a property of Reznichenko}, Topol. Appl., {\bf 123} (2002), 135--143.

\bibitem{Ke} A.~Kechris, {\em Classical Descriptive Set Theory}, Springer-Verlag, New York, 1995.

\bibitem{Kunen}
K.~Kunen, {\em Set Theory: An Introduction to Independence Proofs}, North-Holland, 1980.

 \bibitem{LPT} A.~Leiderman, V.~Pestov, A.~Tomita, {\em On topological groups admitting a base at identity indexed with $\w^\w$}, Fund. Math. (accepted) (http://arxiv.org/abs/1511.07062).

 \bibitem{LRZ} F.~Lin, A.~Ravsky, J.~Zhang, {\em Countable tightness and $\mathfrak G$-bases on free topological groups}, preprint\newline  (http://arxiv.org/abs/1512.07515).

\bibitem{LT} G.~Luk\'acs, B.~Tsaban, {\em Cofinal types of abelian topological groups: coproducts and free groups}, IVth Workshop on Coverings, Selections, and Games in Topology (June 25-30, 2012),  Caserta, Italy\newline (http://u.cs.biu.ac.il/~tsaban/spmc12/Section2/Lukacs.pdf).

\bibitem{LV99} A.~Louveau, B.~Velickovic, {\em Analytic ideals and cofinal types}, Ann. Pure Appl. Logic {\bf 99}:1-3 (1999) 171--195.

\bibitem{Mama} A.~Mamatelashvili, {\em Tukey order on sets of compact subsets of topological spaces}, Ph.D. Thesis, University of Pittsburgh, 2014. (http://d-scholarship.pitt.edu/21920/).

%\bibitem{McN} R.~McCoy, I.~Ntantu, {\em Topological properties of spaces of continuous functions},  Springer, Berlin, 2006.

\bibitem{Mi66} E.~Michael, {\em $\aleph_0$-spaces}, J. Math. Mech. {\bf 15} (1966) 983--1002.

\bibitem{Monk} D.~Monk, {\em On general boundedness and dominating cardinals}, Notre Dame J. Formal Logic {\bf 45}:3 (2004), 129--146

%\bibitem{Morris_functor} S.~A.~Morris, {\em Varieties of topological groups and left adjoint functors},
%J. Austr. Math. Soc. {\bf 16} (1973), 220--227.

% \bibitem{nummela} E.~C.~Nummela, {\em Uniform free topological groups and Samuel compactifications},Topol. Appl. {\bf 13} (1982), 77--83.

\bibitem{Moore} J.T.~Moore, {\em The proper forcing axiom}, Proc. Intern. Congress of Math. Volume II, 3--29, Hindustan Book Agency, New Delhi, 2010.

\bibitem{Mrowka1} S.~Mr\'owka, {\em On completely regular spaces}, Fund. Math. {\bf 41} (1954) 105--106.

\bibitem{Mrowka2} S.~Mr\'owka, {\em Some set-theoretic constructions in topology}, Fund. Math. {\bf 94}:2 (1977) 83--92.

%\bibitem{Mysior} A.~Mysior, {\em A regular space which is not completely regular}, Proc. Amer. Math. Soc. {\bf 81}:4 (1981) 652--653.

%\bibitem{Okun} O.~Okunev, {\em On analyticity in cosmic spaces}, Comment. Math. Univ. Carolin. {\bf 34}:1 (1993), 185--190.

\bibitem{O'M} P.~O'Meara, {\em On paracompactness in function spaces with the compact-open topology}, Proc. Amer. Math. Soc. {\bf 29} (1971) 183--189.

\bibitem{MSak} M.~Sakai, {\em Function spaces with a countable $cs\sp *$-network at a point}, Topology Appl. {\bf 156}:1 (2008), 117--123.

 %\bibitem{P} V.~G.~Pestov, {\em Neighbourhoods of identity in free topological groups}, Vestnik Moskov. Univ. Ser. I Mat. Mekh. 1985, no. 3, 8--10, 101.  (Russain); English translation: Moscow Univ. Math. Bull. {\bf 40} (1985), no. 3, 8--12.

% \bibitem{Rai} D.~A.~Raikov, {\em Free locally convex spaces for uniform spaces}, Mat. Sb. (N.S.) {\bf 63(105)} (1964), 582--590.

% \bibitem{RD} W.~Roelcke and S.~Dierolf, {\em Uniform Structures on Topological Groups and Their Quotients}, McGraw-Hill, 1981.

%\bibitem{Rob} Robinson

% \bibitem{Rolewicz}S.~Rolewicz, {\em Metric linear spaces}, PWN, Warszawa, 1972.

% \bibitem{Sipa1}Ol'ga Sipacheva, {\em The topology of a free topological group}. Fundam. Prikl. Mat.  {\bf 9}  (2003),  no. 2, 99--204,
%  (Russain); English translation: J. Math. Sci. (N. Y.)  {\bf 131}:4  (2005), 5765--5838.

% \bibitem{Sipa2} Ol'ga Sipacheva, {\em Free Boolean topological groups}, Axioms, {\bf 4} (2015), 492--517.

\bibitem{Sol15} S.~Solecki, {\em Tukey reduction among analytic directed orders},  Zb. Rad. (Beogr.) Selected topics in combinatorial analysis, {\bf 17}({\bf 25}) (2015), 209--220.

\bibitem{ST} S.~Solecki, S.~Todorcevic, {\em Cofinal types of topological directed orders}, Ann. Inst. Fourier (Grenoble) {\bf 54}:6 (2004), 1877--1911.

\bibitem{Tal80} M.~Talagrand, {\em Compacts de fonctions mesurables et filtres non mesurables}, Studia Math. {\bf 67}:1 (1980), 13--43.

\bibitem{Tka}  M.G.~Tkachenko, {\em Chains and cardinals}, Dokl. Akad. Nauk SSSR {\bf 239} (1978), 546--549.

 %\bibitem{Tkach}M.~G.~Tkachenko, {\em On completeness of free abelian topological groups}, Soviet Math. Dokl. {\bf 27} (1983), 341--345.

 \bibitem{Tk}
V.V.~ Tkachuk, {\em A space $C_p(X)$ is dominated by irrationals if and only if it is $K$-analytic}, Acta Math. Hungar. {\bf 107} (2005), 253--265.

\bibitem{Todo} S.~Todor\v cevi\'c, {\em Partition problems in topology}, American Mathematical Society, Providence, RI, 1989.



\bibitem{Tod85} S.~Todorcevic, {\em Directed sets and cofinal types}, Trans. Amer. Math. Soc. {\bf 290}:2 (1985) 711--723.

\bibitem{Tuk} J.W.~Tukey, {\em Convergence and uniformity in topology}, Princeton University Press, 1940.



% \bibitem{Usp}  V.~V.~Uspenski\u{\i}, {\em Free topological groups of metrizable spaces},  Math. USSR-Izv. {\bf 37} (1991), no. 3,  657--680.


\bibitem{Vaug} J.~Vaughan, {\em Small uncountable cardinals and topology}, in: Open problems in topology, 195--218, North-Holland, Amsterdam, 1990.



\end{thebibliography}
\end{document}